\documentclass[a4paper,11pt,reqno,noindent]{amsart}
\usepackage[centertags]{amsmath}
\usepackage{mathtools} 
\usepackage{amsfonts,amssymb,amsthm,dsfont,cases,amscd,esint,enumerate}
\usepackage[T1]{fontenc}
\usepackage[english,frenchb]{babel}
\usepackage[applemac]{inputenc}
\usepackage{newlfont}
\usepackage{color}
\usepackage[body={15cm,21.5cm},centering]{geometry} 
\usepackage{fancyhdr}
\usepackage{enumerate}
\usepackage{caption}
\usepackage{subcaption}
\usepackage{tikz}

\theoremstyle{plain}
\newtheorem{theor10}{Theorem}
\newenvironment{theor1}
  {\pushQED{\qed}\begin{theor10}}
  {\popQED\end{theor10}}
\newtheorem{cor10}{Corollary}

\newtheorem{prop10}{Proposition}

\newtheorem{lem0}{Lemma}[section]
\newenvironment{lem}
  {\pushQED{\qed}\begin{lem0}}
  {\popQED\end{lem0}}
\theoremstyle{plain}
\newtheorem{theor0}[lem0]{Theorem}

\newtheorem{prop0}[lem0]{Proposition}
\newenvironment{prop}
  {\pushQED{\qed}\begin{prop0}}
  {\popQED\end{prop0}}
\newtheorem{cor0}[lem0]{Corollary}
\newenvironment{cor}
  {\pushQED{\qed}\begin{cor0}}
  {\popQED\end{cor0}}
\theoremstyle{definition}
\newtheorem{defin0}[lem0]{Definition}

\newtheorem{rems0}[lem0]{Remarks}

\newtheorem{rem0}[lem0]{Remark}
\newenvironment{rem}
  {\pushQED{\qed}\begin{rem0}}
  {\popQED\end{rem0}}
\newtheorem{ex0}[lem0]{Example}

\newtheorem{exs0}[lem0]{Examples}

\newtheorem{notation0}[lem0]{Notation}

\newtheorem{heur0}[lem0]{Heuristics}

\newtheorem{ass}[lem0]{Assumption}

\newenvironment{as}
  {\pushQED{\qed}\begin{ass}}
  {\popQED\end{ass}}

\numberwithin{equation}{section}
\mathchardef\emptyset="001F

\newcommand{\e}{\varepsilon}
\newcommand{\Log}{|\!\log\e|}

\newcommand{\vi}{\operatorname{v}}
\newcommand{\Vi}{\operatorname{V}}

\newcommand{\pr}{\operatorname{p}}
\newcommand{\pre}{\operatorname{p}}
\newcommand{\md}{\operatorname{m}}
\newcommand{\dd}{\operatorname{d}}
\newcommand{\Hd}{\operatorname{H}}
\newcommand{\Ed}{\operatorname{E}}
\newcommand{\GLu}{$(\operatorname{GL}_1)$}
\newcommand{\GLd}{$(\operatorname{GL}_2)$}
\newcommand{\GLt}{$(\operatorname{GL}_3)$}
\newcommand{\GP}{$(\operatorname{GP})$}
\newcommand{\GLup}{$(\operatorname{GL}_1')$}
\newcommand{\GLdp}{$(\operatorname{GL}_2')$}
\newcommand{\GLtp}{$(\operatorname{GL}_3')$}
\newcommand{\GLqp}{$(\operatorname{GL}_4')$}

\newcommand{\R}{\mathbb R}
\newcommand{\Z}{\mathbb Z}
\newcommand{\C}{\mathbb C}
\newcommand{\T}{\mathbb T}
\newcommand{\Jb}{\mathbb J}
\newcommand{\B}{\mathcal B}

\newcommand{\Pc}{\mathcal P}
\newcommand{\Ec}{\mathcal E}

\newcommand{\Tc}{\mathcal T}
\newcommand{\D}{\mathcal D}

\newcommand{\M}{\mathcal M}
\newcommand{\Tr}{\operatorname{Tr}}
\newcommand{\Sb}{\mathbb S}
\newcommand{\Id}{\operatorname{Id}}
\newcommand{\osc}{\operatorname{osc}}
\newcommand{\uloc}{{\operatorname{uloc}}}
\newcommand{\Div}{\operatorname{div}\,}
\newcommand{\loc}{{\operatorname{loc}}}
\newcommand{\curl}{{\operatorname{curl}\,}}
\newcommand{\per}{{\operatorname{per}}}
\newcommand{\Ld}{\operatorname{L}}
\newcommand{\Ex}{\operatorname{ex}}

\newcommand{\cvf}[1]{\mathrel{\mathop{\xrightharpoonup{#1}}}}
\newcommand{\step}[1]{\noindent \textit{Step} #1.}
\newcommand{\substep}[1]{\noindent \textit{Substep} #1.}
\def\nab{\nabla}
\def\({\left(}
\def\){\right)}

\numberwithin{equation}{section}
\usepackage{color}
\usepackage[colorlinks,citecolor=black,urlcolor=black]{hyperref}

\title[Mean-field dynamics for Ginzburg-Landau vortices]{Mean-field dynamics for Ginzburg-Landau vortices\\with pinning and forcing}
\author{Mitia Duerinckx and Sylvia Serfaty}
\date{}

\begin{document}
\selectlanguage{english}
\maketitle

\begin{abstract}
We consider the time-dependent 2D Ginzburg-Landau equation in the whole plane with terms modeling impurities and applied currents. The Ginzburg-Landau vortices are then subjected to three forces: their mutual repulsive Coulomb-like interaction, the applied current pushing them in a fixed direction, and the pinning force attracting them towards the impurities. The competition between the three is expected to lead to
complicated glassy effects.
 We rigorously study the limit in which the number of vortices $N_\e$ blows up as the inverse Ginzburg-Landau parameter $\e$ goes to $0$, and we derive via a modulated energy method fluid-like mean-field evolution equations. These results hold for parabolic, conservative, and mixed-flow dynamics in appropriate regimes of $N_\e\uparrow\infty$.
Finally, we briefly discuss some natural homogenization questions raised by this study.
\end{abstract}

\setcounter{tocdepth}{1}
\tableofcontents

\section{Introduction}

\subsection{General overview}

Superconductors are materials that lose their resistivity at sufficiently low temperature (or low pressure), which allows them to carry electric currents without energy dissipation.
Another important property of these materials is the so-called Meissner effect: (moderate) external magnetic fields are completely expelled from the sample. If the external field is much too strong, the superconducting material returns to a normal state. In the case of a type-II superconductor, an intermediate regime is possible between two critical values of the external field: the material is then in a mixed state, allowing a partial penetration of the external field through ``vortex filaments''.
This mixed state has however a major drawback: when an electric current is applied, it flows through the sample, inducing a Lorentz-like force that sets the vortices in motion, and hence, since vortices are flux filaments, their movement generates an electric field in the direction of the electric current, which dissipates energy and destroys the superconductivity property.

While ordinary superconductors need extreme cooling to achieve superconductivity, the discovery of high-temperature superconductors from the 1980s onwards has given an major boost to technological applications, as the critical temperature of such materials is now reached with only liquid nitrogen.
These high-temperature superconductors happen to be in practice strongly of type~II and, as such, they show vortices for a very wide range of values of the applied magnetic field.
Most technological applications of superconductors therefore occur in this mixed state, and it is crucial to design ways to prevent vortices from moving in order to recover the desired property of dissipation-free current flow. For that purpose a common attempt consists in introducing normal impurities in the material, which are meant to destroy superconductivity locally and therefore ``pin down'' the vortices to their locations if the applied current is not too strong.

With these applications in mind, there is a strong interest in the physics community in understanding the precise effect of such impurities (which are typically randomly scattered around the sample) on the statics and dynamics of vortices.
Of particular interest is the critical applied current needed to depin the vortices from their pinning sites, as well as the slow motion of vortices --- named {\it creep} --- in the disordered sample when the applied current has a small intensity and thermal or quantum effects are taken into consideration.
The competition between vortex interactions and disorder actually leads to complicated glassy effects that are still largely not understood and have attracted much attention in the theoretical physics community these last decades~\cite{BFGLV-94,Giamarchi-Bhatta-02,Giamarchi-09}.
The richness of the dynamic phase diagram in terms of the different tunable parameters is particularly striking~\cite{LLWXL-03,Reichhardt-17}.
In the sequel, we study the collective dynamics of many vortices in a (2D section of a) type-II superconductor with applied current and impurities, and we wish to establish in various regimes the correct mean-field equations describing the vortex matter. We may view this work as a first step to identify proper questions towards a mathematical understanding of the glassy properties of such systems (cf.~Section~\ref{sec:homog} for further comments and questions).

\medskip
The phenomenology of superconductivity is accurately described by the (mesoscopic) Ginzburg-Landau theory. Restricting ourselves to a 2D section of a superconducting material, we rather consider the simpler 2D Ginzburg-Landau model, and vortex filaments are replaced by ``point vortices''. We refer e.g.\@ to~\cite{Tinkham,Tilley} for further reference on these models, and to~\cite{SS-book} for a mathematical introduction. The (mesoscopic) impurities in the material are usually modeled by introducing a pinning weight $a : \R^2 \to [0, 1]$, which locally lowers the energy penalty associated with the vortices~\cite{Larkin-70,CDG-95} (see also~\cite{Chapman-00}): regions with $a=1$ correspond to the pure superconducting material, while regions with $a\approx0$ define the normal impurities. In the time-dependent 2D Ginzburg-Landau equation (which is a gradient flow for the corresponding energy), the pinning weight and the applied electric current appear as follows,
\begin{equation}\label{eq:Gorkov-Eliashberg0}
\begin{cases}
\partial_tw_\e=\triangle w_\e+\frac{w_\e}{\e^2}(a-|w_\e|^2),&\text{in $\R^+\times\Omega$},\\
n\cdot\nabla w_\e=iw_\e \Log n\cdot J_{\Ex},&\text{on $\R^+\times\partial\Omega$},\\
w_\e|_{t=0}=w_\e^\circ,
\end{cases}
\end{equation}
where $\Omega$ is a domain of $\R^2$ and $n$ is the outer unit normal on $\partial\Omega$, where $w_\e:\R^+\times\Omega\to\C$ is the complex-valued order parameter describing superconductivity, where $\Log J_{\text{ex}}:\partial\Omega\to\R^2$ is the (critically-scaled) applied electric current, and where $\e>0$ is the inverse Ginzburg-Landau parameter (a characteristic of the material, which is typically very small for real-life superconductors).
More precisely, as first derived by Schmid~\cite{Schmid-66} and by Gor'kov and Eliashberg~\cite{GE-68}, the true Ginzburg-Landau model should be further coupled to electromagnetism, replacing the above equation by a suitable version with magnetic gauge, and in particular the imposed electric current $J_{\Ex}$ should rather appear as a boundary condition for the electric and magnetic fields.\footnote{Note that in this simplified model~\eqref{eq:Gorkov-Eliashberg0} the number of vortices has to be imposed artificially through the boundary condition, while in the model with gauge it is implicitly determined by the value of the external magnetic field.} Since the gauge does not introduce any significant mathematical difficulty, we focus on the above simplified form of the model, and only briefly comment on the case with gauge in Section~\ref{sec:gauge}.
The order parameter $w_\e$ has the following meaning: values $|w_\e|=1$ and $|w_\e|=0$ correspond to superconducting and normal phases, respectively, and the vortices are the zeroes of $w_\e$ with non-zero topological degree. Vortices typically have a core of size of order $\e$, hence they become point-like in the asymptotic limit $\e\downarrow0$. Moreover, a vortex of degree $d$ at a point $x$ carries a (self-interaction) energy $\pi |d|a(x)\Log$, which varies with its location due to the pinning weight~$a$, and implies that vortices are indeed attracted to the minima of the weight, that is, to the normal impurities.

An important variant of this model~\eqref{eq:Gorkov-Eliashberg0} is the corresponding (conservative) Schrödinger flow, with $\partial_tw_\e$ replaced by $i\partial_tw_\e$. This coincides with the so-called Gross-Pitaevskii equation, which is an example of a nonlinear Schrödinger equation and serves as a model for Bose-Einstein condensates and superfluidity~\cite{Aftalion-06,Rougerie-these}, as well as for nonlinear optics~\cite{Arecchi-91}.
As argued e.g.\@ in~\cite{Aranson-Kramer}, there is also physical interest in the ``mixed-flow'' (or ``complex'') Ginzburg-Landau equation, which is a mix between the Ginzburg-Landau and Gross-Pitaevskii equations. Instead of~\eqref{eq:Gorkov-Eliashberg0} we thus turn to the following more general equation, for any $\alpha\ge0$, $\beta\in\R$, $\alpha^2+\beta^2=1$,
\begin{equation}\label{eq:Gorkov-Eliashberg-mixed-flow}
\begin{cases}
(\alpha+i\Log\beta)\partial_tw_\e=\triangle w_\e+\frac{w_\e}{\e^2}(a-|w_\e|^2),&\text{in $\R^+\times\Omega$},\\
n\cdot\nabla w_\e=iw_\e\Log n\cdot J_{\Ex},&\text{on $\R^+\times\partial\Omega$},\\
w_\e|_{t=0}=w_\e^\circ,
\end{cases}
\end{equation}
which allows to consider by the same token both the parabolic or Ginzburg-Landau case ($\alpha=1$, $\beta=0$) and the conservative or Gross-Pitaevskii case ($\alpha=0$, $\beta=1$).

In this context, including both pinning and applied current, the problems that naturally arise are
\begin{itemize}
\smallskip\item to derive from equation~\eqref{eq:Gorkov-Eliashberg-mixed-flow} a simpler discrete problem for the evolution of a fixed number $N$ of point vortices in the asymptotic limit $\e\downarrow0$;
\smallskip\item to derive a mean-field equation describing the evolution of a large number of vortices, either by taking the limit $N \uparrow \infty$ in the discrete problem, or preferably by taking the limit directly in~\eqref{eq:Gorkov-Eliashberg-mixed-flow} when the number of vortices $N_\e$ blows up as $\e\downarrow0$, thus investigating the commutation of the limits $\e\downarrow0$ and $N\uparrow\infty$;
\smallskip\item to derive effective equations in the regime when the impurities are scattered at a small scale, that is, when the pinning weight $a$ oscillates rapidly, by starting either from the mean-field equation, from the discrete problem, or preferably from~\eqref{eq:Gorkov-Eliashberg-mixed-flow}.
\end{itemize}
\smallskip

\noindent
As recalled below, the first question has already been fully answered. In this work, we focus on the second question, which is to derive a mean-field equation for the vortex liquid directly from the mesoscopic model~\eqref{eq:Gorkov-Eliashberg-mixed-flow}.
This naturally leads us to the third question, which however remains largely open: in Section~\ref{sec:homog} we state various conjectures and give a few preliminary results.

\medskip
Let us start by recalling the behavior of a fixed number $N$ of vortices in the asymptotic regime $\e\downarrow0$. A good understanding was achieved in the physics community since the 1990s~\cite{Neu,Dorsey,Peres-Rubinstein,Chapman-Richardson-97}, and various rigorous studies became available shortly after in the parabolic case~\cite{Lin-96a,Lin-96b,Jerrard-Soner-98,Jian-Song-01,SS-04}, in the conservative case~\cite{Colliander-Jerrard-98,Lin-Xin-99,Jerrard-Smets-15, Kurzke-Marzuola-Spirn-15}, as well as in the mixed-flow case~\cite{Tice-10,S-Tice-11}.
As seen there, vortices are subjected to three forces:
\begin{itemize}
\smallskip\item their mutual repulsive Coulomb (logarithmic) interaction;
\smallskip\item the Lorentz-like force $F$ due to the applied current of intensity $J_{\Ex}$;
\smallskip\item the pinning force, equal to $-\nabla h$ in terms of the so-called pinning potential $h:=\log a$ defined by the pinning weight~$a$.
\end{itemize}
\smallskip

\noindent
Neglecting boundary effects, and assuming that all vortices have the same degree $+1$, the effective vortex dynamics is then given by a system of ODEs of the form
\begingroup\allowdisplaybreaks
\begin{gather}\label{ode1}
(\alpha+\Jb\beta)\partial_t x_i = -N^{-1}\nabla_{x_i} W_N(x_1,\ldots,x_N)-\nabla h(x_i) +F(x_i),\qquad 1\le i\le N,\\
W_N(x_1,\ldots,x_N):=-\sum_{i\ne j}^N\log|x_i-x_j|,\nonumber
\end{gather}
\endgroup
where the $x_i$'s are the macroscopic vortex trajectories, and where $\Jb$ denotes the rotation of vectors by angle $\frac\pi2$ in the plane.
The pinning and applied current intensities are parameters which can be tuned,  leading to regimes in which one or two forces dominate over the others, or all are of the same order.
In~\cite{Tice-10}, in the parabolic case, no pinning force is considered and the  regimes  treated lead to the applied force being of the same order as the interaction. In~\cite{S-Tice-11} the pinning and applied forces are chosen to be of the same order, and both dominate the interaction.
In~\cite{Kurzke-Marzuola-Spirn-15}, in the conservative case, the critical scaling is considered, that is, with all forces being of the same order.

In this work, we rather focus on the situation when the number $N_\e$ of vortices in~\eqref{eq:Gorkov-Eliashberg-mixed-flow} is not fixed but depends on $\e$ and blows up as $\e \downarrow 0$, which is a physically more realistic situation in many regimes of applied fields and currents.
We then wish to describe the evolution of the density of the corresponding vortex liquid. In dilute regimes (that is, when $N_\e$ does not blow up too quickly with respect to $\e$), the correct limiting equation is naturally expected to coincide with the mean-field limit of the discrete vortex dynamics~\eqref{ode1} (cf.~\cite{D-15,Serfaty-18}), that is, the following nonlocal nonlinear continuity equation for the mean-field vorticity $\md$,
\begin{align}\label{eq:expect-mfl-ode1}
\partial_t\!\md=\Div\big((\alpha-\Jb\beta)(\nabla h-F-\nabla\triangle^{-1}\!\md)\md\big),
\end{align}
or alternatively, in terms of the mean-field supercurrent density $\vi$ (related to $\md$ via $\md=\curl\!\vi$),
\begin{align}\label{eq:expect-mfl-ode1-vi}
\partial_t\!\vi=\nabla\!\pre+(\alpha-\beta\Jb)(\nabla^\bot h-F^\bot-\vi)\curl\!\vi,\qquad \Div\!\vi=0.
\end{align}
Note that in the conservative case ($\alpha=0$, $\beta=1$) this equation becomes
\begin{align}\label{eq:expect-mfl-ode1-vi+}
\partial_t\!\vi=\nabla\!\pre+(\nabla h-F+\vi^\bot)\curl\!\vi,\qquad \Div\!\vi=0,
\end{align}
which is equivalent to the incompressible 2D Euler equation due to the identity $\vi^\bot\curl\!\vi=(\vi\cdot\nabla)\vi-\frac12\nabla|\!\vi\!|^2$, while the force $\nabla h-F$ plays the role of a background flow.
In the dissipative case $\alpha>0$, as first discovered  in~\cite{Serfaty-15}, the mean-field behavior in nondilute regimes changes drastically and rather leads to {\it compressible} equations. In other words, the limits $\e\downarrow0$ and $N\uparrow\infty$ do not always commute. A heuristic explanation of such behaviors is included in Section~\ref{sec:intro-MFL-GL-heur}.

In the case without pinning and applied current ($h\equiv0$, $F\equiv0$), such mean-field results have already been rigorously established in a number of settings:
\begin{itemize}
\smallskip\item In the conservative case ($\alpha=0$, $\beta=1$), Jerrard and Spirn~\cite{Jerrard-Spirn-15} have shown in the strongly dilute regime $1\ll N_\e \lesssim (\log \Log)^{1/2}$ that the vorticity converges to the solution of~\eqref{eq:expect-mfl-ode1} (which in that case coincides with the 2D Euler equation in vorticity form), while the second author has shown in~\cite{Serfaty-15} in the nondilute regime  $\Log \ll N_\e\ll  \e^{-1}$ that the supercurrent itself converges to the solution of the 2D Euler equation~\eqref{eq:expect-mfl-ode1-vi+}.
\smallskip\item In the parabolic case ($\alpha=1$, $\beta=0$), the convergence of the vorticity to the solution of~\eqref{eq:expect-mfl-ode1}, first formally derived by Chapman, Rubinstein, Schatzman, and E~\cite{CRS-96,WE-94}, has been rigorously established by Kurzke and Spirn~\cite{Kurzke-Spirn-14} in the strongly dilute regime $1\ll N_\e\le (\log \log \Log)^{1/4}$. Next, the second author has shown in~\cite{Serfaty-15} that in the whole moderately dilute regime $1\ll N_\e \ll \Log$ the supercurrent itself converges to the solution of~\eqref{eq:expect-mfl-ode1-vi}, but that in the critical regime $N_\e\simeq\Log$ it converges to a different {compressible} equation.
\end{itemize}

\smallskip
\noindent
In all the other regimes (that are, the moderately dilute regime $1\ll N_\e\lesssim\Log$ in the conservative case and the nondilute regime $\Log\ll N_\e\ll\e^{-1}$ in the parabolic case), justifying the mean-field limit remains an open question --- to the exception of the weakly nondilute regime $\Log\ll N_\e\ll\Log\log\Log$ in the parabolic case, which is further treated in the present work and leads to yet another compressible mean-field equation, thus answering a question raised in~\cite{Serfaty-15}.
All these results assume that the initial data are suitably ``well-prepared''.
Note that the delicate boundary issues are neglected in~\cite{Jerrard-Spirn-15} and~\cite{Serfaty-15}, where the Gross-Pitaevskii or Ginzburg-Landau equation is set for simplicity on the whole plane, while in~\cite{Kurzke-Spirn-14} Dirichlet boundary data on a bounded domain $\Omega$ are further considered.
The results in~\cite{Kurzke-Spirn-14} and~\cite{Jerrard-Spirn-15} rely on a direct method and a careful study of the vortex trajectories, while those in~\cite{Serfaty-15} are based on a ``modulated energy approach'' and rely on the regularity and stability properties of the mean-field equations.

\medskip
The main goal of the present work is to adapt the modulated energy approach of~\cite{Serfaty-15} to the setting with pinning and applied current, thus extending the results of~\cite{Tice-10,S-Tice-11,Kurzke-Marzuola-Spirn-15} to the case with $N_\e \gg 1$ vortices --- in the whole plane for simplicity. The derivation bears several complications compared to the situation in~\cite{Serfaty-15}, in particular due to the lack of sufficient decay at infinity of the various quantities, and also to the fact that the self-interaction energy of each vortex now varies with its location due to the pinning weight.
Next to the parabolic and conservative cases, we also consider the mixed-flow case. We establish the convergence to suitable fluid-like mean-field evolution equations, which in the simplest case take the form~\eqref{eq:expect-mfl-ode1}--\eqref{eq:expect-mfl-ode1-vi} but differ in some regimes, and for which global well-posedness is discussed in the companion article~\cite{D-16}.
Some of these equations are new in the literature, while some others already appeared in the context of 2D fluid dynamics: in the conservative case, for instance, the obtained mean-field equation coincides with the so-called lake equation~\cite{CHL-1,CHL-2} for shallow water flows.
As emphasized above, different regimes for the intensity of the pinning and applied current lead to different limiting equations, and we include a discussion of all of them.

\subsection*{Acknowledgements}
The work of MD is supported by F.R.S.-FNRS (Belgian National Fund for Scientific Research) through a Research Fellowship.
The authors thank Anne-Laure Dalibard, Jean-Pierre Eckmann, and Thierry Giamarchi  for stimulating discussions.

\subsection*{Notation}
Throughout, $C$ denotes various positive constants which depend on controlled quantities and may change from line to line, but do not depend on the small parameter $\e$. We write $\lesssim$ and $\gtrsim$ for $\le$ and $\ge$ up to such a multiplicative constant $C$. We write $a\simeq b$ if both $a\lesssim b$ and $a\gtrsim b$ hold. Given sequences $(a_\e)_\e,(b_\e)_\e\subset\R$, we write $a_\e\ll b_\e$ (or $b_\e\gg a_\e$) if $a_\e/b_\e$ converges to $0$ as the parameter $\e$ goes to $0$. We also write $a_\e\le O(b_\e)$ if $a_\e\lesssim b_\e$, and $a_\e\le o(b_\e)$ if $a_\e\ll b_\e$. We add a subscript $t$ to indicate the further dependence of constants on an upper bound on time $t$, while additional subscripts indicate the dependence on other parameters. A superscript $t$ to a function indicates that this function is evaluated at time~$t$.
For a vector field $G=(G_1,G_2)$ on $\R^2$, we set $G^\bot=(-G_2,G_1)$, $\curl G=\partial_1G_2-\partial_2G_1$, and $\Div G=\partial_1G_1+\partial_2G_2$. We write $\Jb:\R^2\to\R^2$ for the rotation of vectors by angle $\frac\pi2$ in the plane, hence $\Jb G=G^\bot$. We denote by $B(x,r)$ the ball of radius $r$ centered at $x$ in $\R^2$, and we set $B_r:=B(0,r)$ and $B(x):=B(x,1)$.
We let $Q:=[-\frac12,\frac12)^2$ denote the unit square, frequently identified with the $2$-torus $\T^2$.
We write $a\wedge b:=\min\{a,b\}$ and $a\vee b:=\max\{a,b\}$ for $a,b\in\R$. We denote by $\Ld^p_\uloc(\R^2)$ the Banach space of functions that are uniformly locally $\Ld^p$-integrable on $\R^2$, with norm
\[\|f\|_{\Ld^p_\uloc}:=\sup_x\|f\|_{\Ld^p(B(x))},\]
and we similarly define the Sobolev spaces $W^{k,p}_\uloc(\R^2)$. Given a Banach space $X$ and $t>0$, we use the notation $\|\cdot\|_{\Ld^p_tX}$ for the usual norm in $\Ld^p([0,t];X)$.

\subsection{Main results}\label{chap:main}
We first give a precise formulation of the problem under consideration, present our modulated energy approach and underline the main new difficulties, state precise assumptions, discuss the various regimes that our approach allows to consider, and then state our main mean-field results.

\subsubsection{Precise setting}
Since the presence of the boundary creates mathematical difficulties which we do not know how to overcome (due to the possible entrance and exit of vortices), we modify the mesoscopic model~\eqref{eq:Gorkov-Eliashberg-mixed-flow} and consider a suitable version on the whole plane with boundary conditions ``at infinity''. As in~\cite{Tice-10,S-Tice-11}, the boundary conditions can be changed into a bulk force term by a suitable change of phase in the order parameter $w_\e$.  
Also dividing $w_\e$ by the expected density $\sqrt{a}$,
we arrive at the following equation for the modified order parameter $u_\e$,
 \begin{align}\label{eq:GL-1}
\begin{cases}\lambda_\e(\alpha+i\Log \beta)\partial_tu_\e=\triangle u_\e+\frac{a}{\e^2}u_\e(1-|u_\e|^2)\\
\hspace{6cm}+\nabla h\cdot\nabla u_\e+i\Log F^\bot\cdot\nabla u_\e+fu_\e,\\
u_\e|_{t=0}=u_\e^\circ,\end{cases}
\end{align}
with $h:= \log a$, $f:\R^2 \to \R$, and $F:\R^2\to\R^2$, where $F$ is an effective applied force corresponding to the Lorentz-like force generated by the applied current. The parameter $\lambda_\e$ is an appropriate time rescaling needed to obtain a nontrivial limiting dynamics. Within the derivation of~\eqref{eq:GL-1} from~\eqref{eq:Gorkov-Eliashberg-mixed-flow}, the zeroth-order term $f$ takes on the following explicit form (although this is largely unimportant, and the scaling in the corresponding bounds~\eqref{eq:scalingshFf}--\eqref{eq:scalingshFfdec} below may also be substantially relaxed),
\begin{align}\label{eq:formf}
f:=\frac{\triangle\sqrt a}{\sqrt a}-\frac14\Log ^2|F|^2.
\end{align}
The derivation of the modified model~\eqref{eq:GL-1} from equation~\eqref{eq:Gorkov-Eliashberg-mixed-flow} is postponed to Section~\ref{disc-model}, while the global well-posedness of~\eqref{eq:GL-1} is discussed in Section~\ref{chap:wellposed}.
For simplicity we assume that the pinning weight satisfies
\begin{equation}\label{eq:nondegencond}
\frac1C\le a(x)\le1,\qquad\text{for all $x$,}\end{equation}
which avoids degenerate situations: physically one would like to consider a pinning weight $a$ that may vanish, representing true normal inclusions~\cite{CDG-95}, but this is much more delicate mathematically (cf.\@ e.g.~\cite{Andre-Bauman-Phillips-03}).
Setting $F\equiv 0$, $a\equiv 1$, $h\equiv 0$, and $f\equiv 0$, we naturally retrieve the model without pinning and applied current as studied e.g.\@ in~\cite{Kurzke-Spirn-14,Jerrard-Spirn-15,Serfaty-15}, and our results are thus indeed generalizations of those in~\cite{Kurzke-Spirn-14,Serfaty-15}.

Given solutions of the mesoscopic model~\eqref{eq:GL-1}, we wish to establish the convergence of their {\it supercurrent}, defined by 
\begin{equation*}
j_\e:= \langle \nabla u_\e, iu_\e\rangle,
\end{equation*}
where $\langle \cdot, \cdot\rangle$ stands for the scalar product in $\C$ as identified with $\R^2$, that is, $\langle x,y\rangle=\Re(x\bar y)$ for $x,y\in\C$. 
The {\it vorticity} $\mu_\e$ is derived from the supercurrent via $\mu_\e:= \curl j_\e$.
Note that this indeed corresponds to the density of vortices, defined as zeros of $u_\e$ weighted by their degrees, in the sense that
\begin{equation}
\label{muejac}
\mu_\e ~\approx~ 2\pi \sum_i d_i \delta_{x_i},\qquad\text{as $\e \downarrow 0$},
\end{equation}
with $\{x_i\}_i$ the vortex locations and $\{d_i\}_i$ their degrees (this is made rigorous by the so-called Jacobian estimates, e.g.~\cite[Chapter~6]{SS-book}).
In this setting, we wish to show that the rescaled supercurrent $\frac1{N_\e}j_\e$ converges as $\e\downarrow 0$ to a vector field $\vi$ solving a limiting PDE, which as in~\cite{Serfaty-15} is assumed to be regular enough.
The limiting equations are fluid-like equations of the form~\eqref{eq:expect-mfl-ode1-vi}, where the incompressibility condition can however be lost when the density of vortices becomes too large.
Such equations are studied in detail in the companion article~\cite{D-16}, where solutions are shown in most cases to be global and indeed regular enough if the initial data is.
A formal derivation of these mean-field equations is included in Section~\ref{sec:intro-MFL-GL-heur}.

\subsubsection{Modulated energy approach}
In order to establish the convergence of the rescaled supercurrent,
we adapt the modulated energy approach used by the second author in~\cite{Serfaty-15}.
Modulated energy techniques originate in the relative entropy method first designed by DiPerna~\cite{DiPerna-79} and Dafermos~\cite{Dafermos-79,Dafermos-79b} to establish weak-strong stability principles for some hyperbolic systems. This method was later rediscovered by Yau~\cite{Yau-91} for the hydrodynamic limit of the Ginzburg-Landau lattice model, was introduced in kinetic theory by Golse~\cite{Golse-00} for the convergence of suitably scaled solutions of the Boltzmann equation towards solutions of the incompressible Euler equations (cf.\@ e.g.~\cite{StRaymond-09} for the many recent developments on the topic), and first took the form of a modulated {\it energy} method in the work by Brenier~\cite{Brenier-00} on the quasi-neutral limit of the Vlasov-Poisson system.
In the present situation, the method consists in defining a modulated energy, which in the case without pinning takes the form
\begin{equation}\label{mode}
\tilde\Ec_\e:=\int_{\R^2}\frac12\Big(|\nabla u_\e- i u_\e N_\e\!\vi\!|^2 + \frac{1}{2 \e^2} (1-|u_\e|^2)^2\Big),
\end{equation}
where $\vi$ denotes the (postulated) mean-field supercurrent density.
Note that, while the Ginzburg-Landau energy (that is, \eqref{mode} with $\vi= 0$) diverges for configurations $u_\e$ with nonzero degree at infinity,
$$0\ne\deg (u_\e):= \lim_{R\uparrow \infty} \int_{\partial B_R} \langle \nabla u_\e, iu_\e\rangle \cdot n^\bot,$$
the modulated energy may indeed converge (and does if $\vi$ has the correct circulation at infinity).
This modulated energy $\tilde\Ec_\e$ measures the squared distance between the supercurrent $j_\e=\langle\nabla u_\e,iu_\e\rangle$ and the postulated limit $N_\e\!\vi$, in a way that is well adapted to the energy structure.
In order to prove the desired convergence $\frac1{N_\e}j_\e\to\vi$, showing $\tilde\Ec_\e=o(N_\e^2)$ is then sufficient.
Under some regularity assumption on $\vi$, it was proved  in~\cite{Serfaty-15} that, thanks to the suitable limiting equation satisfied by $\vi$, the modulated energy $\tilde\Ec_\e$ satisfies a Grönwall relation, so that if it is initially of order $o(N_\e^2)$, it remains so, yielding the desired convergence $\frac1{N_\e}j_\e\to\vi$.
However, in regimes with $N_\e \lesssim\Log$, the modulated energy $\tilde\Ec_\e$ cannot be of order $o(N_\e^2)$, since each vortex of degree $d$ carries a self-interaction energy $\pi |d| \Log$. For that reason,
we need to renormalize the modulated energy $\tilde\Ec_\e$ by subtracting the (fixed) total self-interaction energy $\pi \sum_i|d_i| \Log$. More precisely, as we will work in a setting where the initial vortices have positive degrees, $\sum_i |d_i|=N_\e$, and as we expect that this remains the case at later times, we consider the modulated energy excess
\begin{align}\label{eq:excess-SS15}
\tilde\D_\e:=\tilde\Ec_\e-\pi N_\e \Log,
\end{align}
and establish a Gr\"onwall relation on this quantity.
The proof requires to use many  tools of vortex analysis developed over the years, cf.~\cite{SS-book}: lower bounds via the Jerrard-Sandier ball construction, Jacobian estimates, and product estimates.

In the case with pinning weight $a$, the modulated energy~\eqref{mode} should naturally be changed into a weighted one,
\begin{equation}\label{modew}
\int_{\R^2}\frac a2\( |\nabla u_\e-iu_\e N_\e\!\vi |^2 + \frac{a}{2\e^2} (1-|u_\e|^2)^2\).\end{equation}
This leads to several notable modifications:
\begin{itemize}
\smallskip\item
A vortex of degree $d$ at a point $x$ now carries a self-interaction energy $\pi |d|a(x)\Log$, which non-trivially depends on the vortex location~$x$. The total self-interaction energy that needs to be subtracted from the modulated energy~\eqref{modew} is thus no longer $\pi N_\e \Log$ but rather, in view of~\eqref{muejac},
$$\pi \sum_i d_i a(x_i) \Log \,\approx\, \frac\Log2 \int_{\R^2}a  \mu_\e.$$
\item
In some regimes of pinning and applied current, the solution $\vi$ of the limiting equation needs to be replaced in the modulated energy~\eqref{modew} by a suitable $\e$-dependent map~$\vi_\e$, which is separately shown to converge to~$\vi$. This amounts to including lower-order terms in the modulated energy.

\smallskip\item
If $\nabla h$, $F$, and $f$ in~\eqref{eq:GL-1} are bounded but not decaying at infinity (which is a natural setting in view of the typical example of a uniform applied current circulating through the sample),
then the modulated energy~\eqref{modew} does usually not remain finite along the flow, which forces us to truncate it at some scale.
In the conservative case, the decay of $\nabla h$, $F$, and $f$ is anyway needed to guarantee the well-posedness of the mesoscopic model~\eqref{eq:GL-1} (cf.~Section~\ref{chap:wellposed}), so that a truncation of~\eqref{modew} is no longer needed, but in that case, due to pinning, the pressure $\pre$ in the mean-field equation~\eqref{eq:expect-mfl-ode1-vi} for~$\vi$ is no longer square-integrable and another truncation argument is required.
\end{itemize}

\smallskip
\noindent
For these reasons, we are lead to considering the following truncated version of the modulated energy~\eqref{modew},
\begin{equation}\label{der2}\Ec_{\e,R}\,:=\,\int_{\R^2}\frac {a\chi_R}2\Big(|\nabla u_\e-iu_\e N_\e\!\vi_\e\!|^2+\frac{a}{2\e^2}(1-|u_\e|^2)^2\Big),\end{equation}
as well as the corresponding excess,
\begin{eqnarray}
\D_{\e,R}&:=&\Ec_{\e,R}-\frac{\Log }2\int_{\R^2} {a\chi_R}\mu_\e\nonumber\\
&=&\int_{\R^2}\frac {a\chi_R}2\Big(|\nabla u_\e-iu_\e N_\e\!\vi_\e\!|^2+\frac{a}{2\e^2}(1-|u_\e|^2)^2-\Log \mu_\e\Big),\label{der}
\end{eqnarray}
where for all $r>0$ we set  $\chi_r:=\chi(\cdot/r)$ for some fixed cut-off function $\chi\in C^\infty_c(\R^2;[0,1])$ with $\chi|_{B_1}=1$ and $\chi|_{\R^2\setminus B_2}=0$, and with $|\nabla\chi|\lesssim\chi^{1/2}(1-\chi)^{1/2}$.\footnote{Such a function $\chi$ is easily constructed by smoothly gluing the choices $\chi(x)= 1-\exp(-\frac1{(|x|-1)_+})$ for $|x|\sim1$ and $\chi(x)= \exp(-\frac1{(2-|x|)_+})$ for $|x|\sim2$. Since $\big|\nabla\big(1- \exp(-\frac1{(|x|-1)_+})\big)\big|\lesssim \sqrt{\exp(-\frac1{(|x|-1)_+})}$ and $\big|\nabla \exp(-\frac1{(2-|x|)_+})\big|\lesssim \sqrt{\exp(-\frac1{(2-|x|)_+})}$, this choice indeed satisfies the bound $|\nabla\chi|\lesssim\chi^{1/2}(1-\chi)^{1/2}$.}
In the sequel, all energy integrals are truncated as above with the cut-off function $\chi_R$, for some scale $R\gg1$ to be later suitably chosen as a function of $\e$.
We write $\Ec_\e:=\Ec_{\e,\infty}$ for the corresponding quantity without the cut-off $\chi_R$ in the definition (formally $R=\infty$), and also $\D_\e:=\sup_{R\ge1}\D_{\e,R}$.
Rather than the $\Ld^2$-norm restricted to the ball $B_R$ centered at the origin, our methods further allow to consider the uniform $\Ld^2_\loc$-norm at the scale $R$: setting $\chi_R^z:=\chi_R(\cdot-z)$, we define
\begin{gather}\Ec_{\e,R}^*:=\sup_z\Ec_{\e,R}^z~,\quad \Ec_{\e,R}^z:=\int_{\R^2}\frac {a\chi_R^z}2\Big(|\nabla u_\e-iu_\e N_\e\!\vi_\e\!|^2+\frac{a}{2\e^2}(1-|u_\e|^2)^2\Big),\label{der3}\\
\D_{\e,R}^*:=\sup_z\D_{\e,R}^z~,\quad\D_{\e,R}^z:=\Ec_{\e,R}^z-\frac\Log2\int_{\R^2} a\chi_R^z\mu_\e,\label{der4}
\end{gather}
where the suprema run over all lattice points $z\in R\Z^2$.

In this setting, the proof is split into two parts: first we show that $\frac1{N_\e}j_\e$ is close to a suitable $\vi_\e$ by means of a Grönwall argument on the modulated energy excess $\D_{\e,R}^*$, which requires some careful vortex analysis, and second we check that $\vi_\e$ converges to $\vi$, which is a soft consequence of the stability of the limiting equation.
In order to establish a Grönwall relation for $\D_{\e,R}^*$, in addition to the problems at infinity created by the non-decay of $\nabla h$ and $F$ that we wish to allow, the presence of the pinning weight introduces important new technical difficulties, as always in the analysis of Ginzburg-Landau.
We mention two of them (cf.~Section~\ref{sec:vortex} for detail):
\begin{itemize}
\smallskip\item
In this weighted setting, the fact that the self-interaction energy of a vortex depends on its location makes it more difficult to a priori control the total number of vortices, and requires {\it localized} estimates, in particular a localized version of the Jerrard-Sandier ball-construction lower bound~\cite{Sandier-98,Jerrard-99} with a very precise error estimate $o(N_\e^2)$.
The usual error in the lower bound is $O(N_\e|\!\log r|)$, where $r$ is the total radius of the balls, so that we need to take $r$ large enough (almost $O(1)$ when $N_\e$ diverges slowly), but here the pinning weight $a$ adds an important difficulty since it may vary significantly over the size of the balls of this construction, thus perturbing the lower bound itself. A particularly careful vortex analysis is therefore needed.
\smallskip\item
Due to truncations, the vortex analysis must further be refined to the setting of the infinite plane with no global energy control, hence no a priori finiteness assumption on the total number of vortices, which yields additional complications.
\end{itemize}

\subsubsection{Assumptions}
For the essential part of the proof, in the dissipative case ($\alpha>0$), it suffices to assume $h\in W^{2,\infty}(\R^2)$ and $F\in W^{1,\infty}(\R^2)^2$ (hence $f\in\Ld^\infty(\R^2)$ in view of~\eqref{eq:formf}), that is, no decay at infinity is needed.
In the conservative case, in contrast, we need to restrict to a decaying setting to ensure the well-posedness of the mesoscopic model~\eqref{eq:GL-1}: more precisely, we assume $\nabla h,F\in W^{1,p}(\R^2)^2$ for some $p<\infty$, $f\in\Ld^2(\R^2)$, and $\Div F=0$. In both cases, in order to ensure strong enough regularity properties of the solution $\vi$ of the mean-field equation, even stronger assumptions on the data are needed and are listed below. Note that we do not try to optimize these regularity assumptions.

\begin{as}\label{as:main}
Let $\alpha\ge0$, $\beta\in\R$, $\alpha^2+\beta^2=1$, $h:\R^2\to\R$, $a:=e^h$, $F:\R^2\to\R^2$, $f:\R^2\to\R$, $u_\e^\circ:\R^2\to\C$, and $\vi_\e^\circ,\vi^\circ:\R^2\to\R^2$ for all $\e>0$.
Assume that~\eqref{eq:formf} and~\eqref{eq:nondegencond} hold, and that the initial data $(u_\e^\circ,\vi_\e^\circ,\vi^\circ)$ are well-prepared as $\e\downarrow0$, in the sense
\begin{multline}\label{eq:convincond}
\D_{\e}^{*,\circ}:=\sup_{R\ge1}\sup_{z\in\R^2}\int_{\R^2} \frac{a\chi_R^z}2\Big(|\nabla u_\e^\circ-iu_\e^\circ N_\e\!\vi_\e^\circ\!|^2+\frac a{2\e^2}(1-|u_\e^\circ|^2)^2\\
-\Log\curl\langle\nabla u_\e^\circ,iu_\e^\circ\rangle\Big)
~\ll~ N_\e^2,
\end{multline}
with $\vi_\e^\circ\to\vi^\circ$ in $\Ld^2_\uloc(\R^2)^2$, and with $\curl\!\vi_\e^\circ$, $\curl\!\vi^\circ\in\Pc(\R^2)$. Assume that $\vi_\e^\circ$ and $\vi^\circ$ are bounded in $W^{1,q}(\R^2)^2$ for all $q>2$. In addition,
\begin{enumerate}[(a)]
\item \emph{Dissipative case ($\alpha>0$, $\beta\in\R$), non-decaying setting:}\\
For some $s>0$, assume that $u_\e^\circ\in H^1_\uloc(\R^2;\C)$, that $h\in W^{s+3,\infty}(\R^2)$, $F\in W^{s+2,\infty}(\R^2)^2$ (hence $f\in W^{1,\infty}(\R^2)$ in view of~\eqref{eq:formf}), that $\vi_\e^\circ$, $\vi^\circ$ are bounded in $W^{s+2,\infty}(\R^2)^2$, and that $\curl\!\vi_\e^\circ$, $\curl\!\vi^\circ$, $\Div(a\!\vi_\e^\circ)$ are bounded in $H^{s+1}\cap W^{s+1,\infty}(\R^2)$.
\item \emph{Conservative case ($\alpha=0$, $\beta=1$), decaying setting:}\\
Assume that $u_\e^\circ\in U+H^2(\R^2;\C)$ for some reference map $U\in \Ld^\infty(\R^2;\C)$ with $\nabla^2U\in H^1(\R^2;\C)$, $\nabla|U|\in\Ld^2(\R^2)$, $1-|U|^2\in\Ld^2(\R^2)$, and $\nabla U\in \Ld^p(\R^2;\C)$ for all $p>2$ (typically we may choose $U$ smooth and equal to $e^{iN_\e\theta}$ in polar coordinates outside a ball at the origin). Assume that $h\in W^{3,\infty}(\R^2)$, $\nabla h\in H^2(\R^2)^2$, $F\in H^3\cap W^{3,\infty}(\R^2)^2$, $f\in H^2\cap W^{2,\infty}(\R^2)$, and that we have $\Div F=0$ pointwise, and $a(x)\to1$ uniformly as $|x|\uparrow\infty$. Assume that $\vi_\e^\circ$, $\vi^\circ$ are bounded in $W^{2,\infty}(\R^2)^2$, and that $\curl\!\vi_\e^\circ$, $\curl\!\vi^\circ$ are bounded in $H^1(\R^2)$.
\qedhere
\end{enumerate}
\end{as}

One may observe that  if $N_\e \le O(\Log)$ the well-preparedness assumption \eqref{eq:convincond} implies that most vortices are initially positive.

\subsubsection{Regimes}
We first comment on the different regimes for the number $N_\e$ of vortices.
A first critical threshold is $N_\e=O(\Log)$, as is clear from energy considerations since in this regime the (concentrated) vortex energy $O(N_\e\Log)$ becomes of the same order as the (diffuse) phase energy $O(N_\e^2)$. Another critical threshold is expected to occur for $N_\e=O(\e^{-1})$ due to the overlap of the vortex cores.
We therefore separately consider the dilute regime $N_\e\ll\Log$, the critical regime $N_\e\simeq\Log$, and the nondilute regime $\Log\ll N_\e\ll\e^{-1}$.
In the dissipative case, these regimes lead to drastically different mean-field behaviors (cf.\@ heuristics in Section~\ref{sec:intro-MFL-GL-heur}).
We do not consider here the superdense regime $N_\e\gtrsim\e^{-1}$, which is of totally different nature since the modulus $|u_\e|$ of the order parameter is then expected to enter the limiting equation, thus leading to different compressible fluid-like equations~\cite{BGSS-09,BGSS-10,BDGSS-10,Carles-Danchin-Saut-12}.

As we can play with the relative strengths of interactions, pinning, and applied current,
we now describe the different possible scalings.
From energy considerations, we expect interactions, pinning, and applied current to be of order $O(N_\e^2)$, $O(N_\e\Log)|\nabla h|$, and $O(N_\e\Log)|F|$, respectively. The critical scaling (such that all effects have the same order)
thus amounts to choosing both $\nabla h$ and $F$ of order $O(\frac{N_\e}\Log)$.
In order for the different effects to give a nontrivial $O(1)$ contribution in the mean-field limit, the time rescaling in~\eqref{eq:GL-1} then needs to be chosen as $\lambda_\e=O(\frac{N_\e}\Log)$.
This leads us to the following critical regimes:
\begin{enumerate}[~~\qquad]
\item[\GLu] Dissipative case --- dilute vortex regime:\\
$\alpha>0$, $1\ll N_\e\ll\Log $, $\lambda_\e=\frac{N_\e}{\Log }$, $F=\lambda_\e\hat F$, $h=\lambda_\e\hat h$ (i.e.~$a=\hat a^{\lambda_\e}$);
\item[\GLd] Dissipative case --- critical vortex regime:\\
$\alpha>0$, $N_\e\simeq\Log $, $\lambda_\e=1$, $F=\hat F$, $h=\hat h$ (i.e.~$a=\hat a$);
\item[\GLt] Dissipative case --- nondilute vortex regime:\\
$\alpha>0$, $\Log\ll N_\e\ll\e^{-1} $, $\lambda_\e=\frac{N_\e}\Log$, $F=\lambda_\e\hat F$, $h=\hat h$ (i.e.~$a=\hat a$);
\item[\GP] Conservative case --- nondilute vortex regime:\\
$\alpha=0$, $\beta=1$, $\Log\ll N_\e\ll\e^{-1}$, $\lambda_\e=\frac{N_\e}{\Log }$, $F=\lambda_\e\hat F$, $h=\hat h$ (i.e.~$a=\hat a$);
\end{enumerate}
where $\hat h$ and $\hat F$ are independent of $\e$, and $\hat h\le0$ is bounded from below.
Just as in~\cite{Serfaty-15} the modulated energy approach does not allow us to treat the conservative case with fewer vortices $N_\e\lesssim\Log$, although in that case the same mean-field behavior is formally expected as in the nondilute regime $\Log\ll N_\e\ll\e^{-1}$ (cf.~Section~\ref{sec:intro-MFL-GL-heur}).
Note that the non-degeneracy condition~\eqref{eq:nondegencond} for the pinning weight $a=e^h$ imposes that the pinning potential $h$ remains uniformly bounded, so that $h$ cannot be chosen of critical order $O(\frac{N_\e}\Log)$ when $N_\e\gg\Log$, which explains the non-critical scaling of $h$ in~\GLt{} and~\GP.

Modifying the time rescaling $\lambda_\e$ and the scaling of $h$,
we may also consider various non-critical scalings, for which the pinning either dominates or is dominated by the interactions. In such cases, the limiting equations are substantially simplified. We consider for instance:
\begin{enumerate}[~~\qquad]
\item[\GLup] Dissipative case --- dilute vortex regime --- very weak interactions:\\
$\alpha>0$, $N_\e\ll\Log $, $\lambda_\e=1$, $F=\hat F$, $h=\hat h$;
\item[\GLdp] Dissipative case --- dilute vortex regime --- weak interactions:\\
$\alpha>0$, $N_\e\ll\Log $, $\frac{N_\e}\Log\ll\lambda_\e\ll1$, $F=\lambda_\e\hat F$, $h=\lambda_\e\hat h$;
\item[\GLtp] Dissipative case --- dilute vortex regime --- strong interactions:\\
$\alpha>0$, $N_\e\ll\Log $, $\lambda_\e=\frac{N_\e}\Log$, $F=\lambda_\e\hat F$, $h=\lambda_\e'\hat h$, $\lambda_\e'\ll\lambda_\e$;
\item[\GLqp] Dissipative case --- critical vortex regime --- strong interactions:\\
$\alpha>0$, $N_\e\simeq\Log $, $\lambda_\e=1$, $F=\hat F$, $h=\lambda_\e'\hat h$, $\lambda_\e'\ll1$;
\end{enumerate}
where again $\hat h$ and $\hat F$ are independent of $\e$, and $\hat h\le0$ is bounded from below.
Since in the present work we are mostly interested in pinning effects, we focus on the regimes~\GLup{} and~\GLdp{}, while for~\GLtp{} and~\GLqp{} the pinning effects vanish in the limit and the situation is thus much easier and closer to~\cite{Serfaty-15}.
For simplicity, subscripts ``$\e$'' are systematically dropped from the data $a,h,F,f$.

\subsubsection{Statement of main results}
We are now in position to state our main results.
We start with the dissipative mixed-flow case, and first consider the dilute and the critical vortex regimes with critical scalings~\GLu{} and~\GLd, or with non-critical scalings~\GLup{} and~\GLdp.
The following result generalizes those in~\cite{Kurzke-Spirn-14,Serfaty-15} to the case with pinning and applied current.
Note that the statements are slightly finer in the parabolic case.
The mean-field equations are fluid-like of the form~\eqref{eq:expect-mfl-ode1-vi}, but the incompressibility condition is lost in the critical vortex regime, as first evidenced in~\cite{Serfaty-15} (cf.\@ heuristics in Section~\ref{sec:intro-MFL-GL-heur}).
In the regimes~\GLu{} and~\GLdp, the weight $a$ naturally disappears from the incompressibility condition $\Div\!\vi=0$ due to the assumption $a=\hat a^{\lambda_\e}\to1$ as $\e\downarrow0$.
Although all the proofs are quantitative, we only include qualitative statements to simplify the exposition.

\begin{theor1}[Dissipative case]\label{th:mainGL}
Let Assumption~\ref{as:main}(a) hold, where in particular the initial data $(u_\e^\circ,\vi_\e^\circ,\vi^\circ)$ satisfy the well-preparedness condition~\eqref{eq:convincond}.
For all $\e>0$, let $u_\e\in\Ld^\infty_\loc(\R^+;H^1_\uloc(\R^2;\C))$ denote the unique global solution of~\eqref{eq:GL-1} in $\R^+\times\R^2$.
Then, the following hold for the supercurrent density $j_\e:=\langle\nabla u_\e,iu_\e\rangle$.
\begin{enumerate}[(i)]
\item \emph{Regime~\GLu{} with $\log\Log\ll N_\e\ll\Log$, and $\Div(a\!\vi_\e^\circ)=\Div\!\vi^\circ=0$:}\\
We have $\frac1{N_\e}j_\e\to\vi$ in $\Ld^\infty_\loc(\R^+;\Ld^1_\uloc(\R^2)^2)$ as $\e\downarrow0$, where $\vi$ is the unique global (smooth) solution of
\begin{align}\label{eq:GL1lim}
\qquad\begin{cases}
\partial_t\!\vi=\nabla\!\pre+(\alpha-\Jb\beta)(\nabla^\bot\hat h-\hat F^\bot-2\!\vi)\curl\!\vi,\\
\Div\!\vi=0, \quad\vi\!|_{t=0}=\vi^\circ.
\end{cases}
\end{align}
In the parabolic case $\beta=0$, the same conclusion holds for $1\ll N_\e\lesssim\log\Log$.
\item \emph{Regime~\GLd{} with $\frac{N_\e}\Log\to\lambda\in(0,\infty)$ and $\vi_\e^\circ=\vi^\circ$:}\\
For some $T>0$, we have $\frac1{N_\e}j_\e\to\vi$ in $\Ld^\infty_\loc([0,T);\Ld^1_\uloc(\R^2)^2)$ as $\e\downarrow0$, where $\vi$ is the unique local (smooth) solution of
\begin{align}\label{eq:GL2lim}
\qquad\begin{cases}
\partial_t\!\vi=\alpha^{-1}\nabla(\hat a^{-1}\Div(\hat a\!\vi))+(\alpha-\Jb\beta)(\nabla^\bot\hat h-\hat F^\bot-2\lambda\!\vi)\curl\!\vi,\\
\vi\!|_{t=0}=\vi^\circ,
\end{cases}
\end{align}
in $[0,T)\times\R^2$. In the parabolic case $\beta=0$, this solution $\vi$ can be extended globally, and the above holds with $T=\infty$.
\item \emph{Regime~\GLup{} with $\log\Log\ll N_\e\ll\Log$ and $\vi_\e^\circ=\vi^\circ$:}\\
We have $\frac1{N_\e}j_\e\to\vi$ in $\Ld^\infty_\loc(\R^+;\Ld^1_\uloc(\R^2)^2)$ as $\e\downarrow0$, where $\vi$ is the unique global (smooth) solution of
\begin{align}\label{eq:GL1'lim}
\qquad\begin{cases}
\partial_t\!\vi=\alpha^{-1}\nabla(\hat a^{-1}\Div(\hat a\!\vi))+(\alpha-\Jb\beta)(\nabla^\bot\hat h-\hat F^\bot)\curl\!\vi,\\
\vi\!|_{t=0}=\vi^\circ.
\end{cases}
\end{align}
\item \emph{Regime~\GLdp{} with $\log\Log\ll N_\e\ll\Log$ and $\Div(a\!\vi_\e^\circ)=\Div\!\vi^\circ=0$:}\\
We have $\frac1{N_\e}j_\e\to\vi$ in $\Ld^\infty_\loc(\R^+;\Ld^1_\uloc(\R^2)^2)$ as $\e\downarrow0$, where $\vi$ is the unique global (smooth) solution of
\begin{align}\label{eq:GL2'lim}
\qquad\begin{cases}
\partial_t\!\vi=\nabla\!\pre+(\alpha-\Jb\beta)(\nabla^\bot\hat h-\hat F^\bot)\curl\!\vi,\\
\Div\!\vi=0,\quad \vi\!|_{t=0}=\vi^\circ.
\end{cases}
\end{align}
In the parabolic case $\beta=0$ with $\frac{N_\e}\Log\ll\lambda_\e\lesssim \frac{e^{o(N_\e)}}\Log$, the same conclusion also holds for $1\ll N_\e\lesssim\log\Log$.
\qedhere
\end{enumerate}
\end{theor1}

\begin{rem}
In the regimes~\GLu{} and~\GLdp, the modified data $\vi_\e^\circ$ can for instance be chosen as
\[\vi_\e^\circ:=a^{-1}\nabla^\bot(\Div a^{-1}\nabla)^{-1}\curl\vi^\circ,\]
which indeed satisfies $\Div(a\!\vi_\e^\circ)=0$ and $\curl\!\vi_\e^\circ=\curl\!\vi^\circ$, while the assumption $a\to1$ in $\Ld^\infty(\R^2)$ easily implies $\vi_\e^\circ\to\vi^\circ$ in $\Ld^q(\R^2)^2$ for all $q>2$, hence $\vi_\e^\circ\to\vi^\circ$ in $\Ld^2_\uloc(\R^2)^2$.
\end{rem}

We turn to the nondilute vortex regime~\GLt. The following result is only proven to hold in the parabolic case in the weakly nondilute regime $\Log\ll N_\e\ll\Log\log\Log$, and gives rise to a new degenerate mean-field equation that is studied in detail in the companion article~\cite{D-16}. This result is new even in the case without pinning and applied current, as it indeed treats a regime left open in~\cite{Serfaty-15}.
Note that a slightly stronger well-posedness condition is needed here; this condition is however still reasonable since for any smooth $\vi^\circ$ and any $0<\delta <1$ one may construct a configuration $u_\e^\circ$ that satisfies it, cf.~\cite{SS-book}.

\begin{theor1}[Nondilute parabolic case]\label{th:main-hd}
Let Assumption~\ref{as:main}(a) hold, and assume that the initial data $(u_\e^\circ,\vi_\e^\circ,\vi^\circ)$ satisfy $\vi_\e^\circ=\vi^\circ$ and satisfy the following slightly stronger well-preparedness condition, for some $\delta>0$,
\begin{multline*}
\D_{\e}^{*,\circ}:=\sup_{R\ge1}\sup_{z\in\R^2}\int \frac{a\chi_R^z}2\Big(|\nabla u_\e^\circ-iu_\e^\circ N_\e\!\vi^\circ\!|^2+\frac a{2\e^2}(1-|u_\e^\circ|^2)^2\\
-\Log\curl\langle\nabla u_\e^\circ,iu_\e^\circ\rangle\Big)~\lesssim~ N_\e^{2-\delta}.
\end{multline*}
For some $s>3$, assume in addition that $h\in W^{s+2,\infty}(\R^2)$, $F\in W^{s+1,\infty}(\R^2)^2$, and that $\vi^\circ\in W^{s+1,\infty}(\R^2)^2$, $\md^\circ:=\curl\!\vi^\circ\in \Pc\cap H^s(\R^2)$, and $\dd^\circ:=\Div(a\!\vi^\circ)\in H^{s-1}(\R^2)$.
For all $\e>0$, let $u_\e\in\Ld^\infty_\loc(\R^+;H^1_\uloc(\R^2;\C))$ denote the unique global solution of~\eqref{eq:GL-1} in $\R^+\times\R^2$.
Then, in the regime~\GLt{} with $\Log\ll N_\e\ll\Log\log\Log$, in the parabolic case $\beta=0$, the supercurrent density $j_\e:=\langle\nabla u_\e,iu_\e\rangle$ satisfies $\frac1{N_\e}j_\e\to\vi$ in $\Ld^\infty_\loc(\R^+;\Ld^1_\uloc(\R^2)^2)$ as $\e\downarrow0$, where $\vi$ is the unique global (smooth) solution of
\begin{equation}\label{eq:GL3lim}
\begin{cases}
\partial_t\!\vi=-(\hat F^\bot+2\!\vi)\,\curl\!\vi,\\
\vi\!|_{t=0}=\vi^\circ.
\end{cases}
\qedhere
\end{equation}
\end{theor1}

\begin{rem}
As explained in Section~\ref{sec:intro-MFL-GL-heur}, the same mean-field result is expected to hold in the whole nondilute regime $\Log\ll N_\e\ll\e^{-1}$ (up to a suitable well-preparedness condition), but this remains an open question. A corresponding result is also expected in the dissipative mixed-flow case, but then the correct limiting equation is actually unclear since the local well-posedness of the mixed-flow version of the degenerate equation~\eqref{eq:GL3lim}, that is,
\begin{align*}
\partial_t\!\vi=-(\alpha-\Jb\beta)(\hat F^\bot+2\!\vi)\,\curl\!\vi,
\end{align*}
remains unresolved~\cite{D-16}.
\end{rem}

\medskip
We finally turn to the conservative case in the regime~\GP.
For $N_\e\gg\Log$, the well-preparedness condition~\eqref{eq:convincond} is naturally simplified, as the vortex self-interaction energy is no longer dominant.
Note that the pinning force $-\nabla\hat h$ is absent from the limiting equation since in the regime~\GP{} the interaction and the applied current dominate. The pinning weight $a=\hat a$ nevertheless remains in the incompressibility condition $\Div(\hat a\!\vi)=0$.
The mean-field equation is then a variant of the 2D Euler equation~\eqref{eq:expect-mfl-ode1-vi+} and is known as the lake equation in the context of 2D shallow water fluid dynamics (cf.~e.g.~\cite{CHL-1,CHL-2}).
The following result generalizes that in~\cite{Serfaty-15} to the case with pinning and applied current.

\begin{theor1}[Conservative case]\label{th:mainGP}
Let Assumption~\ref{as:main}(b) hold, and assume that the initial data satisfy $\vi_\e^\circ=\vi^\circ$ and satisfy the following simplified well-preparedness condition,
\[\Ec_{\e}^\circ:=\int_{\R^2} \frac a2\Big(|\nabla u_\e^\circ-iu_\e^\circ N_\e\!\vi^\circ\!|^2+\frac a{2\e^2}(1-|u_\e^\circ|^2)^2\Big)~\ll~ N_\e^2.\]
For all $\e>0$, let $u_\e\in\Ld^\infty_\loc(\R^+; U+H^2(\R^2;\C))$ denote the unique global solution of~\eqref{eq:GL-1} in $\R^+\times\R^2$.
Then, in the regime~\GP{} with $\Log\ll N_\e\ll\e^{-1}$, we have $\frac1{N_\e}j_\e\to \vi$ in $\Ld^\infty_\loc(\R^+;(\Ld^1+\Ld^2)(\R^2)^2)$ as $\e\downarrow0$, where $\vi$ is the unique global (smooth) solution of
\begin{equation}\label{eq:GPlim}
\begin{cases}\partial_t\!\vi=\nabla\!\pre-(\hat F-2\!\vi^\bot)\,\curl\!\vi,\\
\Div( \hat a\!\vi)=0,\quad \vi^t\!|_{t=0}=\vi^\circ.
\end{cases}
\qedhere
\end{equation}
\end{theor1}

\begin{rem}
As explained in Section~\ref{sec:intro-MFL-GL-heur}, the same mean-field limit result is actually expected to hold for all $1\ll N_\e\ll\e^{-1}$ (cf.\@ indeed~\cite{Jerrard-Spirn-15} for the other extreme regime $1\ll N_\e\lesssim (\log\Log)^{1/2}$), but this remains an open question.
As in~\cite{Serfaty-15}, we need to restrict here to the nondilute regime $N_\e\gg\Log$ due to the difficulty of controlling the velocity of individual vortices, which is related to the lack of control on $\int_{\R^2}|\partial_t u_\e|^2$.
Note however that in the dilute regime the conservative vortex dynamics formally behaves like the conservative flow for Coulomb particles and that the mean-field limit of the latter system can be rigorously established by a modulated energy approach~\cite{Serfaty-18}.
\end{rem}

\medskip

The structure of the mean-field equations~\eqref{eq:GL1lim}--\eqref{eq:GPlim} is more transparent when expressed in terms of the mean-field vorticity $\md:=\curl\!\vi$. In the case of~\eqref{eq:GL1lim} (and correspondingly for~\eqref{eq:GPlim}), the vorticity $\md$ satisfies a nonlocal nonlinear continuity equation,
\begin{align}\label{eq:GL1lim-bis}
\begin{cases}\partial_t\!\md=\Div\!\big((\alpha-\Jb\beta)(\nabla\hat h-\hat F+2\!\vi^\bot)\md\big),\\
\curl\!\vi=\md,\quad\Div\!\vi=0.\end{cases}
\end{align}
In the case of~\eqref{eq:GL2lim}, the vorticity $\md$ satisfies a similar equation coupled with a convection-diffusion equation for the divergence $\dd:=\Div(\hat a\!\vi)$,
\begin{gather}\label{eq:GL2lim-bis}
\begin{cases}\partial_t\!\md=\Div\!\big((\alpha-\Jb\beta)(\nabla\hat h-\hat F+2\lambda\!\vi^\bot)\md\big),\\
\partial_t\!\dd-\alpha^{-1}\triangle\dd+\alpha^{-1}\Div(\dd\!\nabla\hat h)=\Div\big((\alpha-\Jb\beta)(\nabla^\bot \hat h-\hat F^\bot-2\lambda\!\vi)\hat a\!\md\big),\\
\curl\!\vi=\md,\quad\Div(\hat a\!\vi)=\dd,
\end{cases}
\end{gather}
while the convection-diffusion equation becomes degenerate in the case of~\eqref{eq:GL3lim} and then takes on the following guise, in terms of $\theta:=\Div\!\vi$,
\begin{gather}\label{eq:GL3lim-bis}
\begin{cases}\partial_t\!\md=\Div\!\big((-\hat F+2\lambda\!\vi^\bot)\md\big),\\
\partial_t\theta=\Div\big((-\hat F^\bot-2\lambda\!\vi)\md\big),\\
\curl\!\vi=\md,\quad\Div\!\vi=\theta.
\end{cases}
\end{gather}
A detailed study of these families of equations is provided in the companion article~\cite{D-16}, including global existence results for rough initial data.
In the cases~\eqref{eq:GL1'lim} and~\eqref{eq:GL2'lim}, which correspond to scalings with negligible interactions, the limiting vorticity $\md$ rather satisfies a simple linear continuity equation,
\begin{align}\label{eq:GL1'lim-bis}
\partial_t\!\md=\Div\!\big((\alpha-\Jb\beta)(\nabla\hat h-\hat F)\md\big).
\end{align}
Let us emphasize the nonlocal character of~\eqref{eq:GL1lim-bis}--\eqref{eq:GL3lim-bis}: in~\eqref{eq:GL1lim-bis} and~\eqref{eq:GL3lim-bis} the equations $\curl\!\vi=\md$ and $\Div\!\vi=\theta$ are (formally) solved as
\[\vi=\nabla^\bot\triangle^{-1}\md+\nabla\triangle^{-1}\theta,\]
while in~\eqref{eq:GL2lim-bis} the equations $\curl\!\vi=\md$ and $\Div(a\!\vi)=\dd$ lead to
\[\vi=a^{-1}\nabla^\bot(\Div\! a^{-1}\nabla)^{-1}\md+\nabla(\Div\! a\nabla)^{-1}\dd.\]

\subsection{Heuristic derivation of the mean-field equations}\label{sec:intro-MFL-GL-heur}
In order to illustrate the structure of the 2D mesoscopic model~\eqref{eq:GL-1} and the importance of a careful vortex analysis, we now give a short heuristic derivation of the mean-field equations~\eqref{eq:GL1lim}--\eqref{eq:GPlim}. This derivation brings a more intuitive explanation of the compressibility of the mean-field equations in the nondilute dissipative case, and it further predicts the expected behavior in the different regimes for which our analysis fails.
For simplicity of the discussion, we focus here on the simpler case without pinning and applied current, thus considering the following version of~\eqref{eq:GL-1},
\begin{align}\label{eq:GL-1-intro}
\lambda_\e(\alpha+i\Log \beta)\partial_tu_\e=\triangle u_\e+\frac{u_\e}{\e^2}(1-|u_\e|^2).
\end{align}
Next to the supercurrent density $j_\e$ and the vorticity $\mu_\e$, we define the vortex velocity
\[V_\e:=2\langle\nabla u_\e,i\partial_t u_\e\rangle,\]
the Ginzburg-Landau energy density
\[e_\e:=\frac12\Big(|\nabla u_\e|^2+\frac{1}{2\e^2}(1-|u_\e|^2)^2\Big),\]
and the stress-energy tensor
\begin{align*}
(S_\e)_{kl}:=\langle\partial_ku_\e,\partial_lu_\e\rangle - \frac{\delta_{kl}}2\Big(|\nabla u_\e|^2+\frac{1}{2\e^2}(1-|u_\e|^2)^2\Big).
\end{align*}
The definition of $V_\e$ easily leads to the following algebraic identities (cf.~\cite{SS-prod}),
\begin{align}\label{eq:identity-1-intro}
\partial_tj_\e=V_\e+\nabla\langle \partial_tu_\e,iu_\e\rangle,\qquad\partial_t\mu_\e=\curl V_\e.
\end{align}
By~\eqref{eq:GL-1-intro}, we further find the following identities for the divergence of the supercurrent density
\begin{align}\label{eq:identity-2-intro0}
\Div\! j_\e&=\langle \triangle u_\e,iu_\e\rangle=\lambda_\e\alpha\langle\partial_tu_\e,iu_\e\rangle-\frac{\lambda_\e\beta\Log}2\partial_t(1-|u_\e|^2),
\end{align}
for the divergence of the stress-energy tensor
\begin{align}\label{pdivs-intro}
\Div\! S_\e&=\left\langle \nabla u_\e,\triangle u_\e+\frac{u_\e}{\e^2}(1-|u_\e|^2)\right\rangle=\lambda_\e\alpha \langle \nabla u_\e,\partial_t u_\e\rangle+\frac{\lambda_\e \Log\beta}2V_\e,
\end{align}
and for the time derivative of the energy density
\begin{align*}
\partial_te_\e=\Div\langle\nabla u_\e,\partial_tu_\e\rangle-\lambda_\e\alpha |\partial_tu_\e|^2.
\end{align*}
Using~\eqref{pdivs-intro} to replace $\langle\nabla u_\e,\partial_tu_\e\rangle$, this last identity rather takes on the following guise,
\begin{align}\label{eq:deriv-time-endens}
\lambda_\e\alpha \partial_te_\e=\Div\!\Div \!S_\e-\frac{\lambda_\e \Log\beta}2\Div\! V_\e-\lambda_\e^2\alpha^2|\partial_tu_\e|^2.
\end{align}
If there is no excess energy, the Ginzburg-Landau energy is expected to split into a (concentrated) vortex energy of order $O(N_\e\Log)$ and a (diffuse) phase energy of order $O(N_\e^2)$. Since the quantity $|1-|u_\e|^2|$ is bounded by $\e(e_\e)^{1/2}$, it is therefore formally of order $O(\e (N_\e\Log+N_\e^2)^{1/2})$, which is negligible as soon as $N_\e$ is much smaller than~$O(\e^{-1})$. Choosing the critical scaling $\lambda_\e:=\frac{N_\e}\Log$, the above identities~\eqref{eq:identity-2-intro0}, \eqref{pdivs-intro}, and~\eqref{eq:deriv-time-endens} then become
\begin{gather}\label{eq:identity-2-intro}
\Div\!\frac{j_\e}{N_\e}\approx\alpha\frac{\langle\partial_tu_\e,iu_\e\rangle}{\Log},\\
2\Div\! \frac{S_\e}{N_\e^2}=2\alpha \frac{\langle \nabla u_\e,\partial_t u_\e\rangle}{N_\e\Log}+\beta\frac{V_\e}{N_\e},\label{eq:identity-3-intro}\\
\alpha \partial_t\frac{2e_\e}{N_\e\Log}=2\Div\!\Div\!\frac{S_\e}{N_\e^2}-\beta\Div\!\frac{V_\e}{N_\e}-2\alpha^2 \frac{|\partial_tu_\e|^2}{\Log^2}.\label{eq:identity-4-intro}
\end{gather}
In order to take weak limits in these equations and to characterize the limiting evolution, we need to establish a priori bounds on all the terms and to find relations between the weak limits of the various quantities.
In the limit $\e\downarrow0$, vortices become point-like and the vorticity $\mu_\e$ looks like a sum of $N_\e$ Dirac masses, cf.~\eqref{muejac}.
We may thus formally assume that the rescaled vorticity $\frac1{N_\e}{\mu_\e}$ converges weakly-* to some probability measure $\md\in\Ld^\infty(\R^+;\Pc(\R^2))$. Similarly, the vortex velocity $V_\e$ concentrates at the vortex locations, and we may assume that its rescaled version $\frac1{N_\e}V_\e$ converges weakly-* to some measure $\Vi\in\Ld^\infty_\loc(\R^+;\M(\R^2)^2)$.
For $p<2$ the rescaled supercurrent density $\frac1{N_\e}j_\e$ may be assumed to be bounded in $\Ld^{p}_\loc(\R^2)$ and thus to converge weakly to some limit $\vi\in\Ld^\infty_\loc(\R^+;\Ld^p_\loc(\R^2)^2)$, but it cannot converge in $\Ld^2_\loc(\R^2)$ due to energy concentration. In short,
\begin{align}\label{eq:convergences-intro-MFLGL}
\frac1{N_\e}\mu_\e\cvf*\md,\qquad \frac1{N_\e}V_\e\cvf*\Vi,\qquad \frac1{N_\e}j_\e\cvf{}\vi.
\end{align}
Quadratic quantities such as $e_\e\approx\frac12|j_\e|^2$ and $|\partial_tu_\e|^2$ have a part that concentrates at vortex locations in the limit $\e\downarrow0$, and their concentrated and diffuse parts must be analyzed separately.
If there is no excess energy, the concentrated part of the energy density $e_\e\approx\frac12|j_\e|^2$ should coincide with the vortex self-interaction energy $\frac12\Log\mu_\e\approx\frac12N_\e\Log\!\md$ (this is made precise by the Jerrard-Sandier ball construction lower bound~\cite{Sandier-98,Jerrard-99}), while the diffuse part should be given by $\frac12N_\e^2|\!\vi\!|^2$ in terms of the weak limit $\vi$ of $\frac1{N_\e}j_\e$, cf.~\eqref{eq:convergences-intro-MFLGL}. Such properties could be phrased in terms of defect measures for the convergence of $\frac1{N_\e}j_\e$ in $\Ld^2_\loc(\R^2)$, cf.~\cite{SS-prod}. Similarly, if there is no excess energy, the concentrated part of $|\partial_tu_\e|^2$ should coincide with $$\frac12\Log \mu_\e^{-1}|V_\e|^2\approx \frac12 N_\e\Log\!\md^{-1}\!|\!\Vi\!|^2$$ in terms of the vortex velocity and the vorticity (this is made precise by the so-called product estimate~\cite{SS-prod}), while identity~\eqref{eq:identity-2-intro} in the form $$\alpha^{2}|\partial_tu_\e|^2\approx\alpha^{2}|\langle\partial_tu_\e,iu_\e\rangle|^2\approx\lambda_\e^{-2}|\Div j_\e|^2$$ suggests that the diffuse part of $\alpha^{2}|\partial_tu_\e|^2$ should simply be given by $\Log^2|\Div\!\vi\!|^2$. In short,
\begin{gather}\label{eq:approx-e-eps}
2e_\e\approx |j_\e|^2\approx N_\e\Log\!\md+N_\e^2|\!\vi\!|^2,\\
2\alpha^2|\partial_tu_\e|^2\approx2\Log^2|\!\Div\!\vi\!|^2+\alpha^2N_\e\Log\!\md^{-1}\!|\!\Vi\!|^2.\label{eq:approx-V-eps}
\end{gather}
Let us now turn to the limit of the stress-energy tensor $S_\e\approx j_\e\otimes j_\e-\frac{\Id}2|j_\e|^2$.
Due to the isotropy of the vortex core energy, in link with equipartition properties of the Ginzburg-Landau energy~\cite{Kurzke-Spirn-10}, the stress-energy tensor $S_\e$ should not be sensitive to the concentrated part of $j_\e$ in $\Ld^2_\loc(\R^2)$,
and we simply expect $\frac1{N_\e^2}S_\e\approx \vi\otimes\vi-\frac\Id2|\!\vi\!|^2$ in terms of the weak limit $\vi$ of $\frac1{N_\e}j_\e$ (see also~\cite[Chapter~13]{SS-book}). In particular,
\begin{align}\label{eq:approx-S-eps}
\Div\frac{S_\e}{N_\e^2}\approx \Div\Big(\vi\otimes\vi-\frac\Id2|\!\vi\!|^2\Big)=\vi^\bot\!\md+\vi\Div \!\vi.
\end{align}
Inserting the convergences~\eqref{eq:convergences-intro-MFLGL} and the identifications~\eqref{eq:approx-e-eps}, \eqref{eq:approx-V-eps}, and~\eqref{eq:approx-S-eps} into identities~\eqref{eq:identity-2-intro}, \eqref{eq:identity-3-intro}, and \eqref{eq:identity-4-intro}, we obtain after straightforward simplifications,
\begin{gather}
\Div \!\vi \approx\alpha\frac{\langle\partial_tu_\e,iu_\e\rangle}{\Log},\label{eq:identity-5-intro}\\
2\!\vi^\bot\!\md+2\!\vi\Div\!\vi\approx2\alpha\frac{\langle \nabla u_\e,\partial_t u_\e\rangle}{N_\e\Log}+\beta\!\Vi,\label{eq:identity-6-intro}\\
\alpha\partial_t\!\md+2\alpha \lambda_\e\!\vi\cdot\partial_t\!\vi\approx 2\Div(\vi^\bot\!\md)+2\!\vi\cdot\nabla\Div\!\vi-\beta\Div\!\Vi-\lambda_\e\alpha^2\!\md^{-1}\!|\!\Vi\!|^2.\label{eq:identity-7-intro}
\end{gather}
Further inserting~\eqref{eq:identity-5-intro} into~\eqref{eq:identity-1-intro}, we obtain
\begin{gather}
\alpha\partial_t\!\vi\approx\alpha\!\Vi+\lambda_\e^{-1}\nabla\Div\!\vi,\qquad\partial_t\!\md=\curl\!\Vi.\label{eq:identity-5-0-intro}
\end{gather}
We now separately consider the conservative and the dissipative cases.

\smallskip
\begin{itemize}
\item[$\bullet$] \emph{Conservative case ($\alpha=0$, $\beta=1$).}\\
Identity~\eqref{eq:identity-5-intro} yields $\Div\!\vi=0$, while identity~\eqref{eq:identity-6-intro} takes the form $\Vi=2\!\vi^\bot\!\md$. Injecting this into~\eqref{eq:identity-5-0-intro} then leads to
\[\partial_t\!\md=2\Div(\vi\!\md),\qquad\curl\!\vi=\md,\qquad\Div\!\vi=0,\]
or alternatively,
\[\partial_t\!\vi=\nabla\!\pre+2\!\vi^\bot\!\curl\!\vi,\qquad\Div\!\vi=0.\]
In the regime $1\ll N_\e\ll\e^{-1}$ with the critical choice $\lambda_\e=\frac{N_\e}\Log$, the rescaled supercurrent density $\frac1{N_\e}j_\e$ is thus expected to converge to the solution $\vi$ of this incompressible 2D Euler equation.

\smallskip
\item[$\bullet$] \emph{Dissipative case ($\alpha>0$, $\alpha^2+\beta^2=1$).}\\
Injecting~\eqref{eq:identity-5-0-intro} into~\eqref{eq:identity-7-intro} yields
\begin{align}\label{eq:intro-correction-eqnMFLGL-md0}
\partial_t\!\md\approx \frac2\alpha\Div(\vi^\bot\!\md)-\frac\beta\alpha\Div\!\Vi-\lambda_\e\!\Vi\cdot(2\!\vi+\alpha\!\md^{-1}\!\Vi).
\end{align}
Comparing with~\eqref{eq:identity-5-0-intro} in the form $\partial_t\!\md=\curl\!\Vi$, we deduce in the parabolic case ($\alpha=1$, $\beta=0$) that $\Vi=-2\!\vi\!\md$, while a more careful computation in the general mixed-flow case leads to $\Vi=-2\alpha\!\vi\!\md+2\beta\!\vi^\bot\!\md$. Injecting this into~\eqref{eq:identity-5-0-intro}, we obtain
\begin{align}\label{eq:intro-correction-eqnMFLGL}
\partial_t\!\vi\approx(\lambda_\e\alpha)^{-1}\nabla\Div\!\vi+2(-\alpha\!\vi+\beta\!\vi^\bot)\,\curl\!\vi.
\end{align}
We need to distinguish between three regimes:
\begin{itemize}
\item \emph{Dilute regime $1\ll N_\e\ll\Log$:}\\
As $\lambda_\e\ll1$, equation~\eqref{eq:intro-correction-eqnMFLGL-md0} and the identification of $\Vi$ then yield
\begin{align*}\label{eq:intro-correction-eqnMFLGL-md0}
\partial_t\!\md=\Div(2(\alpha\!\vi^\bot+\beta\!\vi)\md),
\end{align*}
while equation~\eqref{eq:intro-correction-eqnMFLGL} together with~\eqref{eq:identity-5-intro} leads to $\Div\vi=0$, so that we deduce, using the relation $\Div\vi=0$ in the form $\vi=\nabla^\bot\triangle^{-1}\md$, and setting $\pre:=-2\triangle^{-1}\Div((-\alpha\!\vi+\beta\!\vi^\bot)\md)$,
\[\partial_t\!\vi=\nabla\!\pre+2(-\alpha\!\vi+\beta\!\vi^\bot)\,\curl\!\vi,\qquad \Div\!\vi=0.\]
\item \emph{Critical regime $N_\e\simeq\Log$ with $\lambda_\e\to\lambda\in(0,\infty)$:}\\
Equation~\eqref{eq:intro-correction-eqnMFLGL} then becomes
\[\lambda\partial_t\!\vi=\alpha^{-1}\nabla\Div\!\vi+2\lambda(-\alpha\!\vi+\beta\!\vi^\bot)\,\curl\!\vi.\]
\item \emph{Nondilute regime $\Log\ll N_\e\ll\e^{-1}$:}\\
As $\lambda_\e\gg1$, equation~\eqref{eq:intro-correction-eqnMFLGL} then becomes
\[\partial_t\!\vi=2(-\alpha\!\vi+\beta\!\vi^\bot)\,\curl\!\vi.\]
\end{itemize}
In these different regimes, with the critical choice $\lambda_\e=\frac{N_\e}\Log$, the rescaled supercurrent density $\frac1{N_\e}j_\e$ is thus expected to converge to the solution $\vi$ of one of the above equations.
\end{itemize}

\smallskip\noindent
This careful heuristic argument therefore allows to predict the whole family of announced mean-field evolutions~\eqref{eq:GL1lim}--\eqref{eq:GPlim}, and formally explains the (a priori unexpected) higher variety of possible behavior in the dissipative case depending on the vortex density regime. Note however that this formal argument relies on important unproven assumptions such as the absence of energy excess and the equipartition of energy, which are bypassed by the modulated energy approach.

\subsection{Case with gauge}\label{sec:gauge}
In the dissipative case, it is interesting to make the computations also in the case with magnetic gauge, which is the relevant physical model for superconductors. The evolution equation~\eqref{eq:Gorkov-Eliashberg-mixed-flow} is then replaced by the following, as first derived by Schmid~\cite{Schmid-66} and by Gor'kov and Eliashberg~\cite{GE-68}, here written in the mixed-flow case, with strong (critically scaled) applied electric current $\Log J_{\text{ex}}:\partial\Omega\to\R^2$ and applied magnetic field $\Log H_{\text{ex}}:\partial\Omega\to\R$ at the boundary, and with a non-uniform pinning weight $a$,
\begin{align*}
\begin{cases}
(\alpha+i\Log \beta)(\partial_tw_\e-iw_\e\Psi_\e)=\nabla_{B_\e}^2 w_\e+\frac{w_\e}{\e^2}(a-|w_\e|^2),&\text{in $\R^+\times\Omega$,}\\
\sigma(\partial_tB_\e-\nabla\Psi_\e)=\nabla^\bot\curl B_\e+\langle iw_\e,\nabla_{B_\e}w_\e\rangle,&\text{in $\R^+\times\Omega$,}\\
\curl B_\e=\Log H_{\text{ex}},&\text{on $\R^+\times\partial\Omega$,}\\
n\cdot\nabla_{B_\e} w_\e=iw_\e \Log n\cdot J_{\text{ex}},&\text{on $\R^+\times\partial\Omega$,}\\
w_\e|_{t=0}=w_\e^\circ,\end{cases}
\end{align*}
where $B_\e:\R^+\times\R^2\to\R^2$ is the gauge of the magnetic field $\curl B_\e$, where $\Psi_\e:\R^+\times\R^2\to\R$ is the gauge of the electric field $-\partial_tB_\e+\nabla\Psi_\e$, where $\nabla_{B_\e}:=\nabla-iB_\e$ denotes the usual covariant derivative, and where the real parameter $\sigma\ge0$ characterizes the relaxation time of the magnetic field.
As the presence of the boundary creates important mathematical difficulties, we again modify the above mesoscopic model and consider a suitable version on the whole plane with boundary conditions ``at infinity''. As in~\cite{Tice-10,S-Tice-11}, the boundary conditions can be changed into a bulk force term by a suitable change of phase in the unknown functions. Also dividing $w_\e$ by the expected density $\sqrt{a}$ and making a suitable choice of the gauge $\Psi_\e$, we arrive at the following equation for the couple $(u_\e,A_\e)$ replacing the triplet $(w_\e,B_\e,\Psi_\e)$,
\begin{align*}
\begin{cases}
\lambda_\e(\alpha+i\Log \beta)\partial_tu_\e=\nabla_{A_\e}^2 u_\e+\frac{au_\e}{\e^2}(1-|u_\e|^2)\\
\hspace{4cm}+\nabla h\cdot\nabla_{A_\e}u_\e+i\Log F^\bot\cdot\nabla_{A_\e}u_\e+fu_\e,&\text{in $\R^+\times\Omega$,}\\
\sigma\partial_tA_\e=\nabla^\bot \curl A_\e +a\langle iu_\e,\nabla_{A_\e}u_\e\rangle-\frac12\Log aF^\bot(1-|u_\e|^2),&\text{in $\R^+\times\Omega$,}\\
u_\e|_{t=0}=u_\e^\circ,
\end{cases}
\end{align*}
with $h:=\log a$, $f:\R^2\to\R$, and $F:\R^2\to\R^2$, where $F$ and $f$ are given explicitly in terms of $a$, $J_{\text{ex}}$, and $H_{\text{ex}}$. We refer to~\cite[Section~2]{S-Tice-11} for the detail of the derivation of this equation from the above model.
Natural quantities associated with this transformed model are the gauge-invariant supercurrent and vorticity,
\[j_\e:=\langle\nabla_{A_\e}u_\e,iu_\e\rangle,\qquad \mu_\e:=\curl(j_\e+A_\e),\]
and the electric field
\[E_\e:=-\partial_tA_\e.\]
We believe that the derivation of mean-field limit results from this gauged version of the model~\eqref{eq:GL-1} does  not cause any major difficulty, and can be achieved following the kind of computations performed in~\cite[Appendix~C]{Serfaty-15}.
Formally, the corresponding results to Theorem~\ref{th:mainGL} are the convergences
\[\frac{j_\e}{N_\e}\to\vi,\qquad\frac{\mu_\e}{N_\e}\to\md:=\curl\vi+\Hd,\qquad\frac{\curl A_\e}{N_\e}\to \Hd,\qquad\frac{E_\e}{N_\e}\to \Ed,\]
where the limiting triplet $(\vi,H,E)$ satisfies, in the dilute regime~\GLu{},
\begin{align}\label{eq:GL1gauge}
\begin{cases}
\partial_t\!\vi-\Ed=\nabla\!\pr+(\alpha-\Jb\beta)(\nabla^\bot\hat h-\hat F^\bot-2\!\vi)\md,\\
\partial_t\!\Hd=-\curl\Ed,\\
-\sigma\!\Ed=\vi+\nabla^\bot\Hd,\quad \Div\vi=0,
\end{cases}\end{align}
and in the critical regime~\GLd{},
\begin{align}\label{eq:GL2gauge}
\begin{cases}
\partial_t\!\vi-\Ed=\alpha^{-1}\nabla(\hat a^{-1}\Div(\hat a\!\vi))+(\alpha-\Jb\beta)(\nabla^\bot\hat h-\hat F^\bot-2\lambda\!\vi)\md,\\
\partial_t\!\Hd=-\curl \Ed,\\
-\sigma\!\Ed=\vi+\nabla^\bot\Hd,
\end{cases}\end{align}
while in the non-critical scalings~\GLup--\GLdp{} the equations are obtained from the above by removing the nonlinear interaction terms $\vi\!\md$. The structure of these equations is maybe more transparent at the level of the vorticity $\md:=\curl\vi+\Hd$: the system~\eqref{eq:GL1gauge} takes the form
\begin{align*}
\begin{cases}
\partial_t\!\md=\Div\big((\alpha-\Jb\beta)(\nabla\hat h-\hat F+2\!\vi^\bot)\md\big),\\
\sigma\partial_t\!\Hd-\triangle\Hd+\Hd=\md,\\
\Div\vi=0,\quad\curl\vi=\md-H
\end{cases}\end{align*}
while~\eqref{eq:GL2gauge} becomes for $\sigma>0$,
\begin{align*}
\begin{cases}
\partial_t\!\md=\Div\big((\alpha-\Jb\beta)(\nabla\hat h-\hat F+2\!\vi^\bot)\md\big),\\
\partial_t\!\dd-\alpha^{-1}\triangle\dd+\alpha^{-1}\Div(\dd\nabla\hat h)+\frac1\sigma\dd\\
\hspace{3cm}=-\frac1\sigma\hat a\nabla\hat h\cdot\nabla^\bot H+\Div\big((\alpha-\Jb\beta)(\nabla^\bot\hat h-\hat F^\bot-2\lambda\!\vi)\hat a\!\md\big),\\
\sigma\partial_t\!\Hd-\triangle\Hd+\Hd=\md,\\
\Div(\hat a\!\vi)=\dd,\quad\curl\vi=\md-H,
\end{cases}\end{align*}
that is, a continuity equation for $\md$ coupled with a linear heat equation for $\Hd$, and in the case~\eqref{eq:GL2gauge} further coupled with a convection-diffusion equation for the divergence $\dd:=\Div(\hat a\!\vi)$. For simplicity, we only focus in this work on the model without gauge~\eqref{eq:GL-1}.

\subsection{Further questions: homogenization regimes}\label{sec:homog}

So far, we have considered the mean-field regimes for the vortices with a pinning force $\nabla h$ which varies at the macroscopic scale.
However, the most interesting situation from the modeling viewpoint is to let the pinning weight $a$ oscillate quickly at some mesoscopic scale $\eta_\e\ll1$.
In real-life materials, the way in which the impurities are inserted typically leads them to be uniformly and randomly scattered in the sample.
This is naturally modeled as
\begin{align}\label{eq:def-pin-a}
\hat a(x):=\hat a^0\big(x,\tfrac1{\eta_\e}x\big)^{\eta_\e},
\end{align}
where for all $x$ the function $\hat a^0(x,\cdot)$ is a typical realization of some ($\e$-independent) nonnegative stationary random field, and the pinning force then takes the form
\begin{align}\label{eq:def-pin-h}
\nabla\hat h(x)=\nabla_2\hat h^0\big(x,\tfrac1{\eta_\e}x\big)+\eta_\e\nabla_1\hat h^0\big(x,\tfrac1{\eta_\e}x\big),
\end{align}
in terms of $\hat h:=\log\hat a$ and $\hat h^0:=\log\hat a^0$.
We refer to $\eta_\e$ as the ``pin separation'', and for simplicity we assume that $\hat a^0$ is periodic in its second variable.

This leads to the question of combining the mean-field limit for the Ginzburg-Landau vortex dynamics with a homogenization limit. In other words, can one perform the derivation of a limiting equation as $\e\downarrow 0$, $N_\e\uparrow\infty$, and $\eta_\e\downarrow 0$, and in which regimes does it hold?
While the homogenization of the (static) Ginzburg-Landau energy functional with pinning has been studied in some settings~\cite{ASS-01,Andre-Bauman-Phillips-03,DosSantos-13}, we believe that these homogenization questions in the dynamical case are particularly challenging.
They are in fact already very hard for just a finite number of vortices: studying the limit as $\eta \downarrow 0$ of the discrete dynamics~\eqref{ode1} with pinning force of the form~\eqref{eq:def-pin-h}
is a homogenization question for a system of nonlinear coupled ODEs
and is notoriously difficult.
This difficulty is related to the complexity of the collective effects of the interacting vortices and to the possible ``glassy'' properties predicted by physicists for such systems~\cite{Giamarchi-Bhatta-02} due to the subtle competition between vortex interactions and disorder.
Justifying suitable homogenized mean-field equations is thus a crucial question since such equations should enclose all the key dynamical properties of vortex matter; we briefly comment on it below.

\subsubsection{Diagonal and non-diagonal regimes}

As explained in Section~\ref{chap:modarghomog}, our modulated energy methods are not adapted to include homogenization effects: they only allow to treat a diagonal regime, that is, when the pin separation $\eta_\e$ tends very slowly to $0$, in which case the homogenization limit can simply be performed {\it after} the mean-field limit.
The limiting behavior of the rescaled supercurrent $\frac1{N_\e}j_\e$ is then reduced to that of the mean-field equations~\eqref{eq:GL1lim}--\eqref{eq:GL2'lim} with wiggly pinning force~\eqref{eq:def-pin-h}, that is, a (periodic) homogenization problem for the mean-field equations.

\begin{cor}\label{cor:homogdiag}
Let the same assumptions hold as in Theorem~\ref{th:mainGL}, with a wiggly pinning weight~\eqref{eq:def-pin-a}. In the regime~\GLd, we restrict to the parabolic case.
Then there exists a sequence $\eta_{\e,0}\downarrow0$ (depending on all the data of the problem) such that for $\eta_{\e,0}\ll\eta_\e\ll1$
the same conclusions hold as in Theorem~\ref{th:mainGL} in the form $\frac1{N_\e}j_\e-\tilde\vi_\e\to0$, where $\tilde\vi_\e$ denotes the unique global (smooth) solution of the corresponding mean-field equation~\eqref{eq:GL1lim}--\eqref{eq:GL2'lim} with $\nabla\hat h(x)$ replaced by the wiggly pinning force $\nabla_2 \hat h^0(x,\frac1{\eta_\e}x)$.
\end{cor}

In non-diagonal regimes, as our modulated energy approach fails, we only manage to justify the following minor rigorous result:
In the case with negligible interactions and negligible applied current, that is,
\[\alpha>0,~N_\e\ll\Log,~\tfrac{N_\e}\Log\ll\lambda_\e\lesssim1,~h=\lambda_\e\hat h,~F=\lambda_\e'\hat F,~\lambda_\e'\ll\lambda_\e,\]
the vorticity is shown to remain ``stuck'' in the limit, that is, to converge at all times to its initial data (cf.\@ Proposition~\ref{prop:smallforcing}). This is a particular case of the \emph{stick-slip phenomenon} discussed below.
The rigorous treatment of all other regimes, including the commutation of the limits $\e\downarrow 0$, $N\uparrow\infty$, and $\eta\downarrow 0$, is left as an open question. For particle systems with smooth interactions, this commutation problem is easier to settle and is discussed in the forthcoming work~\cite{DS-18}.

\subsubsection{Homogenization of mean-field equations}
In view of Corollary~\ref{cor:homogdiag}, it is natural to consider the homogenization limit of the mean-field equations~\eqref{eq:GL1lim}--\eqref{eq:GL2'lim} with wiggly pinning force $\nabla\hat h(x)=\nabla_2 \hat h^0(x,\frac1{\eta_\e}x)$.
This topic is very delicate on its own, with the same kind of difficulties as for the homogenization of the discrete system~\eqref{ode1} of coupled ODEs.
We first consider the scaling with negligible vortex interactions, which leads to a well-defined linear limiting equation, and we discuss its stick-slip properties, before turning to the general nonlinear case.

\begin{enumerate}[(i)]
\item\emph{Negligible interactions: linear stick-slip law.}\\
In the regime of negligible vortex interactions (cf.~\GLup--\GLdp), particles are independent and the mean-field equations are reduced to a linear continuity equation~\eqref{eq:GL1'lim-bis} for the vorticity (with a compressible vector field), which is much easier to handle.
The homogenization of such an equation is easily understood in 1D~\cite{ACJ-96}, but it becomes surprisingly more subtle in higher dimensions: the 2D periodic case was first investigated by Menon~\cite{Menon-02} and is still partially open.
The situation becomes much simpler if the applied current $\hat F$ is a constant and if the wiggly pinning weight is independent of the macroscopic variable, that is,
\begin{align}\label{eq:def-a-pin+}
\hat a(x):=\hat a^0\big(\tfrac1{\eta_\e}x\big)^{\eta_\e},\qquad \nabla\hat h(x)=\nabla\hat h_0(\tfrac1{\eta_\e}x).
\end{align}
The wiggly linear continuity equation for the mean-field vorticity $\tilde\md_\e$ then takes the form
\[\partial_t\tilde\md_\e=\Div\!\Big((\alpha-\Jb\beta)\big(\nabla\hat h_0(\tfrac\cdot{\eta_\e})-\hat F\big)\,\tilde\md_\e\Big),\]
which is known as a \emph{washboard system} in the physics literature.
The homogenization of this equation is a particular case of the nonlinear results in~\cite{Dalibard-09} (see also~\cite{E-92,Jabin-Tzavaras-09} in the incompressible case and~\cite{Frenod-Hamdache-96,Dalibard-08} in the linear Hamiltonian case), but a more accurate asymptotic description without well-preparedness assumption is postponed to a forthcoming work~\cite{DS-18}.

\noindent
The behavior of the vorticity $\tilde\md_\e$ is intuitively easily understood:
If $\hat F=0$, the vorticity is attracted towards the local wells of the pinning potential $\eta_\e\hat h_0(\frac\cdot{\eta_\e})$. Otherwise, a constant applied force $\hat F\neq 0$ can be absorbed into the term $\nabla\hat h_0(\frac\cdot{\eta_\e})$ by adding an affine function to the pinning potential, which effectively tilts the potential landscape into a washboard-shaped graph. Beyond some positive value of the intensity $|\hat F|$, the tilted potential has no local minimum, leading the particle to fall in the direction of $\hat F$, while below this critical value the vorticity remains pinned.
Such a behavior is known as a \emph{stick-slip law}, and the critical value of the applied force corresponds to the so-called \emph{depinning current}.
More precisely, the dynamics of the homogenized vorticity $\tilde\md$ is characterized by a linear transport equation
\[\partial_t\tilde\md=-\Div(V(\hat F)\,\tilde\md),\]
with homogenized velocity field given by
\begin{align}\label{eq:def-eff-vel}
V(\hat F):=-\int_Q\Gamma^{\hat F}(y)\,d\mu^{\hat F}(y),
\end{align}
where $\mu^{\hat F}$ is an invariant measure for the dynamics associated with the periodic vector field $\Gamma^{\hat F}:=(\alpha-\Jb\beta)(\nabla\hat h_0-\hat F)$ on the torus $Q$.
The {stick-slip} behavior is easily recovered from this formula (cf.\@ Figure~\ref{fig:response-lin}): for small $|\hat F|$ any invariant measure $\mu^{\hat F}$ is concentrated at fixed points, hence $V(\hat F)=0$, meaning that the vorticity gets stuck, while for large $|\hat F|$ the measure $\mu^{\hat F}$ becomes non-trivial, hence $V(\hat F)\ne0$, meaning that the vorticity is transported.
Note that the response $\hat F\mapsto V(\hat F)$ is not smooth at the depinning threshold, but typically has a square-root behavior,
\begin{align}\label{eq:1/2-law}
V(\hat F)\propto(|\hat F|-|\hat F_{c}|)^{1/2},
\end{align}
for $|\hat F|$ close to the critical intensity $|\hat F_{c}|$, cf.~\cite{DS-18}.
Such a frictional stick-slip dynamics is well-known in various 1D systems~\cite{Bhattacharya-99,Grunewald-05,Dirr-Yip-06}.

\item\emph{Non-negligible interactions: nonlinear stick-slip law.}\\
In the regimes~\GLu{} and~\GLd, vortex interactions can no longer be neglected in the mean-field equations~\eqref{eq:GL1lim} and~\eqref{eq:GL2lim}. Considering these equations with wiggly pinning force~\eqref{eq:def-a-pin+} and taking the homogenization limit, a formal $2$-scale expansion leads to nonlocal nonlinear homogenized continuity equations for the homogenized vorticity~$\tilde\md$: setting $W(\tilde\vi;\hat F)(x):=V\big(\hat F-2\tilde\vi^\bot(x)\big)$ with $V$ defined as in~\eqref{eq:def-eff-vel}, we find in the case~\eqref{eq:GL1lim},
\begin{align}\label{eq:GL1lim-ter}
\begin{cases}
\partial_t\tilde\md=-\Div\!\big(W(\tilde\vi;\hat F)\,\tilde\md\big),\\
\curl\tilde\vi=\tilde\md,
\end{cases}\end{align}
and in the case~\eqref{eq:GL2lim} with $\alpha=1$, $\beta=0$,
\begin{align*}
\begin{cases}
\partial_t\tilde\md=-\Div\!\big(W(\tilde\vi;\hat F)\,\tilde\md\big),\\
\partial_t\tilde\dd=\triangle\tilde\dd-\Div\!\big(W(\tilde\vi;\hat F)^\bot\tilde\md\big),\\
\curl\tilde\vi=\tilde\md,\quad\Div\!\tilde\vi=\tilde\dd.
\end{cases}\end{align*}
A rigorous justification of this homogenization limit is particularly challenging due to the nonlocal nonlinear character of the mean-field equations~\eqref{eq:GL1lim}--\eqref{eq:GL2lim} and to their strong instability as $\eta_\e\downarrow0$.
As shown in a forthcoming work~\cite{DS-18}, these questions can be partially solved if Coulomb interactions in~\eqref{eq:GL1lim} are replaced by smooth interactions, that is, if we rather consider a mean-field equation of the form
\[\partial_t\tilde\md_\e=\Div\big((\nabla\hat h(\tfrac\cdot{\eta_\e})-\hat F-2\nabla g\ast\tilde\md_\e)\,\tilde\md_\e\big),\]
for some smooth interaction potential $g$. The relation $\curl\tilde\vi=\tilde\md$ in the formal homogenized equation~\eqref{eq:GL1lim-ter} is then replaced by $\tilde\vi=\nabla^\bot g\ast\tilde\md$.
Note however that the well-posedness of the homogenized equation remains unclear since the vector field $W(\tilde\vi;\hat F)$ is in general not Lipschitz continuous even for smooth $\tilde\vi$ due to~\eqref{eq:1/2-law}.

\noindent
Heuristically, the stick-slip picture remains the same as in the case of negligible interactions:
For small $\hat F$ the vorticity $\tilde\md$ first spreads due to the vortex repulsive interaction until the interaction force $\tilde\vi$ becomes small enough such that $W(\tilde\vi;\hat F)=0$ and the vorticity then remains stuck.
The mean velocity of the system
\[V_m(\hat F):=\lim_{t\uparrow\infty}\frac1t\int_0^t\int_{\R^2}W(\tilde\vi^s,\hat F)\,d\tilde\md^s\,ds\]
is thus expected to satisfy a similar stick-slip law.
Nevertheless, the precise picture should be very different at the depinning threshold:
the mean velocity is expected to be non-smooth, but, compared to the case without interaction~\eqref{eq:1/2-law}, the value $|\hat F_c|$ of the threshold and the value $\frac12$ of the depinning exponent are expected to be radically different, in link with the glassy properties of the system, as predicted in the physics literature~\cite{Narayan-Fisher-92,NSTL-92,Chauve-Giamarchi-LeDoussal-00} (see also~\cite[Section~5]{Giamarchi-Bhatta-02}).
Indeed, due to the competition between the pinning potential and the vortex interaction, the vortices are expected to move as a coherent elastic object in a heterogeneous medium, yielding very particular glassy properties, but a rigorous justification is still missing.

\noindent
Since vortices are elastically coupled by the interaction, the problem is formally analogous to the motion of elastic systems in disordered media, which is indeed the framework considered in the above-cited physics papers.
In this spirit, a considerable attention has been devoted in the physics community to the simpler Quenched Edwards-Wilkinson model for elastic interface motion in disordered media~\cite{Kardar-97,BN-04}.
These questions are also related (although again for different models) to the recent rigorous homogenization results for the forced mean curvature equation and for more general geometric Hamilton-Jacobi equations~\cite{Armstrong-Cardaliaguet-15}.
\end{enumerate}

\begin{rem}
Although deriving a nonlinear stick-slip law based on the mesoscopic model seems out of reach, a rigorous analysis is possible on a very short timescale:
For $t=O(\eta_\e)$, in each (mesoscopic) periodicity cell, the vorticity is shown to concentrate on the support of the invariant measure associated with the initial vector field (cf.\@ Proposition~\ref{prop:locrelax}). This mesoscopic initial-boundary layer result is in agreement with the above description of the dynamics on larger timescales as transport takes place ``along'' invariant measures.
\end{rem}

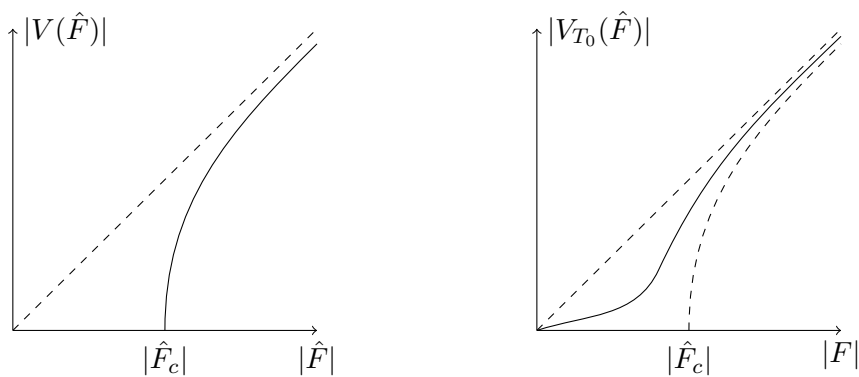
\begin{figure}[!htp]
\begin{center}
\begin{subfigure}[b]{0.4\textwidth}
\centering\begin{tikzpicture}[scale=0.4]
\draw[->](0,0)--(5,0)node[below]{$|\hat F_{c}|$}--(10,0)node[below]{$|\hat F|$};
\draw[->](0,0) -- (0,10)node[right]{$|V(\hat F)|$};
\draw[dashed] (0,0) to (10,10);
\draw (5,0) to[out=90, in=180+45] (10,9.5);
\end{tikzpicture}
\caption{\label{fig:response-lin} No thermal noise: stick-slip law.}
\end{subfigure}
\qquad
\begin{subfigure}[b]{0.4\textwidth}
\centering\begin{tikzpicture}[scale=0.4]
\draw[->](0,0)--(5,0)node[below]{$|\hat F_c|$}--(10,0)node[below]{$|F|$};
\draw[->](0,0) -- (0,10)node[right]{$|V_{T_0}(\hat F)|$};
\draw[dashed] (0,0) to (10,10);
\draw[dashed] (5,0) to[out=90, in=180+45] (10,9.5);
\draw (0,0) to[out=18,in=180+65] (4,2) to[out=65,in=180+44](10,9.75);
\end{tikzpicture}
\caption{\label{fig:response-temp}With thermal noise: Arrhenius law.}
\end{subfigure}
\caption{Typical current-velocity characteristics in the case of negligible vortex interactions.}
\end{center}
\end{figure}

\subsubsection{System with thermal noise}
Different stochastic variants of the Ginzburg-Landau equation have been introduced in the physics literature in order to model the effect of thermal noise in type-II superconductors~\cite{Schmid-69,Hohenberg-Halperin-77,Deang-Du-Gunzburger-01} (see also~\cite{Stoof-99,Gardiner-Anglin-Fudge-02,Gardiner-Davis-03,Swislocki-Deuar-16} for corresponding stochastic versions of the mixed-flow Gross-Pitaevskii equation to model thermal and quantum noise in Bose-Einstein condensates).
Although we do not study here the mean-field limit problem for such models, we expect that for a finite number $N$ of vortices in the limit $\e\downarrow0$ the thermal noise acts on the vortices as~$N$ independent Brownian motions: more precisely, in the regime~\GLu, the limiting trajectories $(x_i)_{i=1}^N$ of the~$N$ vortices are expected to satisfy the following system of coupled SDEs instead of~\eqref{ode1} (cf.\@ e.g.~\cite[Section~III.B]{WE-94}),
\begin{gather}\label{eq:partsyst-temp}
(\alpha+\Jb\beta)dx_i=\Big(-N^{-1}\nabla_{x_i}W_N(x_1,\ldots,x_N)-\nabla\hat h(x_i)+\hat F(x_i)\Big)dt+\sqrt{2T}dB^t_i,\\
W_N(x_1,\ldots,x_N):=-\pi\sum_{i\ne j}^N\log |x_i-x_j|,\nonumber
\end{gather}
where $B_1,\ldots,B_N$ are $N$ independent 2D Brownian motions.
Such macroscopic phenomenological models, where the thermal noise acts via random Langevin kicks, are abundantly used by physicists~\cite{BFGLV-94,Giamarchi-Bhatta-02,Reichhardt-17}.
In the case of a diverging number of vortices $N_\e\gg1$, in the regime~\GLu{}, it is then natural to postulate that a good phenomenological model for the (formal) mean-field supercurrent density $\vi$ is given as the mean-field limit of the particle system~\eqref{eq:partsyst-temp}, that is, the following viscous version of~\eqref{eq:GL1lim},
\begin{align}\label{eq:GL1lim-temp}
\begin{cases}
\partial_t\!\vi=\nabla\!\pre+(\alpha-\Jb\beta)(\nabla^\bot\hat h-\hat F^\bot-2\!\vi)\curl\!\vi+T\triangle\!\vi,\\
\Div\!\vi=0, \quad\vi\!|_{t=0}=\vi^\circ.
\end{cases}
\end{align}
In the regimes~\GLd{} and~\GLt{} we rather consider corresponding viscous versions of~\eqref{eq:GL2lim} and~\eqref{eq:GL3lim},
while in the regimes~\GLup{} and~\GLdp{} these viscous equations should be replaced by their versions without interaction term.
In this viscous context, we may now consider the homogenization problem for the mean-field model~\eqref{eq:GL1lim-temp} with wiggly pinning force $\nabla\hat h(x)=\nabla\hat h_0(\frac1{\eta_\e}x)$.
We naturally restrict attention to the critical scaling for the temperature, that is, $T:=\eta_\e T_0$ for some fixed $T_0>0$.
We first consider the scaling with negligible vortex interactions before turning to the general nonlinear case.

\begin{enumerate}[(i)]
\item\emph{Negligible interactions: Arrhenius law.}\\
If interactions are neglected, we are reduced to the following wiggly linear continuity equation for the mean-field vorticity $\tilde\md_\e$,
\begin{align}\label{eq:subcrit-temp}
\partial_t\tilde\md_\e=\Div\big((\alpha-\Jb\beta)(\nabla_2\hat h_0(\tfrac\cdot{\eta_\e})-\hat F)\,\tilde\md_\e\big)+\eta_\e T_0\triangle\tilde\md_\e.
\end{align}
The homogenization of this equation is a particular case of the nonlinear results in~\cite{Dalibard-06}, although the argument can be considerably simplified here, cf.~\cite{DS-18}.
The dynamics of the homogenized vorticity $\tilde\md$ is characterized by a linear transport equation
\[\partial_t\tilde\md=-\Div(V_{T_0}(\hat F)\,\tilde\md),\]
with homogenized velocity field given by the following viscous analogue of~\eqref{eq:def-eff-vel},
\begin{align}\label{eq:def-VT0}
V_{T_0}(\hat F):=-\int_Q\Gamma^{\hat F}d\mu_{T_0}^{\hat F},
\end{align}
where $\mu_{T_0}^{\hat F}$ is the $T_0$-viscous invariant measure for the dynamics associated with the periodic vector field $\Gamma^{\hat F}:=(\alpha-\Jb\beta)(\nabla\hat h^0-\hat F)$ on the torus $Q$, that is, the unique probability measure on $Q$ satisfying
\begin{align*}
T_0\triangle\mu^{\hat F}_{T_0}+\Div\!\!\big(\Gamma^{\hat F}\mu^{\hat F}_{T_0}\big)=0.
\end{align*}
For $T_0>0$, since the viscous invariant measure $\mu^{\hat F}_{T_0}$ vanishes nowhere on $Q$, we find $V_{T_0}(\hat F)\ne0$ for all $\hat F\ne0$: the vorticity can never get stuck in local wells of the pinning potential. The precise behavior of $V_{T_0}(\hat F)$ for $\hat F$ close to $0$ is of particular interest. Heuristically, the current $\hat F\ne0$ tilts the energy landscape, and the energy barriers of size $\osc\hat h_0:=\max\hat h_0-\min\hat h_0$ are overcome by thermal activation even for small $F_0\ne0$. The velocity law for this so-called thermally assisted flux flow is expected to satisfy the classical Arrhenius law from statistical thermodynamics (cf.\@ e.g.~\cite[Section~5.1]{Giamarchi-Bhatta-02}),
\begin{align}\label{eq:Arrhenius}
V_{T_0}(\hat F)\propto T_0^{-1}\exp\big(\!-T_0^{-1}\osc\hat h_0\big)\,\hat F,
\end{align}
for $|\hat F|\ll T_0\ll1$, that is, the response is linear but exponentially small with respect to the inverse temperature. This is easily checked in 1D~\cite{DS-18} and is related to the Eyring-Kramers formula~\cite{BEGK-04,Helffer-Klein-Nier-04}.
The typical velocity law is plotted in Figure~\ref{fig:response-temp}.

\item\emph{Non-negligible interactions: creep law.}\\
We turn to the homogenization limit of equation~\eqref{eq:GL1lim-temp} with wiggly pinning force $\nabla\hat h(x)=\nabla\hat h_0(\frac1{\eta_\e}x)$ and with $T=\eta_\e T_0$. A  formal $2$-scale expansion leads to the nonlocal nonlinear homogenized continuity equation~\eqref{eq:GL1lim-ter} for the homogenized vorticity $\tilde\md$ with $W(\tilde\vi;\hat F)$ replaced by its viscous analogue $W_{T_0}(\tilde\vi;\hat F)(x):=V_{T_0}(\hat F-2\tilde\vi^\bot(x))$ with $V_{T_0}$ defined as in~\eqref{eq:def-VT0}.
A rigorous justification of this homogenization limit is particularly challenging but we show in a forthcoming work~\cite{DS-18} that it can be entirely solved if Coulomb interactions in~\eqref{eq:GL1lim-temp} are replaced by smooth interactions, and the homogenized equation is then well-posed.

\noindent
As in the case without temperature, due to the competition between pinning and vortex interactions, the precise dynamical properties of the homogenized vorticity are expected to change dramatically
with respect to the case of negligible interactions,
in link with the expected glassy properties of the system~\cite{Giamarchi-Bhatta-02}.
The main manifestation is visible in the low-current low-temperature limit ($|\hat F|\ll T_0\ll1$), where the linear Arrhenius law~\eqref{eq:Arrhenius} is now expected to break down, being replaced by a so-called \emph{creep law}: the mean velocity is expected to depend nonlinearly on the current and to have all vanishing derivatives with respect to $\hat F$ at $0$.
This was first predicted by physicists for related elastic interface motion models~\cite{Nat-87,Ioffe-Vinokur-87} and then adapted to vortex systems~\cite{FGLV-89,Nattermann-90,Giamarchi-LeDoussal-95,Chauve-Giamarchi-LeDoussal-98,Chauve-Giamarchi-LeDoussal-00} (see also~\cite[Section~5]{Giamarchi-Bhatta-02} and references therein), but a rigorous justification is still missing.
Note that the key influence of vortex interactions on the dynamics is exemplified in a simplified 1D model in~\cite[Section~IV]{WE-94}.
\end{enumerate}

\newpage
\section{Discussion of the mesoscopic model}
For future reference, note that in each of the considered regimes~\GLu, \GLd, \GLt, \GLup, \GLdp, and \GP, due to the explicit choice~\eqref{eq:formf} of the zeroth-order term $f$, the following scalings hold,
\begin{enumerate}[(a)]
\item \emph{Dissipative case, non-decaying setting:}
\begin{gather}\label{eq:scalingshFf}
\|\nabla h\|_{W^{1,\infty}}\lesssim 1\wedge\lambda_\e,\qquad\|F\|_{W^{1,\infty}}\lesssim \lambda_\e,\\
\|f\|_{W^{1,\infty}}\lesssim 1\wedge\lambda_\e+\lambda_\e^2\Log^2\lesssim \lambda_\e^2\Log^2;\nonumber
\end{gather}
\item \emph{Conservative case, decaying setting:}
\begin{gather}\label{eq:scalingshFfdec}
\|\nabla h\|_{H^1\cap W^{1,\infty}}\lesssim 1,\qquad\|F\|_{H^1\cap W^{1,\infty}}\lesssim \lambda_\e,\\
\|f\|_{H^1\cap W^{1,\infty}}\lesssim 1+\lambda_\e^2\Log^2\lesssim N_\e^2.\nonumber
\end{gather}
\end{enumerate}

\subsection{Derivation of the modified mesoscopic model}\label{disc-model}
In this section we justify the modified model~\eqref{eq:GL-1} based on the 2D mixed-flow Ginzburg-Landau model~\eqref{eq:Gorkov-Eliashberg-mixed-flow} without gauge.
For that purpose, as in~\cite{Tice-10,S-Tice-11}, we transform the rescaled order parameter $\frac1{\sqrt a}w_\e$ in order to turn the Neumann boundary condition into a homogeneous one, which makes the applied electric current $J_{\text{ex}}$ appear as a bulk term in the equation. For that purpose, we assume that $a=1$ holds on the boundary $\partial\Omega$, and that the total incoming current equals the total outgoing current, that is, $\int_{\partial\Omega} n\cdot J_{\text{ex}}=0$. We then have $\int_{\partial\Omega} an\cdot J_{\text{ex}}=0$, so that there exists a unique solution $\psi\in H^1(\Omega)$ of
\[\begin{cases}
\Div( a\nabla\psi)=0,&\text{in $\Omega$},\\
n\cdot\nabla\psi=n\cdot J_{\text{ex}},&\text{on $\partial \Omega$}.
\end{cases}\]
Defining the modified order parameter $u_\e:= e^{-i\Log \psi}\frac1{\sqrt a}w_\e$, a straightforward computation leads to
\begin{align}\label{eq:GL-1-0}
\begin{cases}\lambda_\e(\alpha+i\Log \beta)\partial_tu_\e=\triangle u_\e+\frac{au_\e}{\e^2}(1-|u_\e|^2)\\
\hspace{4cm}+\nabla h\cdot\nabla u_\e+i\Log F^\bot\cdot\nabla u_\e+ fu_\e,&\text{in $\R^+\times\Omega$,}\\n\cdot\nabla (u_\e\sqrt a)=0,&\text{on $\R^+\times\partial\Omega$,}\\u_\e|_{t=0}=u_\e^\circ,\end{cases}
\end{align}
where we have set
\begin{align}\label{eq:definhFf}
h:=\log a,\qquad F:=-2\nabla^\bot\psi,\qquad \text{and}\qquad f:=\frac{\triangle\sqrt a}{\sqrt a}-\frac14\Log ^2|F|^2.
\end{align}
Note that the vector field $F$ satisfies $\Div F=\curl(aF)=0$.
In order to avoid delicate boundary issues\footnote{Another way to avoid boundary issues is to rather consider the equation on the torus. The total degree of the order parameter $u_\e$ on a period would then however vanish: in order to describe a non-trivial vorticity with distinguished sign, we should rather work with the Ginzburg-Landau model with gauge. As explained in Section~\ref{sec:gauge}, working with the gauge does not cause any major difficulty, but it makes all computations heavier, which we wanted to avoid.},
a natural approach consists in sending the boundary $\partial\Omega$ to infinity and studying the corresponding problem on the whole plane $\R^2$.
The assumption $a|_{\partial\Omega}=1$ is then replaced by
\[a(x)\to1 \quad\text{(that is, $h(x)\to0$)}\qquad \text{and}\qquad \nabla h(x)\to0, \qquad\text{as $|x|\uparrow\infty$},\]
while $F,f$ are simply assumed to be bounded.
Noting that this condition implies $2\nabla\sqrt a=\sqrt a\nabla h\to0$ at infinity, the Neumann boundary condition in~\eqref{eq:GL-1-0} formally translates into $\frac x{|x|}\cdot\nabla u_\e\to0$ at infinity.
Further imposing the natural condition $|u_\e|\to1$ at infinity, we look for a global solution $u_\e:\R^+\times\R^2\to\C$ of~\eqref{eq:GL-1} with fixed total degree $\deg u_\e=N_\e\in\Z$, and with
\begin{align*}
|u_\e|\to1\qquad\text{and}\qquad \frac x{|x|}\cdot \nabla u_\e\to0,\qquad\text{as $|x|\uparrow\infty$}.
\end{align*}
If the fields $F$ and $f$ do not decay at infinity,
the solution $u_\e$ may display a possibly complicated advection structure at infinity, as explained in Section~\ref{chap:wellposed} below: it is then unclear whether the above properties at infinity are satisfied and even whether the total degree of $u_\e$ is well-defined.
As a more precise description of $u_\e$ at infinity is anyway not relevant for our purposes, it is not pursued here.

For simplicity, we may rather truncate $F$ and $f$ at infinity, thus focusing on the local behavior of the solution $u_\e$ in a bounded set. In the conservative case, our results are limited to this decaying setting. Note that one of the conditions $\Div\! F=\curl(aF)=0$ must then be relaxed: we may for instance truncate $\psi$ and define $F$ via formula~\eqref{eq:definhFf}, so that only the condition $\Div\! F=0$ is preserved. Since there is no advection at infinity in this setting, the solution $u_\e$ will be shown to satisfy the desired properties at infinity.

\begin{rem}
Rather than normalizing $w_\e$ by the expected density $\sqrt a$, another natural choice is to normalize by a minimizer $\gamma_\e$ of the weighted Ginzburg-Landau energy~\cite{Lassoued-Mironescu}, that is, a nonvanishing  solution of
\[\begin{cases}-\triangle \gamma_\e=\frac{\gamma_\e}{\e^2}(a-|\gamma_\e|^2), &\text{in $\Omega$,}\\n\cdot \nabla\gamma_\e=0,&\text{on $\partial\Omega$}.\end{cases}\]
Setting $\tilde u_\e:=e^{-i\Log\psi}\frac1{\gamma_\e}w_\e$ with $\psi$ as before, we find
\begin{align*}
\lambda_\e(\alpha+i\Log\beta)\partial_t\tilde u_\e=\triangle \tilde u_\e+\frac{\gamma_\e^2\tilde u_\e}{\e^2}(1-|\tilde u_\e|^2)+\nabla \tilde h\cdot\nabla \tilde u_\e+i\Log \tilde F^\bot\cdot \nabla \tilde u_\e+\tilde f\tilde u_\e,
\end{align*}
in terms of $\tilde h:=\log\gamma_\e^2$, $\tilde F:=-2\nabla^\bot\psi$, and $\tilde f:=-\frac14|F|^2$, and we are thus reduced to a similar equation as before.
\end{rem}

\subsection{Well-posedness of the modified mesoscopic model}\label{chap:wellposed}
In this section, we address the global well-posedness of the modified mesoscopic model~\eqref{eq:GL-1}, both in the dissipative and in the conservative cases. In the dissipative case, global well-posedness is established in the space $\Ld^\infty_\loc(\R^+;H^1_{\uloc}(\R^2;\C))$ for general non-decaying data $h,F,f$, but no precise description of the solution at infinity is obtained, due to a possibly subtle advection structure at infinity: it is not even clear whether the total degree of the solution is well-defined.
This difficulty originates in the possibility of instantaneous creation of many vortex dipoles at infinity for fixed $\e>0$ due to pinning and applied current, although these dipoles are shown to necessarily disappear at infinity in the limit $\e\downarrow0$ e.g.\@ as a consequence of our mean-field results.
In contrast, in the conservative case, we must restrict to decaying data $h,F,f$, in which case no advection can occur at infinity. As is classical since the work of Bethuel and Smets~\cite{Bethuel-Smets-07} (see also~\cite{Miot-09}), we then consider global well-posedness in an affine space $\Ld^\infty_\loc(\R^+;U_\e+H^1(\R^2;\C))$ for some ``reference map'' $U_\e$, which is typically chosen smooth and equal (in polar coordinates) to $e^{iN_\e\theta}$ outside a ball at the origin, for some given $N_\e\in\Z$, thus imposing for $u_\e$ a fixed total degree $N_\e$ at infinity. More generally, we may consider the following space of admissible reference maps,
\begin{multline*}
E_1(\R^2):=\{U\in \Ld^{\infty}(\R^2;\C): \nabla^2 U\in H^1(\R^2;\C),\nabla|U|\in\Ld^2(\R^2),1-|U|^2\in\Ld^2(\R^2),\\
\nabla U\in\Ld^p(\R^2;\C)~\forall p>2\}.
\end{multline*}
Our global well-posedness results are summarized in the following; finer results and detailed proofs are given in Appendix~\ref{app:wellposed}, including additional regularity statements.

\begin{prop}[Well-posedness of the mesoscopic model]\label{prop:globGL}$  $
\begin{enumerate}[(i)]
\item \emph{Dissipative case ($\alpha>0$, $\beta\in\R$), non-decaying setting:}\\
Let $h\in W^{1,\infty}(\R^2)$, $a:=e^h$, $F\in\Ld^\infty(\R^2)^2$, $f\in\Ld^\infty(\R^2)$, and $u_\e^\circ\in H^1_\uloc(\R^2;\C)$. Then there exists a unique global solution $u_\e\in\Ld^\infty_\loc(\R^+; H^1_\uloc(\R^2;\C))$ of~\eqref{eq:GL-1} in $\R^+\times\R^2$ with initial data $u^\circ_\e$, and this solution satisfies $\partial_tu_\e\in \Ld^\infty_\loc(\R^+;\Ld^2_\uloc(\R^2;\C))$.
\item \emph{Conservative case ($\alpha=0$, $\beta=1$), decaying setting:}\\
Let $h\in W^{3,\infty}(\R^2)$, $\nabla h\in H^2(\R^2)^2$, $a:=e^h$, $F\in H^3\cap W^{3,\infty}(\R^2)^2$ with $\Div F=0$, $f\in H^2\cap W^{2,\infty}(\R^2)$, and $u_\e^\circ\in U+H^2(\R^2;\C)$ for some $U\in E_1(\R^2)$. Then there exists a unique global solution $u_\e\in\Ld^\infty_\loc(\R^+; U+H^2(\R^2;\C))$ of~\eqref{eq:GL-1} in $\R^+\times\R^2$ with initial data $u^\circ_\e$, and this solution satisfies $\partial_t u_\e\in\Ld^\infty_\loc(\R^+;\Ld^2(\R^2;\C))$.
\qedhere
\end{enumerate}
\end{prop}

\begin{proof}
Item~(i) follows from Proposition~\ref{prop:globGLappnondec}.
We turn to item~(ii). By Proposition~\ref{prop:globGLapp}(ii), the assumptions in the above statement ensure the existence of a unique global solution $u_\e\in\Ld^\infty_\loc(\R^+; U+H^2(\R^2;\C))$. This directly implies that $\triangle u_\e$, $\nabla h\cdot\nabla u_\e$, $F^\bot\cdot\nabla u_\e$, and $fu_\e$ belong to $\Ld^\infty_\loc(\R^+;\Ld^2(\R^2;\C))$. Using the Sobolev embedding of $H^1(\R^2)$ into $\Ld^2\cap\Ld^6(\R^2)$, and decomposing $u_\e(1-|u_\e|^2)$ in terms of $u_\e=U+\hat u_\e$ with $\hat u_\e\in \Ld^\infty_\loc(\R^+;H^2(\R^2;\C))$, we further deduce that $u_\e(1-|u_\e|^2)$ belongs to $\Ld^\infty_\loc(\R^+;\Ld^2(\R^2;\C))$. Inserting this into equation~\eqref{eq:GL-1} yields the claimed integrability of $\partial_tu_\e$.
\end{proof}

Although a detailed proof is given in Appendix~\ref{app:wellposed}, we include here a brief description of the strategy. In the dissipative case with decaying data $h,F,f$, the arguments in~\cite{Bethuel-Smets-07,Miot-09} are easily adapted to the present context with both pinning and applied current. The conservative regime is more delicate and we then use the structure of the equation to make a change of variables that usefully transforms the first-order terms into zeroth-order ones. The additional regularity assumptions in item~(ii) above are precisely needed for this transformation to be well-behaved. Finally, the general result stated in item~(i) for the dissipative case with non-decaying data is deduced from the corresponding result with decaying data by a careful approximation argument in the space $H^1_\uloc(\R^2;\C)$.

\section{Preliminaries on the mean-field equations}\label{chap:prel-limit-eqn}
As explained, it is convenient to first compare the rescaled supercurrent density $\frac1{N_\e}j_\e$ with an intermediate $\e$-dependent approximation $\vi_\e:\R^+\times\R^2\to\R^2$, which is better adapted to the $\e$-dependence of the pinning potential and which is shown in a second step to converge to the correct limit $\vi$.
In all considered regimes, we derive equations for $\vi_\e$ of the form
\begin{align}\label{eq:genveps}
\partial_t \!\vi_\e=\nabla\!\pre_\e+\Gamma_\e\,\curl\!\vi_\e,\qquad \vi_\e\!|_{t=0}=\vi_\e^\circ,
\end{align}
for some smooth pressure $\pre_\e:\R^2\to\R$ and some smooth vector field $\Gamma_\e:\R^2\to\R^2$. The pressure will either be taken proportional to $a^{-1}\Div(a\!\vi_\e)$, or be  the Lagrange multiplier associated with the constraint $\Div(a\!\vi_\e)=0$.
Until Section~\ref{chap:MFL-GL}, we only manipulate these quantities $\vi_\e,\pre_\e,\Gamma_\e$ formally, while the suitable choice of the equation will be exploited later. In order to ensure that all our computations are licit, the following integrability and smoothness assumptions are needed.

\begin{as}\label{as:apveps}$  $
\begin{enumerate}[(a)]
\item \emph{Dissipative case ($\alpha>0$, $\beta\in\R$):}\\
There exists some $T>0$ such that for all $\e>0$, $t\in[0,T)$, and $q>2$,
\begin{gather*}
\|(\vi_\e^t,\nabla\!\vi_\e^t)\|_{(\Ld^2+\Ld^q)\cap\Ld^\infty}\lesssim_{t,q}1,\quad\|\curl\!\vi_\e^t\!\|_{\Ld^1\cap\Ld^\infty}\lesssim_t1,\quad \|\!\Div(a\!\vi_\e^t)\|_{\Ld^2\cap\Ld^\infty}\lesssim_t1,\\
\|\!\pre_\e^t\!\|_{\Ld^2\cap\Ld^\infty}\lesssim_t\lambda_\e^{-1/2}\wedge\lambda_\e^{-1},\quad\|\nabla\!\pre_\e\!\|_{\Ld^2_t\Ld^2}\lesssim_t1\wedge\lambda_\e^{-1},\\
\|\partial_t\!\vi_\e^t\!\|_{\Ld^2\cap\Ld^\infty}\lesssim_t1+\lambda_\e^{-1/2},\quad\|\partial_t\!\vi_\e\!\|_{\Ld^2_t\Ld^2}\lesssim_t1,\quad\|\partial_t\!\pre_\e^t\!\|_{\Ld^2_t\Ld^2}\lesssim_t\lambda_\e^{-1},\\
\|\Gamma_\e^t\|_{W^{1,\infty}}\lesssim_t1,\quad\|\partial_t\Gamma_\e\|_{\Ld^2_t\Ld^2}\lesssim_t1.
\end{gather*}
\item \emph{Conservative case ($\alpha=0$, $\beta=1$):}\\
There exists some $T>0$ such that for all $\e>0$, $t\in[0,T)$, $q>2$, and $2<p<\infty$,
\begin{gather*}
\|(\vi_\e^t,\nabla\!\vi_\e^t)\|_{(\Ld^2+\Ld^q)\cap\Ld^\infty}\lesssim_{t,q}1,\quad\|\curl\!\vi_\e^t\!\|_{\Ld^1\cap\Ld^\infty}\lesssim_t1\\
\|\!\pre_\e^t\!\|_{\Ld^q\cap\Ld^\infty}\lesssim_{t,q}1,\quad\|\nabla\!\pre_\e^t\!\|_{\Ld^2\cap\Ld^\infty}\lesssim_t1,\quad\|\partial_t\!\vi_\e^t\!\|_{\Ld^2}\lesssim_t1,\quad\|\partial_t\!\pre_\e^t\!\|_{\Ld^p}\lesssim_{t,p}1,\\
\|\Gamma_\e^t\|_{W^{1,\infty}}\lesssim_t1,\quad\|\partial_t\Gamma_\e^t\|_{\Ld^2}\lesssim_t1.\qedhere
\end{gather*}
\end{enumerate}
\end{as}

In the present section, we introduce the relevant choices for equation~\eqref{eq:genveps} and we show that the corresponding solutions $\vi_\e$ exist and satisfy all the properties of Assumption~\ref{as:apveps}. Three different choices are considered,
\begin{itemize}
\item \emph{Dissipative case (cf.~Theorem~\ref{th:mainGL}):}\\
In Section~\ref{chap:MFL-GL}, the rescaled supercurrent $\frac1{N_\e}j_\e$ is shown to remain close to the solution $\vi_\e$ of the following equation,
\begin{gather}\label{eq:GLv1}
\qquad\partial_t\!\vi_\e=\nabla\!\pre_\e+\Gamma_\e\curl\!\vi_\e,\qquad \vi_\e\!|_{t=0}=\vi_\e^\circ,\\
\qquad\Gamma_\e:=\lambda_\e^{-1}(\alpha-\Jb\beta)\Big(\nabla^\bot h- F^\bot-\frac{2N_\e}\Log\vi_\e\Big),
\qquad \pre_\e:=(\lambda_\e\alpha a)^{-1}\Div(a\!\vi_\e);\nonumber
\end{gather}
\item\emph{Nondilute parabolic case (cf.~Theorem~\ref{th:main-hd}):}\\
In Section~\ref{chap:mean-field-hd}, the rescaled supercurrent $\frac1{N_\e}j_\e$ is shown to remain close to the solution $\vi_\e$ of the following equation,
\begin{gather}\label{eq:GLv1-hd}
\qquad\partial_t\!\vi_\e=\nabla\!\pre_\e+\Gamma_\e\curl\!\vi_\e,\qquad \vi_\e\!|_{t=0}=\vi^\circ,\\
\qquad\Gamma_\e:=\lambda_\e^{-1}\Big(\nabla^\bot h- F^\bot-\frac{2N_\e}\Log\vi_\e\Big),
\qquad \pre_\e:=(\lambda_\e a)^{-1}\Div(a\!\vi_\e);\nonumber
\end{gather}
\item\emph{Conservative case (cf.~Theorem~\ref{th:mainGP}):}\\
In Section~\ref{chap:MFL-GP}, the rescaled supercurrent $\frac1{N_\e}j_\e$ is shown to remain close to the solution $\vi_\e$ of the following equation,
\begin{gather}\label{eq:limGP}
\qquad\partial_t\!\vi_\e=\nabla\!\pre_\e+\Gamma_\e\curl\!\vi_\e,\qquad\Div(a\!\vi_\e)=0,\qquad\vi_\e\!|_{t=0}=\vi_\e^\circ,\\
\qquad\Gamma_\e:=-\lambda_\e^{-1}\Big(\nabla^\bot h-F^\bot-\frac{2N_\e}\Log\vi_\e\Big)^\bot.\nonumber
\end{gather}
\end{itemize}
In addition, using the choice of the scalings for $\lambda_\e,h,F$ in each regime, we show how to pass to the limit $\e\downarrow0$ in these equations, which is indeed needed to conclude the proofs of Theorems~\ref{th:mainGL}, \ref{th:main-hd}, and~\ref{th:mainGP}.

\subsection{Dissipative case}\label{chap:dissip-lim}

Let us examine the vorticity formulation of equation~\eqref{eq:GLv1} for~$\vi_\e$. In terms of $\md_\e:=\curl\!\vi_\e$ and $\dd_\e:=\Div(a\!\vi_\e)$, it takes the form of a nonlocal nonlinear continuity equation for the vorticity~$\md_\e$, coupled with a convection-diffusion equation for the divergence $\dd_\e$,
\begin{gather}\label{eq:vortform}
\begin{cases}\partial_t\!\md_\e=-\Div(\Gamma_\e^\bot\!\md_\e),\\
\partial_t\!\dd_\e-(\alpha\lambda_\e)^{-1}\triangle\!\dd_\e+(\alpha\lambda_\e)^{-1}\Div(\dd_\e\!\nabla h)=\Div(a\Gamma_\e\!\md_\e),\\
\curl\!\vi_\e=\md_\e,\quad\Div(a\!\vi_\e)=\dd_\e,\\
\md_\e\!|_{t=0}=\curl\!\vi_\e^\circ,\quad \dd_\e\!|_{t=0}=\Div(a\!\vi_\e^\circ).\end{cases}
\end{gather}
A detailed study of this kind of equations is performed in the companion article~\cite{D-16}, including global existence results for vortex-sheet initial data. The following proposition in particular states that a local solution $\vi_\e$ always exists and satisfies the various properties of Assumption~\ref{as:apveps}(a) under suitable regularity assumptions on the initial data $\vi_\e^\circ$.
Note that in the regimes~\GLu{} and~\GLdp{}, due to the choice $\lambda_\e\downarrow0$, the solution $\vi_\e$ is expected to converge to the solution $\vi$ of some incompressible equation with the constraint $\Div\!\vi=0$, so that we refer to \GLu{} and~\GLdp{} as the {\it incompressible regimes}, and to \GLd{} and~\GLup{} as the {\it compressible regimes}.
Some additional work is required in the incompressible regimes since we then need to make clear the link with the limiting incompressible equations, in particular in order to establish global existence in the mixed-flow case.

\begin{prop}\label{prop:GLvprop}$  $
Let $\alpha>0$, $\beta\in\R$, $h:\R^2\to\R$, $a:=e^{h}$, $F:\R^2\to\R^2$, and let $\vi_\e^\circ:\R^2\to\R^2$ be bounded in $W^{1,q}(\R^2)^2$ for all $q>2$ and satisfy $\curl\!\vi_\e^\circ\in\Pc(\R^2)$.
For some $s>0$, assume that $h\in W^{s+3,\infty}(\R^2)$, $F\in W^{s+2,\infty}(\R^2)^2$, that $\vi_\e^\circ$ is bounded in $W^{s+2,\infty}(\R^2)^2$, and that $\curl\!\vi_\e^\circ$ and $\Div (a\!\vi_\e^\circ)$ are bounded in $H^{s+1}(\R^2)$.
\begin{enumerate}[(i)]
\item \emph{Compressible regimes $\lambda_\e\simeq1$ (that is,~\GLd--\GLup):}\\
There exist $T>0$ (independent of $\e$) and a unique (local) solution $\vi_\e$ of~\eqref{eq:GLv1} in $[0,T)\times\R^2$, in the space $\Ld^\infty_\loc([0,T);\vi_\e^\circ+H^{2}\cap W^{2,\infty}(\R^2)^2)$. Moreover, all the properties of Assumption~\ref{as:apveps}(a) are satisfied, that is, for all $\e>0$, $t\in[0,T)$, and $q>2$,
\begin{gather*}
\qquad\|(\vi_\e^t,\nabla\!\vi_\e^t)\|_{(\Ld^2+\Ld^q)\cap\Ld^\infty}\lesssim_{t,q}1,\quad\|\curl\!\vi_\e^t\!\|_{\Ld^1\cap\Ld^\infty}\lesssim_t1,\quad\|\!\Div(a\!\vi_\e^t)\|_{\Ld^2\cap\Ld^\infty}\lesssim_t1,\\
\qquad\|\!\pre_\e^t\!\|_{\Ld^2\cap\Ld^\infty}\lesssim_t1,\quad\|\nabla\!\pre_\e^t\!\|_{\Ld^2}\lesssim_t1,\quad\|\partial_t\!\vi_\e^t\!\|_{\Ld^2\cap\Ld^\infty}\lesssim_t1,\quad\|\partial_t\!\pre_\e^t\!\|_{\Ld^2_t\Ld^2}\lesssim_t1.
\end{gather*}
In the parabolic case ($\beta=0$), the solution $\vi_\e$ can be extended globally, that is, $T=\infty$. In the scaling with negligible interactions~\GLup, in the dissipative mixed-flow case, the existence time $T$ tends to infinity as $\e\downarrow0$.
\item \emph{Incompressible regimes $\lambda_\e\ll1$ (that is,~\GLu--\GLdp):}\\
Further assume $\Div (a\!\vi_\e^\circ)=0$. There exist $T>0$ (independent of $\e$) and a unique (local) solution $\vi_\e$ of~\eqref{eq:GLv1} in $\R^+\times\R^2$, in the space $\Ld^\infty_\loc([0,T);\vi_\e^\circ+H^{2}\cap W^{2,\infty}(\R^2)^2)$. Moreover, all the properties of Assumption~\ref{as:apveps}(a) are satisfied, that is, for all $t\in[0,T)$ and $q>2$,
\begingroup\allowdisplaybreaks
\begin{gather*}
\qquad\|(\vi_\e^t,\nabla\!\vi_\e^t)\|_{(\Ld^2+\Ld^q)\cap\Ld^\infty}\lesssim_{t,q}1,\quad\|\curl\!\vi_\e^t\!\|_{\Ld^1\cap\Ld^\infty}\lesssim_t1,\quad\|\Div(a\!\vi_\e^t)\|_{\Ld^2\cap\Ld^\infty}\lesssim_t1,\\
\qquad\|\!\pre_\e^t\!\|_{\Ld^2\cap\Ld^\infty}\lesssim_t\lambda_\e^{-1/2},\quad\|\nabla\!\pre_\e\!\|_{\Ld^2_t\Ld^2}\lesssim_t1,\quad\|\partial_t\!\pre_\e^t\!\|_{\Ld^2_t\Ld^2}\lesssim_t\lambda_\e^{-1},\\
\qquad\|\partial_t\!\vi_\e^t\!\|_{\Ld^2\cap\Ld^\infty}\lesssim_t\lambda_\e^{-1/2},\quad\|\partial_t\!\vi_\e\!\|_{\Ld^2_t\Ld^2}\lesssim_t1.
\end{gather*}
\endgroup
In the parabolic case ($\beta=0$), the solution $\vi_\e$ can be extended globally, that is, $T=\infty$. In the dissipative mixed-flow case, the existence time $T$ tends to infinity as $\e\downarrow0$.
\qedhere
\end{enumerate}
\end{prop}

\begin{proof}
We split the proof into five steps.
Item~(i) is proved in Step~1, except the global existence in the regime~\GLup, which is postponed to the last step. The proof of item~(ii) is given in Steps~2--4.

\medskip
\noindent\step1 Compressible regimes~\GLd--\GLup.

\nopagebreak
Let $s>0$ be non-integer. The assumption $\|\hat h\|_{W^{s+3,\infty}}$, $\|\hat F\|_{W^{s+2,\infty}}\lesssim1$ leads to $\|\lambda_\e^{-1}(\nabla^\bot h-F^\bot)\|_{W^{s+2,\infty}}\lesssim1$ in the considered regimes, and also $\lambda_\e^{-1}{N_\e}/\Log\lesssim1$ and $\lambda_\e\simeq1$.
Further using the assumptions on the initial data $\vi_\e^\circ$, it follows from~\cite[Theorems~2--3]{D-16}
that there exists a unique (local) solution $\vi_\e\in\Ld^\infty_\loc([0,T);\vi_\e^\circ+H^{2}\cap W^{2,\infty}(\R^2)^2)$ of~\eqref{eq:GLv1} in $[0,T)\times\R^2$ with initial data $\vi_\e^\circ$, for some $T\gtrsim1$. Moreover, it is shown in~\cite{D-16} that this solution satisfies for all $t\in[0,T)$,
\begin{align}\label{eq:resD16}
\|\!\vi_\e^t-\vi_\e^\circ\!\|_{H^2\cap W^{2,\infty}}\lesssim_t1,\quad\|(\md_\e^t,\dd_\e^t)\|_{H^1\cap W^{1,\infty}}\lesssim_t1,\quad \int_{\R^2}\md_\e^t=1,\quad\md_\e^t\ge0.
\end{align}
In the parabolic case, it actually follows from~\cite[Theorem~1]{D-16} that the solution is global, that is, $T=\infty$. We now quickly argue that all the claimed properties of $\vi_\e$ follow from~\eqref{eq:resD16}.
Combining~\eqref{eq:resD16} with the assumption that $\vi_\e^\circ$ is bounded in $W^{1,q}(\R^2)^2$ for all $q>2$, we find
\[\|(\vi_\e^t,\nabla\!\vi_\e^t)\|_{(\Ld^2+\Ld^q)\cap\Ld^\infty}\lesssim_{t,q}1.\]
The choice $\pre_\e=(\lambda_\e\alpha a)^{-1}\dd_\e$ with $\lambda_\e\simeq1$ leads to
\[\|\!\pre_\e^t\!\|_{H^1\cap W^{1,\infty}}\lesssim\|\!\dd_\e^t\!\|_{H^1\cap W^{1,\infty}}\lesssim_t1.\]
Inserting this information into equation~\eqref{eq:GLv1}, we deduce
\[\|\partial_t\!\vi_\e^t\!\|_{\Ld^2\cap \Ld^{\infty}}\lesssim\|\nabla\!\pre_\e^t\!\|_{\Ld^2\cap \Ld^{\infty}}+\|\Gamma_\e^t\!\md_\e^t\!\|_{\Ld^2\cap \Ld^{\infty}}\lesssim_t1.\]
Testing the convection-diffusion equation $\partial_t\!\dd_\e-(\lambda_\e\alpha)^{-1}(\triangle\!\dd_\e-\Div(\dd_\e\!\nabla h))=\Div(a\Gamma_\e\!\md_\e)$ with $\partial_t\!\dd_\e$ yields
\begin{align*}
\int_{\R^2}|\partial_t\!\dd_\e\!|^2+\frac12(\lambda_\e\alpha)^{-1}\partial_t\int_{\R^2}|\nabla\!\dd_\e\!|^2=-\int_{\R^2}\partial_t\!\dd_\e\Div\!\big((\lambda_\e\alpha)^{-1}\!\dd_\e\!\nabla h-a\Gamma_\e\!\md_\e\big),
\end{align*}
and hence, integrating in time, with $\lambda_\e\simeq1$,
\begin{eqnarray*}
\lefteqn{\|\partial_t\!\dd_\e\!\|_{\Ld^2_t\Ld^2}^2+\frac12(\lambda_\e\alpha)^{-1}\|\nabla\!\dd_\e\!\|_{\Ld^2}^2}\\
&\lesssim& \|\nabla\!\dd_\e^\circ\!\|_{\Ld^2}^2+\|\partial_t\!\dd_\e\!\|_{\Ld^2_t\Ld^2}\big(\|\!\dd_\e\!\|_{\Ld^2_tH^1}+\|a\Gamma_\e\|_{\Ld^\infty_tW^{1,\infty}}\|\!\md_\e\!\|_{\Ld^2_tH^1}\big)\\
&\lesssim_t&1+\|\partial_t\!\dd_\e\!\|_{\Ld^2_t\Ld^2}.
\end{eqnarray*}
Absorbing the last right-hand side term, we conclude
\begin{align}\label{eq:bounddtpreps}
\|\partial_t\!\pre_\e\!\|_{\Ld^2_t\Ld^2}\lesssim\|\partial_t\!\dd_\e\!\|_{\Ld^2_t\Ld^2}\lesssim_t1.
\end{align}
All the claimed properties of $\vi_\e$ follow.

\medskip
\noindent\step2 Estimates for convection-diffusion equations with large diffusivity.

In the incompressible regimes \GLu--\GLdp, the conclusion does not follow as in Step~1 since the corresponding choice $\pre_\e=(\lambda_\e\alpha a)^{-1}\Div(a\!\vi_\e)$ now contains the large prefactor $(\lambda_\e\alpha)^{-1}\gg1$. In particular, equation~\eqref{eq:vortform} for the divergence $\dd_\e:=\Div(a\!\vi_\e)$ takes the form
\begin{gather}\label{eq:vortformresc}
\partial_t\!\dd_\e-(\lambda_\e\alpha)^{-1}\triangle\!\dd_\e+\alpha^{-1}\Div(\dd_\e\!\nabla \hat h)
=\Div(a\Gamma_\e\!\md_\e),
\end{gather}
with a large prefactor $(\lambda_\e\alpha)^{-1}\gg1$ in front of the Laplacian and with initial data $\dd_\e^\circ:=\Div(a\!\vi_\e^\circ)=0$.
In this step, we consider the model convection-diffusion equation
\[\partial_tw-\nu\triangle w+\Div(w\nabla\hat h)=\Div g, \qquad w|_{t=0}=0,\]
with large diffusivity $\nu\gg1$. As the initial condition vanishes, a direct adaptation of~\cite[Lemma~2.3]{D-16} yields the following bounds: for all $\nu\gtrsim1$,
\begin{enumerate}[(a)]
\item for all $s\ge0$, there is a constant $C$ only depending on an upper bound on $s$ and $\|\nabla\hat h\|_{W^{s,\infty}}$ such that
\[\|w^t\|_{H^s}+\nu^{1/2}\|\nabla w\|_{\Ld^2_tH^s}\le C\big(\tfrac t\nu\big)^{1/2}e^{C\frac t\nu}\|g\|_{\Ld^\infty_tH^s}\le Ct^{1/2}e^{Ct}\|g\|_{\Ld^\infty_tH^s};\]
\item there is a constant $C$ only depending on an upper bound on $\|\nabla\hat h\|_{\Ld^\infty}$ such that
\[\|w^t\|_{\dot H^{-1}}\le Ce^{Ct}\|g\|_{\Ld^2_t\Ld^2};\]
\item for all $1\le p,q\le \infty$, there is a constant $C$ only depending on an upper bound on $\|\nabla\hat h\|_{\Ld^\infty}$ such that
\[\|w\|_{\Ld^p_t\Ld^q}\le C\big(\tfrac t\nu\big)^{1/2}e^{C(\frac t\nu)^2}\|g\|_{\Ld^p_t\Ld^q}\le Ct^{1/2}e^{Ct^2}\|g\|_{\Ld^p_t\Ld^q}.\]
\end{enumerate}
In particular, the same bounds as in~\cite[Lemma~2.3]{D-16} hold uniformly with respect to the large diffusivity $\nu\gg1$. Further adapting the proof of~\eqref{eq:bounddtpreps} in Step~1 above, we easily find
\begin{enumerate}[(d)]
\item there is a constant $C$ only depending on an upper bound on $\|\nabla\hat h\|_{W^{1,\infty}}$ such that
\[\|\partial_tw\|_{\Ld^2_t\Ld^2}\le\|\nabla g\|_{\Ld^2_t\Ld^2}+C\big(\tfrac t\nu\big)^{1/2}e^{C\frac t\nu}\|g\|_{\Ld^\infty_t\Ld^2}\le Ct^{1/2}e^{Ct}\|g\|_{\Ld^\infty_tH^1}.\]
\end{enumerate}

\medskip
\noindent\step3 Incompressible regimes~\GLu--\GLdp.

In the vorticity formulation~\eqref{eq:vortform}, the large prefactor $(\lambda_\e\alpha)^{-1}\gg1$ does not affect the equation for the vorticity $\md_\e$, but only the equation for the divergence $\dd_\e$, which now takes the form~\eqref{eq:vortformresc}. However, for the choice $\dd_\e^\circ=0$, the result of Step~2 ensures that the estimates for $\dd_\e$ used in~\cite{D-16} hold uniformly with respect to the large prefactor. Hence, as in Step~1, using the assumptions on the initial data, the proof of~\cite[Theorems~2--3]{D-16} shows that in the incompressible regimes there exists a unique (local) solution $\vi_\e\in\Ld^\infty_\loc([0,T);\vi^\circ+H^{2}\cap W^{2,\infty}(\R^2)^2)$ of~\eqref{eq:GLv1} in $[0,T)\times\R^2$ with initial data $\vi^\circ$, for some $T\gtrsim1$. Moreover, it is shown in~\cite{D-16} that this solution satisfies for all $t\in[0,T)$,
\begin{align}\label{eq:resD16bis}
\|\!\vi_\e^t-\vi_\e^\circ\!\|_{H^{2}\cap W^{2,\infty}}\lesssim_t1,\quad\|(\md_\e^t,\dd_\e^t)\|_{H^1\cap W^{1,\infty}}\lesssim_t1,\quad \int_{\R^2}\md_\e^t=1,\quad\md_\e^t\ge0.
\end{align}
In the parabolic case, it actually follows from~\cite[Theorem~1]{D-16} that the solution is global, that is, $T=\infty$. We now quickly argue that all the claimed properties of $\vi_\e$ follow from~\eqref{eq:resD16bis}.
By definition~\eqref{eq:GLv1}, we find $\|\Gamma_\e^t\|_{W^{1,\infty}}\lesssim_t1$.
Combining~\eqref{eq:resD16bis} with the assumption that $\vi_\e^\circ$ is bounded in $W^{1,q}(\R^2)^2$ for all $q>2$, we obtain
\[\|(\vi_\e^t,\nabla\!\vi_\e^t)\|_{(\Ld^2+\Ld^q)\cap\Ld^\infty}\lesssim_{t,q}1.\]
Using~\eqref{eq:GLv1} in the form $\pre_\e=(\lambda_\e\alpha a)^{-1}\dd_\e$, and applying items~(a)--(c) of Step~2, we find
\[\|\!\pre_\e^t\!\|_{H^1\cap W^{1,\infty}}\lesssim \lambda_\e^{-1}\|\!\dd_\e^t\!\|_{H^1\cap W^{1,\infty}}\lesssim_t\lambda_\e^{-1/2}\|a\Gamma_\e\!\md_\e\!\|_{\Ld^\infty_t(H^1\cap W^{1,\infty})}\lesssim_t\lambda_\e^{-1/2},\]
where the last inequality follows from~\eqref{eq:resD16bis}. Similarly, using the choice $h=\lambda_\e\hat h$ in the form
\[\nabla\!\pre_\e=(\lambda_\e\alpha)^{-1}\nabla (a^{-1}\!\dd_\e)=(\alpha a)^{-1}(\lambda_\e^{-1}\nabla \!\dd_\e-\dd_\e\!\nabla\hat h),\]
item~(a) of Step~2 yields
\[\|\nabla\!\pre_\e\!\|_{\Ld^2_t\Ld^2}\lesssim_t\lambda_\e^{-1}\|\nabla\!\dd_\e\!\|_{\Ld^2_t\Ld^2}+\|\!\dd_\e\!\|_{\Ld^\infty_t\Ld^2}\lesssim_t\|a\Gamma_\e\!\md_\e\!\|_{\Ld^\infty_t\Ld^2}\lesssim_t1.\]
Inserting this information into equation~\eqref{eq:GLv1}, we deduce
\[\|\partial_t\!\vi_\e^t\!\|_{\Ld^2\cap\Ld^\infty}\lesssim\|\nabla\!\pre_\e^t\!\|_{\Ld^2\cap\Ld^\infty}+\|\Gamma_\e^t\|_{\Ld^{\infty}}\|\!\md_\e^t\!\|_{\Ld^2\cap\Ld^\infty}\lesssim_t\lambda_\e^{-1/2},\]
and similarly
\[\|\partial_t\!\vi_\e\!\|_{\Ld^2_t\Ld^2}\lesssim\|\nabla\!\pre_\e\!\|_{\Ld^2_t\Ld^2}+\|\Gamma_\e\|_{\Ld^\infty_t\Ld^\infty}\|\!\md_\e\!\|_{\Ld^2_t\Ld^2}\lesssim_t1.\]
Finally, item~(d) of Step~2 yields
\[\|\partial_t\!\pre_\e\!\|_{\Ld^2_t\Ld^2}\lesssim\lambda_\e^{-1}\|\partial_t\!\dd_\e\!\|_{\Ld^2_t\Ld^2}\lesssim_t\lambda_\e^{-1}\|a\Gamma_\e\!\md_\e\!\|_{\Ld^\infty_tH^1}\lesssim_t\lambda_\e^{-1}.\]
All the claimed properties of $\vi_\e$ follow.

\medskip
\noindent\step4 Global existence in the mixed-flow incompressible regimes.

\nopagebreak
The energy estimates of~\cite[Lemma~4.1(iii)]{D-16} yield
\[\|\!\vi_\e^t-\vi_\e^\circ\!\|_{\Ld^2}\lesssim_t1.\]
Using this estimate and $\int_{\R^2}|\!\md_\e^t\!|=1$ for all~$t$, and arguing as in~\cite[Step~1 of the proof of Lemma~4.5]{D-16}, we find
\begin{multline}\label{eq:boundvepsLinfty}
\|\!\vi_\e^t\!\|_{\Ld^\infty}\lesssim_t1+\|\!\md_\e^t\!\|_{\Ld^\infty}^{1/2}\log^{1/2}\big(2+\|\!\md_\e^t\!\|_{\Ld^\infty}\big)\\
+\|\!\Div(\vi_\e^t-\vi_\e^\circ)\|_{\Ld^2}\log^{1/2}\big(2+\|\!\Div(\vi_\e^t-\vi_\e^\circ)\|_{\Ld^2\cap\Ld^\infty}\big).
\end{multline}
Item~(a) of Step~2 yields
\begin{align*}
\|\!\dd_\e^t\!\|_{\Ld^2}\lesssim_t \lambda_\e^{1/2}\|a\Gamma_\e\!\md_\e\!\|_{\Ld^\infty_t\Ld^2}&\lesssim_t \lambda_\e^{1/2}\|\!\vi_\e-\vi_\e^\circ\!\|_{\Ld^\infty_t\Ld^2}\|\!\md_\e\!\|_{\Ld^\infty_t\Ld^\infty}+\lambda_\e^{1/2}\|\!\md_\e\!\|_{\Ld^\infty_t\Ld^2}\\
&\lesssim_t \lambda_\e^{1/2}\|\!\md_\e\!\|_{\Ld^\infty_t\Ld^\infty}+\lambda_\e^{1/2}\|\!\md_\e\!\|_{\Ld^\infty_t\Ld^\infty}^{1/2},
\end{align*}
and hence, in terms of $\Div(\vi_\e-\vi_\e^\circ)=a^{-1}\!\dd_\e-\lambda_\e\nabla\hat h\cdot(\vi_\e-\vi_\e^\circ)$,
\begin{align*}
\|\!\Div(\vi_\e^t-\vi_\e^\circ)\|_{\Ld^2}&\lesssim_t \lambda_\e^{1/2}(1+\|\!\md_\e\!\|_{\Ld^\infty_t\Ld^\infty}).
\end{align*}
Inserting this into~\eqref{eq:boundvepsLinfty}, we find
\begin{align}\label{eq:boundvepsLinfty2}
\|\!\vi_\e^t\!\|_{\Ld^\infty}\lesssim_t(1+\|\!\md_\e\!\|_{\Ld^\infty_t\Ld^\infty})\log^{1/2}\big(2+\|\!\md_\e\!\|_{\Ld^\infty_t\Ld^\infty}+\|\!\Div\!\vi_\e^t\!\|_{\Ld^\infty}\big).
\end{align}
Item~(c) of Step~2 yields
\begin{align*}
\|\!\dd_\e^t\!\|_{\Ld^\infty}\lesssim_t\lambda_\e^{1/2}\|a\Gamma_\e\!\md_\e\!\|_{\Ld^\infty_t\Ld^\infty}\lesssim \lambda_\e^{1/2}(1+\|\!\vi_\e\!\|_{\Ld^\infty_t\Ld^\infty})\|\!\md_\e\!\|_{\Ld^\infty_t\Ld^\infty},
\end{align*}
or alternatively, in terms of $\Div\!\vi_\e=a^{-1}\!\dd_\e-\lambda_\e\nabla\hat h\cdot\vi_\e$,
\begin{align*}
\|\!\Div\!\vi_\e^t\!\|_{\Ld^\infty}\lesssim_t \lambda_\e^{1/2}(1+\|\!\vi_\e\!\|_{\Ld^\infty_t\Ld^\infty})(1+\|\!\md_\e\!\|_{\Ld^\infty_t\Ld^\infty}).
\end{align*}
Combining this with~\eqref{eq:boundvepsLinfty2} leads to
\begin{align*}
\|\!\Div\!\vi_\e^t\!\|_{\Ld^\infty}\lesssim_t \lambda_\e^{1/2}(1+\|\!\md_\e\!\|_{\Ld^\infty_t\Ld^\infty}^2)\log^{1/2}\big(2+\|\!\md_\e\!\|_{\Ld^\infty_t\Ld^\infty}+\|\!\Div\!\vi_\e^t\!\|_{\Ld^\infty}\big).
\end{align*}
Estimating $\log^{1/2}$ by $\log$, applying the inequality $a\log b\le b+a\log a$ to the choices $a=1+\|\!\md_\e\!\|_{\Ld^\infty_t\Ld^\infty}^2$ and $b=2+\|\!\md_\e\!\|_{\Ld^\infty_t\Ld^\infty}+\|\!\Div\!\vi_\e^t\!\|_{\Ld^\infty}$, and using $\lambda_\e\ll1$ to absorb the term $\|\!\Div\!\vi_\e^t\!\|_{\Ld^\infty}$ appearing in the right-hand side,
we find
\begin{align*}
\|\!\Div\!\vi_\e^t\!\|_{\Ld^\infty}\lesssim_t \lambda_\e^{1/2}(1+\|\!\md_\e\!\|_{\Ld^\infty_t\Ld^\infty}^2)\log\big(2+\|\!\md_\e\!\|_{\Ld^\infty_t\Ld^\infty}\big),
\end{align*}
so that~\eqref{eq:boundvepsLinfty2} finally takes the form
\begin{align*}
\|\!\vi_\e^t\!\|_{\Ld^\infty}\lesssim_t(1+\|\!\md_\e\!\|_{\Ld^\infty_t\Ld^\infty})\log^{1/2}\big(2+\|\!\md_\e\!\|_{\Ld^\infty_t\Ld^\infty}\big).
\end{align*}
In particular, we deduce the following estimates,
\begin{align*}
\|\!\vi_\e^t\!\|_{\Ld^\infty}\lesssim_t1+\|\!\md_\e\!\|_{\Ld^\infty_t\Ld^\infty}^2\qquad\text{and}\qquad\|\!\dd_\e^t\!\|_{\Ld^\infty}\lesssim_t\lambda_\e^{1/2}(1+\|\!\md_\e\!\|_{\Ld^\infty_t\Ld^\infty}^3).
\end{align*}
The result in~\cite[Lemma~4.3(i)]{D-16} then yields the following bound on the vorticity $\md_\e$,
\begin{align*}
\|\!\md_\e^t\!\|_{\Ld^\infty}\,&\lesssim\,\exp\Big(Ct\big(1+\|\!\dd_\e\!\|_{\Ld^\infty_t\Ld^\infty}+\lambda_\e\|\!\vi_\e\!\|_{\Ld^\infty_t\Ld^\infty}\big)\Big)\\
&\lesssim_t\,\exp\big(Ct\lambda_\e^{1/2}(1+\|\!\md_\e\!\|_{\Ld^\infty_t\Ld^\infty}^3)\big).
\end{align*}
As $\lambda_\e\ll1$, this bound easily implies that for all $T>0$ there exists $\e_0(T)>0$ such that for all $0<\e<\e_0(T)$ the vorticity~$\md_\e^t$ (if it exists) remains bounded in $\Ld^\infty(\R^2)$ for all $t\in[0,T]$. Then repeating the arguments in~\cite[Sections~4.2--4.3]{D-16}, this a priori bound on the vorticity allows to deduce existence and uniqueness of a solution on the whole time interval $[0,T]$. This proves that the existence time blows up as $\e\downarrow0$.

\medskip
\noindent\step5 Global existence in the mixed-flow compressible regime~\GLup.

\nopagebreak
Just as in~\eqref{eq:boundvepsLinfty} above, we obtain the bounds $\|\!\vi_\e^t-\vi_\e^\circ\!\|_{\Ld^2}\lesssim_t1$ and
\begin{multline}\label{eq:boundvepsLinfty3}
\|\!\vi_\e^t\!\|_{\Ld^\infty}\lesssim_t1+\|\!\md_\e^t\!\|_{\Ld^\infty}^{1/2}\log^{1/2}\big(2+\|\!\md_\e^t\!\|_{\Ld^\infty}\big)\\
+\|\!\Div(\vi_\e^t-\vi_\e^\circ)\|_{\Ld^2}\log^{1/2}\big(2+\|\!\Div(\vi_\e^t-\vi_\e^\circ)\|_{\Ld^2\cap\Ld^\infty}\big).
\end{multline}
Considering the equation~\eqref{eq:vortform} for $\dd_\e$, the a priori estimates in~\cite[Lemma~2.3]{D-16} yield
\begin{multline*}
\|\!\dd_\e^t\!\|_{\Ld^2}\lesssim_t1+\|a\Gamma_\e\!\md_\e\!\|_{\Ld^\infty_t\Ld^2}\lesssim_t 1+\|\!\md_\e\!\|_{\Ld^\infty_t\Ld^2}+\|\!\md_\e\!\|_{\Ld^\infty_t\Ld^\infty}\|\!\vi_\e-\vi_\e^\circ\!\|_{\Ld^\infty_t\Ld^2}\\
\lesssim_t 1+\|\!\md_\e\!\|_{\Ld^\infty_t\Ld^\infty},
\end{multline*}
and also
\begin{align*}
\|\!\dd_\e^t\!\|_{\Ld^\infty}&\lesssim_t1+\|a\Gamma_\e\!\md_\e\!\|_{\Ld^\infty_t\Ld^\infty}\lesssim_t 1+\|\!\md_\e\!\|_{\Ld^\infty_t\Ld^\infty}(1+\|\!\vi_\e\!\|_{\Ld^\infty_t\Ld^\infty}).
\end{align*}
As by definition $\Div(\vi_\e^t-\vi_\e^\circ)=a^{-1}(\dd_\e^t-\dd_\e^\circ)-\nabla h\cdot (\vi_\e^t-\vi_\e^\circ)$, these estimates take the form
\begin{gather}\label{eq:boundvepsLinfty4}
\|\!\Div(\vi_\e^t-\vi_\e^\circ)\|_{\Ld^2}\lesssim_t1+\|\!\md_\e\!\|_{\Ld^\infty_t\Ld^\infty},\\
\|\!\Div\!\vi_\e^t\!\|_{\Ld^\infty}\lesssim_t (1+\|\!\md_\e\!\|_{\Ld^\infty_t\Ld^\infty})(1+\|\!\vi_\e\!\|_{\Ld^\infty_t\Ld^\infty}).\nonumber
\end{gather}
Injecting these estimates into~\eqref{eq:boundvepsLinfty3} yields
\begin{align*}
\|\!\vi_\e^t\!\|_{\Ld^\infty}&\lesssim_t1+\|\!\md_\e^t\!\|_{\Ld^\infty}^{1/2}\log^{1/2}\big(2+\|\!\md_\e^t\!\|_{\Ld^\infty}\big)\\
&\qquad\qquad+(1+\|\!\md_\e\!\|_{\Ld^\infty_t\Ld^\infty})\log^{1/2}\big((1+\|\!\md_\e\!\|_{\Ld^\infty_t\Ld^\infty})(1+\|\!\vi_\e\!\|_{\Ld^\infty_t\Ld^\infty})\big)\\
&\lesssim_t(1+\|\!\md_\e\!\|_{\Ld^\infty_t\Ld^\infty})\log^{1/2}\big(2+\|\!\md_\e\!\|_{\Ld^\infty_t\Ld^\infty}+\|\!\vi_\e\!\|_{\Ld^\infty_t\Ld^\infty}\big).
\end{align*}
Estimating $\log^{1/2}$ by $\log$, applying the inequality $a\log b\le b+a\log a$ to the choices $a:=1+\|\!\md_\e\!\|_{\Ld^\infty_t\Ld^\infty}$ and $b:=2+\|\!\md_\e\!\|_{\Ld^\infty_t\Ld^\infty}+\frac1K\|\!\vi_\e\!\|_{\Ld^\infty_t\Ld^\infty}$, and choosing $K\simeq_t1$ large enough to absorb the term $\|\!\vi_\e\!\|_{\Ld^\infty_t\Ld^\infty}$ appearing in the right-hand side, we find
\begin{align*}
\|\!\vi_\e\!\|_{\Ld^\infty_t\Ld^\infty}&\lesssim_t(1+\|\!\md_\e\!\|_{\Ld^\infty_t\Ld^\infty})\log\big(2+\|\!\md_\e\!\|_{\Ld^\infty_t\Ld^\infty}\big),
\end{align*}
so that~\eqref{eq:boundvepsLinfty4} takes the form,
\begin{align*}
\|\!\Div\!\vi_\e\!\|_{\Ld^\infty_t\Ld^\infty}\lesssim_t (1+\|\!\md_\e\!\|_{\Ld^\infty_t\Ld^\infty})^2\log\big(2+\|\!\md_\e\!\|_{\Ld^\infty_t\Ld^\infty}\big).
\end{align*}
The result in~\cite[Lemma~4.3(i)]{D-16} then gives the following bound on the vorticity $\md_\e$, in the considered regime~\GLup,
\[\|\!\md_\e^t\!\|_{\Ld^\infty}\lesssim\exp\bigg(Ct\Big(1+\frac{N_\e}\Log\|(\vi_\e,\Div\!\vi_\e)\|_{\Ld^\infty_t\Ld^\infty}\Big)\bigg)\lesssim_t\exp\bigg(\frac{CtN_\e}\Log\|\!\md_\e\!\|_{\Ld^\infty_t\Ld^\infty}^3\bigg).\]
As $N_\e\ll\Log$, this bound easily implies that for all $T>0$ there exists $\e_0(T)>0$ such that for all $0<\e<\e_0(T)$ the vorticity~$\md_\e^t$ (if it exists) remains bounded in $\Ld^\infty(\R^2)$ for all $t\in[0,T]$. Then repeating the arguments in~\cite[Sections~4.2--4.3]{D-16}, existence and uniqueness of a solution on the whole time interval $[0,T]$ follows from this a priori bound. This proves that the existence time blows up as $\e\downarrow0$.
\end{proof}

We now show how to pass to the limit in equation~\eqref{eq:GLv1} as $\e\downarrow0$, which is easily achieved e.g.\@ by a Grönwall argument on the $\Ld^2$-distance between $\vi_\e$ and the solution $\vi$ of the limiting equation.

\begin{lem}\label{lem:lastlimGL}
Let the same assumptions hold as in Proposition~\ref{prop:GLvprop}, and let $\vi_\e:[0,T)\times\R^2\to\R^2$ be the corresponding (local) solution of~\eqref{eq:GLv1}, for some $T>0$ (independent of~$\e$). Assume that $\vi_\e^\circ\to\vi^\circ$ in $\Ld^2_\uloc(\R^2)^2$ as $\e\downarrow0$. The following hold.
\begin{enumerate}[(i)]
\item \emph{Regime \GLu:}\\
We have $\vi_\e\to\vi$ in $\Ld^\infty_\loc([0,T);\Ld^2_\uloc(\R^2)^2)$ as $\e\downarrow0$, where $\vi\in\Ld^\infty_\loc(\R^+;\vi^\circ+\Ld^2(\R^2)^2)$ is the unique global (smooth) solution of
\begin{align}\label{eq:limGL1}
\begin{cases}
\partial_t\!\vi=\nabla\!\pre+(\alpha-\Jb\beta)(\nabla^\bot\hat h-\hat F^\bot-2\!\vi)\curl\!\vi,\\
\Div\!\vi=0,\quad\vi\!|_{t=0}=\vi^\circ.
\end{cases}
\end{align}
\item \emph{Regime~\GLd{} with $\frac{N_\e}\Log\to\lambda\in(0,\infty)$ and $\vi_\e^\circ=\vi^\circ$:}\\
We have $\vi_\e\to\vi$ in $\Ld^\infty_\loc([0,T);\Ld^2(\R^2)^2)$ as $\e\downarrow0$, where $\vi\in \Ld^\infty_\loc([0,T);\vi^\circ+\Ld^2(\R^2)^2)$ is the unique local (smooth) solution of
\begin{align}\label{eq:limGL2}
\begin{cases}
\partial_t\!\vi=\alpha^{-1}\nabla(\hat a^{-1}\Div(\hat a\!\vi))+(\alpha-\Jb\beta)(\nabla^\bot\hat h-\hat F^\bot-2\lambda\!\vi)\curl\!\vi,\\
\vi\!|_{t=0}=\vi^\circ.
\end{cases}
\end{align}
\item \emph{Regime~\GLup{} with $\vi_\e^\circ=\vi^\circ$:}\\
We have $\vi_\e\to\vi$ in $\Ld^\infty_\loc([0,T);\Ld^2(\R^2)^2)$ as $\e\downarrow0$, where $\vi\in\Ld^\infty_\loc(\R^+;\vi^\circ+\Ld^2(\R^2)^2)$ is the unique global (smooth) solution of
\begin{align}\label{eq:limGL4}
\begin{cases}
\partial_t\!\vi=\alpha^{-1}\nabla(\hat a^{-1}\Div(\hat a\!\vi))+(\alpha-\Jb\beta)(\nabla^\bot\hat h-\hat F^\bot)\curl\!\vi,\\
\vi\!|_{t=0}=\vi^\circ.
\end{cases}
\end{align}
\item \emph{Regime~\GLdp:}\\
We have $\vi_\e\to\vi$ in $\Ld^\infty_\loc([0,T);\Ld^2_\uloc(\R^2)^2)$ as $\e\downarrow0$, where $\vi\in\Ld^\infty_\loc(\R^+;\vi^\circ+\Ld^2(\R^2)^2)$ is the unique global (smooth) solution of
\begin{equation}\label{eq:limGL3}
\begin{cases}
\partial_t\!\vi=\nabla\!\pre+(\alpha-\Jb\beta)(\nabla^\bot\hat h-\hat F^\bot)\curl\!\vi,\\
\Div\!\vi=0,\quad\vi\!|_{t=0}=\vi^\circ.
\end{cases}\qedhere
\end{equation}
\end{enumerate}
\end{lem}

\begin{proof}
We treat each of the four regimes separately. We denote by $\xi_R^z(x):=e^{-|x-z|/R}$ the exponential cut-off at the scale $R\ge1$ centered at $z\in R\Z^2$.

\medskip
\step1 Regime~\GLu.
\nopagebreak

Using the choice of the scalings for $\lambda_\e,h,F$ in the regime~\GLu, with $\lambda_\e=\frac{N_\e}\Log\ll1$, and setting $a_\e:=a=\hat a^{\lambda_\e}$, equation~\eqref{eq:GLv1} takes on the following guise,
\begin{gather*}
\partial_t\!\vi_\e=\nabla\!\pre_\e+(\alpha-\Jb\beta)(\nabla^\bot\hat h-\hat F^\bot-2\!\vi_\e)\,\curl\!\vi_\e,\quad\pre_\e:=(\lambda_\e\alpha a_\e)^{-1}\Div(a_\e\!\vi_\e),
\end{gather*}
with initial data $\vi_\e\!|_{t=0}=\vi_\e^\circ\to\vi^\circ$ in $\Ld^2_\uloc(\R^2)^2$.
As $\lambda_\e\to0$, it is then formally clear from the vorticity formulation of this equation that $\vi_\e$ should converge to the solution $\vi$ of~\eqref{eq:limGL1}.

The existence and uniqueness of a global smooth solution $\vi\in\Ld^\infty_\loc(\R^+;\vi^\circ+\Ld^2(\R^2)^2)$ of~\eqref{eq:limGL1} are established in~\cite[Theorems~1 and~3]{D-16}. Moreover, we show that the following estimates hold for all $t\ge0$ and $R,\theta>0$,
\begin{gather}\label{eq:boundvv}
\|\!\vi^t\!\|_{W^{1,\infty}}\lesssim_t1,\quad
\|(\vi^t,\pre^t)\|_{\Ld^2(B_R)}\lesssim_{t,\theta}R^\theta,\quad\|\curl\!\vi^t\!\|_{\Ld^1}=1.
\end{gather}
The bounds on $\vi$ are indeed direct consequences of the results in~\cite{D-16} together with the regularity assumptions on the data (in particular $\vi^\circ\in\Ld^q(\R^2)^2$ for all $q>2$).
It remains to check the bound on the pressure $\pre$. Taking the divergence of both sides of equation~\eqref{eq:limGL1}, we obtain the following equation for the pressure $\pre^t$, for all $t\ge0$,
\[-\triangle\!\pre^t=\Div\!\big((\alpha-\Jb\beta)(\nabla^\bot\hat h-\hat F^\bot-2\!\vi^t)\curl\!\vi^t\big).\]
By Riesz potential theory, we deduce for all $2<q<\infty$,
\[\|\!\pre^t\!\|_{\Ld^q}\lesssim_q(1+\|\!\vi^t\!\|_{\Ld^\infty})\|\curl\!\vi^t\!\|_{\Ld^{\frac{2q}{2+q}}}\lesssim(1+\|\!\vi^t\!\|_{\Ld^\infty})\big(\|\curl\!\vi^t\!\|_{\Ld^1}+\|\nabla\!\vi^t\!\|_{\Ld^\infty}\big)\lesssim_t1,\]
and the bound on the pressure $\pre$ follows.

We turn to the convergence $\vi_\e\to\vi$ in $\Ld^\infty_\loc([0,T);\Ld^2_\uloc(\R^2)^2)$ and argue by a Grönwall argument. Using the equations for $\vi_\e,\vi$, we find
\begin{multline}\label{eq:lastlimcomputGL1}
\partial_t\int_{\R^2} a_\e\xi_R^z|\!\vi_\e-\vi\!|^2=2\int_{\R^2} a_\e\xi_R^z(\vi_\e-\vi)\cdot\nabla(\pre_\e-\pre)-4\alpha\int_{\R^2} a_\e\xi_R^z|\!\vi_\e-\vi\!|^2\curl\!\vi_\e\\
+2\int_{\R^2} a_\e\xi_R^z(\alpha-\Jb\beta)(\nabla^\bot\hat h-\hat F^\bot-2\!\vi\!\big)\cdot(\vi_\e-\vi)\,\curl(\vi_\e-\vi).
\end{multline}
Integrating by parts in the first term, decomposing
\[\Div(a_\e\xi_R^z(\vi_\e-\vi))=a_\e\nabla\xi_R^z\cdot(\vi_\e-\vi)+\lambda_\e\alpha a_\e\xi_R^z\!\pre_\e-\lambda_\e a_\e\xi_R^z\nabla\hat h \cdot\vi,\]
noting that the second right-hand side term in~\eqref{eq:lastlimcomputGL1} is nonpositive, and using the following weighted Delort-type identity (as e.g.~in~\cite{D-16}),
\begin{eqnarray}\label{eq:weightedDelort}
\lefteqn{(\vi_\e-\vi)\,\curl(\vi_\e-\vi)}\nonumber\\
&=&a_\e^{-1}(\vi_\e-\vi)^\bot\Div(a_\e(\vi_\e-\vi))-\frac12a_\e^{-1}|\!\vi_\e-\vi\!|^2\nabla^\bot a_\e-a_\e^{-1}(\Div(a_\e S_{\vi_\e-\vi}))^\bot\\
&=&\lambda_\e\alpha\!\pre_\e(\vi_\e-\vi)^\bot-\lambda_\e(\nabla\hat h\cdot\vi)(\vi_\e-\vi)^\bot-\frac{\lambda_\e}2|\!\vi_\e-\vi\!|^2\nabla^\bot\hat h-a_\e^{-1}(\Div(a_\e S_{\vi_\e-\vi}))^\bot,\nonumber
\end{eqnarray}
in terms of the stress-energy tensor $S_{w}:=w\otimes w-\frac12|w|^2\Id$, we deduce
\begin{multline*}
\partial_t\int_{\R^2} a_\e\xi_R^z|\!\vi_\e-\vi\!|^2\le-2\int_{\R^2} a_\e(\pre_\e-\pre)\nabla\xi_R^z\cdot(\vi_\e-\vi)-2\lambda_\e\alpha\int_{\R^2} a_\e\xi_R^z\pre_\e(\pre_\e-\pre)\\
+2\lambda_\e\int_{\R^2} a_\e\xi_R^z(\pre_\e-\pre)\vi\cdot\nabla\hat h+2\lambda_\e\alpha\int_{\R^2} a_\e\xi_R^z\!\pre_\e(\alpha-\Jb\beta)(\nabla^\bot\hat h-\hat F^\bot-2\!\vi)\cdot(\vi_\e-\vi)^\bot\\
-2\lambda_\e\int_{\R^2} a_\e\xi_R^z(\nabla\hat h\cdot\vi)(\alpha-\Jb\beta)(\nabla^\bot\hat h-\hat F^\bot-2\!\vi)\cdot(\vi_\e-\vi)^\bot\\
-\lambda_\e\int_{\R^2} a_\e\xi_R^z|\!\vi_\e-\vi\!|^2(\alpha-\Jb\beta)(\nabla^\bot\hat h-\hat F^\bot-2\!\vi)\cdot\nabla^\bot\hat h\\
-2\int_{\R^2} a_\e S_{\vi_\e-\vi}:\nabla\Big(\xi_R^z(\alpha\Jb+\beta)(\nabla^\bot\hat h-\hat F^\bot-2\!\vi)\Big),
\end{multline*}
and hence, using~\eqref{eq:boundvv} in the form $\|\!\vi^t\!\|_{W^{1,\infty}}\lesssim1$, the assumption $\|(\nabla\hat h,\hat F)\|_{W^{1,\infty}}\lesssim1$, the property $|\nabla\xi_R^z|\lesssim R^{-1}\xi_R^z$ of the exponential cut-off, and the pointwise estimate $|S_{w}|\lesssim|w|^2$, we obtain
\begin{multline*}
\partial_t\int_{\R^2} a_\e\xi_R^z|\!\vi_\e-\vi\!|^2\le (R^{-2}-\lambda_\e\alpha)\int_{\R^2} a_\e\xi_R^z|\!\pre_\e\!|^2\\
+C_t(R^{-2}+\lambda_\e)\int_{\R^2} a_\e\xi_R^z(|\!\pre\!|^2+|\!\vi\!|^2)+C_t\int_{\R^2} a_\e\xi_R^z|\!\vi_\e-\vi\!|^2.
\end{multline*}
Choosing $R=\lambda_\e^{-n}$ for some $n\ge1$, we obtain $R^{-2}\ll\lambda_\e$ hence $R^{-2}-\lambda_\e\alpha<0$ for $\e$ small enough. Using~\eqref{eq:boundvv} to estimate the second right-hand side term then yields
\begin{align*}
\partial_t\int_{\R^2} a_\e\xi_R^z|\!\vi_\e-\vi\!|^2\lesssim_{t,\theta} R^{2\theta}\lambda_\e+\int_{\R^2} a_\e\xi_R^z|\!\vi_\e-\vi\!|^2\lesssim \lambda_\e^{1-2n\theta}+\int_{\R^2} a_\e\xi_R^z|\!\vi_\e-\vi\!|^2.
\end{align*}
For $\theta>0$ small enough, the conclusion follows from the Grönwall inequality.

\medskip
\step2 Regime~\GLd.

Using the choice of the scalings for $\lambda_\e,h,F$ in the regime~\GLd, equation~\eqref{eq:GLv1} takes on the following guise,
\begin{align*}
\partial_t\!\vi_\e=\alpha^{-1}\nabla(\hat a^{-1}\Div(\hat a\!\vi_\e))+\bigg((\alpha-\Jb\beta)\Big(\nabla^\bot \hat h-\hat F^\bot-\frac{2N_\e}\Log\!\vi_\e\Big)\bigg)\,\curl\!\vi_\e,
\end{align*}
with initial data $\vi_\e\!|_{t=0}=\vi^\circ$. As $\frac{N_\e}\Log\to\lambda\in(0,\infty)$, it is formally clear that $\vi_\e$ should converge to the (local) solution $\vi$ of equation~\eqref{eq:limGL2}. Existence and uniqueness of $\vi$ are given by Proposition~\ref{prop:GLvprop} just as for $\vi_\e$,
and the following bounds hold for all $t\in[0,T)$,
\begin{align}\label{eq:boundvveps2}
\|(\vi^t,\vi^t_\e)\|_{W^{1,\infty}}\lesssim_t1,\qquad \|\curl\!\vi^t\!\|_{\Ld^1}=1.
\end{align}
Using the equations for $\vi_\e,\vi$, we find
\begin{multline*}
\partial_t\int_{\R^2} \hat a\xi_R^z|\!\vi_\e-\vi\!|^2\\
=2\alpha^{-1}\int_{\R^2} \hat a\xi_R^z(\vi_\e-\vi)\cdot\nabla(\hat a^{-1}\Div(\hat a(\vi_\e-\vi)))-\frac{4\alpha N_\e}\Log\int_{\R^2} \hat a\xi_R^z|\!\vi_\e-\vi\!|^2\curl\!\vi_\e\\
+2\int_{\R^2} \hat a\xi_R^z\bigg((\alpha-\Jb\beta)\Big(\nabla^\bot\hat h-\hat F^\bot-\frac{2N_\e}\Log\vi\Big)\bigg)\cdot(\vi_\e-\vi)(\curl\!\vi_\e-\curl\!\vi)\\
-4\Big(\frac{N_\e}\Log-\lambda\Big)\int_{\R^2} \hat a\xi_R^z(\vi_\e-\vi)\cdot(\alpha-\Jb\beta)\vi\curl\!\vi.
\end{multline*}
Integrating by parts, using the weighted Delort-type identity~\eqref{eq:weightedDelort} in the form
\begin{multline*}
(\vi_\e-\vi)\,\curl(\vi_\e-\vi)=\hat a^{-1}(\vi_\e-\vi)^\bot\Div(\hat a(\vi_\e-\vi))\\
-\frac12|\!\vi_\e-\vi\!|^2\nabla^\bot \hat h-\hat a^{-1}(\Div(\hat a S_{\vi_\e-\vi}))^\bot,
\end{multline*}
using the properties~\eqref{eq:boundvveps2} of $\vi_\e,\vi$, the assumption $\|(\nabla\hat h,\hat F)\|_{W^{1,\infty}}\lesssim1$, and simplifying the terms as in Step~1, we easily deduce
\begin{multline*}
\partial_t\int_{\R^2} \hat a\xi_R^z|\!\vi_\e-\vi\!|^2\le-2\alpha^{-1}\int_{\R^2} \hat a^{-1}\xi_R^z|\Div(\hat a(\vi_\e-\vi))|^2\\
+C_t\int_{\R^2} \xi_R^z|\!\vi_\e-\vi\!||\Div(\hat a(\vi_\e-\vi))|+C_t\int_{\R^2} \hat a\xi_R^z|\!\vi_\e-\vi\!|^2+C_t\Big|\frac{N_\e}\Log-\lambda\Big|,
\end{multline*}
hence $\partial_t\int_{\R^2} \hat a\xi_R^z|\!\vi_\e-\vi\!|^2\lesssim C_t\int_{\R^2} \hat a\xi_R^z|\!\vi_\e-\vi\!|^2+o_t(1)$, and the conclusion now follows from the Grönwall inequality, letting $R\uparrow\infty$.

\medskip
\noindent\step3 Regime~\GLup.

Using the choice of the scalings for $\lambda_\e,h,F$ in the regime~\GLup, equation~\eqref{eq:GLv1} takes on the following guise,
\begin{align*}
\partial_t\!\vi_\e=\alpha^{-1}\nabla(\hat a^{-1}\Div(\hat a\!\vi_\e))+(\alpha-\Jb\beta)\Big(\nabla^\bot\hat h-\hat F^\bot-\frac{2N_\e}\Log\!\vi_\e\Big)\curl\!\vi_\e,
\end{align*}
with initial data $\vi_\e\!|_{t=0}=\vi^\circ$. As by assumption $\frac{N_\e}\Log\to0$, it is formally clear that $\vi_\e$ should converge to the solution $\vi$ of equation~\eqref{eq:limGL4} as $\e\downarrow0$. Existence, uniqueness, and regularity of this (global) solution $\vi$ are given by Proposition~\ref{prop:GLvprop} just as for $\vi_\e$, and the convergence result follows as in Step~2 (with $\lambda=0$).

\medskip
\noindent\step4 Regime~\GLdp.

Using the choice of the scalings for $\lambda_\e,h,F$ in the regime~\GLdp, equation~\eqref{eq:GLv1} takes the following form, with $a_\e:=\hat a^{\lambda_\e}$,
\begin{gather*}
\partial_t\!\vi_\e=\nabla\!\pre_\e+(\alpha-\Jb\beta)\Big(\nabla^\bot\hat h-\hat F^\bot-\frac{2\lambda_\e^{-1}N_\e}\Log\!\vi_\e\Big)\curl\!\vi_\e,\\
\pre_\e:=(\lambda_\e\alpha a_\e)^{-1}\Div(a_\e\!\vi_\e),
\end{gather*}
with initial data $\vi_\e\!|_{t=0}=\vi_\e^\circ\to\vi^\circ$ in $\Ld^2_\uloc(\R^2)^2$. As by assumption $\lambda_\e^{-1}\frac{N_\e}\Log\to0$, it is formally clear that $\vi_\e$ should converge to the solution $\vi$ of equation~\eqref{eq:limGL3} as $\e\downarrow0$. Existence, uniqueness, and regularity of this (global) solution $\vi$ are given by Proposition~\ref{prop:GLvprop} just as for $\vi_\e$,
and the convergence result follows as in Step~1.
\end{proof}

\subsection{Nondilute parabolic case}\label{chap:deg-eqn-exist-bound}

Let us examine the vorticity formulation of equation~\eqref{eq:GLv1-hd} for $\vi_\e$. As in~\eqref{eq:vortform}, in terms of $\md_\e:=\curl\!\vi_\e$ and $\dd_\e:=\Div(a\!\vi_\e)$, it takes on the following guise,
\begin{gather*}
\begin{cases}\partial_t\!\md_\e=-\Div(\Gamma_\e^\bot\!\md_\e),\\
\partial_t\!\dd_\e-\lambda_\e^{-1}\triangle\!\dd_\e+\lambda_\e^{-1}\Div(\dd_\e\!\nabla h)=\Div(a\Gamma_\e\!\md_\e),\\
\curl\!\vi_\e=\md_\e,\quad\Div(a\!\vi_\e)=\dd_\e,\\
\md_\e\!|_{t=0}=\curl\!\vi^\circ,\quad\dd_\e\!|_{t=0}=\Div(a\!\vi^\circ).
\end{cases}
\end{gather*}
In the present nondilute regime, as $\lambda_\e\uparrow\infty$, the diffusion tends to be degenerate and more work is thus needed to ensure the validity of uniform a priori estimates. The key consists in suitably exploiting the well-posedness of the degenerate limiting equation, studied in~\cite{D-16}. As an immediate corollary of such estimates, we also deduce that $\vi_\e$ converges to the solution $\vi$ of this degenerate equation.

\begin{prop}\label{prop:apriori-est-veps-hd}
Let $h:\R^2\to\R$, $a:=e^h$, $F:\R^2\to\R^2$, and let $\vi_\e^\circ:\R^2\to\R^2$ be bounded in $W^{1,q}(\R^2)^2$ for all $q>2$ and satisfy $\curl\!\vi_\e^\circ\in\Pc(\R^2)$. For some $s>0$, assume that $h\in W^{s+6,\infty}(\R^2)$, $F\in W^{s+5,\infty}(\R^2)^2$, that $\vi_\e^\circ$ is bounded in $W^{s+5,\infty}(\R^2)^2$, that $\curl\!\vi_\e^\circ$ is bounded in $H^{s+4}(\R^2)$, and that $\Div(a\!\vi_\e^\circ)$ is bounded in $H^{s+3}(\R^2)$.\\
In the regime~\GLt{} with $\vi_\e^\circ=\vi^\circ$, there exists a unique (global) solution $\vi_\e$ of~\eqref{eq:GLv1-hd} in $\R^+\times\R^2$, in the space $\Ld^\infty_\loc(\R^+;\vi^\circ+H^{s+4}(\R^2)^2)$.
Moreover, all the properties of Assumption~\ref{as:apveps}(a) are satisfied: for all $T>0$ and $q>2$, there is some $\e_0(T)>0$\footnote{Only depending on an upper bound on $T$, $s$, $s^{-1}$, $\|\hat h\|_{W^{s+6,\infty}}$, $\|(\hat F,\vi^\circ)\|_{W^{s+5,\infty}}$, $\|\!\vi^\circ\!\|_{W^{1,q}}$, $\|\!\md^\circ\!\|_{H^{s+4}}$, and $\|\!\dd^\circ\!\|_{H^{s+3}}$.} such that for all $0<\e<\e_0(T)$ and $0\le t\le T$,
\begin{gather}
\|(\vi_\e^t,\nabla\!\vi_\e^t)\|_{(\Ld^2+\Ld^q)\cap\Ld^\infty}\lesssim_{t,q}1,\qquad\|\!\md_\e^t\!\|_{\Ld^1\cap\Ld^\infty}\lesssim_{t}1,\qquad \|\partial_t\!\vi_\e^t\!\|_{\Ld^2\cap\Ld^\infty}\lesssim_{t}1,\nonumber\\
\|\!\dd_\e^t\!\|_{\Ld^2\cap\Ld^\infty}\lesssim_{t}1,\qquad\|\nabla\!\dd_\e^t\!\|_{\Ld^2\cap\Ld^\infty}\lesssim_{t}1,\qquad\|\partial_t\!\dd_\e^t\!\|_{\Ld^2}\lesssim_{t}1.
\label{eq:apriori-est-veps-hd}
\end{gather}
In addition, there holds $\vi_\e\to\vi$ in $\Ld^\infty_\loc(\R^+;\vi^\circ+H^{s+3}(\R^2)^2)$ as $\e\downarrow0$,
where $\vi$ is the unique (global) solution of
\begin{equation}\label{eq:GL3lim-conv}
\begin{cases}\partial_t\!\vi=-(\hat F^\bot+2\!\vi)\,\curl\!\vi,\\
\vi\!|_{t=0}=\vi^\circ,
\end{cases}
\end{equation}
in $\R^+\times\R^d$, in the space $\Ld^\infty_\loc(\R^+;\vi^\circ+H^{s+4}\cap W^{s+4,\infty}(\R^2)^2)$.
\end{prop}

\begin{proof}
Direct estimates on $\vi_\e$ as in~\cite{D-16} are not uniform with respect to $\lambda_\e\gg1$. As we show, however, exploiting strong a priori estimates on the limiting solution $\vi$ allows to deduce the desired uniform estimates on~$\vi_\e$. We split the proof into two steps.

\medskip
\noindent\step1 A priori estimates.

Let $s>0$, and assume that $\hat h\in W^{s+3,\infty}(\R^2)$, $\hat F\in W^{s+2,\infty}(\R^2)^2$, and that there exists a unique global solution $\vi$ of equation~\eqref{eq:GL3lim-conv} with $\vi\in\Ld^\infty_\loc(\R^+;\vi^\circ+\Ld^2(\R^2)^2)\cap \Ld^\infty_\loc(\R^+;W^{s+2,\infty}(\R^2)^2)$ and with $\md:=\curl\!\vi$, $\dd:=\Div(\hat a\!\vi)\in\Ld^\infty_\loc(\R^+;H^{s+2}(\R^2))$.
Also assume that there exists a unique global solution $\vi_\e$ of~\eqref{eq:GLv1-hd} in $\Ld^\infty_\loc(\R^+;\vi^\circ+H^{s+2}(\R^2))$.
In this step, we consider the nondilute regime $\lambda_\e\gg1$, and we show that for any fixed $t\ge0$ we have for all $\e>0$ small enough (that is, for all $\lambda_\e$ large enough),
\begin{gather}\label{eq:apriori-est-global-vmd-deg}
\|\!\vi_\e-\vi\!\|_{\Ld^\infty_tH^{s+1}}+\|\!\md_\e-\md\!\|_{\Ld^\infty_tH^{s+1}}+\|\!\dd_\e-\dd\!\|_{\Ld^\infty_tH^s}\le C_t\lambda_\e^{-1},\\
\|\!\dd_\e-\dd\!\|_{\Ld^\infty_tH^{s+1}}\le C_t\lambda_\e^{-1/2},\nonumber
\end{gather}
hence in particular,
\begin{gather}\label{eq:apriori-est-global-vmd-deg-boundedness}
\|\!\vi_\e-\vi^\circ\!\|_{\Ld^\infty_tH^{s+2}}+\|\!\md_\e\!\|_{\Ld^\infty_tH^{s+1}}+\|\!\dd_\e\!\|_{\Ld^\infty_tH^{s+1}}\le C_t,
\end{gather}
where the constant $C_t$ only depends on an upper bound on $\lambda_\e^{-1}$, $s$, $s^{-1}$, $\|\hat h\|_{W^{s+3,\infty}}$, $\|\hat F\|_{W^{s+2,\infty}}$, $\|\!\vi\!\|_{\Ld^\infty_tW^{s+2,\infty}}$, $\|(\md,\dd)\|_{\Ld^\infty_tH^{s+2}}$, $\|\!\vi-\vi^\circ\!\|_{\Ld^\infty_t\Ld^2}$, and on time $t$.
We split the proof into six further substeps. In this step, we use the notation $\lesssim_t$ for $\le$ up to a constant $C_t>0$ as above, and we use the notation $\lesssim$ for $\le$ up to a constant that depends only on an upper bound on $\lambda_\e^{-1}$, $\|\hat h\|_{W^{s+3,\infty}}$, and on $\|\hat F\|_{W^{s+2,\infty}}$.

\medskip
\noindent\substep{1.1} Notation.

Define $\delta\!\vi_\e:=\lambda_\e(\vi_\e-\vi)$, $\delta\!\md_\e:=\curl\delta\!\vi_\e=\lambda_\e(\md_\e-\md)$, and $\delta\!\dd_\e:=\Div(\hat a\delta\!\vi_\e)=\lambda_\e(\dd_\e-\dd)$.
Given the choice of the scalings, equation~\eqref{eq:GLv1-hd} for $\vi_\e$ takes on the following guise,
\begin{align}\label{eq:GLv1-hd-rewrite}
\partial_t\!\vi_\e=\lambda_\e^{-1}\nabla(\hat a^{-1}\!\dd_\e)+\big(\lambda_\e^{-1}\nabla^\bot\hat h-\hat F^\bot-2\!\vi_\e\!\big)\md_\e,
\end{align}
and hence, decomposing $\vi_\e=\vi+\lambda_\e^{-1}\delta\!\vi_\e$,
\begin{multline*}
\partial_t\!\vi+\lambda_\e^{-1}\partial_t\delta\!\vi_\e=-(\hat F^\bot+2\!\vi)\md+\lambda_\e^{-1}\Big(\nabla(\hat a^{-1}\!\dd)+\md\nabla^\bot\hat h-\hat F^\bot\delta\!\md_\e-2\!\vi\!\delta\!\md_\e-2\!\md\!\delta\!\vi_\e\Big)\\
+\lambda_\e^{-2}\Big(\nabla(\hat a^{-1}\delta\!\dd_\e)+\delta\!\md_\e\!\nabla^\bot\hat h-2\delta\!\vi_\e\!\delta\!\md_\e\Big).
\end{multline*}
Injecting equation~\eqref{eq:GL3lim-conv} for $\vi$ and multiplying both sides by $\lambda_\e$, we obtain the following equation for $\delta\!\vi_\e$,
\begin{align}\label{eq:deltaveps}
\partial_t\delta\!\vi_\e=\lambda_\e^{-1}\nabla(\hat a^{-1}\delta\!\dd_\e)+(W_\e-2\lambda_\e^{-1}\delta\!\vi_\e)\delta\!\md_\e-2\!\md\!\delta\!\vi_\e+G,
\end{align}
with initial data $\delta\!\vi_\e\!|_{t=0}=0$,
where we have set
\begin{align*}
G:=\nabla(\hat a^{-1}\!\dd)+\md\nabla^\bot\hat h,\qquad W_\e:=\lambda_\e^{-1}\nabla^\bot\hat h-\hat F^\bot-2\!\vi.
\end{align*}
Taking the curl of~\eqref{eq:deltaveps} leads to
\begin{align}\label{eq:deltamdeps-todegene}
\partial_t\delta\!\md_\e=-\Div((W_\e^\bot-2\lambda_\e^{-1}\delta\!\vi_\e^\bot)\delta\!\md_\e)+2\delta\!\vi_\e^\bot\cdot\nabla\!\md-2\!\md\delta\!\md_\e+\curl G,
\end{align}
while applying $\Div(\hat a\,\cdot)$ yields
\begin{multline}\label{eq:deltaddeps-todegene}
\partial_t\delta\!\dd_\e=\lambda_\e^{-1}\triangle\delta\!\dd_\e-\lambda_\e^{-1}\Div(\delta\!\dd_\e\!\nabla\hat h)\\
+\Div(\hat a(W_\e-2\lambda_\e^{-1}\delta\!\vi_\e)\delta\!\md_\e)-2\Div(\hat a\!\md\!\delta\!\vi_\e)+\Div(\hat aG),
\end{multline}
with initial data $\delta\!\md_\e\!|_{t=0}=0$ and $\delta\!\dd_\e\!|_{t=0}=0$.
Proving the result~\eqref{eq:apriori-est-global-vmd-deg} thus amounts to establishing uniform a priori estimates for the solutions $\delta\!\vi_\e$, $\delta\!\md_\e$, and $\delta\!\dd_\e$ of the above equations.

\medskip
\noindent\substep{1.2} $\Ld^2$-estimate on $\delta\!\vi_\e$ and $\delta\!\md_\e$.

\nopagebreak
In this step, we show that
\begin{align}\label{eq:boundL2-deltaveps}
\|\delta\!\vi_\e\!\|_{\Ld^\infty_t\Ld^2}+\|\delta\!\md_\e\!\|_{\Ld^\infty_t(\dot H^{-1}\cap\Ld^2)}+\|\delta\!\dd_\e\!\|_{\Ld^\infty_t\dot H^{-1}}~\lesssim_t1.
\end{align}
On the one hand, from equation~\eqref{eq:deltaveps}, noting that $-2\lambda_\e^{-1}\delta\!\vi_\e\delta\!\md_\e-2\!\md\!\delta\!\vi_\e=-2\!\md_\e\!\delta\!\vi_\e$, we find by integration by parts,
\begin{align*}
\partial_t\int_{\R^2}\hat a|\delta\!\vi_\e\!|^2&=-2\lambda_\e^{-1}\int_{\R^2}\hat a^{-1}|\delta\!\dd_\e\!|^2+2\int_{\R^2}\hat a\delta\!\vi_\e\cdot\big(W_\e\delta\!\md_\e-2\!\md_\e\!\delta\!\vi_\e+G\big)\\
&\le 2\int_{\R^2}\hat a\delta\!\vi_\e\cdot\big(W_\e\delta\!\md_\e+G\big),
\end{align*}
and hence, using the Cauchy-Schwarz inequality and injecting the definition of $G$ and $W_\e$,
\begin{align*}
\partial_t\Big(\int_{\R^2}\hat a|\delta\!\vi_\e^t\!|^2\Big)^{1/2}&\le\|W_\e^t\|_{\Ld^\infty}\Big(\int_{\R^2}\hat a|\delta\!\md_\e^t\!|^2\Big)^{1/2}+\Big(\int_{\R^2}\hat a|G^t|^2\Big)^{1/2}\\
&\lesssim\|(\nabla\hat h,\hat F,\vi^t)\|_{\Ld^\infty}\|\delta\!\md_\e^t\!\|_{\Ld^2}+\|\!\dd^t\!\|_{H^1}+\|\!\md^t\!\|_{\Ld^2}\\
&\lesssim_t 1+\|\delta\!\md_\e^t\!\|_{\Ld^2},
\end{align*}
that is,
\begin{align}\label{eq:boundL2-deltaveps-pre}
\|\delta\!\vi_\e\!\|_{\Ld^\infty_t\Ld^2}
\lesssim_t \|\delta\!\md_\e\!\|_{\Ld^\infty_t\Ld^2}+1.
\end{align}
On the other hand, equation~\eqref{eq:deltamdeps-todegene} yields by integration by parts,
\begin{multline*}
\partial_t\int_{\R^2}|\delta\!\md_\e\!|^2=\int_{\R^2}|\delta\!\md_\e\!|^{2}\Div(-W_\e^\bot+2\lambda_\e^{-1}\delta\!\vi_\e^\bot)\\
-4\int_{\R^2}|\delta\!\md_\e\!|^{2}\!\md+2\int_{\R^2}\delta\!\md_\e\big(2\delta\!\vi_\e^\bot\cdot\nabla\!\md+\curl G\big),
\end{multline*}
and hence, decomposing $\Div(\lambda_\e^{-1}\delta\!\vi_\e^\bot)=-\lambda_\e^{-1}\delta\!\md_\e=\md-\md_\e\le\md$,
\begin{align*}
\partial_t\int_{\R^2}|\delta\!\md_\e\!|^2&\le\int_{\R^2}|\delta\!\md_\e\!|^{2}\curl W_\e+2\int_{\R^2}\delta\!\md_\e\big(2\delta\!\vi_\e^\bot\cdot\nabla\!\md+\curl G\big)\\
&\le\|\nabla W_\e\|_{\Ld^\infty}\|\delta\!\md_\e\!\|_{\Ld^2}^{2}+4\|\nabla\!\md\!\|_{\Ld^\infty}\|\delta\!\vi_\e\!\|_{\Ld^{2}}\|\delta\!\md_\e\!\|_{\Ld^2}+2\|\curl G\|_{\Ld^2}\|\delta\!\md_\e\!\|_{\Ld^2}.
\end{align*}
Injecting the definitions of $G$ and $W_\e$ with $\lambda_\e^{-1}\lesssim1$, and using~\eqref{eq:boundL2-deltaveps-pre} to estimate the $\Ld^2$-norm of $\delta\!\vi_\e$ in the right-hand side, we deduce
\begin{align*}
\partial_t\|\delta\!\md_\e^t\!\|_{\Ld^2}&\lesssim_t \|\delta\!\md_\e^t\!\|_{\Ld^2}+\|\delta\!\vi_\e^t\!\|_{\Ld^{2}}+1\lesssim_t\|\delta\!\md_\e\!\|_{\Ld^\infty_t\Ld^2}+1.
\end{align*}
Combining this with~\eqref{eq:boundL2-deltaveps-pre} and with the obvious estimate $\|(\delta\!\md_\e,\delta\!\dd_\e)\|_{\dot H^{-1}}\lesssim\|\delta\!\vi_\e\!\|_{\Ld^2}$, the conclusion~\eqref{eq:boundL2-deltaveps} follows from the Grönwall inequality.

\medskip
\noindent\substep{1.3} $H^{s+1}$-estimate on $\delta\!\md_\e$.
\nopagebreak

In this step, we show that
\begin{multline}\label{eq:derboundHs-deltamdeps}
\partial_t\|\delta\!\md_\e\!\|_{H^{s+1}}\lesssim_t1+\|\delta\!\md_\e\!\|_{H^{s+1}}+\|\delta\!\dd_\e\!\|_{H^{s}}\\
+\lambda_\e^{-1}\big(\|\delta\!\md_\e\!\|_{H^{s+1}}^2+\|\delta\!\md_\e\!\|_{H^{s+1}}\|\delta\!\dd_\e\!\|_{H^{s+1}}\big).
\end{multline}
Arguing as in~\cite[Proof of Lemma~2.2]{D-16}, with $s>0$, the time derivative of the $H^s$-norm of the vorticity $\delta\!\md_\e$ is computed as follows,
\begingroup\allowdisplaybreaks
\begin{align*}
\partial_t\|\delta\!\md_\e\!\|_{H^{s+1}}&\le\frac12\|\Div(W_\e^\bot-2\lambda_\e^{-1}\delta\!\vi_\e^\bot)\|_{\Ld^\infty}\|\delta\!\md_\e\!\|_{H^{s+1}}+\|[\langle\nabla\rangle^{s+1}\Div,W_\e^\bot]\delta\!\md_\e\!\|_{\Ld^2}\\
&\hspace{1cm}+2\lambda_\e^{-1}\|[\langle\nabla\rangle^{s+1}\Div,\delta\!\vi_\e^\bot]\delta\!\md_\e\!\|_{\Ld^2}+2\|\delta\!\vi_\e^\bot\cdot\nabla\!\md\!\|_{H^{s+1}}\\
&\hspace{2cm}+2\|\!\md\delta\!\md_\e\!\|_{H^{s+1}}+\|\curl G\|_{H^{s+1}}\\
&\lesssim(\|W_\e\|_{W^{s+2,\infty}}+\|\!\md\!\|_{W^{s+1,\infty}})\|\delta\!\md_\e\!\|_{H^{s+1}}+\|\!\md\!\|_{H^{s+2}}\|\delta\!\vi_\e\!\|_{H^{s+1}}\\
&\hspace{1cm}+\lambda_\e^{-1}\big(\|\delta\!\vi_\e\!\|_{W^{1,\infty}}\|\delta\!\md_\e\!\|_{H^{s+1}}+\|\delta\!\md_\e\!\|_{\Ld^\infty}\|\delta\!\vi_\e\!\|_{H^{s+2}}\big)+\|\curl G\|_{H^{s+1}}.
\end{align*}
\endgroup
Injecting the definition of $G$ and $W_\e$ with $\lambda_\e^{-1}\lesssim1$, and using the Sobolev embedding for $\Ld^\infty(\R^2)$ into $H^{s+1}(\R^2)$ with $s>0$, we find
\begin{align}\label{eq:derboundHs-deltamdeps-pre}
\partial_t\|\delta\!\md_\e\!\|_{H^{s+1}}\lesssim_t\|\delta\!\md_\e\!\|_{H^{s+1}}+\|\delta\!\vi_\e\!\|_{H^{s+1}}
+\lambda_\e^{-1}\|\delta\!\vi_\e\!\|_{H^{s+2}}\|\delta\!\md_\e\!\|_{H^{s+1}}+1.
\end{align}
Decomposing $\delta\!\vi_\e=\hat a^{-1}\nabla^\bot(\Div\hat  a^{-1}\nabla)^{-1}\delta\!\md_\e+\nabla(\Div\hat  a\nabla)^{-1}\delta\!\dd_\e$, we appeal to e.g.~\cite[Lemma~2.6]{D-16} in the form
\begin{align}\label{eq:bound-pot-deltav-md}
\|\delta\!\vi_\e\!\|_{H^{r+1}}\lesssim\|\delta\!\md_\e\!\|_{\dot H^{-1}\cap H^{r}}+\|\delta\!\dd_\e\!\|_{\dot H^{-1}\cap H^{r}},
\end{align}
with successively $r=s$ and $r=s+1$. Injecting this into~\eqref{eq:derboundHs-deltamdeps-pre}, and using the result~\eqref{eq:boundL2-deltaveps} of Substep~1.2 in the form $\|(\delta\!\md_\e,\delta\!\dd_\e)\|_{\dot H^{-1}}\lesssim_t1$, the conclusion~\eqref{eq:derboundHs-deltamdeps} follows.

\medskip
\noindent\substep{1.4} $H^{s+1}$-estimate on $\delta\!\dd_\e$ without loss of regularity.

In this step we show that
\begin{align}\label{eq:boundHs+-deltaddeps}
\lambda_\e^{-1/2}\|\delta\!\dd_\e\!\|_{\Ld^\infty_tH^{s+1}}\lesssim_t 1+\|\delta\!\md_\e\!\|_{\Ld^\infty_tH^{s+1}}+\|\delta\!\dd_\e\!\|_{\Ld^\infty_tH^{s}}+\lambda_\e^{-1}\|\delta\!\md_\e\!\|_{\Ld^\infty_tH^{s+1}}^2.
\end{align}
Equation~\eqref{eq:deltaddeps-todegene} for the divergence $\delta\!\dd_\e$ takes the form $\partial_t\delta\!\dd_\e=\lambda_\e^{-1}\triangle\delta\!\dd_\e+\Div\!H_\e$, where we have set
\[H_\e:=-\lambda_\e^{-1}\delta\!\dd_\e\nabla\hat h+a(W_\e-2\lambda_\e^{-1}\delta\!\vi_\e)\delta\!\md_\e-2a\!\md\!\delta\!\vi_\e+aG.\]
Testing this equation with $(-\triangle)^{-1}\langle\nabla\rangle^{2(s+1)}\partial_t\delta\!\dd_\e$, arguing as in~\cite[Proof of Lemma~2.3(i)]{D-16}, we find
\begin{align*}
\lambda_\e^{-1}\|\delta\!\dd_\e\!\|_{\Ld^\infty_tH^{s+1}}^2\le \int_0^t\|H_\e^u\|_{H^{s+1}}^2du,
\end{align*}
and hence, injecting the definitions of $H_\e$, $G$, and $W_\e$, with $s>0$,
\begin{multline*}
\lambda_\e^{-1}\|\delta\!\dd_\e\!\|_{\Ld^\infty_tH^{s+1}}^2\lesssim_t \lambda_\e^{-2}\int_0^t\|\delta\!\dd_\e^u\!\|_{H^{s+1}}^2du+\lambda_\e^{-2}\|\delta\!\vi_\e\!\|_{\Ld^\infty_tH^{s+1}}^2\|\delta\!\md_\e\!\|_{\Ld^\infty_tH^{s+1}}^2\\
+\|\delta\!\md_\e\!\|_{\Ld^\infty_tH^{s+1}}^2+\|\delta\!\vi_\e\!\|_{\Ld^\infty_tH^{s+1}}^2+1.
\end{multline*}
The Grönwall inequality with $\lambda_\e^{-1}\lesssim1$ then yields
\begin{align*}
\lambda_\e^{-1}\|\delta\!\dd_\e\!\|_{\Ld^\infty_tH^{s+1}}^2\lesssim_t 1+\|\delta\!\md_\e\!\|_{\Ld^\infty_tH^{s+1}}^2+\|\delta\!\vi_\e\!\|_{\Ld^\infty_tH^{s+1}}^2+\lambda_\e^{-2}\|\delta\!\vi_\e\!\|_{\Ld^\infty_tH^{s+1}}^2\|\delta\!\md_\e\!\|_{\Ld^\infty_tH^{s+1}}^2.
\end{align*}
The conclusion~\eqref{eq:boundHs+-deltaddeps} follows from this together with the bound~\eqref{eq:bound-pot-deltav-md} and with the result~\eqref{eq:boundL2-deltaveps} of Substep~1.2 in the form $\|(\delta\!\md_\e,\delta\!\dd_\e)\|_{\dot H^{-1}}\lesssim_t1$.

\medskip
\noindent\substep{1.5} $H^{s}$-estimate on $\delta\!\dd_\e$ with loss of regularity.

In this step, we show that
\begin{align}\label{eq:derboundHs-deltaddeps}
\partial_t\|\delta\!\dd_\e\!\|_{H^s}\lesssim_t 1+\|\delta\!\dd_\e\!\|_{H^s}+\|\delta\!\md_\e\!\|_{H^{s+1}}+\lambda_\e^{-1}\big(\|\delta\!\md_\e\!\|_{H^{s+1}}^2+\|\delta\!\dd_\e\!\|_{H^{s}}\|\delta\!\md_\e\!\|_{H^{s+1}}\big).
\end{align}
Equation~\eqref{eq:deltaddeps-todegene} for the divergence $\delta\!\dd_\e$ yields after integration by parts,
\begin{align*}
\partial_t\|\delta\!\dd_\e\!\|_{H^s}^2&\le-2\lambda_\e^{-1}\int_{\R^2}|\nabla\langle\nabla\rangle^s\delta\!\dd_\e\!|^2+2\lambda_\e^{-1}\int_{\R^2}\langle\nabla\rangle^s(\delta\!\dd_\e\!\nabla\hat h)\cdot\nabla\langle\nabla\rangle^s\delta\!\dd_\e\\
&\hspace{1cm}+2\int_{\R^2}\big(\langle\nabla\rangle^s\delta\!\dd_\e\big)\Div\langle\nabla\rangle^s\big(a(W_\e-2\lambda_\e^{-1}\delta\!\vi_\e)\delta\!\md_\e-2a\!\md\!\delta\!\vi_\e+aG\big)\\
&\le\lambda_\e^{-1}\|\delta\!\dd_\e\!\nabla\hat h\|_{H^s}^2+2\|\delta\!\dd_\e\!\|_{H^s}\big(\|a(W_\e-2\lambda_\e^{-1}\delta\!\vi_\e)\delta\!\md_\e+aG\|_{H^{s+1}}\\
&\hspace{6cm}+2\|\!\md\!\delta\!\dd_\e\!\|_{H^s}+2\|a\delta\!\vi_\e\!\cdot\nabla\!\md\!\|_{H^s}\big),
\end{align*}
and hence, injecting the definition of $G$ and $W_\e$,
\begin{align*}
\partial_t\|\delta\!\dd_\e\!\|_{H^s}&\lesssim_t 1+\|\delta\!\dd_\e\!\|_{H^s}+\|\delta\!\md_\e\!\|_{H^{s+1}}+\|\delta\!\vi_\e\!\|_{H^s}+\lambda_\e^{-1}\|\delta\!\vi_\e\!\|_{H^{s+1}}\|\delta\!\md_\e\!\|_{H^{s+1}}.
\end{align*}
The result~\eqref{eq:derboundHs-deltaddeps} follows from this together with the bound~\eqref{eq:bound-pot-deltav-md} and with the result~\eqref{eq:boundL2-deltaveps} of Substep~1.2 in the form $\|(\delta\!\md_\e,\delta\!\dd_\e)\|_{\dot H^{-1}}\lesssim_t1$.

\medskip
\noindent\substep{1.6} Proof of~\eqref{eq:apriori-est-global-vmd-deg} and~\eqref{eq:apriori-est-global-vmd-deg-boundedness}.

Injecting~\eqref{eq:boundHs+-deltaddeps} into~\eqref{eq:derboundHs-deltamdeps} with $\lambda_\e^{-1}\lesssim1$, we find
\begin{multline*}
\partial_t\|\delta\!\md_\e\!\|_{\Ld^\infty_tH^{s+1}}\lesssim_t1+\|\delta\!\md_\e\!\|_{\Ld^\infty_tH^{s+1}}+\|\delta\!\dd_\e\!\|_{\Ld^\infty_tH^{s}}\\
+\lambda_\e^{-1/2}\big(\|\delta\!\md_\e\!\|_{\Ld^\infty_tH^{s+1}}^2+\|\delta\!\dd_\e\!\|_{\Ld^\infty_tH^{s}}^2\big)+\lambda_\e^{-3/2}\|\delta\!\md_\e\!\|_{\Ld^\infty_tH^{s+1}}^3.
\end{multline*}
Together with~\eqref{eq:derboundHs-deltaddeps}, this yields
\begin{eqnarray*}
\lefteqn{\partial_t\big(\|\delta\!\md_\e\!\|_{\Ld^\infty_tH^{s+1}}+\|\delta\!\dd_\e\!\|_{\Ld^\infty_tH^s}\big)}\\
&\lesssim_t&1+\|\delta\!\md_\e\!\|_{\Ld^\infty_tH^{s+1}}+\|\delta\!\dd_\e\!\|_{\Ld^\infty_tH^{s}}\\
&&\hspace{1cm}+\lambda_\e^{-1/2}\big(\|\delta\!\md_\e\!\|_{\Ld^\infty_tH^{s+1}}^2+\|\delta\!\dd_\e\!\|_{\Ld^\infty_tH^{s}}^2\big)+\lambda_\e^{-3/2}\|\delta\!\md_\e\!\|_{\Ld^\infty_tH^{s+1}}^3\\
&\lesssim_t&1+\|\delta\!\md_\e\!\|_{\Ld^\infty_tH^{s+1}}+\|\delta\!\dd_\e\!\|_{\Ld^\infty_tH^{s}}+\lambda_\e^{-3/4}(\|\delta\!\md_\e\!\|_{\Ld^\infty_tH^{s+1}}+\|\delta\!\dd_\e\!\|_{\Ld^\infty_tH^{s}})^3,
\end{eqnarray*}
and hence, by time integration,
\begin{align*}
\|\delta\!\md_\e\!\|_{\Ld^\infty_tH^{s+1}}+\|\delta\!\dd_\e\!\|_{\Ld^\infty_tH^s}&\le C_t\Big(1+\lambda_\e^{-3/4}\big(\|\delta\!\md_\e\!\|_{\Ld^\infty_tH^{s+1}}+\|\delta\!\dd_\e\!\|_{\Ld^\infty_tH^s}\big)^3\Big).
\end{align*}
For any $t\ge0$, choosing $\e>0$ small enough such that $\lambda_\e^{-3/4}\le(2C_t)^{-3}$, we obtain
\begin{align*}
\|\delta\!\md_\e\!\|_{\Ld^\infty_tH^{s+1}}+\|\delta\!\dd_\e\!\|_{\Ld^\infty_tH^s}\lesssim_t1.
\end{align*}
Combining this with the bound~\eqref{eq:bound-pot-deltav-md} and with the result~\eqref{eq:boundL2-deltaveps} of Substep~1.2 in the form of $\|(\delta\!\md_\e,\delta\!\dd_\e)\|_{\dot H^{-1}}\lesssim_t1$, we deduce
\begin{align*}
\|\delta\!\vi_\e\!\|_{\Ld^\infty_tH^{s+1}}+\|\delta\!\md_\e\!\|_{\Ld^\infty_tH^{s+1}}+\|\delta\!\dd_\e\!\|_{\Ld^\infty_tH^s}\lesssim_t1.
\end{align*}
Injecting this into the result~\eqref{eq:boundHs+-deltaddeps} of Substep~1.4, we find
\begin{align*}
\|\delta\!\dd_\e\!\|_{\Ld^\infty_tH^{s+1}}\lesssim_t\lambda_\e^{1/2},
\end{align*}
and the conclusion~\eqref{eq:apriori-est-global-vmd-deg} follows. Further decomposing $\vi_\e=\vi+\lambda_\e^{-1}\delta\!\vi_\e$, these results yield
\[\|\!\md_\e\!\|_{\Ld^\infty_tH^{s+1}}+\|\!\dd_\e\!\|_{\Ld^\infty_tH^{s+1}}\lesssim_t1.\]
Combining this again with~\eqref{eq:bound-pot-deltav-md}, we obtain
\begin{align*}
\|\!\vi_\e-\vi^\circ\!\|_{\Ld^\infty_tH^{s+2}}&\lesssim\|\!\md_\e-\md^\circ\!\|_{\Ld^\infty_tH^{s+1}}+\|\!\dd_\e-\dd^\circ\!\|_{\Ld^\infty_tH^{s+1}}+\|\!\vi_\e-\vi^\circ\!\|_{\Ld^\infty_t\Ld^2}\\
&\lesssim_t1+\lambda_\e^{-1}\big(\|\delta\!\md_\e\!\|_{\Ld^\infty_tH^{s+1}}+\|\delta\!\dd_\e\!\|_{\Ld^\infty_tH^{s+1}}+\|\delta\!\vi_\e\!\|_{\Ld^\infty_t\Ld^2}\big)\lesssim_t1,
\end{align*}
and the conclusion~\eqref{eq:apriori-est-global-vmd-deg-boundedness} follows.

\medskip
\noindent\step2 Conclusion.

Let $s>1$, and assume that $\hat h\in W^{s+5,\infty}(\R^2)$, $\hat F\in W^{s+4,\infty}(\R^2)^2$, $\vi^\circ\in W^{s+3,\infty}(\R^2)^2$, $\curl\!\vi^\circ\in H^{s+3}\cap W^{s+3,\infty}(\R^2)$, and $\Div(a\!\vi^\circ)\in H^{s+2}(\R^2)$. In this step, we use the notation $\lesssim$ for $\le$ up to a constant that depends only on an upper bound on the norms of these data and on $s$ and $(s-1)^{-1}$, and we write $\lesssim_t$ to indicate the further dependence on an upper bound on time~$t$.

Under these assumptions we know from~\cite[Theorem~4]{D-16} that equation~\eqref{eq:GL3lim-conv} admits a unique global solution $\vi\in\Ld^\infty_\loc(\R^+;\vi^\circ+H^{s+3}\cap W^{s+3,\infty}(\R^2)^2)$, which implies in particular
\[\|\!\vi-\vi^\circ\!\|_{\Ld^\infty_t H^{s+3}}+\|\!\vi\!\|_{\Ld^\infty_tW^{s+3,\infty}}+\|(\md,\dd)\|_{\Ld^\infty_tH^{s+2}}\lesssim_t1.\]
In addition, we know from~\cite[Theorem~1(i)]{D-16} that equation~\eqref{eq:GLv1-hd} also admits a unique global solution $\vi_\e\in\Ld^\infty_\loc(\R^+;\vi^\circ+H^{s+3}(\R^2)^2)$.
We may thus apply the result of Step~1, which for any $t\ge0$ yields for all $\e>0$ small enough,
\[\|\!\vi_\e-\vi^\circ\!\|_{\Ld^\infty_tH^{s+2}}+\|\!\md_\e\!\|_{\Ld^\infty_tH^{s+1}}+\|\!\dd_\e\!\|_{\Ld^\infty_tH^{s+1}}\lesssim_t1.\]
As $s>1$, this implies by the Sobolev embedding,
\[\|\!\vi_\e-\vi^\circ\!\|_{\Ld^\infty_t(H^3\cap W^{2,\infty})}+\|\!\md_\e\!\|_{\Ld^\infty_t(H^2\cap W^{1,\infty})}+\|\!\dd_\e\!\|_{\Ld^\infty_t(H^2\cap W^{1,\infty})}\lesssim_t1,\]
and hence, using these bounds in equation~\eqref{eq:GLv1-hd-rewrite},
\[\|\partial_t\!\vi_\e\!\|_{\Ld^\infty_t(H^1\cap\Ld^\infty)}+\|\partial_t\!\dd_\e\!\|_{\Ld^\infty_t\Ld^2}\lesssim_t1.\]
The desired estimates follow.
Finally, the result~\eqref{eq:apriori-est-global-vmd-deg} of Step~1 with $\lambda_\e\gg1$ directly implies the convergence $\vi_\e\to\vi$ in $\Ld^\infty_\loc(\R^+;\vi^\circ+H^{s+2}(\R^2)^2)$.
\end{proof}

\subsection{Conservative case}\label{chap:GP-lim}

Let us examine the vorticity formulation of equation~\eqref{eq:limGP} for $\vi_\e$. In terms of $\md_\e:=\curl\!\vi_\e$, it takes the form of a nonlocal nonlinear continuity equation for the vorticity $\md_\e$,
\begin{gather}\label{eq:limGPvort}
\begin{cases}\partial_t\!\md_\e=-\Div(\Gamma_\e^\bot\!\md_\e),\\
\curl\!\vi_\e=\md_\e,\quad\Div(a\!\vi_\e)=0,\\
\md_\e\!|_{t=0}=\curl\!\vi_\e^\circ.\end{cases}
\end{gather}
Given the form of $\Gamma_\e$ in~\eqref{eq:limGP}, this equation is a variant of the 2D Euler equation in vorticity form and is known as the lake equation in the context of 2D fluid dynamics (cf.~e.g.~\cite{CHL-1,CHL-2}): the pinning weight $a$ corresponds to the effect of a varying depth in shallow water, while the forcing $\nabla h-F$ is similar to a background flow.
A detailed study of this kind of equations is performed in the companion article~\cite{D-16}. The following proposition states that a solution $\vi_\e$ always exists globally and satisfies the various properties of Assumption~\ref{as:apveps}(b), under suitable regularity assumptions on the initial data $\vi_\e^\circ$.

\begin{prop}\label{prop:GPvprop}
Let $h:\R^2\to\R$, $a:=e^h$, $F:\R^2\to\R^2$,
and let $\vi_\e^\circ:\R^2\to\R^2$ be bounded in $W^{1,q}(\R^2)^2$ for all $q>2$ and satisfy $\curl\!\vi_\e^\circ\in\Pc(\R^2)$.
Assume that $h\in \Ld^\infty(\R^2)$, $\nabla h,F\in \Ld^4\cap W^{2,\infty}(\R^2)^2$, that $a(x)\to1$ uniformly as $|x|\uparrow\infty$, that $\vi_\e^\circ$ is bounded in $W^{2,\infty}(\R^2)^2$ with $\Div(a\!\vi_\e^\circ)=0$, and that $\curl\!\vi_\e^\circ$ is bounded in $H^{1}(\R^2)$.\\
In the regime~\GP,
there exists a unique (global) solution $\vi_\e$ of~\eqref{eq:limGP} in $\R^+\times\R^2$, in the space $\Ld^\infty_\loc(\R^+;\vi_\e^\circ+H^{2}\cap W^{1,\infty}(\R^2)^2)$. Moreover, all the properties of Assumption~\ref{as:apveps}(b) are satisfied, that is, for all $t\ge0$ and $q>2$ and $2<p<\infty$,
\begin{gather*}
\|(\vi_\e^t,\nabla\!\vi_\e^t)\|_{(\Ld^2+\Ld^q)\cap\Ld^\infty}\lesssim_{t,q}1,\quad\|\curl\!\vi_\e^t\!\|_{\Ld^1\cap\Ld^\infty}\lesssim_t1,\\
\|\!\pre_\e^t\!\|_{\Ld^q\cap\Ld^{\infty}}\lesssim_{t,q}1,\quad\|\nabla\!\pre_\e^t\!\|_{\Ld^2\cap\Ld^\infty}\lesssim_t1,\quad\|\partial_t\!\vi_\e^t\!\|_{\Ld^2}\lesssim_t1,\quad\|\partial_t\!\pre_\e^t\!\|_{\Ld^p}\lesssim_{t,p}1.
\end{gather*}
In addition, for all $\theta>0$ and $\varrho\ge1$, setting $\pre_{\e,\varrho}:=\chi_\varrho\pre_\e$, we have for all $t\ge0$,
\begin{equation}\label{eq:diffpressuretrunc}
\|\nabla(\pre_{\e,\varrho}^t-\pre_\e^t)\|_{\Ld^2}\lesssim_{\theta,t} \varrho^{\theta-2}+\int_{|x|>\varrho}|\curl\!\vi_\e^\circ\!|^2.\qedhere
\end{equation}
\end{prop}

\begin{proof}
We split the proof into three steps.

\medskip
\noindent\step1 Preliminary.

In this step, we prove the following Meyers-type elliptic regularity estimate: if $b\in\Ld^\infty(\R^2)$ satisfies $\frac12\le b\le 1$ pointwise, and $b(x)\to1$ uniformly as $|x|\uparrow\infty$, then for all $g\in\Ld^1\cap\Ld^2(\R^2)^2$ the decaying solution $v$ of equation $-\Div(b\nabla v)=\Div g$ satisfies for all $2<q<\infty$,
\[\|v\|_{\Ld^{q}}\lesssim_{q}\|g\|_{\Ld^{\frac{2q}{q+2}}\cap\Ld^2}\lesssim\|g\|_{\Ld^1\cap\Ld^2}.\]

Let $b\in\Ld^\infty(\R^2)$ be fixed with $\frac12\le b\le 1$ pointwise and $b(x)\to1$ uniformly as $|x|\uparrow\infty$.
Set $b_{r}:=\chi_{r}+b(1-\chi_{r})$ and decompose the equation for $v$ as follows,
\[-\Div(b_{r}\nabla v)=\Div\!\big(g+(b-b_{r})\nabla v\big).\]
Let $1<p<2$. Meyers' perturbative argument~\cite{Meyers-63} gives a value $\kappa_p>0$ such that, if $\tilde b\in\Ld^\infty(\R^2)$ satisfies $\kappa_p\le \tilde b\le1$, then for all $k\in\Ld^1\cap\Ld^2(\R^2)^2$ the decaying solution $w$ of equation $-\Div(\tilde b\nabla w)=\Div k$ satisfies $\|\nabla w\|_{\Ld^p}\lesssim_p\|k\|_{\Ld^{p}}$.
By definition, for $r$ large enough, the truncated coefficient $b_{r}$ satisfies $\kappa_p\le b_r\le 1$, hence
\[\|\nabla v\|_{\Ld^p}\lesssim_p\|g+(b-b_{r})\nabla v\|_{\Ld^{p}}.\]
Using the elementary energy estimate $\|\nabla v\|_{\Ld^2}\lesssim\|g\|_{\Ld^2}$, and noting that $b_r=b$ in $\R^2\setminus B_{2r}$, Hölder's inequality yields
\[\|\nabla v\|_{\Ld^p}\lesssim_p\|g\|_{\Ld^p}+\|\nabla v\|_{\Ld^{p}(B_{2r})}\lesssim\|g\|_{\Ld^p}+r^{2(\frac1p-\frac12)}\|\nabla v\|_{\Ld^2}\lesssim\|g\|_{\Ld^p}+r^{2(\frac1p-\frac12)}\|g\|_{\Ld^2}.\]
Rather decomposing the equation for $v$ as follows,
\[-\triangle v=\Div(g+(b-1)\nabla v),\]
we deduce from Riesz potential theory, with $2<q:=\frac{2p}{2-p}<\infty$,
\[\|v\|_{\Ld^{q}}\lesssim_{q} \|g\|_{\Ld^{p}}+\|\nabla v\|_{\Ld^{p}}.\]
Combining this with the above, the conclusion follows.

\medskip
\noindent\step2 Proof of Assumption~\ref{as:apveps}(b).

The assumption $\|\hat h\|_{W^{3,\infty}}$, $\|(\nabla\hat h,\hat F)\|_{\Ld^4\cap W^{2,\infty}}\lesssim1$ yields $\|\lambda_\e^{-1}(\nabla^\bot h-F^\bot)\|_{\Ld^4\cap W^{2,\infty}}\lesssim1$ in the considered regime, and also $\lambda_\e^{-1}\frac{N_\e}\Log=1$ and $\lambda_\e^{-1}\lesssim1$. Using the assumptions on the initial data $\vi_\e^\circ$, it follows from~\cite[Theorems~1 and~3]{D-16} that there exists a unique (global) solution $\vi_\e\in\Ld^\infty_\loc(\R^+;\vi_\e^\circ+H^{2}\cap W^{1,\infty}(\R^2)^2)$ of~\eqref{eq:limGP} in $\R^+\times\R^2$ with initial data $\vi_\e^\circ$. Moreover, it is shown in~\cite{D-16} that this solution satisfies in particular, for all $t\ge0$,
\begin{align}\label{eq:resD16GP}
\|\!\vi_\e^t-\vi_\e^\circ\!\|_{H^2\cap W^{1,\infty}}\lesssim_t1,\quad\|\!\md_\e^t\!\|_{H^1\cap\Ld^\infty}\lesssim_t1,\quad\int_{\R^2}\md_\e^t=1,\quad \md_\e^t\ge0.
\end{align}
(In order to ensure $\vi_\e\in\Ld^\infty_\loc(\R^+;\vi_\e^\circ+H^2(\R^2)^2)$, the results in~\cite{D-16} would actually require $\nabla h,F,\vi^\circ\in W^{s+2,\infty}(\R^2)^2$ for some $s>0$ due to the use of the Sobolev embedding for $H^{s+1}(\R^2)$ into $W^{s,\infty}(\R^2)$ in~\cite[Proof of Lemma~4.6]{D-16}. However, this use of the Sobolev embedding is easily replaced by an a priori estimate for $\vi_\e$ in $W^{s,\infty}(\R^2)^2$, for which it is already enough to assume $\nabla h,F,\vi^\circ\in W^{2,\infty}(\R^2)^2$, cf.~\cite[Lemma~4.7]{D-16}.)

We argue that all the claimed properties of $\vi_\e$ follow from~\eqref{eq:resD16GP}. Combining~\eqref{eq:resD16GP} with the assumption that $\vi_\e^\circ$ is bounded in $W^{1,q}(\R^2)^2$ for all $q>2$, we obtain
\[\|(\vi_\e^t,\nabla\!\vi_\e^t)\|_{(\Ld^2+\Ld^q)\cap\Ld^\infty}\lesssim_{t,q}1.\]
Applying $\Div(\hat a\,\cdot)$ to both sides of equation~\eqref{eq:limGP}, we find the following equation for the pressure, in the considered regime~(GP),
\begin{align}\label{eq:pressureeqn}
-\Div(\hat a\nabla\!\pre_\e^t)=\Div(\hat a\Gamma_\e^t\!\md_\e^t)=-\Div\!\big(\hat a\!\md_\e^t(\lambda_\e^{-1}\nabla^\bot\hat h-\hat F^\bot-2\!\vi_\e^t)^\bot\big).
\end{align}
An energy estimate directly yields
\begin{align}\label{eq:boundpressureH1}
\|\nabla\!\pre_\e^t\!\|_{\Ld^2}&\lesssim\|\hat a\!\md_\e^t(\lambda_\e^{-1}\nabla^\bot\hat h-\hat F^\bot-2\!\vi_\e^t)^\bot\|_{\Ld^2}\lesssim_t1,
\end{align}
and similarly, first differentiating both sides of~\eqref{eq:pressureeqn},
\begin{align}\label{eq:boundpressureH2}
\|\nabla^2\!\pre_\e^t\!\|_{\Ld^2}\lesssim\|\nabla\!\pre_\e^t\!\|_{\Ld^2}+\big\|\nabla\big(\hat a\!\md_\e^t(\lambda_\e^{-1}\nabla^\bot\hat h-\hat F^\bot-2\!\vi_\e^t)^\bot\big)\big\|_{\Ld^2}\lesssim_t1.
\end{align}
Inserting~\eqref{eq:boundpressureH1} into~\eqref{eq:limGP} yields
\[\|\partial_t\!\vi_\e^t\!\|_{\Ld^2}\le\|\nabla\!\pre_\e^t\!\|_{\Ld^2}+\|\Gamma_\e^t\!\md_\e^t\!\|_{\Ld^2}\lesssim_{t}1.\]
Applying to equation~\eqref{eq:pressureeqn} the Meyers-type result of Step~1, we find for all $2<q<\infty$,
\[\|\!\pre_\e^t\!\|_{\Ld^q}\lesssim_q\|\hat a\!\md_\e^t(\lambda_\e^{-1}\nabla^\bot \hat h-\hat F^\bot-2\!\vi_\e^t)^\bot\|_{\Ld^1\cap\Ld^2}\lesssim_t1.\]
Combining this with~\eqref{eq:boundpressureH2}, we deduce from the Sobolev embedding $\|\!\pre_\e^t\!\|_{\Ld^q\cap\Ld^\infty}\lesssim_{q,t}1$ for all $q>2$.
Differentiating both sides of~\eqref{eq:pressureeqn} with respect to the time variable, the Meyers-type result of Step~1 further yields for all $2<q<\infty$,
\begin{align*}
\|\partial_t\!\pre_\e^t\!\|_{\Ld^q}&\lesssim_q\big\|\hat a\partial_t\big(\!\md_\e^t(\lambda_\e^{-1}\nabla^\bot \hat h-\hat F^\bot-2\!\vi_\e^t)^\bot\big)\big\|_{\Ld^1\cap\Ld^2}\\
&\lesssim\|\!\md_\e^t\!\|_{\Ld^2\cap\Ld^\infty}\|\partial_t\!\vi_\e^t\!\|_{\Ld^2}+\|\Gamma_\e^t\partial_t\!\md_\e^t\|_{\Ld^1\cap\Ld^2}\\
&\lesssim_t1+\|\Gamma_\e^t\partial_t\!\md_\e^t\|_{\Ld^1\cap\Ld^2}.
\end{align*}
Using equation~\eqref{eq:limGPvort} to estimate the time derivative of the vorticity, and using that $\|\lambda_\e^{-1}\nabla\hat h-\hat F\|_{\Ld^4\cap W^{1,\infty}}\lesssim1$, we find
\begin{multline*}
\|\Gamma_\e^t\partial_t\!\md_\e^t\!\|_{\Ld^1\cap\Ld^2}\lesssim\|\Gamma_\e^t\|_{\Ld^4\cap\Ld^\infty}^2\|\nabla\!\md_\e^t\!\|_{\Ld^2}+\|\Gamma_\e^t\|_{W^{1,\infty}}^2\|\!\md_\e^t\!\|_{\Ld^1\cap\Ld^2}\\
\lesssim_t\|\Gamma_\e^t\|_{\Ld^4\cap W^{1,\infty}}^2\lesssim1+\|\!\vi_\e^t\!\|_{\Ld^4\cap W^{1,\infty}}^2\lesssim_t1,
\end{multline*}
hence $\|\partial_t\!\pre_\e^t\!\|_{\Ld^q}\lesssim_{t,q}1$. All the claimed properties of $\vi_\e$ follow.

\medskip
\noindent\step3 Proof of~\eqref{eq:diffpressuretrunc}.
\nopagebreak

For all $t\ge0$, testing equation~\eqref{eq:pressureeqn} with $(1-\chi_\varrho)\pre_\e^t$, and using $|\nabla\chi_\varrho|\lesssim \varrho^{-1}(1-\chi_\varrho)^{1/2}$ and the inequality $2xy\le x^2+y^2$, we find
\begin{eqnarray*}
\lefteqn{\int_{\R^2} \hat a(1-\chi_\varrho)|\nabla\!\pre_\e^t\!|^2}\\
&=&\int_{\R^2} \hat a\pre_\e^t\nabla\chi_\varrho\cdot \nabla\!\pre_\e^t-\int_{\R^2} \hat a(1-\chi_\varrho)\nabla\!\pre_\e^t\cdot\Gamma_\e^t\md_\e^t+\int_{\R^2}\hat a\!\pre_\e^t\nabla\chi_\varrho\cdot \Gamma_\e^t\md_\e^t\\
&\le &\frac12\int_{\R^2}\hat a(1-\chi_\varrho)|\nabla\!\pre_\e^t\!|^2+C\varrho^{-2}\int_{\varrho\le |x|\le 2\varrho}|\!\pre_\e^t\!|^2+C\int_{\R^2} (1-\chi_\varrho)|\Gamma_\e^t|^2|\!\md_\e^t\!|^2.
\end{eqnarray*}
Absorbing the first right-hand side term and recalling that Step~2 gives $\|\Gamma_\e^t\|_{\Ld^\infty}$, $\|\!\md_\e^t\!\|_{\Ld^2}\lesssim_t1$ and $\|\!\pre_\e^t\!\|_{\Ld^{p}}\lesssim_{p,t}1$ for all $p>2$, we obtain with Hölder's inequality,
\begin{align*}
\int_{\R^2} (1-\chi_\varrho)|\nabla\!\pre_\e^t\!|^2&\lesssim_t \varrho^{-2}\int_{\varrho\le |x|\le 2\varrho}|\!\pre_\e^t\!|^2+\int_{\R^2} (1-\chi_\varrho)|\!\md_\e^t\!|^2\\
&\lesssim_{p,t} \varrho^{-\frac4p}+\int_{\R^2} (1-\chi_\varrho)|\!\md_\e^t\!|^2,
\end{align*}
and hence, for all $p>2$,
\begin{align*}
\|\nabla(\pre_{\e,\varrho}^t-\pre_\e^t)\|_{\Ld^2}^2&\lesssim\int_{\R^2} (1-\chi_\varrho)|\nabla\!\pre_\e^t\!|^2+\varrho^{-2}\int_{\varrho\le|x|\le2\varrho}|\!\pre_\e^t\!|^2\\
&\lesssim_{p,t} \varrho^{-\frac4p}+\int_{\R^2} (1-\chi_\varrho)|\!\md_\e^t\!|^2.
\end{align*}
It remains to estimate the last right-hand side term. For all $t\ge0$, using again  the bounds of Step~2 and the estimate $|\nabla\chi_\varrho|\lesssim \varrho^{-1}(1-\chi_\varrho)^{1/2}$, we deduce from~\eqref{eq:limGPvort},
\begin{align*}
\partial_t\int_{\R^2}(1-\chi_\varrho)|\!\md_\e^t\!|^2&=2\int_{\R^2}(1-\chi_\varrho)\md_\e^t\curl(\Gamma_\e^t\!\md_\e^t)\\
&=2\int_{\R^2} |\!\md_\e^t\!|^2\Gamma_\e^t\cdot \nabla^\bot\chi_\varrho-\int_{\R^2}(1-\chi_\varrho)\Gamma_\e^t\cdot\nabla^\bot|\!\md_\e^t\!|^2\\
&=2\int_{\R^2} |\!\md_\e^t\!|^2\Gamma_\e^t\cdot \nabla^\bot\chi_\varrho+\int_{\R^2}|\!\md_\e^t\!|^2\curl\big((1-\chi_\varrho)\Gamma_\e^t\big)\\
&\lesssim_t \varrho^{-1}\int_{\R^2}(1-\chi_\varrho)^{1/2}|\!\md_\e^t\!|^2+\int_{\R^2}(1-\chi_\varrho)|\!\md_\e^t\!|^2\\
&\lesssim_t \varrho^{-2}+\int_{\R^2}(1-\chi_\varrho)|\!\md_\e^t\!|^2,
\end{align*}
and hence, by the Grönwall inequality,
\[\int_{\R^2}(1-\chi_\varrho)|\!\md_\e^t\!|^2\lesssim_t \varrho^{-2}+\int_{\R^2}(1-\chi_\varrho)|\curl\!\vi_\e^\circ\!|^2,\]
and the result~\eqref{eq:diffpressuretrunc} follows.
\end{proof}

We now show how to pass to the limit in equation~\eqref{eq:limGP} as $\e\downarrow0$,
which is easily achieved by a Grönwall argument on the $\Ld^2$-distance between $\vi_\e$ and the solution $\vi$ of the limiting equation. Note that, in the limit, pinning effects only remain in the constraint.

\begin{lem}\label{lem:lastlimGP}
Let the same assumptions hold as in Proposition~\ref{prop:GPvprop}, and let $\vi_\e:\R^+\times\R^2\to\R^2$ be the corresponding (global) solution of~\eqref{eq:limGP}. In the regime \GP{} with $\vi_\e^\circ=\vi^\circ$, we have $\vi_\e\to\vi$ in $\Ld^\infty_\loc(\R^+;\Ld^2(\R^2)^2)$ as $\e\downarrow0$, where $\vi$ is the unique global (smooth) solution of
\begin{equation}\label{eq:limGPfin}
\begin{cases}
\partial_t\!\vi=\nabla\!\pre+(-\hat F+2\!\vi^\bot)\,\curl\!\vi,\\
\Div(\hat a\!\vi)=0,\quad\vi\!|_{t=0}=\vi^\circ.
\end{cases}\qedhere
\end{equation}
\end{lem}

\begin{proof}
Using the choice of the scalings for $\lambda_\e,h,F$ in the regime~\GP, equation~\eqref{eq:limGP} takes on the following guise,
\begin{align*}
\partial_t\!\vi_\e=\nabla\!\pre_\e+\big(\lambda_\e^{-1}\nabla \hat h-\hat F+2\!\vi_\e^\bot\!\big)\curl\!\vi_\e,\qquad\Div(\hat a\!\vi_\e)=0,\qquad\vi_\e\!|_{t=0}=\vi^\circ.
\end{align*}
As $\lambda_\e^{-1}\to0$, it is formally clear that $\vi_\e$ should converge to the solution $\vi$ of equation~\eqref{eq:limGPfin}. Note that existence, uniqueness, and regularity of $\vi$ are given by Proposition~\ref{prop:GPvprop} just as for $\vi_\e$, and we have in particular the following bounds for all $t\ge0$,
\begin{align}\label{eq:boundvvepsGP}
\|(\vi^t,\vi_\e^t)\|_{W^{1,\infty}}\lesssim_t1,\quad \|\curl\!\vi_\e^t\|_{\Ld^1}=1,\quad
\|(\pre^t,\pre_\e^t)\|_{\Ld^\infty}\lesssim_{t}1,
\end{align}
and for all $R,\theta>0$,
\begin{align}\label{eq:boundvvepsGP+}
\|(\vi^t,\vi_\e^t)\|_{\Ld^2(B_R)}\lesssim_{t,\theta}R^\theta,\quad\|(\pre^t,\pre_\e^t)\|_{\Ld^2(B_R)}\lesssim_{t,\theta}R^\theta.
\end{align}
We denote by $\xi_R^z(x):=e^{-|x-z|/R}$ the exponential cut-off at the scale $R\ge1$ centered at $z\in R\Z^2$. Using the equations for $\vi_\e,\vi$, we find
\begin{multline*}
\partial_t\int \hat a\xi_R^z|\!\vi_\e-\vi\!|^2=2\int \hat a\xi_R^z(\vi_\e-\vi)\cdot\nabla(\pre_\e-\pre)\\
+2\int \hat a\xi_R^z(-\hat F+2\!\vi^\bot)\cdot(\vi_\e-\vi)(\curl\!\vi_\e-\curl\!\vi)
+2\lambda_\e^{-1}\int \hat a\xi_R^z\nabla\hat h\cdot(\vi_\e-\vi)\curl\!\vi_\e.
\end{multline*}
Integrating by parts in the first right-hand side term, using the relation $\Div(\hat a\xi_R^z(\vi_\e-\vi))=\hat a\nabla\xi_R^z\cdot(\vi_\e-\vi)$, and using the weighted Delort-type identity~\eqref{eq:weightedDelort} in the form
\begin{align*}
(\vi_\e-\vi)\curl(\vi_\e-\vi)&=-\frac12|\!\vi_\e-\vi\!|^2\nabla^\bot\hat h-\hat a^{-1}(\Div(\hat a S_{\vi_\e-\vi}))^\bot,
\end{align*}
we deduce
\begin{multline*}
\partial_t\int \hat a\xi_R^z|\!\vi_\e-\vi\!|^2=-2\int \hat a\nabla\xi_R^z\cdot(\vi_\e-\vi)(\pre_\e-\pre)-\int \hat a\xi_R^z\nabla^\bot\hat h\cdot(-\hat F+2\!\vi^\bot)|\!\vi_\e-\vi\!|^2\\
+2\int \hat aS_{\vi_\e-\vi}:\nabla(\xi_R^z(\hat F^\bot+2\!\vi))+2\lambda_\e^{-1}\int \hat a\xi_R^z\nabla\hat h\cdot(\vi_\e-\vi)\curl\!\vi_\e,
\end{multline*}
and hence, using~\eqref{eq:boundvvepsGP}--\eqref{eq:boundvvepsGP+}, the assumption $\|(\nabla\hat h,\hat F)\|_{W^{1,\infty}}\lesssim1$, the property $|\nabla\xi_R^z|\lesssim R^{-1}\xi_R^z$ of the exponential cut-off, and the pointwise estimate $|S_{w}|\lesssim|w|^2$,
\begin{align*}
\partial_t\int \hat a\xi_R^z|\!\vi_\e-\vi\!|^2\lesssim_{t,\theta} R^{-2(1-\theta)}+\lambda_\e^{-2}+\int \hat a\xi_R^z|\!\vi_\e-\vi\!|^2.
\end{align*}
Choosing $\theta=\frac12$, the Grönwall inequality yields $\sup_z\int a_\e\xi_R^z|\!\vi_\e-\vi\!|^2\lesssim_{t}R^{-1}+\lambda_\e^{-2}$, and the conclusion follows, letting $R\uparrow\infty$.
\end{proof}

\section{Computations on the modulated energy}\label{chap:modenergy}

In this section, we adapt to the weighted case with pinning and applied current the computations of~\cite{Serfaty-15}:
we compute the time derivative of the modulated energy excess~\eqref{der4} and express it with only quadratic terms in the error instead of terms which initially appear as linear and would thus make a Grönwall argument impossible. These computations are based on algebraic manipulations using all the equations and various appropriate physical quantities that are introduced below.

\subsection{Modulated energy}
We recall the definitions of modulated energy and energy excess~\eqref{der2}--\eqref{der4}. 
In order to prove that the rescaled supercurrent density $N_\e^{-1}j_\e:=N_\e^{-1}\langle\nabla u_\e,iu_\e\rangle$ is close to $\vi_\e$, we follow the strategy of~\cite{Serfaty-15}, considering the following {\it modulated energy}, which is modeled on the weighted Ginzburg-Landau energy, plays the role of an adapted (squared) distance between $j_\e$ and $N_\e\!\vi_\e$, and is localized by means of the cut-off function $\chi_R$ at some scale $R\gg1$ (to be later optimized as a function of $\e$),
\[\Ec_{\e,R}:=\int_{\R^2}\frac {a\chi_R}2\Big(|\nabla u_\e-iu_\e N_\e\!\vi_\e\!|^2+\frac{a}{2\e^2}(1-|u_\e|^2)^2\Big).\]
As usual, this modulated energy needs to be renormalized by subtracting the expected self-interaction energy of the vortices (compare with Lemma~\ref{lem:ballconstr} below), which then yields the following {\it modulated energy excess},
\begin{eqnarray*}
\D_{\e,R}&:=&\Ec_{\e,R}-\frac{\Log }2\int_{\R^2} {a\chi_R}\mu_\e\\
&=&\int_{\R^2}\frac {a\chi_R}2\Big(|\nabla u_\e-iu_\e N_\e\!\vi_\e\!|^2+\frac{a}{2\e^2}(1-|u_\e|^2)^2-\Log \mu_\e\Big).
\end{eqnarray*}
As explained in the introduction, the cut-off $\chi_R$ is not needed in the conservative case, where we only treat the case when $h,F,f$ decay at infinity. We write $\Ec_\e:=\Ec_{\e,\infty}$ for the corresponding quantity without the cut-off $\chi_R$ in the definition (formally $R=\infty$), and also $\D_\e:=\sup_{R\ge1}\D_{\e,R}$.

On the one hand, rather than the $\Ld^2$-norm restricted to the ball $B_R$ centered at the origin, our methods further allow to consider the uniform $\Ld^2_\loc$-norm at the scale $R$: setting $\chi_R^z:=\chi_R(\cdot-z)$, we define
\[\Ec_{\e,R}^*:=\sup_z\Ec_{\e,R}^z,\qquad \Ec_{\e,R}^z:=\int_{\R^2}\frac {a\chi_R^z}2\Big(|\nabla u_\e-iu_\e N_\e\!\vi_\e\!|^2+\frac{a}{2\e^2}(1-|u_\e|^2)^2\Big),\]
where henceforth the supremum always implicitly runs over all lattice points $z\in R\Z^2$, and similarly
\[\D_{\e,R}^*:=\sup_z\D_{\e,R}^z,\qquad\D_{\e,R}^z:=\Ec_{\e,R}^z-\frac\Log2\int_{\R^2} a\chi_R^z\mu_\e.\]
Note that by definition we have for all $x\in \R^2$ and $L>0$,
\begin{align}\label{eq:boundbyEmod}
\|\nabla u_\e-iu_\e N_\e\!\vi_\e\!\|_{\Ld^2(B_L(x))}^2+\e^{-2}\|1-|u_\e|^2\|_{\Ld^2(B_L(x))}^2\lesssim \Big(1+\frac LR\Big)^2\Ec_{\e,R}^*.
\end{align}

On the other hand, in order to simplify computations, we need as in~\cite{Serfaty-15}
to add some suitable lower-order terms, and rather consider, for some other scale $\varrho\gg1$ (to be also later optimized as a function of $\e$),
\[\hat\Ec_{\e,\varrho,R}:=\int_{\R^2}\frac {a}2\Big(\chi_R|\nabla u_\e-iu_\e N_\e\!\vi_\e\!|^2+\frac{a\chi_R}{2\e^2}(1-|u_\e|^2)^2+(1-|u_\e|^2)(N_\e^2\psi_{\e,\varrho,R}+f\chi_R)\Big),\]
and similarly for the modulated energy excess,
\begin{align}\label{eq:defDeR}
\hat\D_{\e,\varrho,R}:=\hat\Ec_{\e,\varrho,R}-\frac{\Log }2\int_{\R^2} {a\chi_R}\mu_\e,
\end{align}
where the function $\psi_{\e,\varrho,R}:\R^2\to\R$ is precisely chosen as follows,
\begin{align}\label{eq:choicepsi}
\psi_{\e,\varrho,R}:=-\chi_R|\!\vi_\e\!|^2+\frac{\Log }{N_\e}\chi_R\vi_\e\!\cdot\,(\nabla^\bot h-F^\bot)+\frac{\lambda_\e\beta\Log }{N_\e}\chi_R\pre_{\e,\varrho}-\frac{\Log }{N_\e}\nabla\chi_R\cdot \vi_\e^\bot,
\end{align}
in terms of the truncated pressure $\pre_{\e,\varrho}:=\chi_\varrho\pre_\e$.
This choice is motivated by the fact that it yields some useful cancellations in the proof of Lemma~\ref{lem:decompcruc} below.
Again, replacing $\chi_R$ and $\pre_{\e,\varrho}$ by $\chi_R^z$ and $\pre_{\e,\varrho}^z=\chi_\varrho^z\pre_\e$, we further define $\hat\Ec_{\e,\varrho,R}^z$ and $\hat\D_{\e,\varrho,R}^z$ for $z\in \R^2$, and we then set $\hat\Ec_{\e,\varrho,R}^*:=\sup_z\hat\Ec_{\e,\varrho,R}^z$ and $\hat\D_{\e,\varrho,R}^*:=\sup_z\hat\D_{\e,\varrho,R}^z$ (where suprema implicitly run over all lattice points $z\in R\Z^2$).
The additional truncation scale $\rho\gg1$ is introduced here to cure the lack of integrability of the pressure $\pre_\e$ in the conservative case: indeed, the pressure $\pre_\e$ does in general not belong to $\Ld^2(\R^2)$ (cf.~Assumption~\ref{as:apveps}(b) and Proposition~\ref{prop:GPvprop}, which are indeed optimal in that respect), while it does always in the case without pinning and applied current (cf.~\cite{Serfaty-15}).
In the dissipative case this truncation is not needed, so that we may set $\pre_{\e,\infty}:=\pre_\e$ with $\varrho:=\infty$, and we then drop for simplicity the subscript $\varrho$ from the notation, writing $\psi_{\e,R}:=\psi_{\e,\infty,R}$, $\hat\Ec_{\e,R}:=\hat\Ec_{\e,\infty,R}$, etc.

In the dissipative case, as a consequence of~\eqref{eq:scalingshFf} and of Assumption~\ref{as:apveps}(a), $\psi_{\e,R}$ is bounded in $\Ld^p(\R^2)$ uniformly with respect to $R$ for all $2<p\le \infty$ (but not in $\Ld^2(\R^2)$), and using the bound~\eqref{eq:scalingshFf} we have in the considered regimes, for all $t\in[0,T)$ and $\theta>0$,
\begin{gather}\label{eq:boundpsi}
\|\psi_{\e,R}^t\|_{\Ld^2}\lesssim_{t,\theta} 1+\frac{\Log}{N_\e}(\lambda_\e R^\theta+1\wedge\lambda_\e^{1/2}+R^{-1+\theta}),\\
\|\partial_t\psi_{\e,R}\|_{\Ld^2_t\Ld^2}\lesssim_{t} 1+\frac{\Log }{N_\e}.\nonumber
\end{gather}
In the conservative case, in the considered regime~\GP, the bound~\eqref{eq:scalingshFfdec} and Assumption~\ref{as:apveps}(b) rather yield,
for all $t\in[0,T)$ and $\theta>0$,
\begin{gather}\label{eq:boundpsi-2}
\|\psi_{\e,\varrho,R}^t\|_{\Ld^2}+\|\partial_t\psi_{\e,\varrho,R}^t\|_{\Ld^2}\lesssim_{t,\theta} 1+\frac{\Log}{N_\e}\lambda_\e\varrho^\theta\lesssim\varrho^\theta.
\end{gather}
Based on these estimates, the following lemma states that the additional terms in $\hat\Ec_{\e,\varrho,R}$ are indeed of lower order, so that $\hat\Ec_{\e,\varrho,R}$ is equivalent to the modulated energy $\Ec_{\e,R}$.

\begin{lem}[Neglecting lower-order terms]\label{lem:apestu}
Let $h:\R^2\to\R$, $a:=e^h$, $F:\R^2\to\R^2$ satisfy~\eqref{eq:scalingshFf} or~\eqref{eq:scalingshFfdec}, let $u_\e:[0,T)\times\R^2\to\C$, and let $\vi_\e:[0,T)\times\R^2\to\R^2$ be as in Assumption~\ref{as:apveps} for some $T>0$.
Further assume that $0<\e\ll1$ and $\varrho,R\gg1$ satisfy for some $\theta>0$, in the dissipative case,
\begin{align}\label{eq:lemapestuAS}
\e\big(N_\e^2+N_\e\Log(\lambda_\e R^\theta+1\wedge\lambda_\e^{1/2}+R^{-1+\theta})+R\lambda_\e^2\Log^2\big)\ll N_\e\Big(1\wedge\frac{N_\e}\Log\Big)^{1/2},
\end{align}
or in the conservative case,
\begin{align}\label{eq:lemapestuAS-dec}
\e N_\e^2(\varrho^\theta+R)\ll N_\e\Big(1\wedge\frac{N_\e}\Log\Big)^{1/2}.
\end{align}
Then for all $z\in\R^2$ we have
\[|\hat\Ec_{\e,\varrho,R}^{z,t}-\Ec_{\e,R}^{z,t}|=|\hat\D_{\e,\varrho,R}^{z,t}-\D_{\e,R}^{z,t}|\lesssim_t o(N_\e)\Big(1\wedge \frac{N_\e}\Log\Big)^{1/2}(\Ec_{\e,R}^{z,t})^{1/2}.\qedhere\]
\end{lem}

\begin{proof}
We focus on the dissipative case, as the other is similar. The Cauchy-Schwarz inequality yields
\begin{align*}
|\hat\Ec_{\e,R}^z-\Ec_{\e,R}^z|&\lesssim\int_{\R^2} |1-|u_\e|^2|(N_\e^2|\psi_{\e,R}^z|+|f|\chi_R^z)\\
&\le\Big(\int_{\R^2} \chi_R^z(1-|u_\e|^2)^2\Big)^{1/2}\big(N_\e^2\|(\chi_R^z)^{-1/2}\psi_{\e,R}^z\|_{\Ld^2}+\|f\|_{\Ld^2(B_{2R}(z))}\big)\\
&\lesssim\e(\Ec_{\e,R}^z)^{1/2}\big(N_\e^2\|(\chi_R^z)^{-1/2}\psi_{\e,R}^z\|_{\Ld^2}+R\|f\|_{\Ld^\infty}\big).
\end{align*}
Arguing just as in~\eqref{eq:boundpsi}, using~\eqref{eq:scalingshFf}, Assumption~\ref{as:apveps}(a), and the fact that $|\chi_R^{-1/2}\nabla\chi_R|\lesssim R^{-1}\mathds1_{B_{2R}}$,  the choice~\eqref{eq:choicepsi} of $\psi_{\e,R}$ yields, for all $\theta>0$,
\begin{align*}
\|(\chi_R^z)^{-1/2}\psi_{\e,R}^z\|_{\Ld^2}\lesssim_{t,\theta} 1+\frac{\Log}{N_\e}(\lambda_\e R^\theta+1\wedge\lambda_\e^{1/2}+R^{-1+\theta}).
\end{align*}
Combined with~\eqref{eq:scalingshFf} and with assumption~\eqref{eq:lemapestuAS}, this proves the result.
\end{proof}

\subsection{Physical quantities and identities}
In addition to the \emph{supercurrent density} $j_\e:=\langle \nabla u_\e,iu_\e\rangle$ and to the \emph{vorticity} $\mu_\e:=\curl j_\e$, we define the {\it vortex velocity}
\[V_\e:=2\langle \nabla u_\e,i\partial_tu_\e\rangle.\]
The following identities are easily checked from these definitions (cf.~\cite{SS-prod}),
\begin{align}\label{eq:identity-1}
\partial_tj_\e=V_\e+\nabla\langle \partial_tu_\e,iu_\e\rangle,\qquad\partial_t\mu_\e=\curl V_\e,
\end{align}
and also, using equation~\eqref{eq:GL-1} for $u_\e$,
\begin{multline}\label{eq:identity-2}
\Div j_\e=\langle \triangle u_\e,iu_\e\rangle=\lambda_\e\alpha\langle\partial_tu_\e,iu_\e\rangle-j_\e\cdot\nabla h\\
-\frac{\lambda_\e\beta\Log}2\partial_t(1-|u_\e|^2)+\frac\Log2 F^\bot\cdot\nabla(1-|u_\e|^2).
\end{multline}
In the same vein as when introducing the modulated energy and energy excess, we define the following {\it modulated vorticity} and {\it modulated velocity},
\begin{align}
\tilde \mu_\e&:=\curl(N_\e\!\vi_\e+\langle \nabla u_\e-iu_\e N_\e\!\vi_\e, iu_\e\rangle)=\mu_\e+\curl(N_\e \!\vi_\e(1-|u_\e|^2)),\label{eq:defmutilde}\\
\tilde V_{\e,\varrho}&:=2\langle \nabla u_\e-iu_\e N_\e \!\vi_\e,i(\partial_t u_\e-iu_\e N_\e\!\pre_{\e,\varrho})\rangle=V_\e-N_\e \!\vi_\e\!\partial _t|u_\e|^2+N_\e\!\pre_{\e,\varrho}\!\nabla|u_\e|^2.\label{eq:defVtilde}
\end{align}
We also consider the {\it weighted Ginzburg-Landau energy density}
\[e_\e:=\frac a2\Big(|\nabla u_\e|^2+\frac{a}{2\e^2}(1-|u_\e|^2)^2+(1-|u_\e|^2)f\Big).\]
Another key quantity is the  $2\times 2$ {\it stress-energy tensor} $S_\e$,
\begin{align}\label{eq:defS}
(S_\e)_{kl}:=a\langle \partial_ku_\e,\partial_lu_\e\rangle -\frac a2\Id \Big(|\nabla u_\e|^2+ \frac a{2\e^2}(1-|u_\e|^2)^2+(1-|u_\e|^2)f\Big),
\end{align}
and its modulated version $\tilde S_\e$,
\begin{multline}\label{eq:defSmod}
(\tilde S_\e)_{kl}:=a\Big(\langle \partial_k u_\e-iu_\e N_\e \!\vi_{\e,k},\partial_l u_\e-iu_\e N_\e \!\vi_{\e,l}\rangle+ N_\e^2(1-|u_\e|^2)\vi_{\e,k}\!\vi_{\e,l}\Big)\\
-\frac a2\Id\Big(|\nabla u_\e-iu_\e N_\e \!\vi_\e\!|^2+\frac a{2\e^2}(1-|u_\e|^2)^2+(1-|u_\e|^2)(N_\e^2|\!\vi_\e\!|^2+f)\Big).
\end{multline}
The following pointwise estimates are abundantly used in the sequel.
\begin{lem}\label{lem:pointest}
We have
\begin{align*}
|j_\e-N_\e v_\e|&\le |\nabla u_\e-iu_\e N_\e \!\vi_\e\!|+|\nabla u_\e-iu_\e N_\e \!\vi_\e\!||1-|u_\e|^2|+N_\e|\!\vi_\e\!||1-|u_\e|^2|,\\
|\mu_\e|&\le2|\nabla u_\e|^2\le4|\nabla u_\e-iu_\e N_\e \!\vi_\e\!|^2+4N_\e^2|\!\vi_\e\!|^2+4N_\e^2|1-|u_\e|^2||\!\vi_\e\!|^2,\\
|V_\e|&\le2\big(|\nabla u_\e-iu_\e N_\e \!\vi_\e\!||\partial_tu_\e|+N_\e |\!\vi_\e\!||\partial_tu_\e|+N_\e|1-|u_\e|^2||\!\vi_\e\!||\partial_tu_\e|\big),\\
|\tilde V_{\e,\varrho}|&\le 2|\partial_tu_\e||\nabla u_\e-iu_\e N_\e \!\vi_\e\!|+2N_\e|\!\pre_{\e,\varrho}\!||\nabla u_\e-iu_\e N_\e \!\vi_\e\!|\\
&\hspace{2cm}+2N_\e|\!\pre_{\e,\varrho}\!||1-|u_\e|^2||\nabla u_\e-iu_\e N_\e \!\vi_\e\!|,\\
|\partial_t|u_\e||&\le|\partial_t u_\e-iu_\e N_\e\!\pre_{\e}\!|,\\
|\nabla|u_\e||&\le|\nabla u_\e-iu_\e N_\e \!\vi_\e\!|.\qedhere
\end{align*}
\end{lem}

\begin{proof}
The first estimate is obtained as follows,
\begin{align*}
|j_\e-N_\e \!\vi_\e\!|&\le|\langle \nabla u_\e-iu_\e N_\e \!\vi_\e,iu_\e\rangle|+N_\e|1-|u_\e|^2||\!\vi_\e\!|\\
&\le|\nabla u_\e-iu_\e N_\e \!\vi_\e\!|+|\nabla u_\e-iu_\e N_\e \!\vi_\e\!||1-|u_\e|^2|+N_\e|\!\vi_\e\!||1-|u_\e|^2|,
\end{align*}
and the estimates on $V_\e$ and $\tilde V_{\e,\varrho}$ similarly follow from the definitions.
The estimate on $\mu_\e$ is a direct consequence of the representation $\mu_\e=\curl\langle\nabla u_\e,iu_\e\rangle=2\langle\nabla_2u_\e,i\nabla_1u_\e\rangle$. Finally noting that
\[|\partial_tu_\e-iu_\e N_\e\!\pre_\e|^2=|\partial_t|u_\e||^2+|u_\e|^2\Big|\partial_t\frac{u_\e}{|u_\e|}-i\frac{u_\e}{|u_\e|}N_\e\!\pre_\e\Big|^2,\]
the result on $\partial_t|u_\e|$ follows, and the result on $\nabla|u_\e|$ is obtained similarly.
\end{proof}

\subsection{Divergence of the modulated stress-energy tensor}

In the following lemma we explicitly compute the divergence of the modulated stress-energy tensor: as already mentioned, it plays a crucial role in the sequel in order to replace some linear terms in the error by quadratic ones (cf.\@ Step~3 of the proof of Lemma~\ref{lem:decompcruc} below).

\begin{lem}\label{lem:divS}
Let $u_\e:[0,T)\times\R^2\to\C$ be a solution of~\eqref{eq:GL-1} as in Proposition~\ref{prop:globGL}, and let $\vi_\e:[0,T)\times\R^2\to\R^2$ be as in Assumption~\ref{as:apveps}. Defining by $(\Div \tilde S_\e)_k:=\sum_{l}\partial_l(\tilde S_\e)_{kl}$ the divergence of the $2$-tensor $\tilde S_\e$, we have
\begin{align*}
&\Div \tilde S_\e=a\lambda_\e\alpha\left\langle \partial_tu_\e-iu_\e N_\e\!\pre_{\e,\varrho},\nabla u_\e-iu_\e N_\e \!\vi_\e\right\rangle-a\mu_\e(N_\e \!\vi_\e^\bot\!-\tfrac12\Log F)\\
&\hspace{3cm}+aN_\e(N_\e \!\vi_\e-j_\e)^\bot\curl\!\vi_\e+\frac{a\lambda_\e\beta}2\Log \tilde V_{\e,\varrho}\\
&\hspace{2cm}+aN_\e(N_\e \!\vi_\e-j_\e)(\Div\!\vi_\e+\nabla h\cdot \!\vi_\e-\lambda_\e\alpha\pre_{\e,\varrho})-\frac a2(1-|u_\e|^2)\nabla f\\
&\hspace{1cm}-\frac a2\nabla h\Big(|\nabla u_\e-iu_\e N_\e \!\vi_\e\!|^2+\frac{a}{\e^2}(1-|u_\e|^2)^2+ (1-|u_\e|^2)(N_\e^2|\!\vi_\e\!|^2+f)\Big)\\
&\hspace{-0.1cm}+a\lambda_\e\alpha N_\e^2 \! \vi_\e\!\pre_{\e,\varrho}(1-|u_\e|^2)-\frac{a\lambda_\e\beta}2N_\e\Log \pre_{\e,\varrho}\!\nabla|u_\e|^2+\frac a2N_\e\Log (F^\bot\cdot\nabla|u_\e|^2)\vi_\e.
\qedhere
\end{align*}
\end{lem}

\begin{proof}
On the one hand, a direct computation yields, for the stress-energy tensor,
\begin{multline}\label{pdivs}
\Div S_\e=a\left\langle \nabla u_\e,\triangle u_\e+\frac{au_\e}{\e^2}(1-|u_\e|^2)+\nabla h\cdot \nabla u_\e+fu_\e\right\rangle\\
-\frac a2\nabla h\Big(|\nabla u_\e|^2+\frac{a}{\e^2}(1-|u_\e|^2)^2+ (1-|u_\e|^2)f\Big)-\frac a2(1-|u_\e|^2)\nabla f.
\end{multline}
On the other hand, the modulated stress-energy tensor can be decomposed as
\begin{align*}
\tilde S_\e=S_\e-aN_\e \!\vi_\e\otimes j_\e-aN_\e j_\e\otimes \vi_\e+aN_\e^2\!\vi_\e\!\otimes \vi_\e-\frac {aN_\e}2\Id\big(N_\e|\!\vi_\e\!|^2-2\!\vi_\e\!\cdot\, j_\e\big),
\end{align*}
which, combined with~\eqref{pdivs}, yields
\begin{multline*}
\Div \tilde S_\e=a\left\langle \nabla u_\e,\triangle u_\e+\frac{au_\e}{\e^2}(1-|u_\e|^2)+\nabla h\cdot \nabla u_\e+fu_\e\right\rangle\\
-\frac a2\nabla h\Big(|\nabla u_\e|^2+\frac{a}{\e^2}(1-|u_\e|^2)^2+ (1-|u_\e|^2)f\Big)-\frac a2(1-|u_\e|^2)\nabla f\\
-aN_\e\Big(j_\e\nabla h\cdot \!\vi_\e+\vi_\e\!\nabla h\cdot j_\e-N_\e \!\vi_\e\!\nabla h\cdot \!\vi_\e+\frac12N_\e|\!\vi_\e\!|^2\nabla h-\vi_\e\!\cdot\, j_\e\nabla h\Big)\\
-aN_\e j_\e\Div\!\vi_\e-aN_\e(\vi_\e\!\cdot\,\nabla)j_\e-aN_\e \!\vi_\e\!\Div j_\e-aN_\e(j_\e\cdot \nabla)\vi_\e+aN_\e^2\!\vi_\e\!\Div\!\vi_\e\\
+aN_\e^2(\vi_\e\!\cdot\,\nabla)\vi_\e\!-aN_\e^2\sum_l\vi_{\e,l}\nabla\!\vi_{\e,l}+aN_\e\sum_l \vi_{\e,l}\nabla j_{\e,l}+aN_\e\sum_l j_{\e,l}\nabla\!\vi_{\e,l},
\end{multline*}
where we denote by $\vi_{\e,l}$ and $j_{\e,l}$ the $l$-th component of the vector fields $\vi_\e$ and $j_\e$, respectively.
Noting that $(F\cdot\nabla)G-\sum_lF_l\nabla G_l=F^\bot\curl G$, and using equation~\eqref{eq:GL-1} for $u_\e$, this becomes
\begin{multline}\label{eq:lemdivSpreres}
\Div \tilde S_\e=a\lambda_\e\left\langle(\alpha+i\beta\Log )\partial_tu_\e, \nabla u_\e\right\rangle-a\Log \langle \nabla u_\e,iF^\bot\cdot\nabla u_\e\rangle\\
-\frac a2\nabla h\Big(|\nabla u_\e|^2+N_\e^2|\!\vi_\e\!|^2-2N_\e \!\vi_\e\!\cdot\, j_\e+\frac{a}{\e^2}(1-|u_\e|^2)^2+ (1-|u_\e|^2)f\Big)\\
-\frac a2(1-|u_\e|^2)\nabla f-aN_\e\big(j_\e\nabla h\cdot \!\vi_\e+\vi_\e\!\nabla h\cdot j_\e-N_\e \!\vi_\e\!\nabla h\cdot \!\vi_\e\!\big)\\
+aN_\e\big(-\vi_\e^\bot\!\mu_\e+(N_\e \!\vi_\e-j_\e)^\bot\curl\!\vi_\e-\vi_\e\!\Div j_\e+(N_\e \!\vi_\e-j_\e)\Div\!\vi_\e\big).
\end{multline}
Using identity~\eqref{eq:identity-2}, the first right-hand side term can be rewritten as
\begin{eqnarray*}
\lefteqn{\lambda_\e\left\langle(\alpha+i\beta\Log )\partial_tu_\e,\nabla u_\e\right\rangle}\\
&=&\lambda_\e\alpha\left\langle \partial_tu_\e-iu_\e N_\e\!\pre_{\e,\varrho},\nabla u_\e-iu_\e N_\e \!\vi_\e\right\rangle+N_\e\lambda_\e\alpha \!\vi_\e \langle \partial_tu_\e,i u_\e\rangle\\
&&\quad+N_\e\lambda_\e\alpha\pre_{\e,\varrho} j_\e -N_\e^2\lambda_\e\alpha|u_\e|^2\pre_{\e,\varrho} \!\vi_\e+\frac{\lambda_\e\beta}2\Log V_\e\\
&=&\lambda_\e\alpha\left\langle \partial_tu_\e-iu_\e N_\e\!\pre_{\e,\varrho},\nabla u_\e-iu_\e N_\e \!\vi_\e\right\rangle+N_\e \!\vi_\e(\Div j_\e+j_\e\cdot\nabla h)\\
&&\quad+\frac12N_\e\Log (F^\bot\cdot\nabla|u_\e|^2)\vi_\e+\frac{\lambda_\e\beta}2N_\e\Log\!\vi_\e\! \partial_t(1-|u_\e|^2)+N_\e\lambda_\e\alpha\pre_{\e,\varrho} j_\e\\
&&\quad-N_\e^2\lambda_\e\alpha|u_\e|^2\!\pre_{\e,\varrho} \!\vi_\e+\frac{\lambda_\e\beta}2\Log V_\e.
\end{eqnarray*}
Inserting this into~\eqref{eq:lemdivSpreres}, recombining $|\nabla u_\e|^2+N_\e^2|\!\vi_\e\!|^2-2N_\e \!\vi_\e\!\cdot\, j_\e=|\nabla u_\e-iu_\e N_\e \!\vi_\e\!|^2+N_\e^2(1-|u_\e|^2)|\!\vi_\e\!|^2$, noting that $\langle \nabla u_\e,iF^\bot\cdot\nabla u_\e\rangle=-\frac12F\mu_\e$, and using~\eqref{eq:defVtilde} to transform the vortex velocity $V_\e$ into its modulated version $\tilde V_{\e,\varrho}$,
we obtain
\begin{multline*}
\Div \tilde S_\e=a\lambda_\e\alpha\left\langle \partial_tu_\e-iu_\e N_\e\!\pre_{\e,\varrho},\nabla u_\e-iu_\e N_\e \!\vi_\e\right\rangle+aN_\e \!\vi_\e(\Div j_\e+j_\e\cdot\nabla h)\\
+\frac a2N_\e\Log (F^\bot\cdot\nabla|u_\e|^2)\vi_\e+\lambda_\e\alpha aN_\e\pre_{\e,\varrho} j_\e-aN_\e^2\lambda_\e\alpha|u_\e|^2\pre_{\e,\varrho} \!\vi_\e\\
+\frac{a\lambda_\e\beta}2\Log\tilde V_{\e,\varrho}-\frac{a\lambda_\e\beta}2N_\e\Log\pre_{\e,\varrho}\!\nabla|u_\e|^2-a\mu_\e(N_\e\!\vi_\e^\bot-\tfrac12\Log F)\\
-\frac a2\nabla h\Big(|\nabla u_\e-iu_\e N_\e \!\vi_\e\!|^2+\frac{a}{\e^2}(1-|u_\e|^2)^2+ (1-|u_\e|^2)(N_\e^2|\!\vi_\e\!|^2+f)\Big)\\
-\frac a2(1-|u_\e|^2)\nabla f-aN_\e\big(j_\e\nabla h\cdot \!\vi_\e+\vi_\e\!\nabla h\cdot j_\e-N_\e \!\vi_\e\!\nabla h\cdot \!\vi_\e\!\big)\\
+aN_\e\big((N_\e \!\vi_\e-j_\e)^\bot\curl\!\vi_\e-\vi_\e\!\Div j_\e+(N_\e \!\vi_\e-j_\e)\Div\!\vi_\e\big),
\end{multline*}
and the result follows after straightforward simplifications.
\end{proof}

\subsection{Time derivative of the modulated energy excess}
We establish the following decomposition of the time derivative of the modulated energy excess $\hat\D_{\e,\varrho,R}$. As will be seen in Sections~\ref{chap:MFL-GL}--\ref{chap:MFL-GP}, mean-field limit results are then reduced to the estimation of the different terms in this decomposition. To simplify notation, it is stated here with truncations centered at $z=0$, but the corresponding result of course also holds uniformly for all translations $z\in \R^2$.
\begin{lem}\label{lem:decompcruc}
Let $\alpha\ge0$, $\beta\in\R$, and let $h:\R^2\to\R$, $a:=e^h$, $F:\R^2\to\R^2$, $f:\R^2\to\R$ satisfy~\eqref{eq:scalingshFf} or~\eqref{eq:scalingshFfdec}. Let $u_\e:[0,T)\times\R^2\to\C$ and $\vi_\e:[0,T)\times\R^2\to\R^2$ be solutions of~\eqref{eq:GL-1} and of~\eqref{eq:genveps} as in Proposition~\ref{prop:globGL} and in Assumption~\ref{as:apveps}, respectively.
Let $0<\e\ll1$, $\varrho, R\gg1$, and let $\bar\Gamma_\e:[0,T)\times\R^2\to\R^2$ be a given vector field with $\|\bar\Gamma_\e^t\|_{W^{1,\infty}}\lesssim_t1$. Then, we have
\begin{align*}
\partial_t\hat\D_{\e,\varrho,R}=~&I_{\e,\varrho,R}^S+I_{\e,\varrho,R}^V+I_{\e,\varrho,R}^E+I_{\e,\varrho,R}^D+I_{\e,\varrho,R}^H+I_{\e,\varrho,R}^d+I_{\e,\varrho,R}^g+I_{\e,\varrho,R}^n+I_{\e,\varrho,R}',
\end{align*}
in terms of
\begingroup\allowdisplaybreaks
\begin{align*}
I_{\e,\varrho,R}^S:=~&-\int_{\R^2}\chi_R\nabla\bar\Gamma_\e^\bot: \tilde S_\e,\\
I_{\e,\varrho,R}^V:=~&\int_{\R^2} \frac {a\chi_R\Log}2\,\tilde V_{\e,\varrho}\cdot \Big(-\lambda_\e\beta \Gamma_\e^\bot+ \nabla^\bot h-F^\bot-\frac{2N_\e}\Log \!\vi_\e\Big),\\
I_{\e,\varrho,R}^E:=~&-\int_{\R^2} \frac {a\chi_R\Log}2\,\Gamma_\e\cdot \Big( \nabla^\bot h-F^\bot-\frac{2N_\e}\Log\vi_\e\Big)\mu_\e,\\
I_{\e,\varrho,R}^D:=~&-\int_{\R^2} \lambda_\e \alpha a\chi_R|\partial_tu_\e-iu_\e N_\e\!\pre_{\e,\varrho}\!|^2\\
&\qquad-\int_{\R^2} \lambda_\e\alpha a\chi_R \Gamma_\e^\bot\cdot\left\langle \partial_tu_\e-iu_\e N_\e\!\pre_{\e,\varrho},\nabla u_\e-iu_\e N_\e \!\vi_\e\right\rangle,\\
I_{\e,\varrho,R}^H:=~&\int_{\R^2} \frac {a\chi_R}2\Gamma_\e^\bot\cdot\nabla h\Big(|\nabla u_\e-iu_\e N_\e \!\vi_\e\!|^2+\frac{a}{\e^2}(1-|u_\e|^2)^2-\Log \mu_\e\Big),
\end{align*}
and
\begin{align*}
I_{\e,\varrho,R}^d:=~&\int_{\R^2} a\chi_RN_\e\Big(\bar\Gamma_\e^\bot\cdot(j_\e-N_\e \!\vi_\e)+\langle \partial_t u_\e-iu_\e N_\e\!\pre_{\e,\varrho},iu_\e\rangle\Big)\\
&\hspace{6cm}\times\big(\Div\!\vi_\e+\vi_\e\cdot\,\nabla h-\lambda_\e\alpha\pre_{\e,\varrho}\big),\\
I_{\e,\varrho,R}^g:=~&\int_{\R^2} a\chi_RN_\e(N_\e \!\vi_\e-j_\e)\cdot(\Gamma_\e-\bar\Gamma_\e)\curl\!\vi_\e+\int_{\R^2} \frac {a\chi_R}2 \lambda_\e\beta\Log\tilde V_{\e,\varrho}\cdot(\Gamma_\e-\bar\Gamma_\e)^\bot\\
&\qquad+\int_{\R^2} \lambda_\e\alpha a\chi_R (\Gamma_\e-\bar\Gamma_\e)^\bot\cdot\left\langle \partial_tu_\e-iu_\e N_\e\!\pre_{\e,\varrho},\nabla u_\e-iu_\e N_\e \!\vi_\e\right\rangle\\
&\qquad+\int_{\R^2} \frac {a\chi_R}2(\bar\Gamma_\e-\Gamma_\e)^\bot\cdot\nabla h\Big(|\nabla u_\e-iu_\e N_\e \!\vi_\e\!|^2+\frac{a}{\e^2}(1-|u_\e|^2)^2\Big)\\
&\qquad+\int_{\R^2} a\chi_R(\bar\Gamma_\e-\Gamma_\e)\cdot (N_\e \!\vi_\e+\tfrac12\Log F^\bot)\mu_\e\\
&\qquad+\int_{\R^2} a\chi_R{\lambda_\e\beta}N_\e\Log(\bar\Gamma_\e-\Gamma_\e)^\bot\cdot \!\vi_\e\partial_t|u_\e|^2,\\
I_{\e,\varrho,R}^n:=~&-\int_{\R^2}\nabla\chi_R\cdot\tilde S_\e\cdot\bar\Gamma_\e^\bot\\
&\qquad- \int_{\R^2} a\nabla\chi_R\cdot\Big(\langle\partial_tu_\e-iu_\e N_\e\!\pre_{\e,\varrho},\nabla u_\e-iu_\e N_\e \!\vi_\e\rangle+\frac{\Log }2\tilde V_{\e,\varrho}^\bot\Big),
\end{align*}
\endgroup
and where the error $I_{\e,\varrho,R}'$ is estimated as follows: in the dissipative case, in the considered regimes,
\begin{align}\label{eq:decompDeR-rest}
\int_0^t|I_{\e,\varrho,R}'|\lesssim_{t}\e R(N_\e^2+\Log^2)(\Ec_{\e,R}^*)^{1/2},
\end{align}
or in the conservative case, in the considered regime~\GP, for all $\theta>0$,
\begin{equation}\label{eq:decompDeR-rest-dec}
|I_{\e,\varrho,R}'|\lesssim_{t,\theta}\e N_\e\Ec_{\e,R}^*+N_\e(\Ec_{\e,R}^*)^{1/2}\|\nabla(\pre_\e-\pre_{\e,\varrho})\|_{\Ld^2}+\e N_\e^2\varrho^\theta(\Ec_{\e,R}^*)^{1/2}.
\qedhere
\end{equation}
\end{lem}

\begin{proof}
We split the proof into three steps, first computing the time derivative $\partial_t\hat\Ec_{\e,\varrho,R}$, then deducing an expression for $\partial_t\hat\D_{\e,\varrho,R}$, and finally introducing the modulated stress-energy tensor to replace the linear terms by quadratic ones, which are better suited for the Grönwall argument.

\medskip
\noindent\step1 Time derivative of the modulated energy.

In this step, we prove the following identity,
\begin{multline}\label{eq:timederenergy}
\partial_t\hat\Ec_{\e,\varrho,R}=- \int_{\R^2} a\nabla \chi_R\cdot\langle \partial_t u_\e,\nabla u_\e-iu_\e N_\e \!\vi_\e\rangle+\int_{\R^2} \frac{aN_\e^2}2\partial_t\big((1-|u_\e|^2)(\psi_{\e,\varrho,R}-\chi_R|\!\vi_\e\!|^2)\big)\\
+\int_{\R^2} N_\e a\chi_R\langle \partial_tu_\e,iu_\e\rangle(\Div\!\vi_\e+\vi_\e\!\cdot\,\nabla h)\\
+\int_{\R^2} a\chi_R\Big(N_\e(N_\e \!\vi_\e-j_\e)\cdot\partial_t\!\vi_\e-\lambda_\e\alpha|\partial_tu_\e|^2-N_\e \!\vi_\e\!\cdot\, V_\e-\frac{\Log }2F^\bot\cdot V_\e\Big).
\end{multline}
For that purpose, let us first compute the time derivative of the modulated energy density
\begin{multline}
\frac12\partial_t\Big(\chi_R|\nabla u_\e-iu_\e N_\e \!\vi_\e\!|^2+\frac{a\chi_R}{2\e^2}(1-|u_\e|^2)^2+(1-|u_\e|^2)(N_\e^2\psi_{\e,\varrho,R}+f\chi_R)\Big)\label{eq:derenergydens-1}\\
=\chi_R\langle \nabla u_\e-iu_\e N_\e \!\vi_\e,\nabla \partial_t u_\e-iu_\e N_\e\partial_t \!\vi_\e-i\partial_tu_\e N_\e \!\vi_\e\rangle-\chi_R\langle \partial_tu_\e,\frac{au_\e}{\e^2}(1-|u_\e|^2)\rangle\\
+\frac12\partial_t\big((1-|u_\e|^2)(N_\e^2\psi_{\e,\varrho,R}+f\chi_R)\big).
\end{multline}
Note that the first right-hand side term can be rewritten as
\begingroup\allowdisplaybreaks
\begin{eqnarray}
\lefteqn{\langle \nabla u_\e-iu_\e N_\e \!\vi_\e,\nabla \partial_t u_\e-iu_\e N_\e \partial_t \!\vi_\e-i\partial_tu_\e N_\e \!\vi_\e\rangle}\nonumber\\
&=&\langle \nabla u_\e,\nabla \partial_tu_\e\rangle -N_\e\partial_t\!\vi_\e\!\cdot\, j_\e-N_\e \!\vi_\e\!\cdot\, \langle \nabla u_\e,i\partial_tu_\e\rangle-N_\e \!\vi_\e\!\cdot\, \langle iu_\e,\nabla\partial_t u_\e\rangle\nonumber\\
&&\hspace{3cm}+\frac{N_\e^2}2|u_\e|^2\partial_t |\!\vi_\e\!|^2+\frac{N_\e^2}2|\!\vi_\e\!|^2\partial_t|u_\e|^2\nonumber\\
&=&\Div \langle \nabla u_\e, \partial_tu_\e\rangle -\langle \partial_tu_\e,\triangle u_\e\rangle-N_\e\partial_t\!\vi_\e\!\cdot\, j_\e-N_\e \!\vi_\e\!\cdot\, \langle \nabla u_\e,i\partial_tu_\e\rangle\nonumber\\
&&\hspace{3cm}-N_\e \!\vi_\e\!\cdot\, (\partial_tj_\e-\langle i\partial_tu_\e,\nabla u_\e\rangle)+\frac{N_\e^2}2\partial_t(|u_\e|^2|\!\vi_\e\!|^2)\nonumber\\
&=&\Div \langle \nabla u_\e, \partial_tu_\e\rangle -\langle \partial_tu_\e,\triangle u_\e\rangle-N_\e \!\vi_\e\!\cdot\,\partial_tj_\e-N_\e j_\e\cdot\partial_t\!\vi_\e+\frac{N_\e^2}2\partial_t (|u_\e|^2|\!\vi_\e\!|^2),\hspace{1cm}\label{eq:derenergydens-2}
\end{eqnarray}\endgroup
where
\begin{multline}
\Div\langle \nabla u_\e, \partial_tu_\e\rangle =\Div\langle \partial_t u_\e,\nabla u_\e-iu_\e N_\e \!\vi_\e\rangle+\Div(N_\e \!\vi_\e\langle \partial_tu_\e,iu_\e\rangle)\\
=\Div\langle \partial_t u_\e,\nabla u_\e-iu_\e N_\e \!\vi_\e\rangle+N_\e\langle \partial_tu_\e,iu_\e\rangle\Div\!\vi_\e+N_\e \!\vi_\e\!\cdot\, (\partial_t j_\e-V_\e).\label{eq:derenergydens-3}
\end{multline}
Combining~\eqref{eq:derenergydens-1}, \eqref{eq:derenergydens-2} and~\eqref{eq:derenergydens-3}, the time derivative of the energy density takes on the following guise, after straightforward simplifications,
\begin{multline*}
\frac12\partial_t\Big(\chi_R|\nabla u_\e-iu_\e N_\e \!\vi_\e\!|^2+\frac{a\chi_R}{2\e^2}(1-|u_\e|^2)^2+(1-|u_\e|^2)(N_\e^2\psi_{\e,\varrho,R}+f\chi_R)\Big)\\
=\chi_R\Div\langle \partial_t u_\e,\nabla u_\e-iu_\e N_\e \!\vi_\e\rangle+N_\e\chi_R\langle \partial_tu_\e,iu_\e\rangle\Div\!\vi_\e-N_\e\chi_R \!\vi_\e\!\cdot\, V_\e\\
+N_\e\chi_R(N_\e \!\vi_\e-j_\e)\cdot\partial_t\!\vi_\e\!
-\chi_R\left\langle \partial_tu_\e,\triangle u_\e+\frac{au_\e}{\e^2}(1-|u_\e|^2)\right\rangle\\
+\frac12\partial_t\big((1-|u_\e|^2)(N_\e^2\psi_{\e,\varrho,R}-N_\e^2\chi_R|\!\vi_\e\!|^2+f\chi_R)\big).
\end{multline*}
Integrating this identity in space yields
\begin{multline*}
\partial_t\int_{\R^2} \frac{a}2\Big(\chi_R|\nabla u_\e-iu_\e N_\e \!\vi_\e\!|^2+\frac{a\chi_R}{2\e^2}(1-|u_\e|^2)^2+(1-|u_\e|^2)(N_\e^2\psi_{\e,\varrho,R}+f\chi_R)\Big)\\
=\int_{\R^2} a\chi_R\Big(N_\e\langle \partial_tu_\e,iu_\e\rangle\Div\!\vi_\e-N_\e \!\vi_\e\!\cdot\, V_\e+N_\e(N_\e \!\vi_\e-j_\e)\cdot\partial_t\!\vi_\e\Big)\\
-\int_{\R^2} a\chi_R\Big\langle \partial_tu_\e,\triangle u_\e+\frac{au_\e}{\e^2}(1-|u_\e|^2)\Big\rangle- \int_{\R^2}\nabla( a\chi_R)\cdot\langle \partial_t u_\e,\nabla u_\e-iu_\e N_\e \!\vi_\e\rangle\\
+\int_{\R^2} \frac a2\partial_t\big((1-|u_\e|^2)(N_\e^2\psi_{\e,\varrho,R}-N_\e^2\chi_R|\!\vi_\e\!|^2+f\chi_R)\big).
\end{multline*}
Decomposing $\nabla(a\chi_R)=a\chi_R\nabla h+a\nabla\chi_R$, and using~\eqref{eq:GL-1} in the form
\begin{multline*}
\left\langle \partial_tu_\e,\triangle u_\e+\frac{au_\e}{\e^2}(1-|u_\e|^2)+\nabla h\cdot\nabla u_\e\right\rangle\\
=\left\langle \partial_tu_\e,\lambda_\e(\alpha+i\beta\Log )\partial_t u_\e-i\Log F^\bot\cdot\nabla u_\e-fu_\e\right\rangle\\
=\lambda_\e\alpha|\partial_tu_\e|^2+\frac{\Log }2F^\bot\cdot V_\e-\frac12f\partial_t|u_\e|^2,
\end{multline*}
the result~\eqref{eq:timederenergy} follows after straightforward simplifications.

\medskip
\noindent\step2 Time derivative of the modulated energy excess.
\nopagebreak

In this step, we prove the following identity,
\begingroup\allowdisplaybreaks
\begin{multline}\label{eq:timederenergyexcess}
\partial_t\hat\D_{\e,\varrho,R}=\int_{\R^2} \frac {a\chi_R}2\tilde V_{\e,\varrho}\cdot (\Log  (\nabla^\bot h-F^\bot)-2N_\e \!\vi_\e)\\
+\int_{\R^2} a\chi_RN_\e(N_\e \!\vi_\e-j_\e)\cdot\Gamma_\e\curl\!\vi_\e
-\int_{\R^2} \lambda_\e\alpha a\chi_R|\partial_tu_\e-iu_\e N_\e\!\pre_{\e,\varrho}\!|^2\\
+\int_{\R^2} a\chi_R N_\e\langle\partial_tu_\e-iu_\e N_\e\!\pre_{\e,\varrho},iu_\e\rangle(\Div\!\vi_\e+\vi_\e\!\cdot\,\nabla h-\lambda_\e\alpha\pre_{\e,\varrho})\\
-\int_{\R^2} a\nabla\chi_R\cdot\Big(\langle\partial_tu_\e-iu_\e N_\e\!\pre_{\e,\varrho},\nabla u_\e-iu_\e N_\e \!\vi_\e\rangle+\frac{\Log }2\tilde V_{\e,\varrho}^\bot\Big)\\
+\int_{\R^2} a\chi_RN_\e (N_\e \!\vi_\e-j_\e)\cdot\nabla(\pre_\e-\pre_{\e,\varrho})+\int_{\R^2}\frac{aN_\e^2}2\partial_t\big((1-|u_\e|^2)(\psi_{\e,\varrho,R}-\chi_R|\!\vi_\e\!|^2)\big)\\
-\int_{\R^2} a N_\e^2\!\pre_{\e,\varrho}(1-|u_\e|^2)\big( \!\vi_\e\!\cdot\,\nabla\chi_R+\chi_R (\Div\!\vi_\e+\vi_\e\!\cdot\,\nabla h)\big)\\
+\int_{\R^2}\frac{aN_\e\Log}2\partial_t(1-|u_\e|^2)\bigg(\vi_\e^\bot\!\cdot\nabla\chi_R-\lambda_\e\beta{\chi_R}\!\pre_{\e,\varrho}-{\chi_R} \!\vi_\e\!\cdot\, \Big(\nabla^\bot h-F^\bot-2\frac{N_\e}\Log \!\vi_\e\Big)\bigg)\\
+\int_{\R^2} \frac{aN_\e\Log }2 \pre_{\e,\varrho}\!\nabla(1-|u_\e|^2)\cdot\bigg(\nabla^\bot\chi_R+\chi_R\Big(\nabla^\bot h-2F^\bot-2\frac{N_\e}\Log \!\vi_\e\Big)\bigg).
\end{multline}\endgroup
Noting that identity~\eqref{eq:identity-1} implies
\begin{multline*}
\Log \int_{\R^2} a\chi_R\partial_t\mu_\e=\Log \int_{\R^2} a\chi_R\curl V_\e\\
=-\Log \int_{\R^2} a \chi_RV_\e\cdot\nabla^\bot h-\Log \int_{\R^2} a V_\e\cdot\nabla^\bot \chi_R,
\end{multline*}
it is immediate to deduce from~\eqref{eq:timederenergy} the following identity for the time derivative of the modulated energy excess,
\begin{multline}\label{eq:timederenergyexcess-0}
\partial_t\hat\D_{\e,\varrho,R}=\int_{\R^2} \frac {a\chi_R}2V_\e\cdot (\Log  (\nabla^\bot h-F^\bot)-2N_\e \!\vi_\e)\\
+\int_{\R^2} aN_\e\chi_R \langle \partial_t u_\e,iu_\e\rangle (\Div\!\vi_\e+\vi_\e\!\cdot\,\nabla h)
+\int_{\R^2} a\chi_RN_\e(N_\e \!\vi_\e-j_\e)\cdot\partial_t\!\vi_\e\\
-\int_{\R^2} \lambda_\e \alpha a\chi_R|\partial_tu_\e|^2 +\int_{\R^2}\frac{aN_\e^2}2\partial_t\big((1-|u_\e|^2)(\psi_{\e,\varrho,R}-\chi_R|\!\vi_\e\!|^2)\big)\\
- \int_{\R^2} a\nabla \chi_R\cdot \bigg(\langle\partial_t u_\e,\nabla u_\e-iu_\e N_\e \!\vi_\e\rangle+\frac{\Log }2V_\e^\bot\bigg).
\end{multline}
Now using equation~\eqref{eq:genveps} for the time evolution of $\vi_\e$ and an integration by parts, we find
\begin{eqnarray*}
\lefteqn{\int_{\R^2} a\chi_RN_\e(N_\e \!\vi_\e-j_\e)\cdot\partial_t\!\vi_\e}\\
&=&\int_{\R^2} a\chi_RN_\e(N_\e \!\vi_\e-j_\e)\cdot\Gamma_\e\curl\!\vi_\e+\int_{\R^2} a\chi_RN_\e (N_\e \!\vi_\e-j_\e)\cdot\nabla\!\pre_\e\\
&=&\int_{\R^2} a\chi_RN_\e(N_\e \!\vi_\e-j_\e)\cdot\Gamma_\e\curl\!\vi_\e+\int_{\R^2} a\chi_RN_\e (N_\e \!\vi_\e-j_\e)\cdot\nabla(\pre_\e-\pre_{\e,\varrho})\\
&&\quad-\int_{\R^2} a\chi_RN_\e\!\pre_{\e,\varrho}(N_\e\Div\!\vi_\e-\Div j_\e)-\int_{\R^2} a\chi_RN_\e\!\pre_{\e,\varrho}\nabla h\cdot(N_\e \!\vi_\e- j_\e)\\
&&\hspace{4cm}-\int_{\R^2} aN_\e\!\pre_{\e,\varrho}\nabla \chi_R\cdot(N_\e \!\vi_\e- j_\e).
\end{eqnarray*}
Combining this with identity~\eqref{eq:identity-2} yields
\begingroup\allowdisplaybreaks
\begin{eqnarray*}
\lefteqn{\int_{\R^2} a\chi_RN_\e(N_\e \!\vi_\e-j_\e)\cdot\partial_t\!\vi_\e}\\
&=&\int_{\R^2} a\chi_RN_\e(N_\e \!\vi_\e-j_\e)\cdot\Gamma_\e\curl\!\vi_\e+\int_{\R^2} a\chi_RN_\e (N_\e \!\vi_\e-j_\e)\cdot\nabla(\pre_\e-\pre_{\e,\varrho})\\
&&-\int_{\R^2} a\chi_RN_\e\!\pre_{\e,\varrho}\nabla h\cdot(N_\e \!\vi_\e- j_\e)-\int_{\R^2} aN_\e\!\pre_{\e,\varrho}\nabla \chi_R\cdot(N_\e \!\vi_\e- j_\e)\\
&&-\int_{\R^2} a\chi_RN_\e\!\pre_{\e,\varrho}\Big(N_\e\Div\!\vi_\e+j_\e\cdot \nabla h-\lambda_\e\alpha\langle\partial_tu_\e,iu_\e\rangle\\
&&\hspace{4cm}+\frac{\Log }2F^\bot\cdot\nabla |u_\e|^2-\frac{\lambda_\e\beta\Log }2\partial_t|u_\e|^2\Big)\\
&=&\int_{\R^2} a\chi_RN_\e(N_\e \!\vi_\e-j_\e)\cdot\Gamma_\e\curl\!\vi_\e+\int_{\R^2} a\chi_RN_\e (N_\e \!\vi_\e-j_\e)\cdot\nabla(\pre_\e-\pre_{\e,\varrho})\\
&&-\int_{\R^2} a\chi_RN_\e^2\pre_{\e,\varrho}(\Div\!\vi_\e+\vi_\e\!\cdot\, \nabla h)
-\int_{\R^2} aN_\e\!\pre_{\e,\varrho}\nabla \chi_R\cdot(N_\e \!\vi_\e- j_\e)\\
&&+\int_{\R^2} a\chi_RN_\e\!\pre_{\e,\varrho}\Big(\lambda_\e\alpha\langle\partial_tu_\e,iu_\e\rangle-\frac{\Log }2F^\bot\cdot\nabla |u_\e|^2+\frac{\lambda_\e\beta\Log }2\partial_t|u_\e|^2\Big).
\end{eqnarray*}\endgroup
Inserting this into~\eqref{eq:timederenergyexcess-0}, we then find
\begingroup\allowdisplaybreaks
\begin{multline}\label{eq:timederenergyexcess-pre}
\partial_t\hat\D_{\e,\varrho,R}=\int_{\R^2} \frac {a\chi_R}2V_\e\cdot \big(\Log  (\nabla^\bot h-F^\bot)-2N_\e \!\vi_\e\big)\\
+\int_{\R^2} a\chi_R N_\e\langle \partial_t u_\e,iu_\e\rangle(\Div\!\vi_\e+\vi_\e\!\cdot\,\nabla h+\lambda_\e\alpha\pre_{\e,\varrho})\\
-\int_{\R^2} a\chi_RN_\e^2\pre_{\e,\varrho}(\Div\!\vi_\e+\vi_\e\!\cdot\, \nabla h)
+\int_{\R^2} a\chi_RN_\e(N_\e \!\vi_\e-j_\e)\cdot\Gamma_\e\curl\!\vi_\e\\
+\int_{\R^2} a\chi_RN_\e (N_\e \!\vi_\e-j_\e)\cdot\nabla(\pre_\e-\pre_{\e,\varrho})
+\int_{\R^2}\frac{aN_\e^2}2\partial_t\big((1-|u_\e|^2)(\psi_{\e,\varrho,R}-\chi_R|\!\vi_\e\!|^2)\big)\\
+\int_{\R^2}\frac{a\chi_R}2N_\e\Log\!\pre_{\e,\varrho}\!\big({\lambda_\e\beta}\partial_t|u_\e|^2-F^\bot\cdot\nabla |u_\e|^2\big)
-\int_{\R^2} \lambda_\e\alpha a\chi_R|\partial_tu_\e|^2\\
- \int_{\R^2} a\nabla \chi_R\cdot \Big(\langle\partial_t u_\e,\nabla u_\e-iu_\e N_\e \!\vi_\e\rangle+\frac{\Log }2V_\e^\bot+N_\e\!\pre_{\e,\varrho}(N_\e \!\vi_\e- j_\e)\Big).
\end{multline}\endgroup
Using identity~\eqref{eq:defVtilde} to turn $V_\e$ into~$\tilde V_{\e,\varrho}$, the first right-hand side term is rewritten as
\begin{multline*}
\int_{\R^2} \frac {a\chi_R}2V_\e\cdot (\Log  (\nabla^\bot h-F^\bot)-2N_\e \!\vi_\e)\\
=\int_{\R^2} \frac {a\chi_R}2\big(\tilde V_{\e,\varrho}-N_\e \!\vi_\e\!\partial_t(1-|u_\e|^2)-N_\e\!\pre_{\e,\varrho}\!\nabla|u_\e|^2\big)\cdot \big(\Log  (\nabla^\bot h-F^\bot)-2N_\e \!\vi_\e\big),
\end{multline*}
while the last right-hand side term becomes
\begin{multline*}
\int_{\R^2} a\nabla \chi_R\cdot \Big(\langle\partial_t u_\e,\nabla u_\e-iu_\e N_\e \!\vi_\e\rangle+\frac{\Log }2V_\e^\bot+N_\e\!\pre_{\e,\varrho}(N_\e \!\vi_\e- j_\e)\Big)\\
=\int_{\R^2} a\nabla\chi_R\cdot\Big(\langle\partial_tu_\e-iu_\e N_\e\!\pre_{\e,\varrho},\nabla u_\e-iu_\e N_\e \!\vi_\e\rangle+N_\e^2\pre_{\e,\varrho} \!\vi_\e(1-|u_\e|^2)\\
+\frac{\Log }2\tilde V_{\e,\varrho}^\bot-\frac{N_\e\Log }2 \!\vi_\e^\bot\!\partial_t(1-|u_\e|^2)-\frac{N_\e\Log }2\pre_{\e,\varrho}\!\nabla^\bot|u_\e|^2\Big).
\end{multline*}
Further decomposing
\begin{align*}
|\partial_tu_\e|^2&=|\partial_tu_\e-iu_\e N_\e\!\pre_{\e,\varrho}\!|^2+2N_\e\!\pre_{\e,\varrho}\langle\partial_tu_\e-iu_\e N_\e\!\pre_{\e,\varrho},iu_\e\rangle\\
&\hspace{2cm}+N_\e^2|\!\pre_{\e,\varrho}\!|^2-(1-|u_\e|^2)N_\e^2|\!\pre_{\e,\varrho}\!|^2,\\
\langle\partial_tu_\e,iu_\e\rangle&=\langle\partial_tu_\e-iu_\e N_\e\!\pre_{\e,\varrho},iu_\e\rangle+|u_\e|^2 N_\e\!\pre_{\e,\varrho},
\end{align*}
the result~\eqref{eq:timederenergyexcess} easily follows after straightforward simplifications.

\medskip
\noindent\step3 Conclusion.

In the right-hand side of~\eqref{eq:timederenergyexcess}, the term $\int_{\R^2} a\chi_RN_\e(N_\e \!\vi_\e-j_\e)\cdot\Gamma_\e\curl\!\vi_\e$ is linear in $N_\e\!\vi_\e-j_\e$, thus preventing a direct use of a Grönwall argument. As in~\cite{Serfaty-15}, we replace this term by others involving the modulated stress-energy tensor $\tilde S_\e$, which is indeed a nicer {\it quadratic} quantity. For that purpose, let us integrate the result of Lemma~\ref{lem:divS} in space with $\chi_R\bar\Gamma_\e^\bot$, where $\bar\Gamma_\e:[0,T)\to W^{1,\infty}(\R^2)^2$ is a given vector field (we would like to choose $\bar\Gamma_\e=\Gamma_\e$, but a suitable perturbation will be needed),
\begingroup\allowdisplaybreaks
\begin{multline*}
\int_{\R^2}\chi_R\bar\Gamma_\e^\bot\cdot \Div \tilde S_\e=\int_{\R^2} \lambda_\e\alpha a\chi_R \bar\Gamma_\e^\bot\cdot\left\langle \partial_tu_\e-iu_\e N_\e\!\pre_{\e,\varrho},\nabla u_\e-iu_\e N_\e \!\vi_\e\right\rangle\\
-\int_{\R^2} a\chi_R\bar\Gamma_\e\cdot (N_\e \!\vi_\e+\tfrac12\Log F^\bot)\mu_\e
+\int_{\R^2} a\chi_RN_\e(N_\e \!\vi_\e-j_\e)\cdot \bar\Gamma_\e\curl\!\vi_\e\\
+\int_{\R^2} \lambda_\e\beta\frac{a\chi_R}2\Log \bar\Gamma_\e^\bot\cdot \tilde V_{\e,\varrho}-\int_{\R^2} \lambda_\e\beta\frac{a\chi_R}2N_\e\Log \pre_{\e,\varrho}\bar\Gamma_\e^\bot\cdot\nabla|u_\e|^2\\
+\int_{\R^2} a\chi_RN_\e\bar\Gamma_\e^\bot\cdot(N_\e \!\vi_\e-j_\e)(\Div\!\vi_\e+\nabla h\cdot \!\vi_\e-\lambda_\e\alpha\pre_{\e,\varrho})-\int_{\R^2} \frac {a\chi_R}2(1-|u_\e|^2)\bar\Gamma_\e^\bot\cdot \nabla f\\
-\int_{\R^2} \frac {a\chi_R}2\bar\Gamma_\e^\bot\cdot\nabla h\Big(|\nabla u_\e-iu_\e N_\e \!\vi_\e\!|^2+\frac{a}{\e^2}(1-|u_\e|^2)^2+ (1-|u_\e|^2)(N_\e^2|\!\vi_\e\!|^2+f)\Big)\\
+\int_{\R^2} \lambda_\e\alpha a\chi_RN_\e^2\!\pre_{\e,\varrho}(1-|u_\e|^2)(\bar\Gamma_\e^\bot\cdot \!\vi_\e)+\int_{\R^2} \frac {a\chi_R}2N_\e\Log (F^\bot\cdot\nabla|u_\e|^2)(\bar\Gamma_\e^\bot\cdot \!\vi_\e).
\end{multline*}
\endgroup
In the right-hand side, the term $\int_{\R^2} a\chi_RN_\e(N_\e \!\vi_\e-j_\e)\cdot\bar\Gamma_\e\curl\!\vi_\e$ exactly corresponds to the bad term in the right-hand side of~\eqref{eq:timederenergyexcess}. Replacing it by this new expression involving the modulated stress-energy tensor, and treating as errors all the terms involving the difference $\bar\Gamma_\e-\Gamma_\e$, we find
\begingroup\allowdisplaybreaks
\begin{multline*}
\partial_t\hat\D_{\e,\varrho,R}=\sum_{j=0}^3T_{\e,R}^j+I_{\e,\varrho,R}^g+I_{\e,\varrho,R}^n\\
-\int_{\R^2}\chi_R\nabla\bar\Gamma_\e^\bot: \tilde S_\e-\int_{\R^2} \lambda_\e\alpha a\chi_R \Gamma_\e^\bot\cdot\left\langle \partial_tu_\e-iu_\e N_\e\!\pre_{\e,\varrho},\nabla u_\e-iu_\e N_\e \!\vi_\e\right\rangle\\
+\int_{\R^2} \frac {a\chi_R}2\Gamma_\e^\bot\cdot\nabla h\Big(|\nabla u_\e-iu_\e N_\e \!\vi_\e\!|^2+\frac{a}{\e^2}(1-|u_\e|^2)^2\Big)\\
-\int_{\R^2} \lambda_\e\alpha a\chi_R|\partial_tu_\e-iu_\e N_\e\!\pre_{\e,\varrho}\!|^2
+\int_{\R^2} a\chi_R\Gamma_\e\cdot (N_\e \!\vi_\e+\tfrac12\Log F^\bot)\mu_\e\\
+\int_{\R^2} \frac {a\chi_R}2\tilde V_{\e,\varrho}\cdot (-\lambda_\e\beta\Log\Gamma_\e^\bot+\Log  (\nabla^\bot h-F^\bot)-2N_\e \!\vi_\e)\\
+\int_{\R^2} a\chi_R N_\e\big(\langle\partial_tu_\e-iu_\e N_\e\!\pre_{\e,\varrho},iu_\e\rangle+\bar\Gamma_\e^\bot\cdot(j_\e-N_\e \!\vi_\e)\big)(\Div\!\vi_\e+\vi_\e\!\cdot\,\nabla h-\lambda_\e\alpha\pre_{\e,\varrho}),
\end{multline*}\endgroup
where $I_{\e,\varrho,R}^g$ and $I_{\e,\varrho,R}^n$ are given as in the statement, and where we have set
\begin{align*}
T_{\e,\varrho,R}^0:=&\int_{\R^2} a\chi_R N_\e(N_\e\!\vi_\e-j_\e)\cdot\nabla(\pre_\e-\pre_{\e,\varrho}),\\
T_{\e,\varrho,R}^1:=&\int_{\R^2} \frac {a\chi_R}2(1-|u_\e|^2)(N_\e^2|\!\vi_\e\!|^2+f)\bar\Gamma_\e^\bot\cdot\nabla h\\
&\qquad-\int_{\R^2} a N_\e^2\!\pre_{\e,\varrho}(1-|u_\e|^2)\big( \!\vi_\e\!\cdot\,\nabla\chi_R+\chi_R (\Div\!\vi_\e+\vi_\e\!\cdot\,\nabla h)\big)\\
&\qquad+\int_{\R^2} \frac {a\chi_R}2(1-|u_\e|^2)\bar\Gamma_\e^\bot\cdot \nabla f-\int_{\R^2} \lambda_\e\alpha a\chi_RN_\e^2 \pre_{\e,\varrho}(1-|u_\e|^2)\bar\Gamma_\e^\bot\cdot \!\vi_\e,\\
T_{\e,\varrho,R}^2:=&\int_{\R^2} \frac {a\chi_R}2N_\e\Log \big(F^\bot\cdot\nabla(1-|u_\e|^2)\big)\bar\Gamma_\e^\bot\cdot \!\vi_\e\\
&\qquad+\int_{\R^2} \frac{aN_\e\Log }2 \pre_{\e,\varrho}\!\nabla(1-|u_\e|^2)\\
&\hspace{2cm}\cdot\bigg(\nabla^\bot\chi_R+\chi_R\Big(\nabla^\bot h-2F^\bot-\lambda_\e\beta\bar\Gamma_\e^\bot-2\frac{N_\e}\Log \!\vi_\e\Big)\bigg),
\end{align*}
\begin{align*}
T_{\e,\varrho,R}^3:=&\int_{\R^2}\frac{aN_\e^2}2\partial_t\big((1-|u_\e|^2)(\psi_{\e,\varrho,R}-\chi_R|\!\vi_\e\!|^2)\big)\\
&\qquad+\int_{\R^2}\frac{aN_\e\Log}2\partial_t(1-|u_\e|^2)\\
&\hspace{2cm}\times\bigg(\vi_\e^\bot\!\cdot\nabla\chi_R-\lambda_\e\beta{\chi_R}\!\pre_{\e,\varrho}- {\chi_R} \!\vi_\e\!\cdot\, \Big(\nabla^\bot h-F^\bot-2\frac{N_\e}\Log \!\vi_\e\Big)\bigg).
\end{align*}
It remains to estimate these four error terms $T_{\e,\varrho,R}^i$, $0\le i\le3$. We start with $T_{\e,\varrho,R}^0$. In the dissipative case  we take $\varrho=\infty$, hence  $T_{\e,\varrho,R}^0=0$. In the conservative case, using the pointwise estimate of Lemma~\ref{lem:pointest} for $j_\e-N_\e\!\vi_\e$,
and using Assumption~\ref{as:apveps}(b), with in particular
\[\|\nabla(\pre_\e^t-\pre_{\e,\varrho}^t)\|_{\Ld^2\cap\Ld^\infty}\lesssim \|\nabla\!\pre_\e^t\!\|_{\Ld^2\cap\Ld^\infty}+\varrho^{-1}\|\!\pre_{\e,\varrho}^t\!\|_{\Ld^2\cap\Ld^\infty}\lesssim_t1,\]
we find
\begin{align*}
|T_{\e,\varrho,R}^0|&\lesssim_t N_\e \|\nabla u_\e-iu_\e N_\e \!\vi_\e\!\|_{\Ld^2(B_{2R})}\big(\|\nabla(\pre_\e-\pre_{\e,\varrho})\|_{\Ld^2}+\|1-|u_\e|^2\|_{\Ld^2(B_{2R})}\big)\\
&\hspace{3cm}+N_\e^2\|1-|u_\e|^2\|_{\Ld^2(B_{2R})}\|\nabla(\pre_\e-\pre_{\e,\varrho})\|_{\Ld^2}\\
&\lesssim_t\e N_\e\Ec_{\e,R}^*+(1+\e N_\e)N_\e(\Ec_{\e,R}^*)^{1/2}\|\nabla(\pre_\e-\pre_{\e,\varrho})\|_{\Ld^2}.
\end{align*}
Using~\eqref{eq:scalingshFf} or~\eqref{eq:scalingshFfdec}, Assumption~\ref{as:apveps},
and the assumption $\|\bar\Gamma_\e\|_{\Ld^\infty}\lesssim_t1$, we obtain in the considered regimes, in the dissipative case,
\begin{align*}
|T_{\e,\varrho,R}^1|&\lesssim_t\e\big(\lambda_\e^{-1/2}N_\e^2+R\lambda_\e^2\Log^2\big)(\Ec_{\e,R}^*)^{1/2},
\end{align*}
and in the conservative case,
\begin{align*}
|T_{\e,\varrho,R}^1|&\lesssim_t\e(N_\e^2+\lambda_\e^2\Log^2)(\Ec_{\e,R}^*)^{1/2}\lesssim\e N_\e^2(\Ec_{\e,R}^*)^{1/2}.
\end{align*}
Integrating by parts, $T_{\e,\varrho,R}^2$ takes the form
\begin{multline*}
T_{\e,\varrho,R}^2=-\int_{\R^2} \frac{N_\e\Log }2 (1-|u_\e|^2)\\
\times\Div\bigg(a\!\pre_{\e,\varrho}\!\nabla^\bot\chi_R+a\chi_RF^\bot(\bar\Gamma_\e^\bot\cdot \!\vi_\e)+a\!\pre_{\e,\varrho}\!\chi_R\Big(\nabla^\bot h-2F^\bot-\lambda_\e\beta\bar\Gamma_\e^\bot-2\frac{N_\e}\Log \!\vi_\e\!\Big)\bigg),
\end{multline*}
and hence, again using~\eqref{eq:scalingshFf} or~\eqref{eq:scalingshFfdec}, Assumption~\ref{as:apveps},
and the bound $\|\bar\Gamma_\e\|_{W^{1,\infty}}\lesssim1$, we obtain in the considered regimes, for all $\theta>0$, in the dissipative case,
\begin{align*}
|T_{\e,\varrho,R}^2|&\lesssim_{t,\theta} \e N_\e\Log \big(1+R^{-1}\lambda_\e^{-1/2}+\lambda_\e R^\theta\big)(\Ec_{\e,R}^*)^{1/2},
\end{align*}
and in the conservative case,
\begin{align*}
|T_{\e,\varrho,R}^{2}|&\lesssim_{t,\theta} \e N_\e\Log (1+\lambda_\e\varrho^\theta)(\Ec_{\e,R}^*)^{1/2}\lesssim \e N_\e^2\varrho^\theta(\Ec_{\e,R}^*)^{1/2}.
\end{align*}
Finally, we note that the choice~\eqref{eq:choicepsi} of $\psi_{\e,\varrho,R}$ exactly yields
\begin{align*}
T_{\e,\varrho,R}^3&=\int_{\R^2}\frac{aN_\e^2}2(1-|u_\e|^2)\partial_t(\psi_{\e,\varrho,R}-\chi_R|\!\vi_\e\!|^2)\\
&=\int_{\R^2}\frac{aN_\e^2}2(1-|u_\e|^2)(\partial_t\psi_{\e,\varrho,R}-2\chi_R\vi_\e\!\cdot\,\partial_t\!\vi_\e),
\end{align*}
and hence, using \eqref{eq:boundpsi} or~\eqref{eq:boundpsi-2}, and Assumption~\ref{as:apveps}, we obtain in the considered regimes, in the dissipative case,
\begin{align*}
\|T_{\e,\varrho,R}^3\|_{\Ld^1_t}&\lesssim_t \e N_\e^2\Big(1+\frac\Log{N_\e}\Big)(\Ec_{\e,R}^*)^{1/2}\lesssim\e (N_\e^2+N_\e\Log)(\Ec_{\e,R}^*)^{1/2},
\end{align*}
and in the conservative case,
\begin{align*}
|T_{\e,\varrho,R}^{3}|&\lesssim_{t,\theta} \e N_\e^2\varrho^\theta(\Ec_{\e,R}^*)^{1/2}.
\end{align*}
The conclusion follow from the above with $I_{\e,\varrho,R}':=T_{\e,\varrho,R}^0+T_{\e,\varrho,R}^1+T_{\e,\varrho,R}^2+T_{\e,\varrho,R}^3$.
\end{proof}

\section{Vortex analysis}\label{sec:vortex}
In this section, we recall and revisit some standard tools for vortex analysis, which are needed in order to control the various terms appearing in the decomposition of $\partial_t\hat\D_{\e,\varrho,R}$ in Lemma~\ref{lem:decompcruc}. These tools will only be used in the dissipative case, and we restrict in this section to the dilute regime $N_\e\lesssim \Log$.
(Suitable adaptations to the nondilute regime $N_\e\gg\Log$ are postponed to Section~\ref{chap:vort-anal-hd}.)

\subsection{Ball construction lower bounds}
We need a version of the Jerrard-Sandier ball-construction lower bounds~\cite{Sandier-98,Jerrard-99} that is {\it localizable} in order to be adapted both to the weighted case and to the setting of the infinite plane with no finite energy control (hence no a priori bound on the number of vortices), and which further yields very small errors (we need an error of order $o(N_\e^2)$, which gets very small when $N_\e$ diverges slowly).
For that purpose we use the version developed in~\cite{SS-11}, which in particular allows to cover the plane with balls centered at the points of the lattice $R\Z^2$, make the standard ball construction in each ball of the covering, assemble all the constructed balls, and then discard some balls from the collection so as to make it disjoint again. 
The error in the lower bounds given by this ball construction is essentially $N_\e|\!\log r|$, where $r$ is the total radius of the balls, so that we need to take $r$ large enough (almost as large as $O(1)$ when $N_\e$ diverges slowly), but here the pinning weight adds again a difficulty since it may vary significantly over the size of the balls of this construction, thus perturbing the lower bound itself.

The following preliminary result describes the precise contribution of the vortices to the energy, and in particular defines the vortex ``locations''.

\begin{lem}[Localized lower bound]\label{lem:ballconstr}
Let $h:\R^2\to\R$, $a:=e^h$, with $1\lesssim a\le1$, let $u_\e:\R^2\to\C$, $\vi_\e:\R^2\to\R^2$, with $\|\curl\!\vi_\e\!\|_{\Ld^2\cap\Ld^\infty}\lesssim1$. Let $0<\e\ll1$, $N_\e,R\ge1$, and assume that $\log\Ec_{\e,R}^*\ll\Log$.
Then, for some $\bar r\simeq1$, for all $\e>0$ small enough and all $r\in(\e^{1/2},\bar r)$, there exists a locally finite union of disjoint closed balls $\B_{\e,R}^r$, monotone in $r$ and covering the set $\{x:|u_\e(x)|<\frac12\}$, such that for all $z\in R\Z^2$ the sum of the radii of the balls of the collection $\B_{\e,R}^r$ centered at points in $B_R(z)$ is bounded by $r$, and such that, letting $\B_{\e,R}^r:=\biguplus_jB^j$, $B^j:=\bar B(y_j,r_j)$, $d_j:=\deg(u_\e,\partial B^j)$, and defining the point-vortex measure $\nu_{\e,R}^r:=2\pi\sum_jd_j\delta_{y_j}$, the following properties hold,
\begin{enumerate}[(i)]
\item \emph{Localized lower bound:} For all $\phi\in W^{1,\infty}(\R^2)$ with $\phi\ge0$, we have for all $j$,
\begin{multline}\label{eq:lowerboundclaim2}
\qquad\frac12\int_{B^j}\phi\Big(|\nabla u_\e-iu_\e N_\e \!\vi_\e\!|^2+\frac{a}{2\e^2}(1-|u_\e|^2)^2\Big)\ge \pi\phi(y_j)|d_j|\log(\tfrac r\e)\\
-O(r_j\Ec_{\e,R}^*)\|\nabla\phi\|_{\Ld^{\infty}}-O\bigg(r_j^2N_\e^2+|d_j|\log\Big(2+\frac{\Ec_{\e,R}^*}\Log\Big)\bigg)\|\phi\|_{\Ld^{\infty}}.
\end{multline}
Similarly, if $\phi$ is further supported in a ball of radius $R$,
\begin{multline}\label{eq:lowerboundclaim1}
\qquad\frac12\int_{\B_{\e,R}^r}\phi\Big(|\nabla u_\e-iu_\e N_\e \!\vi_\e\!|^2+\frac{a}{2\e^2}(1-|u_\e|^2)^2\Big)\ge \frac{\log(\tfrac r\e) }2\int_{\R^2} \phi |\nu_{\e,R}^r|\\
-O(r\Ec_{\e,R}^*)\|\nabla\phi\|_{\Ld^{\infty}}-O\bigg(r^2N_\e^2+\frac{N_\e^2+\Ec_{\e,R}^*}\Log\log\Big(2+\frac{\Ec_{\e,R}^*}\Log\Big)\bigg)\|\phi\|_{\Ld^{\infty}}.
\end{multline}
\item \emph{Number of vortices:}
\begin{align}\label{eq:boundvortnumb}
\sup_z\int_{B_R(z)}|\nu_{\e,R}^r|\lesssim \frac{N_\e^2+\Ec_{\e,R}^*}\Log.
\end{align}
\item \emph{Jacobian estimate:}
For all $\gamma\in[0,1]$,
\[\sup_z\|\nu_{\e,R}^r-\tilde\mu_\e\|_{(C_c^\gamma(B_{R}(z)))^*}\lesssim r^\gamma\frac{N_\e^2+\Ec_{\e,R}^*}\Log+\e^{\gamma/2}(\Ec_{\e,R}^*+\e^{2}N_\e^2).\qedhere\]
\end{enumerate}
\end{lem}

\begin{proof}
We split the proof into two steps.

\medskip
\noindent\step1 Proof of~(i)--(ii).

We use the notation $\tilde\Ec_{\e,R}^*:=\sup_z\int_{B_{R}(z)}\tilde e_\e$, with
\[\tilde e_\e:=\frac{1}2\Big(|\nabla u_\e-iu_\e N_\e \!\vi_\e\!|^2+\frac{a_{\min}}{2\e^2}(1-|u_\e|^2)^2\Big),\qquad a_{\min}:=\inf_x a(x)\gtrsim1.\]
Note that by assumption we have in particular $\tilde \Ec_{\e,R}^*\lesssim \Ec_{\e,R}^*\lesssim \e^{-1/5}$.
We may apply~\cite[Proposition~2.1]{SS-11} with $\Omega_\e=\R^2$, $A_\e=N_\e\!\vi_\e$, with $\e$ replaced by $\e/\sqrt{a_{\min}}$, and with the open cover $(U_\alpha)_\alpha=(B_R(z))_{z\in R\Z^2}$ (note that the argument in~\cite{SS-11} indeed works identically on the whole space, and that the energy bound is only needed uniformly on all elements of the open cover).  For some $\e_0,C_0,\bar r\simeq1$, for all $\e<\e_0$ and $r\in(\e^{1/2},\bar r)$, we obtain a locally finite collection $\B_{\e,R}^r$ of disjoint closed balls covering the set $\{x:|u_\e(x)|<\frac12\}$, such that for all $B\in\B_{\e,R}^r$ we have
\[\int_B\Big(\tilde e_\e+\frac{N_\e^2}2|\curl\!\vi_\e\!|^2\Big)\ge\pi|d_B|\Big(\log\frac{r}{\e \bar C_B}-C_0\Big),\]
where we have set $d_B:=\deg(u_\e,\partial B)$, and where $\bar C_B$ is defined as in~\cite[(2.4)]{SS-11}. Moreover, the construction in~\cite{SS-11} ensures that $\B_{\e,R}^r$ is monotone in $r$ and that $B_R(z)\cap\B_{\e,R}^r$ has total radius bounded by $r$ for all $z\in R\Z^2$. By~\cite[Lemma~2.1]{SS-11}, we have $\bar C_B\le {16}\Log^{-1} \tilde\Ec_{\e,R}^*\lesssim\Log^{-1}\Ec_{\e,R}^*$, so that the above becomes, for all $B\in\B_{\e,R}^r$,
\begin{align}\label{eq:lowerboundSS11}
\int_B\Big(\tilde e_\e+\frac{N_\e^2}2|\curl\!\vi_\e\!|^2\Big)\ge\pi|d_B|\log(\tfrac{r}{\e})-|d_B|O\bigg(\log\Big(2+\frac{\Ec_{\e,R}^*}\Log\Big)\bigg).
\end{align}
Let $r\in (\e^{1/2},\bar r)$ be fixed, and set $\B_{\e,R}^r=\biguplus_jB^j$, $B^j:=\bar B(y_j,r_j)$, with corresponding degrees $d_j:=d_{B^j}$.
Noting that by assumption we have
\[\int_{B^j}|\curl\!\vi_\e\!|^2\lesssim|B^j|\lesssim r_j^2,\]
the result~\eqref{eq:lowerboundSS11} takes the following form, for all $j$,
\begin{align}\label{eq:lowerboundSS11+}
\int_{B^j}\tilde e_\e\ge\pi|d_j|\log(\tfrac{r}{\e})-|d_j|O\bigg(\log\Big(2+\frac{\Ec_{\e,R}^*}\Log\Big)\bigg)-O(r_j^2N_\e^2).
\end{align}
Using the assumption $\log\Ec_{\e,R}^*\ll\Log$ and the choice $r>\e^{1/2}$, the above right-hand side is bounded from below by $\frac\pi2|d_j|\Log(1-o(1))-O(r_j^2N_\e^2)$, and hence, summing over $B^j\in\B_{\e,R}^r$ with $y_j\in B_{R}(z)$, we find for all $\e>0$ small enough,
\begin{align*}
\frac{\pi}3\Log\sum_{j:y_j\in B_R(z)}|d_j|~\le~ \int_{B_{R+1}(z)\cap\B_{\e,R}^r}\tilde e_\e+O(N_\e^2)\sum_{j:y_j\in B_R(z)}r_j^2 ~\lesssim~ \Ec_{\e,R}^*+r^2N_\e^2,
\end{align*}
and hence, with the choice $r\lesssim1$,
\begin{align}\label{eq:sum|dB|}
\sum_{j:y_j\in B_R(z)}|d_j|\lesssim \frac{N_\e^2+\Ec_{\e,R}^*}\Log,
\end{align}
that is, item~(ii). 
Let us now prove item~(i). Let $\phi\in W^{1,\infty}(\R^2)$, $\phi\ge0$. For all $B^j\in\B_{\e,R}^r$, we have from~\eqref{eq:lowerboundSS11+},
\begin{align*}
\int_{B^j}\phi\tilde e_\e&\ge \phi(y_j)\int_{B^j} \tilde e_\e-r_j\|\nabla\phi\|_{\Ld^\infty}\int_{B^j}\tilde e_\e\\
&\ge \pi\phi(y_j)|d_j|\log(\tfrac r\e)\\
&\qquad-\phi(y_j)|d_j|\,O\bigg(\log \Big(2+\frac{\Ec_{\e,R}^*}\Log\Big)\bigg)-\phi(y_j)O(r_j^2N_\e^2)-r_j\|\nabla\phi\|_{\Ld^\infty}\int_{B^j}\tilde e_\e,
\end{align*}
hence
\begin{multline*}
\int_{B^j}\phi\tilde e_\e\ge \pi\phi(y_j)|d_j|\log(\tfrac r\e)\\
-O\bigg(r_j^2N_\e^2+|d_j|\log\Big(2+\frac{\Ec_{\e,R}^*}\Log\Big)\bigg)\|\phi\|_{\Ld^{\infty}}-O(r_j\Ec_{\e,R}^*)\|\nabla\phi\|_{\Ld^{\infty}}.
\end{multline*}
Further assuming that $\phi$ is supported in $B_R(z)$ for some $z\in R\Z^2$, summing the above with respect to $j$ with $y_j\in B_R$, setting $\nu_{\e,R}^r:=2\pi\sum_{j}d_j\delta_{y_j}$, and using~\eqref{eq:sum|dB|}, we find
\begin{multline*}
\int_{\B_{\e,R}^r}\phi\tilde e_\e\ge \frac{\log(\tfrac r\e) }2\int_{\R^2} \phi |\nu_{\e,R}^r|\\
-O\bigg(r^2N_\e^2+\frac{N_\e^2+\Ec_{\e,R}^*}\Log\log\Big(2+\frac{\Ec_{\e,R}^*}\Log\Big)\bigg)\|\phi\|_{\Ld^{\infty}}
-O(r\Ec_{\e,R}^*)\|\nabla\phi\|_{\Ld^{\infty}}.
\end{multline*}
Item~(i) then follows by definition of $\tilde e_\e$ with $a_{\min}\le a$.

\medskip
\noindent\step2 Proof of~(iii).

\nopagebreak
Using item~(i) and arguing just as in~\cite[Proposition~4.4(5)]{Serfaty-15}, for $\gamma\in[0,1]$, we obtain for all $r\in(\e^{1/2},\bar r)$ and all $\phi\in C^{\gamma}_c(\R^2)$ supported in $B_R(z)$ for some $z\in R\Z^2$,
\begingroup\allowdisplaybreaks
\begin{eqnarray}
\lefteqn{\bigg|\int\phi(\nu_{\e,R}^r-\tilde\mu_\e)\bigg|}\nonumber\\
&\lesssim& r^\gamma|\phi|_{C^{\gamma}} \sum_{j:y_j\in B_R(z)}|d_j|\nonumber\\
&&\quad+\e^{\gamma/2}\|\phi\|_{C^{\gamma}}\int_{B_R}\Big(|\nabla u_\e-iu_\e N_\e\!\vi_\e\!|^2+\frac{(1-|u_\e|^2)^2}{2\e^2}+N_\e|1-|u_\e|^2||\curl\!\vi_\e\!|\Big)\nonumber\\
&\lesssim& r^\gamma \frac{N_\e^2+\Ec_{\e,R}^*}\Log|\phi|_{C^{\gamma}}+\bigg(\e^{\gamma/2}\Ec_{\e,R}^*+\e^{2+\gamma/2}N_\e^2\int_{B_R}|\curl\!\vi_\e\!|^2\bigg)\|\phi\|_{C^{\gamma}},\label{eq:lem3.2iii-lip}
\end{eqnarray}
\endgroup
where $|\cdot|_{C^\gamma}$ denotes the usual Hölder seminorm and where $\|\cdot\|_{C^\gamma}:=|\cdot|_{C^\gamma}+\|\cdot\|_{\Ld^\infty}$.
The result follows from the assumption $\|\curl\!\vi_\e\!\|_{\Ld^2}\lesssim1$.
\end{proof}

In Section~\ref{chap:MFL-GL}, strong estimates are proved on the time derivative of the modulated energy excess $\D_{\e,R}^*$, but these estimates a priori involve the modulated energy $\Ec_{\e,R}^*$. In order to buckle the argument, it is thus crucial to independently find an optimal control on $\Ec_{\e,R}^*$, or equivalently on the number of vortices, in terms of $\D_{\e,R}^*$. Note that in the case without pinning and applied current no cut-off is needed and this difficulty is absent (the excess is then indeed simply defined by $\tilde\D_\e=\tilde\Ec_\e-\pi N_\e\Log$, cf.~\eqref{eq:excess-SS15}). This control of $\Ec_{\e,R}^*$ is the main content of the following result, and  allows to further refine the conclusions of Lemma~\ref{lem:ballconstr} above. Particular attention is needed in the strongly dilute regime $N_\e\lesssim\log\Log$ to ensure an error as small as $o(N_\e^2)$ in the energy lower bound. Various useful corollaries are further included.
In particular, item~(vi) gives an optimal control of the energy inside the small balls, measured in $\Ld^p$ for any $p<2$. Since this $\Ld^p$ result is already enough for our purposes, we do not adapt the more precise Lorentz estimates of~\cite[Corollary~1.2]{S-Tice-08} to the present weighted context, and we instead use a more direct argument adapted from~\cite{Struwe-94}.

\begin{prop}[Refined lower bound]\label{prop:ballconstr}
Let $h:\R^2\to\R$, $a:=e^h$, with $1\lesssim a\le1$ and $\|\nabla h\|_{\Ld^\infty}\lesssim1$, let $u_\e:\R^2\to\C$, $\vi_\e:\R^2\to\R^2$, with $\|\curl\!\vi_\e\!\|_{\Ld^1\cap\Ld^\infty},\|\!\vi_\e\!\|_{\Ld^\infty}\lesssim1$. Let $0<\e\ll1$, $1\ll N_\e\lesssim\Log$, and $R\ge1$ with $\Log\lesssim R\lesssim\Log^n$ for some $n\ge1$, and assume that $\D_{\e,R}^*\lesssim N_\e^2$.
Then $\Ec_{\e,R}^*\lesssim N_\e\Log$ holds for all $\e>0$ small enough.
Moreover, for some $\bar r\simeq1$, for all $\e>0$ small enough and all $r\in(\e^{1/2},\bar r)$, there exists a locally finite union of disjoint closed balls $\B_{\e,R}^r$, monotone in $r$ and covering the set $\{x:|u_\e(x)|<\frac12\}$, and for all $r_0\in(\e^{1/2},\bar r)$ and $r\ge r_0$ there exists a locally finite union of disjoint closed balls $\tilde\B_{\e,R}^{r_0,r}$, monotone in $r$ and covering the set $\{x:||u_\e(x)|-1|\ge\Log^{-1}\}$, such that $\B_{\e,R}^{r_0}\subset\tilde\B_{\e,R}^{r_0,r_0}$, such that for all $z\in R\Z^2$ the sum of the radii of the balls of the collection $\B_{\e,R}^r$ centered at points of $B_R(z)$ is bounded by $r$ and the sum of the radii of the balls of the collection $\tilde\B_{\e,R}^{r_0,r}$ centered at points of $B_R(z)$ is bounded by $Cr$, and such that, letting $\B_{\e,R}^r:=\biguplus_jB^j$, $B^j:=\bar B(y_j,r_j)$, $d_j:=\deg(u_\e,\partial B^j)$, and defining the point-vortex measure $\nu_{\e,R}^r:=2\pi\sum_j d_j\delta_{y_j}$, the following properties hold,
\begin{enumerate}[(i)]
\item \emph{Lower bound:} In the regime $N_\e\gg\log\Log$, we have for all $e^{-o(N_\e)}\le r\ll \frac{N_\e}\Log$ and $z\in \R^2$,
\begin{multline}\label{eq:lowerbound}
\qquad\frac12\int_{\B_{\e,R}^r}a\chi_R^z\Big(|\nabla u_\e-iu_\e N_\e \!\vi_\e\!|^2+\frac{a}{2\e^2}(1-|u_\e|^2)^2\Big)\\
\ge \frac{\Log}2\int_{\R^2} a\chi_R^z |\nu_{\e,R}^r|-o(N_\e^2),
\end{multline}
while in the regime $1\ll N_\e\lesssim\log\Log$ we have for all $e^{-o(N_\e)}\le r\ll1$ and $r_0\le r$ with $\e^{1/2}<r_0\ll\frac{N_\e}\Log$, for all $z\in\R^2$,
\begin{multline}\label{eq:lowerboundbis}
\qquad\frac12\int_{\tilde\B_{\e,R}^{r_0,r}}a\chi_R^z\Big(|\nabla u_\e-iu_\e N_\e \!\vi_\e\!|^2+\frac{a}{2\e^2}(1-|u_\e|^2)^2\Big)\\
\ge \frac{\Log}2\int_{\R^2} a\chi_R^z \nu_{\e,R}^{r_0}-o(N_\e^2).
\end{multline}
\item \emph{Number of vortices:} For $\e^{1/2}<r\ll1$,
\begin{align}\label{eq:numbvort}
\sup_z\int_{B_R(z)}|\nu_{\e,R}^r|~\lesssim~ N_\e,
\end{align}
and moreover in the regime $1\ll N_\e\ll\Log^{1/2}$ the measure $\nu_{\e,R}^r$ is nonnegative for all $e^{-o(1)\frac\Log{N_\e}}\le r<\bar r$.
\item \emph{Jacobian estimate:} For $\e^{1/2}<r\ll1$, for all $\gamma\in[0,1]$,
\begin{gather}
\sup_z\|\nu_{\e,R}^r-\tilde\mu_\e\|_{(C_c^\gamma(B_{R}(z)))^*}\lesssim r^\gamma N_\e+\e^{\gamma/2}N_\e\Log,\label{eq:jacobest}\\
\sup_z\|\mu_\e-\tilde\mu_\e\|_{(C_c^\gamma(B_R(z)))^*}\lesssim \e^\gamma N_\e\Log^{n+1},\label{eq:jacobestbis}
\end{gather}
hence in particular, for all $\gamma\in(0,1]$,
\begin{align}\label{eq:vortmass}
\sup_z\|\tilde\mu_\e\|_{(C_c^\gamma(B_{R}(z)))^*}~\simeq_\gamma~N_\e,\qquad \sup_z\|\mu_\e\|_{(C_c^\gamma(B_{R}(z)))^*}~\simeq_\gamma~ N_\e.
\end{align}
\item \emph{Excess energy estimate:} For all $\phi\in W^{1,\infty}(\R^2)$ supported in a ball of radius $R$,
\begin{multline}\label{eq:excessestim}
\qquad\int_{\R^2}\phi\Big(|\nabla u_\e-iu_\e N_\e\!\vi_\e\!|^2+\frac{a}{2\e^2}(1-|u_\e|^2)^2-\Log\mu_\e\Big)\\
\lesssim (\D_{\e,R}^*+o(N_\e^2))\|\phi\|_{W^{1,\infty}}.
\end{multline}
\item \emph{Energy outside small balls:} In the regime $N_\e\gg\log\Log$, we have for all $e^{-o(N_\e)}\le r<\bar r$ and $z\in\R^2$,
\begin{align}\label{eq:energyoutballs}
\int_{\R^2\setminus\B_{\e,R}^r}a\chi_R^z\Big(|\nabla u_\e-iu_\e N_\e \!\vi_\e\!|^2+\frac{a}{2\e^2}(1-|u_\e|^2)^2\Big)\le \D_{\e,R}^z+o(N_\e^2),
\end{align}
while in the regime $1\ll N_\e\lesssim\log\Log$ we have for all $r\ge e^{-o(N_\e)}$ and $r_0\le r$ with $\e^{1/2}<r_0\ll \frac{N_\e}\Log$, for all $z\in\R^2$,
\begin{align}\label{eq:energyoutballs2}
\int_{\R^2\setminus\tilde\B_{\e,R}^{r_0,r}}a\chi_R^z\Big(|\nabla u_\e-iu_\e N_\e \!\vi_\e\!|^2+\frac{a}{2\e^2}(1-|u_\e|^2)^2\Big)\le \D_{\e,R}^z+o(N_\e^2).
\end{align}
\item \emph{$\Ld^p$-estimate inside small balls:} In the regime $N_\e\gg\log\Log$, we have for all $\e^{1/2}< r<\bar r$ and $1\le p<2$,
\begin{align}\label{eq:Lpstruwe}
\sup_z\int_{\B_{\e,R}^r}\chi_R^z|\nabla u_\e-iu_\e N_\e\!\vi_\e\!|^p\lesssim_p o(N_\e^p),
\end{align}
while in the regime $1\ll N_\e\lesssim\log\Log$ we have for all $r>\e^{1/2}$ and $r_0\le r$ with $\e^{1/2}<r_0\ll \frac{N_\e}\Log$, for all $1\le p<2$,
\begin{equation}\label{eq:Lpstruwe2}
\sup_z\int_{\tilde \B_{\e,R}^{r_0,r}}\chi_R^z|\nabla u_\e-iu_\e N_\e\!\vi_\e\!|^p\lesssim_p o(N_\e^p).
\qedhere
\end{equation}
\end{enumerate}
\end{prop}

\begin{proof}
We split the proof into eight steps. The main work consists in checking that the assumptions imply the optimal bound on the energy $\Ec_{\e,R}^*\lesssim N_\e\Log$. The conclusion is obtained in Step~5 for the regime $\log\Log\lesssim N_\e\lesssim\Log$, but only in Step~7 for the complementary regime $1\ll N_\e\ll\log\Log$. The various other claims are finally deduced in Step~8.

\medskip
\noindent\step1 Rough a priori bound on the energy.

In this step, we prove $\Ec_{\e,R}^*\lesssim R^2\Log^2$, and hence by the choice of $R$ we deduce $\Ec_{\e,R}^*\lesssim\Log^m$ for some $m\ge4$.
Decomposing $\mu_\e=N_\e\curl\!\vi_\e+\curl(j_\e-N_\e\!\vi_\e)$, the assumption $\D_{\e,R}^*\lesssim N_\e^2$ yields for all $z\in \R^2$,
\begin{align}\label{eq:decompEepsR}
\Ec_{\e,R}^z&\le \D_{\e,R}^*+\frac\Log2\int_{\R^2} a\chi_{R}^z\mu_\e\nonumber\\
&\lesssim N_\e^2+{N_\e\Log}\int_{\R^2} a\chi_{R}^z|\curl\!\vi_\e\!|+\Log\int_{\R^2}|\nabla (a\chi_{R}^z)||j_\e-N_\e\!\vi_\e\!|.
\end{align}
Using the pointwise estimate of Lemma~\ref{lem:pointest} for $j_\e-N_\e\!\vi_\e$,
using $|\nabla(a\chi_{R}^z)|\lesssim \mathds1_{B_{2R}(z)}$, $\|\curl\!\vi_\e\!\|_{\Ld^1}\lesssim1$, and $\|\!\vi_\e\!\|_{\Ld^\infty}\lesssim1$, we obtain
\begin{align*}
\Ec_{\e,R}^z&\lesssim \Log^2+\Log\Big(\int_{B_{2R}(z)}(1-|u_\e|^2)^2\Big)^{1/2}\Big(\int_{B_{2R}(z)}|\nabla u_\e-iu_\e N_\e\!\vi_\e\!|^2\Big)^{1/2}\\
&\qquad+R\Log\Big(\int_{B_{2R}(z)}|\nabla u_\e-iu_\e N_\e\!\vi_\e\!|^2\Big)^{1/2}+RN_\e\Log\Big(\int_{B_{2R}(z)}(1-|u_\e|^2)^2\Big)^{1/2}\\
&\lesssim \Log^2+\e\Log\Ec_{\e,R}^*+R\Log(\Ec_{\e,R}^*)^{1/2}.
\end{align*}
Taking the supremum over $z$, and absorbing $\Ec_{\e,R}^*$ into the left-hand side, the result follows.

\medskip
\noindent\step2 Application of Lemma~\ref{lem:ballconstr}.
\nopagebreak

The result of Step~1 yields in particular $\log\Ec_{\e,R}^*\ll\Log$, which allows to apply Lemma~\ref{lem:ballconstr}. For fixed $r\in(\e^{1/2},\bar r)$, let $\B_{\e,R}^r=\biguplus_jB^j$ denote the union of disjoint closed balls given by Lemma~\ref{lem:ballconstr}, and let $\nu_{\e,R}^r$ denote the associated point-vortex measure.
Using Lemma~\ref{lem:ballconstr}(ii) in the form
\begin{align}\label{eq:estnumbervort}
\int_{B_R(z)}|\nu_{\e,R}^r|=\sum_{j:y_j\in B_R(z)}|d_j|\lesssim N_\e+\frac{\Ec_{\e,R}^*}\Log,
\end{align}
Lemma~\ref{lem:ballconstr}(i) gives,
for all $\phi\in W^{1,\infty}(\R^2)$ supported in a ball of radius~$R$, with $\phi\ge0$,
\begin{multline}\label{eq:lowerboundclaim1b}
\frac12\int_{\B_{\e,R}^r}\phi\Big(|\nabla u_\e-iu_\e N_\e \!\vi_\e\!|^2+\frac{a}{2\e^2}(1-|u_\e|^2)^2\Big)\\
\ge \frac{\Log}2\int_{\R^2} \phi |\nu_{\e,R}^r|-O(r\Ec_{\e,R}^*)\|\nabla\phi\|_{\Ld^{\infty}}\\
-O\bigg(r^2N_\e^2+|\!\log r|\Big(N_\e+\frac{\Ec_{\e,R}^*}\Log\Big)+\Big(N_\e+\frac{\Ec_{\e,R}^*}\Log\Big)\log\Big(2+\frac{\Ec_{\e,R}^*}\Log\Big)\bigg)\|\phi\|_{\Ld^{\infty}}.
\end{multline}
We now prove the following consequence of these  bounds, for all $z\in\R^2$,
\begin{multline}\label{eq:energyR2minB}
\int_{\R^2\setminus\B_{\e,R}^r}a\chi_R^z\Big(|\nabla u_\e-iu_\e N_\e \!\vi_\e\!|^2+\frac{a}{2\e^2}(1-|u_\e|^2)^2\Big)\\
\le \D_{\e,R}^z+O\bigg(r\Ec_{\e,R}^*+(|\!\log r|+r\Log)\Big(N_\e+\frac{\Ec_{\e,R}^*}\Log\Big)\\
+\Big(N_\e+\frac{\Ec_{\e,R}^*}\Log\Big)\log\Big(2+\frac{\Ec_{\e,R}^*}\Log\Big)\bigg).
\end{multline}
First, the lower bound~\eqref{eq:lowerboundclaim1b} applied to  $\phi=a\chi_R^z$ is rewritten as follows,
\begin{multline*}
\frac12\int_{\R^2\setminus\B_{\e,R}^r}a\chi_R^z\Big(|\nabla u_\e-iu_\e N_\e \!\vi_\e\!|^2+\frac{a}{2\e^2}(1-|u_\e|^2)^2\Big)\\
\le T_{\e,R}^{r,z}+O\bigg(r\Ec_{\e,R}^*+r^2N_\e^2+|\!\log r|\Big(N_\e+\frac{\Ec_{\e,R}^*}\Log\Big)+\Big(N_\e+\frac{\Ec_{\e,R}^*}\Log\Big)\log\Big(2+\frac{\Ec_{\e,R}^*}\Log\Big)\bigg),
\end{multline*}
where we have set
\[T_{\e,R}^{r,z}:=\frac12\int_{\R^2} a\chi_R^z\Big(|\nabla u_\e-iu_\e N_\e \!\vi_\e\!|^2+\frac{a}{2\e^2}(1-|u_\e|^2)^2-\Log\nu_{\e,R}^r\Big).\]
If $\nu_{\e,R}^r$ was replaced by $\mu_\e$ in this last expression, we would recognize the definition of the excess $\D_{\e,R}^z$, and the result~\eqref{eq:energyR2minB} would follow. Hence, it only remains to check that for all $\phi\in W^{1,\infty}(\R^2)$ supported in a ball of radius $R$,
\begin{align}\label{eq:munuexch}
\Big|\int_{\R^2} \phi(\mu_\e-\nu_{\e,R}^r)\Big|\lesssim r\Big(N_\e+\frac{\Ec_{\e,R}^*}\Log\Big)\|\phi\|_{W^{1,\infty}}+\e^{1/3}\|\phi\|_{W^{1,\infty}}.
\end{align}
Using the result of Step~1 in the form $\e^{1/6}\Ec_{\e,R}^*\lesssim1$, Lemma~\ref{lem:ballconstr}(iii) with $\gamma=1$ yields
\begin{align*}
\Big|\int_{\R^2} \phi(\tilde\mu_\e-\nu_{\e,R}^r)\Big|\lesssim r\Big(N_\e+\frac{\Ec_{\e,R}^*}\Log\Big)\|\phi\|_{W^{1,\infty}}+\e^{1/3}\|\phi\|_{W^{1,\infty}}.
\end{align*}
It remains to replace $\tilde \mu_\e$ by $\mu_\e$ in this estimate. By definition~\eqref{eq:defmutilde}, with $\|\!\vi_\e\!\|_{\Ld^\infty}\lesssim1$ and $|\nabla \phi|\le\mathds1_{B_{R}(z)}\|\phi\|_{W^{1,\infty}}$, and using the result of Step~1 in the form $\e^{2/3} RN_\e(\Ec_{\e,R}^*)^{1/2}\lesssim1$, we find
\begin{align}
\Big|\int_{\R^2} \phi(\tilde\mu_\e-\mu_\e)\Big|&\le N_\e\int_{B_{R}(z)}|\nabla\phi||\!\vi_\e\!||1-|u_\e|^2|\nonumber\\
&\lesssim RN_\e\|\phi\|_{W^{1,\infty}}\Big(\int_{B_{R}(z)}(1-|u_\e|^2)^2\Big)^{1/2}\nonumber\\
&\lesssim \e RN_\e(\Ec_{\e,R}^*)^{1/2}\|\phi\|_{W^{1,\infty}}\lesssim\e^{1/3}\|\phi\|_{W^{1,\infty}},\label{eq:boundmumutildeequ}
\end{align}
and the result~\eqref{eq:munuexch} follows.

\medskip
\noindent\step3 Energy and number of vortices.

In this step, we show that~\eqref{eq:estnumbervort} is essentially an equality, in the following sense: for all $\e^{1/2}<r\ll1$,
\begin{align}\label{eq:energyvort}
\sup_z\int_{\R^2}\chi_R^z|\nu_{\e,R}^r|\lesssim N_\e+\frac{\Ec_{\e,R}^*}\Log \lesssim N_\e+\sup_z\int_{\R^2} \chi_R^z|\nu_{\e,R}^r|.
\end{align}
The lower bound follows from~\eqref{eq:estnumbervort}. We turn to the upper bound.
Since the energy excess satisfies $\D_{\e,R}^z\lesssim N_\e^2$, we deduce from~\eqref{eq:munuexch},
\begin{align}\label{eq:energyvort-pre}
\Ec_{\e,R}^z&\le \D_{\e,R}^z+\frac\Log2\int_{\R^2} a\chi_R^z\mu_\e\nonumber\\
&\le \frac\Log2\int_{\R^2} a\chi_R^z\nu_{\e,R}^r+O\bigg(N_\e^2+r\Log\Big(N_\e+\frac{\Ec_{\e,R}^*}\Log\Big)\bigg).
\end{align}
Taking the supremum in $z$, and absorbing $\Ec_{\e,R}^*$ in the left-hand side with $r\ll1$,
the upper bound in~\eqref{eq:energyvort} follows.

\medskip
\noindent\step4 Bound on the total variation of the vorticity.

In this step, we prove that for all $e^{-o(\Log)}<r\ll1$,
\begin{align}\label{eq:signrough}
\sup_z\int_{\R^2}\chi_R^z|\nu_{\e,R}^r|\le (1+o(1))\sup_z\int_{\R^2}\chi_R^z\nu_{\e,R}^r+O(N_\e).
\end{align}
This result is used in Step~5 below in order to replace $\int a\chi_R^z\nu_{\e,R}^r$ (resp. $\int a\chi_R^z\mu_{\e}$) by $\int \chi_R^z\nu_{\e,R}^r$ (resp. $\int \chi_R^z\mu_{\e}$), which is crucial if we want to avoid integrability assumptions on $\nabla h$, as we do here.

The lower bound~\eqref{eq:lowerboundclaim1b} of Step~2 with $\phi=a\chi_R^y$ yields for all $y\in\R^2$, using the upper bound in~\eqref{eq:energyvort} to replace the energy $\Ec_{\e,R}^*$ in the error terms,
\begin{align*}
\Ec_{\e,R}^y&\ge \frac12\int_{\B_{\e,R}^r}a\chi_R^y\Big(|\nabla u_\e-iu_\e N_\e \!\vi_\e\!|^2+\frac{a}{2\e^2}(1-|u_\e|^2)^2\Big)\\
&\ge \frac{\Log}2\int_{\R^2} a\chi_R^y |\nu_{\e,R}^r|-O\bigg((|\!\log r|+r\Log) \Big(N_\e+\sup_z\int_{\R^2} \chi_R^z|\nu_{\e,R}^r|\Big)\\
&\hspace{5cm}+\Big(N_\e+\sup_z\int_{\R^2} \chi_R^z|\nu_{\e,R}^r|\Big)\log\Big(2+\frac{\Ec_{\e,R}^*}\Log\Big)\bigg).
\end{align*}
For $e^{-o(\Log)}<r\ll1$, using the result of Step~1 in the form $\log\Ec_{\e,R}^*\ll\Log$, we obtain for all $y\in\R^2$,
\begin{align}\label{eq:preboundinfenergy}
\Ec_{\e,R}^y\ge \frac{\Log}2\int_{\R^2} a\chi_R^y |\nu_{\e,R}^r|-o(\Log)\sup_z\int_{\R^2} \chi_R^z|\nu_{\e,R}^r|-o(N_\e\Log).
\end{align}
On the other hand, the upper bound~\eqref{eq:energyvort-pre} yields
\begin{align}\label{eq:preboundsupenergy}
\Ec_{\e,R}^y\le\frac\Log2\int_{\R^2} a\chi_R^y\nu_{\e,R}^r+O(N_\e\Log)+o(1)\Ec_{\e,R}^*,
\end{align}
and thus, taking the supremum over $y$ and absorbing $\Ec_{\e,R}^*$ in the left-hand side,
\begin{align*}
\Ec_{\e,R}^*\le\frac\Log2(1+o(1))\sup_z\int_{\R^2} a\chi_R^z|\nu_{\e,R}^r|+O(N_\e\Log),
\end{align*}
so that~\eqref{eq:preboundsupenergy} takes the form, for all $y\in\R^2$,
\begin{align*}
\Ec_{\e,R}^y\le\frac\Log2\int_{\R^2} a\chi_R^y\nu_{\e,R}^r+O(N_\e\Log)+o(\Log)\sup_z\int_{\R^2} \chi_R^z|\nu_{\e,R}^r|.
\end{align*}
Combining this with~\eqref{eq:preboundinfenergy}, dividing both sides by $\frac12\Log$, and taking the supremum over~$y$, we find
\begin{align*}
\sup_z\int_{\R^2} \chi_R^z (\nu_{\e,R}^r)^-\lesssim\sup_z\int_{\R^2} a\chi_R^z (|\nu_{\e,R}^r|-\nu_{\e,R}^r)\le O(N_\e)+o(1)\sup_z\int \chi_R^z|\nu_{\e,R}^r|,
\end{align*}
hence
\begin{multline*}
\sup_z\int_{\R^2} \chi_R^z |\nu_{\e,R}^r|=\sup_z\int_{\R^2} \chi_R^z (\nu_{\e,R}^r+2(\nu_{\e,R}^r)^-)\\
\le \sup_z\int_{\R^2} \chi_R^z \nu_{\e,R}^r+O(N_\e)+o(1)\sup_z\int_{\R^2} a\chi_R^z|\nu_{\e,R}^r|,
\end{multline*}
and the result~\eqref{eq:signrough} follows after absorbing the last right-hand side term.

\medskip
\noindent\step5 Refined bound on the energy.

In this step, we prove $\Ec_{\e,R}^*\lesssim(N_\e+\log\Log)\Log$. By~\eqref{eq:estnumbervort} this implies in particular $\sup_z\int_{\R^2}\chi_R^z|\nu_{\e,R}^r|\lesssim N_\e+\log\Log$. In the regime $N_\e\gtrsim\log\Log$, these bounds are already the optimal ones. The strongly dilute regime $1\ll N_\e\ll\log\Log$ is treated in Steps~6--7.

Let $e^{-o(\Log)}<r\ll1$ to be suitably chosen later. Using~\eqref{eq:munuexch}, the bound on the energy excess $\D_{\e,R}^*\lesssim N_\e^2$ yields for all $z\in R\Z^2$,
\begin{align*}
\Ec_{\e,R}^z&\le \D_{\e,R}^z+\frac\Log2\int_{\R^2} a\chi_R^z\mu_\e\lesssim N_\e^2+\Log\int_{\R^2} \chi_R^z|\nu_{\e,R}^r|+r(N_\e\Log+\Ec_{\e,R}^*),
\end{align*}
and hence, using the result~\eqref{eq:signrough} of Step~4,
\begin{align*}
\Ec_{\e,R}^*&\lesssim N_\e\Log+\Log\sup_z\int_{\R^2} \chi_R^z\nu_{\e,R}^r+r\Ec_{\e,R}^*.
\end{align*}
Using~\eqref{eq:munuexch} again, and absorbing $\Ec_{\e,R}^*$ in the left-hand side with $r\ll1$, this takes the form
\begin{align}\label{eq:boundenergypreconcl}
\Ec_{\e,R}^*&\lesssim N_\e\Log+\Log\sup_z\int_{\R^2} \chi_R^z\mu_{\e}.
\end{align}
It remains to estimate $\int_{\R^2}\chi_R^z\mu_\e$.
Decomposing $\mu_\e=N_\e\curl\!\vi_\e+\curl(j_\e-N_\e\!\vi_\e)$, using the pointwise estimate of Lemma~\ref{lem:pointest} for $j_\e-N_\e\!\vi_\e$, using $|\nabla \chi_{R}^z|\lesssim R^{-1}\mathds1_{B_{2R}(z)}$, $\|\nabla\chi_R^z\|_{\Ld^2}\lesssim1$, $\|\curl\!\vi_\e\!\|_{\Ld^1}\lesssim1$, $\|\!\vi_\e\!\|_{\Ld^\infty}\lesssim1$, and using the result of Step~1 in the form $\e \Ec_{\e,R}^*\lesssim1$, we find
\begin{align*}
\int_{\R^2} \chi_R^z\mu_\e&=N_\e\int_{\R^2}\chi_R^z\curl\!\vi_\e-\int_{\R^2}\nabla^\bot\chi_R^z\cdot(j_\e-N_\e\!\vi_\e)\\
&\lesssim N_\e+\int_{\R^2}|\nabla\chi_R^z||\nabla u_\e-iu_\e N_\e\!\vi_\e\!|.
\end{align*}
Regarding the last integral, we distinguish between the contributions inside and outside the balls $\B_{\e,R}^r$, with $|\nabla\chi_R^z|\lesssim R^{-1}\mathds1_{B_{2R}(z)}\le R^{-1}\chi_{2R}^z$, $\|\nabla\chi_R^z\|_{\Ld^2}\lesssim1$, and $|B_{2R}(z)\cap\B_{\e,R}^r|\lesssim r^2$,
\begin{align}
\int_{\R^2} \chi_R^z\mu_\e&\lesssim N_\e+\int_{\R^2\setminus\B_{\e,R}^r}|\nabla\chi_R^z||\nabla u_\e-iu_\e N_\e\!\vi_\e\!|+R^{-1}\int_{B_{2R}(z)\cap\B_{\e,R}^r}|\nabla u_\e-iu_\e N_\e\!\vi_\e\!|\nonumber\\
&\lesssim N_\e+\Big(\int_{\R^2\setminus\B_{\e,R}^r}\chi_{2R}^z|\nabla u_\e-iu_\e N_\e\!\vi_\e\!|^2\Big)^{1/2}\nonumber\\
&\hspace{5cm}+rR^{-1}\Big(\int_{B_{2R}(z)}|\nabla u_\e-iu_\e N_\e\!\vi_\e\!|^2\Big)^{1/2}.\label{eq:boundvortpreconcl}
\end{align}
Estimating the last right-hand side term by $rR^{-1}(\Ec_{\e,R}^*)^{1/2}$, using~\eqref{eq:energyR2minB} to estimate the first, using the bound on the energy excess $\D_{\e,R}^*\lesssim N_\e^2$, and noting that $k^{1/2}\log^{1/2}(2+k)\ll k$ holds for $k\gg1$, we obtain
\begin{eqnarray*}
\lefteqn{\int_{\R^2} \chi_R^z\mu_\e\lesssim N_\e+(\D_{\e,R}^*)^{1/2}+rR^{-1}(\Ec_{\e,R}^*)^{1/2}+r^{1/2}(N_\e\Log+\Ec_{\e,R}^*)^{1/2}}\\
&&\hspace{4cm}+\Big(N_\e+\frac{\Ec_{\e,R}^*}\Log\Big)^{1/2}\bigg(|\!\log r|+\log\Big(2+\frac{\Ec_{\e,R}^*}\Log\Big)\bigg)^{1/2}\\
&\lesssim& N_\e+ r^{1/2}(N_\e\Log)^{1/2}+ r^{1/2}(\Ec_{\e,R}^*)^{1/2}+o(1)\frac{\Ec_{\e,R}^*}\Log+|\!\log r|^{1/2}\Big(N_\e+\frac{\Ec_{\e,R}^*}\Log\Big)^{1/2}.
\end{eqnarray*}
Combining this with~\eqref{eq:boundenergypreconcl} yields
\begin{multline*}
\frac{\Ec_{\e,R}^*}\Log\lesssim N_\e+ r^{1/2}(N_\e\Log)^{1/2}\\
+ r^{1/2}(\Ec_{\e,R}^*)^{1/2}+o(1)\frac{\Ec_{\e,R}^*}\Log+|\!\log r|^{1/2}\Big(N_\e+\frac{\Ec_{\e,R}^*}\Log\Big)^{1/2},
\end{multline*}
and hence, 
\begin{align*}
\frac{\Ec_{\e,R}^*}\Log&\lesssim N_\e+|\!\log r|+ r^{1/2}\Log.
\end{align*}
The result follows from the choice $r=\Log^{-2}$.

\medskip
\noindent\step6 Refined lower bound in the strongly dilute regime.

In this step, we study the regime $1\ll N_\e\lesssim\log\Log$, for which the result of Step~5 is not optimal. More precisely, we consider the whole regime $1\ll N_\e\lesssim\Log$ and we show the following: for all $r_0\in(\e^{1/2},\bar r)$ and $r\ge r_0$, there exists a locally finite union of disjoint closed balls $\tilde\B_{\e,R}^{r_0,r}$, monotone in $r$, covering the set $\{x:||u_\e(x)|-1|\ge\Log^{-1}\}$, such that for all $z$ the sum of the radii of the balls intersecting $B_R(z)$ is bounded by $Cr$, and such that for all $\e>0$ small enough, and all $r_0\le r$ satisfying
\begin{align}\label{eq:choicerr0}
\e^{1/2}<r_0\ll \frac{N_\e}{\Log}\frac{N_\e}{N_\e+\log\Log},\qquad e^{-o(N_\e)}\le r\ll1,
\end{align}
we have for all $z\in \R^2$,
\begin{multline}\label{eq:lowerboundref}
\frac12\int_{\tilde\B_{\e,R}^{r_0,r}}a\chi_R^z\Big(|\nabla u_\e-iu_\e N_\e \!\vi_\e\!|^2+\frac{a}{2\e^2}(1-|u_\e|^2)^2\Big)\\
\ge \frac{\Log}2\int_{\R^2} a\chi_R^z \nu_{\e,R}^{r_0}-o(1)\Big(\frac{\Ec_{\e,R}^*}\Log\Big)^2-o(N_\e^2).
\end{multline}
We split the proof into three further substeps.

\medskip
\noindent\substep{6.1} Enlarged balls: in this step, given some fixed $r_0\in(\e^{1/2},\bar r)$, we construct the enlarged collections of balls $\tilde\B_{\e,R}^{r_0,r}$ for $r\ge r_0$.

According to \cite[Proposition~4.8]{SS-book}, and using the energy estimate of Step~5, we have
\[\mathcal{H}^1(\{ x \in B_R(z), ||u_\e(x)|-1|\ge \Log^{-1}\})\le C \e \Log^2\Ec_{\e,R}^*\le C\e\Log^4,\]
where $\mathcal H^1$ denotes the $1$-dimensional Haussdorff measure. From~\cite[Section~4.4.1]{SS-book} and~\cite[Section~2.2]{SS-11}, it follows that we may cover the set $\{ x: ||u_\e(x)|-1|\ge \Log^{-1}\} $ by a locally finite union of disjoint closed balls such that for all $z$ the sum of the radii of the balls intersecting $B_R(z)$ is bounded by $C \e \Log^4$.
We then combine this collection of balls with the collection $\B_{\e, R}^{r_0}$. Inductively merging as in \cite[Lemma 4.1]{SS-book} any two such balls that intersect into a ball with the same total radius, we obtain a new collection 
$\tilde\B_{\e, R}^{r_0}$ of disjoint closed balls that cover the set $\{x:||u_\e(x)|-1|\ge\Log^{-1}\}$, and such that for all $z$ the sum of the radii of the balls intersecting $B_R(z)$ is bounded by $r_0 + C \e \Log^6\le Cr_0$.

Let us now grow the balls of this new collection $\tilde \B_{\e,R}^{r_0}$ following Sandier's ball construction, as described e.g.\@ in~\cite[Theorem 4.2]{SS-book}. This consists in growing simultaneously all the balls  keeping their centers fixed and multiplying their radius by the same factor $t$. If some balls touch at some point during the growth, the corresponding balls are merged into one larger ball containing the previous ones and with the same total radius. This construction ensures that the balls always remain disjoint. Stopping the growth process at some value of the factor $t$, and setting $r=tr_0$, we denote by $\tilde\B_{\e,R}^{r_0,r}$ the corresponding locally finite collection of disjoint closed balls. By construction,
for all $z$, the sum of the radii of the balls that intersect $B_R(z)$ is bounded by $Ct(r_0+C\e\Log^6)\le Cr$. Note that by construction $\B_{\e,R}^{r_0}\subset \tilde\B_{\e,R}^{r_0}=\tilde\B_{\e,R}^{r_0,r_0}$.

\medskip
\noindent\substep{6.2} Preliminary estimate.

According to \cite[Lemma 3.2]{S-Tice-08} (applied with  $c=d$ and $\lambda = 1$), we have, for any $\Sb^1$-valued map $v$  with  degree $d$ on a generic ball $B$ of radius $r$, and for any vector field $A:\partial B\to\R^2$,
\begin{equation*}
\frac12 \int_{\partial B} |\nab v - iv A |^2 + \frac12   \int_B |\curl A|^2\ge \frac{\pi d^2}{r} - \frac{\pi d^2}2 + \frac12 \int_{\partial B} \left|\nab v- iv A - i vd \frac{\tau }{r}\right|^2,
\end{equation*}
where $\tau $ denotes the unit tangent to the circle $\partial B$.
Applying it to $v=\frac{u_\e}{|u_\e|} $ and $A=N_\e\! \vi_\e$, and noting that $|\nabla u_\e-iu_\e F|^2=|u_\e|^2|\nabla\frac{u_\e}{|u_\e|}-i\frac{u_\e}{|u_\e|}F|^2+|\nabla|u_\e||^2$ holds for any real-valued vector field $F$, we obtain the following improved lower bound on annuli: if the condition $||u_\e|-1|\le\Log^{-1}$ holds on $\partial B$, then we have
\begin{multline}\label{stice}
(1+O(\Log^{-1}))\frac12 \int_{\partial B} |\nab u_\e - iu_\e N_\e \!\vi_\e\!|^2   + \frac12   N_\e^2 \int_B |\curl\!\vi_\e\!|^2
\\  \ge  \frac{\pi d^2}{r} -\frac{\pi d^2}2 + \frac12 (1-O(\Log^{-1})) \int_{\partial B} \Big|\nab u_\e- iu_\e N_\e \!\vi_\e - i u_\e d \frac{\tau }{r}\Big|^2.
\end{multline}

\medskip
\noindent\substep{6.3} Proof of~\eqref{eq:lowerboundref}.

Let $r_0>0$ be chosen as in~\eqref{eq:choicerr0}. We start from Lemma~\ref{lem:ballconstr}(i) with $\phi=a\chi_R^z$, combined with the refined energy estimate of Step~5 and the choice of~$r_0$, which yields
\begin{multline}\label{eq:lowerboundr0}
\frac12\int_{\B_{\e,R}^{r_0}}a\chi_R^z\Big(|\nabla u_\e-iu_\e N_\e \!\vi_\e\!|^2+\frac{a}{2\e^2}(1-|u_\e|^2)^2\Big)\\
\ge \frac{\log(\tfrac{r_0}\e)}2\int a\chi_R^z |\nu_{\e,R}^{r_0}|-o(N_\e^2)-C\Big(N_\e+\frac{\Ec_{\e,R}^*}\Log\Big)\log\Big(2+\frac{\Ec_{\e,R}^*}\Log\Big).
\end{multline}
We next need to show that this lower bound for the energy is essentially maintained during the ball growth and merging process, hence holds as well for the collections $\tilde\B_{\e,R}^{r_0,r}$ with $r>r_0$.

Assume that some ball $B=\bar B(y,s)$  gets grown into $B'=\bar B(y,t s)$ without merging, for some $t\ge1$, and assume that $B'\setminus B$ does not intersect $\tilde\B_{\e,R}^{r_0}$, so that $||u_\e|-1|\le\Log^{-1}$ holds on $B'\setminus B$. Let $d$ denote the degree of $B$ (hence of~$B'$). Since by assumption we have
\begin{align}\label{eq:Lip-achi}
|a(x)\chi_R^z(x)-a(y)\chi_R^z(y)|&\le \chi_R^z(y)|a(x)-a(y)|+a(x)|\chi_R^z(x)-\chi_R^z(y)|\nonumber\\
&\le C\big(R^{-1}+\chi_R^z(y)\big)|x-y|,
\end{align}
we may write
\begin{multline*}
\frac12\int_{B' \setminus B} a\chi_R^z\Big(|\nabla u_\e-iu_\e N_\e\!\vi_\e\!|^2+\frac a{2\e^2}(1-|u_\e|^2)^2\Big)\\
\ge\frac{a(y)\chi_R^z(y)}2 \int_{B'\setminus B} |\nabla u_\e-iu_\e N_\e\!\vi_\e\!|^2
-CR^{-1}\int_{B'\setminus B}|\cdot-y| |\nabla u_\e-iu_\e N_\e\!\vi_\e\!|^2\\
-C\chi_R^z(y)\int_{B'\setminus B}|\cdot-y| |\nabla u_\e-iu_\e N_\e\!\vi_\e\!|^2.
\end{multline*}
Using that $|u_\e|\le1+\Log^{-1}$ holds on $B'\setminus B$, the last right-hand side term above is estimated as follows,
\begin{multline}\label{eq:upbound-radcomp}
\int_{B'\setminus B}|\cdot-y| |\nabla u_\e-iu_\e N_\e\!\vi_\e\!|^2\\
\le 2 \int_{B'\setminus B}|\cdot-y|\,|u_\e|^2\Big| \frac{\tau d}{|\cdot-y|}\Big|^2+ 2\int_{B'\setminus B}|\cdot-y|~ \Big|\nabla u_\e-iu_\e N_\e\!\vi_\e - iu_\e\frac{\tau d}{|\cdot-y|}\Big|^2\\
\le Cd^2(t-1)s + 2ts\int_{B'\setminus B} \Big|\nabla u_\e-iu_\e N_\e\!\vi_\e- iu_\e \frac{\tau d}{|\cdot-y|}\Big|^2,
\end{multline}
where $\tau(x)=\frac{(x-y)^\bot}{|x-y|}$ is the unit tangent to the circle centered at $y$, and we may then deduce
\begin{multline}\label{eq:prelowerb}
\frac12\int_{B' \setminus B} a\chi_R^z\Big(|\nabla u_\e-iu_\e N_\e\!\vi_\e\!|^2+\frac a{2\e^2}(1-|u_\e|^2)^2\Big)\\
\ge\frac{a(y)\chi_R^z(y)}2\int_{B'\setminus B} |\nabla u_\e-iu_\e N_\e\!\vi_\e\!|^2
-CtsR^{-1}\int_{B'\setminus B} |\nabla u_\e-iu_\e N_\e\!\vi_\e\!|^2\\
-Cd^2 (t-1)s\chi_R^z(y)
-Cts\chi_R^z(y)\int_{B'\setminus B} \Big|\nabla u_\e-iu_\e N_\e\!\vi_\e - iu_\e \frac{\tau d}{|\cdot-y|}\Big|^2.
\end{multline}
Again using that $||u_\e|-1|\le\Log^{-1}$ holds on $B'\setminus B$, the estimate~\eqref{stice} on the ball $B(y,\rho)$ for $\rho$ integrated between $s$ and $ts$ takes the form
\begin{multline*}
(1+C\Log^{-1}) \frac12\int_{B'\setminus B} |\nab u_\e - iu_\e N_\e \!\vi_\e\!|^2\\
\ge \pi d^2\log t -\frac{\pi}2d^2(t-1)s-\frac12N_\e^2(t-1)s \int_{B'} |\curl\!\vi_\e\!|^2\\
+ (1-C\Log^{-1})\frac12 \int_{B'\setminus B} \Big|\nab u_\e- iu_\e N_\e \!\vi_\e - i u_\e \frac{\tau d}{|\cdot-y|}\Big|^2.
\end{multline*}
Combining this with~\eqref{eq:prelowerb}, we are led to
\begin{multline}\label{eq:stice-int}
(1+C\Log^{-1})\frac12\int_{B' \setminus B} a\chi_R^z\Big(|\nabla u_\e-iu_\e N_\e\!\vi_\e\!|^2+\frac a{2\e^2}(1-|u_\e|^2)^2\Big)\\
\ge a(y)\chi_R^z(y)\pi d^2\log t-C(t-1)s\Big(d^2+N_\e^2\int_{B'} |\curl\!\vi_\e\!|^2\Big)\\
-CtsR^{-1}\int_{B'\setminus B} |\nabla u_\e-iu_\e N_\e\!\vi_\e\!|^2\\
+\Big(\frac{a(y)}2 (1-C\Log^{-1}) -Cts\Big)\chi_R^z(y)\int_{B'\setminus B} \Big|\nabla u_\e-iu_\e N_\e\!\vi_\e - iu_\e \frac{\tau d}{|\cdot-y|}\Big|^2.
\end{multline}
For $\e$ small enough and $ts\le\min\{1,\frac1{4C}\inf a\}=:\tilde r$ (note that by assumption $\tilde r\simeq1$), the last right-hand side term is nonnegative, so that we conclude
\begin{multline}
{(1+C\Log^{-1})\frac12\int_{B' \setminus B} a\chi_R^z\Big(|\nabla u_\e-iu_\e N_\e\!\vi_\e\!|^2+\frac a{2\e^2}(1-|u_\e|^2)^2\Big)}\\
\ge a(y)\chi_R^z(y)\pi d^2\log t-C(t-1)s(d^2+N_\e^2)\\
-CtsR^{-1}\int_{B'\setminus B} |\nabla u_\e-iu_\e N_\e\!\vi_\e\!|^2.\label{eq:lowerboundinflate}
\end{multline}

If the ball $B=\bar B(y,s)$ belongs to the collection $\tilde\B_{\e,R}^{r_0,r}$ for some $r\ge r_0$, only a finite number of balls of the collection $\B_{\e,R}^{r_0}$ are included in the ball $B$. Denote them by $B^j=\bar B(y_j,s_j)$, $j=1,\ldots,k$. By definition, the degree $d$ of $B$ is then equal to $d=\sum_jd_j$, where~$d_j$ denotes the degree of $B^j$.
We may then write
\begin{align*}
a(y)\chi_R^z(y)d^2\ge a(y)\chi_R^z(y)\sum_jd_j
&\ge \sum_ja(y_j)\chi_R^z(y_j)d_j-C\sum_j|d_j||y-y_j|\mathds1_{B_{2R}(z)}(y_j)\\
&\ge \sum_ja(y_j)\chi_R^z(y_j)d_j-Cs\sum_j|d_j|\mathds1_{B_{2R}(z)}(y_j),
\end{align*}
and hence, in terms of the point-vortex measure $\nu_{\e,R}^{r_0}$,
\begin{align}\label{eq:usealower}
a(y)\chi_R^z(y)d^2\ge \frac1{2\pi}\int_Ba\chi_R^z\nu_{\e,R}^{r_0}-Cs\int_{B_{2R}(z)}|\nu_{\e,R}^{r_0}|.
\end{align}
Therefore, if the ball $B=\bar B(y,s)$ belongs to the collection $\tilde\B_{\e,R}^{r_0,r}$ for some $r\ge r_0$ and gets grown without merging into a ball $B'=\bar B(y,t s)$ for some $t\ge1$ with $ts\le \tilde r$, then combining~\eqref{eq:lowerboundinflate} and~\eqref{eq:usealower} yields
\begin{multline*}
(1+C\Log^{-1})\frac12\int_{B' \setminus B} a\chi_R^z\Big(|\nabla u_\e-iu_\e N_\e\!\vi_\e\!|^2+\frac a{2\e^2}(1-|u_\e|^2)^2\Big)\\
\ge \frac{\log t}2\int_B a\chi_R^z\nu_{\e,R}^{r_0}-Cs\log t\int_{B_{2R}(z)}|\nu_{\e,R}^{r_0}|-C(t-1)s\Big(N_\e+\int_{B_{2R}(z)}|\nu_{\e,R}^{r_0}|\Big)^2\\
-CtsR^{-1}\int_{B'\setminus B} |\nabla u_\e-iu_\e N_\e\!\vi_\e\!|^2,
\end{multline*}
hence, using Lemma~\ref{lem:ballconstr}(ii) and the inequality $|\!\log t|\le t-1$ for $t\ge1$,
\begin{multline*}
\frac12\int_{B' \setminus B} a\chi_R^z\Big(|\nabla u_\e-iu_\e N_\e\!\vi_\e\!|^2+\frac a{2\e^2}(1-|u_\e|^2)^2\Big)\\
\ge \frac{\log t}2\int_B a\chi_R^z\nu_{\e,R}^{r_0}-C(t-1)s\Big(N_\e+\frac{\Ec_{\e,R}^*}\Log\Big)^2
-CtsR^{-1}\int_{B'\setminus B} |\nabla u_\e-iu_\e N_\e\!\vi_\e\!|^2.
\end{multline*}

By construction of the ball growth and merging process, this easily implies the following: if a ball $B=\bar B(y_B,s_B)$ belongs to the collection $\tilde\B_{\e,R}^{r_0,r}$ for some $r_0\le r\le \tilde r$, then we have
\begin{multline*}
\frac12\int_{B\setminus \tilde\B_{\e,R}^{r_0}} a\chi_R^z\Big(|\nabla u_\e-iu_\e N_\e\!\vi_\e\!|^2+\frac a{2\e^2}(1-|u_\e|^2)^2\Big)\\
\ge \frac{\log(\tfrac r{r_0})}2\int_B a\chi_R^z\nu_{\e,R}^{r_0}-Cs_B\Big(N_\e+\frac{\Ec_{\e,R}^*}\Log\Big)^2-Cs_BR^{-1}\int_{B\setminus \tilde\B_{\e,R}^{r_0}} |\nabla u_\e-iu_\e N_\e\!\vi_\e\!|^2,
\end{multline*}
hence, using the choice $R\gtrsim\Log$,
\begin{multline*}
\frac12\int_{B\setminus \tilde\B_{\e,R}^{r_0}} a\chi_R^z\Big(|\nabla u_\e-iu_\e N_\e\!\vi_\e\!|^2+\frac a{2\e^2}(1-|u_\e|^2)^2\Big)\\
\ge \frac{\log(\tfrac r{r_0})}2\int_B a\chi_R^z\nu_{\e,R}^{r_0}-Cs_B\Big(N_\e+\frac{\Ec_{\e,R}^*}\Log\Big)^2.
\end{multline*}
Summing this estimate over all the balls $B$ of the collection $\tilde\B_{\e,R}^{r_0,r}$ that intersect $B_{2R}(z)$, and recalling that the sum of the radii of these balls is by construction bounded by $Cr$, we deduce for all $r_0\le r\le\tilde r$,
\begin{multline*}
\frac12\int_{\tilde\B_{\e,R}^{r_0,r}\setminus \tilde\B_{\e,R}^{r_0}} a\chi_R^z\Big(|\nabla u_\e-iu_\e N_\e\!\vi_\e\!|^2+\frac a{2\e^2}(1-|u_\e|^2)^2\Big)\\
\ge \frac{\log(\tfrac r{r_0})}2\int_{\R^2} a\chi_R^z\nu_{\e,R}^{r_0}-Cr\Big(N_\e+\frac{\Ec_{\e,R}^*}\Log\Big)^2.
\end{multline*}
Combining this with~\eqref{eq:lowerboundr0}, and recalling that by definition $\B_{\e,R}^{r_0}\subset\tilde\B_{\e,R}^{r_0}$, we deduce
\begin{multline}\label{eq:lowerboundrefpre}
\frac12\int_{\tilde B_{\e,R}^{r_0,r}}a\chi_R^z\Big(|\nabla u_\e-iu_\e N_\e \!\vi_\e\!|^2+\frac{a}{2\e^2}(1-|u_\e|^2)^2\Big)
\ge \frac{\log(\tfrac r\e)}2\int_{\R^2} a\chi_R^z\nu_{\e,R}^{r_0}\\
-Cr\Big(\frac{\Ec_{\e,R}^*}\Log\Big)^2-o(N_\e^2)-C\Big(N_\e+\frac{\Ec_{\e,R}^*}\Log\Big)\log\Big(2+\frac{\Ec_{\e,R}^*}\Log\Big),
\end{multline}
and hence, using Lemma~\ref{lem:ballconstr}(ii) and the choice~\eqref{eq:choicerr0} of $r$,
\begin{eqnarray*}
\lefteqn{\frac12\int_{\tilde B_{\e,R}^{r_0,r}}a\chi_R^z\Big(|\nabla u_\e-iu_\e N_\e \!\vi_\e\!|^2+\frac{a}{2\e^2}(1-|u_\e|^2)^2\Big)}\\
&\ge& \frac{\Log}2\int_{\R^2} a\chi_R^z \nu_{\e,R}^{r_0}-C|\!\log r|\Big(N_\e+\frac{\Ec_{\e,R}^*}\Log\Big)\\
&&\qquad-Cr\Big(\frac{\Ec_{\e,R}^*}\Log\Big)^2-o(N_\e^2)-C\Big(N_\e+\frac{\Ec_{\e,R}^*}\Log\Big)\log\Big(2+\frac{\Ec_{\e,R}^*}\Log\Big)\\
&\ge& \frac{\Log}2\int_{\R^2} a\chi_R^z \nu_{\e,R}^{r_0}-o(1)\Big(\frac{\Ec_{\e,R}^*}\Log\Big)^2-o(N_\e^2),
\end{eqnarray*}
that is, \eqref{eq:lowerboundref}.

\medskip
\noindent\step7 Optimal bound on the energy.

In this step, we prove $\Ec_{\e,R}^*\lesssim N_\e\Log$, thus completing the result of Step~5 in all regimes $1\ll N_\e\lesssim\Log$. Note that by Step~3 this also implies $\sup_z\int_{\R^2}\chi_R^z|\nu_{\e,R}^r|\lesssim N_\e$.

By Step~5, it only remains to consider the strongly dilute regime $1\ll N_\e\lesssim\log\Log$.
Let $r_0\le r\ll1$ be fixed as in~\eqref{eq:choicerr0}. On the one hand, using the estimate~\eqref{eq:munuexch},
we deduce from the result~\eqref{eq:lowerboundref} of Step~6,
\begin{multline*}
\frac12\int_{\R^2\setminus\tilde\B_{\e,R}^{r_0,r}}a\chi_R^z\Big(|\nabla u_\e-iu_\e N_\e \!\vi_\e\!|^2+\frac{a}{2\e^2}(1-|u_\e|^2)^2\Big)\\
\le\D_{\e,R}^z
+O\bigg(r_0\Log\Big(N_\e+\frac{\Ec_{\e,R}^*}\Log\Big)\bigg)+o(1)\Big(\frac{\Ec_{\e,R}^*}\Log\Big)^2+o(N_\e^2)
\end{multline*}
and hence, using the assumption $\D_{\e,R}^*\lesssim N_\e^2$, the suboptimal energy bound of Step~5, and the choice~\eqref{eq:choicerr0} of~$r_0$,
\begin{align}\label{eq:estRminBrefined}
\frac12\int_{\R^2\setminus\tilde\B_{\e,R}^{r_0,r}}a\chi_R^z\Big(|\nabla u_\e-iu_\e N_\e \!\vi_\e\!|^2+\frac{a}{2\e^2}(1-|u_\e|^2)^2\Big)\lesssim N_\e^2+o(1)\Big(\frac{\Ec_{\e,R}^*}\Log\Big)^2.
\end{align}
On the other hand, combining the estimates~\eqref{eq:boundenergypreconcl} and~\eqref{eq:boundvortpreconcl} (with $\B_{\e,R}^r$ replaced by $\tilde\B_{\e,R}^{r_0,r}$) of Step~5, we find
\begin{align*}
\Ec_{\e,R}^*\lesssim N_\e\Log+\Log\Big(\sup_z\int_{\R^2\setminus\tilde\B_{\e,R}^{r_0,r}}\chi_{R}^z|\nabla u_\e-iu_\e N_\e\!\vi_\e\!|^2\Big)^{1/2}+r\Log R^{-1}(\Ec_{\e,R}^*)^{1/2}.
\end{align*}
Now inserting~\eqref{eq:estRminBrefined} yields
\begin{align*}
\Ec_{\e,R}^*\lesssim N_\e\Log+o(1)\Ec_{\e,R}^*+\Log R^{-1}(\Ec_{\e,R}^*)^{1/2},
\end{align*}
and thus, recalling the choice $R\gtrsim\Log$, and absorbing $\Ec_{\e,R}^*$ in the left-hand side, the result $\Ec_{\e,R}^*\lesssim N_\e\Log$ follows.

\medskip
\noindent\step8 Conclusion.\nopagebreak

The optimal energy bound $\Ec_{\e,R}^*\lesssim N_\e\Log$ is now proved. In the present step, we check that the remaining statements follow from this bound.
We split the proof into seven further substeps.

\medskip
\noindent\substep{8.1} Proof of~(i).

The result~\eqref{eq:lowerbound} follows from~\eqref{eq:lowerboundclaim1b} in Step~2 with $\phi=a\chi_R^z$, combined with the optimal energy bound. Repeating the argument of Step~6 with the optimal energy bound rather than with the suboptimal bound of Step~5, the choice~\eqref{eq:choicerr0} can be replaced by $\e^{1/2}<r_0\ll \frac{N_\e}\Log$. For such a choice of $r_0$, and for $r\ge r_0$ as in~\eqref{eq:choicerr0}, the result~\eqref{eq:lowerboundref} together with the optimal energy bound directly implies the result~\eqref{eq:lowerboundbis} in the strongly dilute regime $1\ll N_\e\lesssim\log\Log$.

\medskip
\noindent\substep{8.2} Proof of~(ii).

The bound~\eqref{eq:numbvort} on the number of vortices follows from the result~\eqref{eq:energyvort} of Step~3 together with the optimal energy bound.
It remains to prove that in the regime $1\ll N_\e\ll\Log^{1/2}$ for $e^{-o(1)\frac\Log{N_\e}}\le r<\bar r$ each ball of the collection $\B_{\e,R}^r$ has a nonnegative degree. This is a refinement of the result of Step~4.
The lower bound~\eqref{eq:lowerboundclaim1b} of Step~2 with $\phi=a\chi_R^z$ can be rewritten as follows, using the optimal energy bound, for all $z\in \R^2$,
\begin{multline*}
{\Log}\int_{\R^2} a\chi_R^z (\nu_{\e,R}^r)^-=\frac{\Log}2\int_{\R^2} a\chi_R^z (|\nu_{\e,R}^r|-\nu_{\e,R}^r)\\\le\Ec_{\e,R}^z-\frac\Log2\int_{\R^2} a\chi_R^z\nu_{\e,R}^r+O\Big(rN_\e\Log+r^2N_\e^2+N_\e|\!\log r|\Big)+o(N_\e^2),
\end{multline*}
and hence, using~\eqref{eq:munuexch} to replace $\nu_{\e,R}^r$ by $\mu_\e$ in the right-hand side, and using the assumption $\D_{\e,R}^z\lesssim N_\e^2$, we find
\begin{align}\label{eq:boundnegpart}
{\Log}\int_{\R^2} \chi_R^z (\nu_{\e,R}^r)^-\lesssim N_\e^2+rN_\e\Log+N_\e|\!\log r|.
\end{align}
Dividing both sides by $\Log$, we deduce for $N_\e\ll\Log^{1/2}$ with $e^{-o(1)\frac\Log{N_\e}}\le r\ll N_\e^{-1}$,
\[\sup_z\int_{\R^2} \chi_R^z (\nu_{\e,R}^r)^-\ll1,\]
which means that for $\e$ small enough there exists no single ball $B^j\in\B_{\e,R}^r$ with negative degree $d_j<0$. This proves the result for $r\ll N_\e^{-1}$. Now for $N_\e^{-1}\lesssim r<\bar r$ the same property must hold, since, by monotonicity of the collection $\B_{\e,R}^r$ with respect to $r$, for any $r>r'$ the degree of a ball $B\in\B_{\e,R}^r$ equals the sum of the degrees of all the balls $B'\in\B_{\e}(r')$ with $B'\subset B$.

\medskip
\noindent\substep{8.3} Proof of~(v).

In the regime $N_\e\gg\log\Log$, for $e^{-o(N_\e)}\le r\ll \frac{N_\e}\Log$, the result~\eqref{eq:energyoutballs} follows from~\eqref{eq:energyR2minB} together with the optimal energy bound. Monotonicity of $\B_{\e,R}^r$ with respect to $r$ then implies~\eqref{eq:energyoutballs} for all $r\ge e^{-o(N_\e)}$ in the regime $N_\e\gg\log\Log$. In the regime $1\ll N_\e\lesssim\log\Log$, it suffices to argue as for~\eqref{eq:energyR2minB} in Step~2, but with the lower bound~\eqref{eq:lowerboundclaim1b} replaced by its refined version~\eqref{eq:lowerboundref}:
for $r_0\le r$ with $\e^{1/2}<r_0\ll \frac{N_\e}\Log$ and $e^{-o(N_\e)}\le r\ll1$,
the estimate~\eqref{eq:lowerboundref} together with~\eqref{eq:munuexch} indeed yields
\begin{multline*}
\frac12\int_{\R^2\setminus\tilde\B_{\e,R}^{r_0,r}}a\chi_R^z\Big(|\nabla u_\e-iu_\e N_\e \!\vi_\e\!|^2+\frac{a}{2\e^2}(1-|u_\e|^2)^2\Big)\\
\le \frac12\int_{\R^2}a\chi_R^z\Big(|\nabla u_\e-iu_\e N_\e \!\vi_\e\!|^2+\frac{a}{2\e^2}(1-|u_\e|^2)^2\Big)-\frac{\Log}2\int_{\R^2} a\chi_R^z \nu_{\e,R}^{r_0}+o(N_\e^2)\\
\le \D_{\e,R}^z+r_0N_\e\Log+o(N_\e^2)=\D_{\e,R}^z+o(N_\e^2),
\end{multline*}
and the result~\eqref{eq:energyoutballs2} follows by monotonicity of $\tilde\B_{\e,R}^{r_0,r}$ with respect to $r$.

\medskip
\noindent\substep{8.4} Proof of~(iii).\nopagebreak

The Jacobian estimate~\eqref{eq:jacobest} follows from Lemma~\ref{lem:ballconstr}(iii) together with the optimal energy bound, and the estimate~\eqref{eq:jacobestbis} with $\gamma=1$ similarly follows from~\eqref{eq:boundmumutildeequ}. The result~\eqref{eq:jacobestbis} for all $\gamma\in[0,1]$ is then obtained by interpolation (as e.g.\@ in~\cite{Jerrard-Soner-98}) provided we also manage to prove, for all $\phi\in\Ld^\infty(\R^2)$ supported in a ball $B_R(z)$,
\begin{align}\label{eq:mumutildeC0}
\Big|\int_{\R^2} \phi(\tilde\mu_\e-\mu_\e)\Big|\lesssim RN_\e\Log\|\phi\|_{\Ld^\infty}.
\end{align}
Let $\phi\in\Ld^\infty(\R^2)$ be supported in $B_R(z)$, for some $z\in \R^2$. By definition~\eqref{eq:defmutilde}, we find
\begin{multline*}
\int_{\R^2} \phi(\tilde\mu_\e-\mu_\e)= N_\e\int_{\R^2} \phi\,\big((1-|u_\e|^2)\curl\!\vi_\e+2\langle\nabla u_\e-iu_\e N_\e\!\vi_\e,u_\e\rangle\cdot\vi_\e^\bot\big)\\
\le N_\e\|\phi\|_{\Ld^\infty}\int_{B_R(z)} \Big(|1-|u_\e|^2||\curl\!\vi_\e\!|+2|\!\vi_\e\!||1-|u_\e|^2||\nabla u_\e-iu_\e N_\e\!\vi_\e|\\
+2|\!\vi_\e\!||\nabla u_\e-iu_\e N_\e\!\vi_\e|\Big),
\end{multline*}
hence
we deduce from the optimal energy bound, with $\|\!\vi_\e\!\|_{\Ld^\infty},\|\curl\!\vi_\e\!\|_{\Ld^2}\lesssim1$,
\begin{align*}
\Big|\int_{\R^2} \phi(\tilde\mu_\e-\mu_\e)\Big|\lesssim \big(\e N_\e^2\Log+RN_\e\Log\big)\|\phi\|_{\Ld^\infty},
\end{align*}
that is,~\eqref{eq:mumutildeC0}.

\medskip
\noindent\substep{8.5} Proof of~(iv) in the regime $N_\e\gg\log\Log$.

\nopagebreak
We focus on the regime $N_\e\gg\log\Log$. Let $\e^{1/2}<r\ll1$ to be later optimized as a function of $\e$.
We write as before $\B_{\e,R}^r=\biguplus_jB^j$, $B^j=\bar B(y_j,r_j)$, we denote by $d_j$ the degree of $B^j$, and we set $\nu_{\e,R}^r=2\pi\sum_jd_j\delta_{y_j}$.
Given $\phi\in W^{1,\infty}(\R^2)$ supported in the ball~$B_R(z)$, we decompose 
\begin{eqnarray}
\lefteqn{\int_{\R^2}\phi\Big(|\nabla u_\e-iu_\e N_\e\!\vi_\e\!|^2+\frac{a}{2\e^2}(1-|u_\e|^2)^2-\Log\nu_{\e,R}^r\Big)}\nonumber\\
&\le&\int_{\R^2\setminus\B_{\e,R}^r}\phi\Big(|\nabla u_\e-iu_\e N_\e\!\vi_\e\!|^2+\frac{a}{2\e^2}(1-|u_\e|^2)^2\Big)\nonumber\\
&&+\sum_j\bigg|\int_{B^j}\phi\Big(|\nabla u_\e-iu_\e N_\e\!\vi_\e\!|^2+\frac{a}{2\e^2}(1-|u_\e|^2)^2\Big)-2\pi\phi(y_j)d_j\Log\bigg|\nonumber\\
&\le&\|a^{-1}\phi\|_{\Ld^\infty}\int_{\R^2\setminus\B_{\e,R}^r}a\chi_R^z\Big(|\nabla u_\e-iu_\e N_\e\!\vi_\e\!|^2+\frac{a}{2\e^2}(1-|u_\e|^2)^2\Big)\nonumber\\
&&+\|a^{-1}\phi\|_{\Ld^\infty}\sum_{j}\chi_R^z(y_j)\bigg|\int_{B^j}a\Big(|\nabla u_\e-iu_\e N_\e\!\vi_\e\!|^2+\frac{a}{2\e^2}(1-|u_\e|^2)^2\Big)\nonumber\\
&&\hspace{9cm}-2\pi a(y_j)d_j\Log\bigg|\nonumber\\
&&+r\|a^{-1}\phi\|_{W^{1,\infty}}\int_{B_{2R}(z)\cap\B_{\e,R}^r}\Big(|\nabla u_\e-iu_\e N_\e\!\vi_\e\!|^2+\frac{a}{2\e^2}(1-|u_\e|^2)^2\Big).\label{eq:boundTer}
\end{eqnarray}
Combined with the optimal energy bound, the localized lower bound~\eqref{eq:lowerboundclaim2} in Lemma~\ref{lem:ballconstr}(i) with $\phi=a$ yields for all $j$,
\begin{multline*}
\frac12\int_{B^j}a\Big(|\nabla u_\e-iu_\e N_\e \!\vi_\e\!|^2+\frac{a}{2\e^2}(1-|u_\e|^2)^2\Big)\\
\ge \pi a(y_j)|d_j|\Log-O\big(r_jN_\e\Log+|d_j||\!\log r|+|d_j|\log N_\e\big),
\end{multline*}
hence
\begin{multline*}
\bigg|\int_{B^j}a\Big(|\nabla u_\e-iu_\e N_\e \!\vi_\e\!|^2+\frac{a}{2\e^2}(1-|u_\e|^2)^2\Big)-2\pi a(y_j)|d_j|\Log\bigg|\\
\le\int_{B^j}a\Big(|\nabla u_\e-iu_\e N_\e \!\vi_\e\!|^2+\frac{a}{2\e^2}(1-|u_\e|^2)^2\Big)-2\pi a(y_j)|d_j|\Log\\
+O\big(r_jN_\e\Log+|d_j||\!\log r|+|d_j|\log N_\e\big).
\end{multline*}
Noting that $\chi_R^z(y_j)\le\chi_R^z(y)+O(R^{-1}r_j)\chi_{2R}^z(y_j)$ holds for $y\in B_j$, we obtain
\begin{multline*}
\chi_R^z(y_j)\bigg|\int_{B^j}a\Big(|\nabla u_\e-iu_\e N_\e \!\vi_\e\!|^2+\frac{a}{2\e^2}(1-|u_\e|^2)^2\Big)-2\pi a(y_j)|d_j|\Log\bigg|\\
\le\int_{B^j}a\chi_R^z\Big(|\nabla u_\e-iu_\e N_\e \!\vi_\e\!|^2+\frac{a}{2\e^2}(1-|u_\e|^2)^2\Big)-2\pi a(y_j)\chi_R^z(y_j)|d_j|\Log\\
+\chi_{2R}^z(y_j)O\big(r_jN_\e\Log+|d_j||\!\log r|+|d_j|\log N_\e\big).
\end{multline*}
Inserting this into~\eqref{eq:boundTer}, and using the bound of item~(ii) on the number of vortices, we find
\begingroup\allowdisplaybreaks
\begin{multline*}
\int_{\R^2}\phi\Big(|\nabla u_\e-iu_\e N_\e\!\vi_\e\!|^2+\frac{a}{2\e^2}(1-|u_\e|^2)^2-\Log\nu_{\e,R}^r\Big)\\
\le\|a^{-1}\phi\|_{\Ld^\infty}\int_{\R^2} a\chi_{R}^z\Big(|\nabla u_\e-iu_\e N_\e\!\vi_\e\!|^2+\frac{a}{2\e^2}(1-|u_\e|^2)^2-\Log\nu_{\e,R}^r\Big)\\
+r\|a^{-1}\phi\|_{W^{1,\infty}}\int_{B_{2R}(z)\cap\B_{\e,R}^r}\Big(|\nabla u_\e-iu_\e N_\e\!\vi_\e\!|^2+\frac{a}{2\e^2}(1-|u_\e|^2)^2\Big)\\
+O\big(rN_\e\Log+N_\e|\!\log r|+N_\e\log N_\e\big)\|\phi\|_{\Ld^\infty},
\end{multline*}
\endgroup
where the second right-hand side term is estimated by $rN_\e\Log\|a^{-1}\phi\|_{W^{1,\infty}}$, and where the bound~\eqref{eq:munuexch} can be used to replace $\nu_{\e,R}^r$ by $\mu_\e$ in both sides up to an error of order $(rN_\e\Log+1)\|\phi\|_{\Ld^{\infty}}$. In the present regime $N_\e\gg\log\Log$, we may choose $e^{-o(N_\e)}\le r\ll \frac{N_\e}\Log$, and the conclusion~\eqref{eq:excessestim} follows for that choice.

\medskip
\noindent\substep{8.6} Proof of~(iv) in the regime $1\ll N_\e\lesssim\log\Log$.

We turn to the regime $1\ll N_\e\lesssim\log\Log$, in which case the proof of~(iv) needs to be adapted in the spirit of the computations in Step~6.
Let $\phi\in W^{1,\infty}(\R^2)$ be supported in the ball $B_R(z)$, and let $e^{-o(1)\frac\Log{N_\e}}\le r_0\ll \frac{N_\e}\Log$. First arguing as in Substep~8.5 with this choice of $r_0$, we obtain
\begin{multline}\label{eq:(iv)-elt1}
\int_{\B_{\e,R}^{r_0}}\phi\Big(|\nabla u_\e-iu_\e N_\e\!\vi_\e\!|^2+\frac{a}{2\e^2}(1-|u_\e|^2)^2-\log(\tfrac{r_0}\e)\nu_{\e,R}^{r_0}\Big)\\
\le \|a^{-1}\phi\|_{\Ld^\infty}\int_{\B_{\e,R}^{r_0}}a\chi_R^z\Big(|\nabla u_\e-iu_\e N_\e\!\vi_\e\!|^2+\frac{a}{2\e^2}(1-|u_\e|^2)^2-\log(\tfrac{r_0}\e)\nu_{\e,R}^{r_0}\Big)\\
+o(N_\e^2)\|\phi\|_{W^{1,\infty}}.
\end{multline}

Now we consider the modified ball collection $\tilde\B_{\e,R}^{r_0,r}$ with $r\ge r_0$, as constructed in Step~6.1.
Assume that some ball $B=\bar B(y,s)$ gets grown into $B'=\bar B(y,ts)$ without merging, for some $t\ge1$, and assume that $B'\setminus B$ does not intersect $\tilde\B_{\e,R}^{r_0}$, so that by construction $||u_\e|-1|\le\Log^{-1}$ holds on $B'\setminus B$. Let $d$ denote the degree of $B$ (hence of $B'$). We may then decompose
\begin{multline*}
\bigg|\frac12\int_{B'\setminus B}\phi\Big(|\nabla u_\e-iu_\e N_\e\!\vi_\e\!|^2+\frac{a}{2\e^2}(1-|u_\e|^2)^2\Big)-\phi(y)\pi d\log t\bigg|\\
\le \|a^{-1}\phi\|_{\Ld^\infty}\bigg|\frac12\int_{B'\setminus B}a\chi_R^z\Big(|\nabla u_\e-iu_\e N_\e\!\vi_\e\!|^2+\frac{a}{2\e^2}(1-|u_\e|^2)^2\Big)- a(y)\chi_R^z(y)\pi d\log t\bigg|\\
+\|a^{-1}\phi\|_{W^{1,\infty}}\frac{1}2\int_{B'\setminus B}|\cdot-y|\chi_R^z\Big(|\nabla u_\e-iu_\e N_\e\!\vi_\e\!|^2+\frac{a}{2\e^2}(1-|u_\e|^2)^2\Big),
\end{multline*}
and hence, decomposing $\chi_R^z(x)\le \chi_R^z(y)+O(R^{-1})$ for all $x\in B'\setminus B$,
\begin{multline*}
\bigg|\frac12\int_{B'\setminus B}\phi\Big(|\nabla u_\e-iu_\e N_\e\!\vi_\e\!|^2+\frac{a}{2\e^2}(1-|u_\e|^2)^2\Big)-\phi(y)\pi d\log t\bigg|\\
\le \|a^{-1}\phi\|_{\Ld^\infty}\bigg|\frac12\int_{B'\setminus B}a\chi_R^z\Big(|\nabla u_\e-iu_\e N_\e\!\vi_\e\!|^2+\frac{a}{2\e^2}(1-|u_\e|^2)^2\Big)- a(y)\chi_R^z(y)\pi d\log t\bigg|\\
+\frac{\chi_R^z(y)}2\|a^{-1}\phi\|_{W^{1,\infty}}\int_{B'\setminus B}|\cdot-y|\Big(|\nabla u_\e-iu_\e N_\e\!\vi_\e\!|^2+\frac{a}{2\e^2}(1-|u_\e|^2)^2\Big)\\
+CtsR^{-1}\|\phi\|_{W^{1,\infty}}\int_{B'\setminus B}\Big(|\nabla u_\e-iu_\e N_\e\!\vi_\e\!|^2+\frac{a}{2\e^2}(1-|u_\e|^2)^2\Big).
\end{multline*}
Arguing as in~\eqref{eq:upbound-radcomp} yields
\begin{multline}\label{eq:increasball-(iv)}
\bigg|\frac12\int_{B'\setminus B}\phi\Big(|\nabla u_\e-iu_\e N_\e\!\vi_\e\!|^2+\frac{a}{2\e^2}(1-|u_\e|^2)^2\Big)-\phi(y)\pi d\log t\bigg|\\
\le \|a^{-1}\phi\|_{\Ld^\infty}\bigg|\frac12\int_{B'\setminus B}a\chi_R^z\Big(|\nabla u_\e-iu_\e N_\e\!\vi_\e\!|^2+\frac{a}{2\e^2}(1-|u_\e|^2)^2\Big)- a(y)\chi_R^z(y)\pi d\log t\bigg|\\
+\chi_R^z(y)\|a^{-1}\phi\|_{W^{1,\infty}}\bigg(Cd^2(t-1)s
+ts\int_{B'\setminus B}\Big|\nabla u_\e-iu_\e N_\e\!\vi_\e-iu_\e\frac{\tau d}{|\cdot-y|}\Big|^2\bigg)\\
+CtsR^{-1}\|\phi\|_{W^{1,\infty}}\int_{B'\setminus B}|\nabla u_\e-iu_\e N_\e\!\vi_\e\!|^2
+Cts\|a^{-1}\phi\|_{W^{1,\infty}}\int_{B'\setminus B}\frac{a}{\e^2}(1-|u_\e|^2)^2.
\end{multline}
Recalling the improved lower bound~\eqref{eq:stice-int}, and combining it with the bound of item~(ii) on the number of vortices, and with the assumption $\|\curl\!\vi_\e\!\|_{\Ld^\infty}\lesssim1$, we find
\begin{multline*}
(1+O(\Log^{-1}))\frac12\int_{B' \setminus B} a\chi_R^z\Big(|\nabla u_\e-iu_\e N_\e\!\vi_\e\!|^2+\frac a{2\e^2}(1-|u_\e|^2)^2\Big)\\
\ge a(y)\chi_R^z(y)\pi d^2\log t-C(t-1)sN_\e^2-CtsR^{-1}\int_{B'\setminus B} |\nabla u_\e-iu_\e N_\e\!\vi_\e\!|^2\\
+\Big(\frac{a(y)}2 (1-C\Log^{-1}) -Cts\Big)\chi_R^z(y)\int_{B'\setminus B} \Big|\nabla u_\e-iu_\e N_\e\!\vi_\e - iu_\e \frac{\tau d}{|\cdot-y|}\Big|^2.
\end{multline*}
and hence, injecting this estimate into~\eqref{eq:increasball-(iv)}, we deduce for $\e$ small enough and $ts\ll1$,
\begin{multline*}
\bigg|\frac12\int_{B'\setminus B}\phi\Big(|\nabla u_\e-iu_\e N_\e\!\vi_\e\!|^2+\frac{a}{2\e^2}(1-|u_\e|^2)^2\Big)-\phi(y)\pi d\log t\bigg|\\
\le C\|\phi\|_{W^{1,\infty}}\bigg(\frac12\int_{B'\setminus B}a\chi_R^z\Big(|\nabla u_\e-iu_\e N_\e\!\vi_\e\!|^2+\frac{a}{2\e^2}(1-|u_\e|^2)^2\Big)- a(y)\chi_R^z(y)\pi d\log t\bigg)\\
+C(t-1)sN_\e^2\|\phi\|_{W^{1,\infty}}
+CtsR^{-1}\|\phi\|_{W^{1,\infty}}\int_{B'\setminus B}|\nabla u_\e-iu_\e N_\e\!\vi_\e\!|^2\\
+Cts\|\phi\|_{W^{1,\infty}}\int_{B'\setminus B}\frac{a}{\e^2}(1-|u_\e|^2)^2.
\end{multline*}
Using the bound of item~(ii) on the number of vortices, we find
\begin{align*}
\Big|\phi(y)\pi d\log t-\frac{\log t}2\int_B\phi \nu_{\e,R}^{r_0}\Big|&\le \frac{s\log t}2\|\nabla\phi\|_{\Ld^\infty}\int_B|\nu_{\e,R}^{r_0}|\le C(t-1)sN_\e\|\nabla\phi\|_{\Ld^\infty},
\end{align*}
so that the above becomes
\begin{multline*}
\bigg|\frac12\int_{B'\setminus B}\phi\Big(|\nabla u_\e-iu_\e N_\e\!\vi_\e\!|^2+\frac{a}{2\e^2}(1-|u_\e|^2)^2\Big)-\frac{\log t}2\int_{B'}\phi\nu_{\e,R}^{r_0}\bigg|\\
\le C\|\phi\|_{W^{1,\infty}}\bigg(\int_{B'\setminus B}a\chi_R^z\Big(|\nabla u_\e-iu_\e N_\e\!\vi_\e\!|^2+\frac{a}{2\e^2}(1-|u_\e|^2)^2\Big)-{\log t}\int_{B'}a\chi_R^z\nu_{\e,R}^{r_0}\\
+C(t-1)sN_\e^2
+CtsR^{-1}\int_{B'\setminus B}|\nabla u_\e-iu_\e N_\e\!\vi_\e\!|^2
+Cts\int_{B'\setminus B}\frac{a}{\e^2}(1-|u_\e|^2)^2\bigg).
\end{multline*}
By construction of the ball growth and merging process, this easily implies the following: if a ball $B=\bar B(y_B,s_B)$ belongs to the collection $\tilde\B_{\e,R}^{r_0,r}$ for some $r_0\le r\ll1$, then we have
\begin{multline*}
\bigg|\frac12\int_{B\setminus\tilde\B_{\e,R}^{r_0}}\phi\Big(|\nabla u_\e-iu_\e N_\e\!\vi_\e\!|^2+\frac{a}{2\e^2}(1-|u_\e|^2)^2\Big)-\frac{\log (\tfrac r{r_0})}2\int_{B}\phi\nu_{\e,R}^{r_0}\bigg|\\
\le C\|\phi\|_{W^{1,\infty}}\bigg(\int_{B\setminus \tilde\B_{\e,R}^{r_0}}a\chi_R^z\Big(|\nabla u_\e-iu_\e N_\e\!\vi_\e\!|^2+\frac{a}{2\e^2}(1-|u_\e|^2)^2\Big)- \log(\tfrac r{r_0})\int_Ba\chi_R^z\nu_{\e,R}^{r_0}\\
+Cs_BN_\e^2
+Cs_BR^{-1}\int_{B\setminus \tilde\B_{\e,R}^{r_0}}|\nabla u_\e-iu_\e N_\e\!\vi_\e\!|^2
+Cs_B\int_{B\setminus \tilde\B_{\e,R}^{r_0}}\frac{a}{\e^2}(1-|u_\e|^2)^2\bigg).
\end{multline*}
Summing this estimate over all balls $B$ of the collection $\tilde\B_{\e,R}^{r_0,r}$ that intersect $B_R(z)$, and recalling that the sum of the radii of these balls is by construction bounded by $Cr$,
\begin{multline}\label{eq:increasball-(iv)-rewr}
\bigg|\frac12\int_{\tilde\B_{\e,R}^{r_0,r}\setminus\tilde\B_{\e,R}^{r_0}}\phi\Big(|\nabla u_\e-iu_\e N_\e\!\vi_\e\!|^2+\frac{a}{2\e^2}(1-|u_\e|^2)^2\Big)-\frac{\log (\tfrac r{r_0})}2\int_{\R^2}\phi\nu_{\e,R}^{r_0}\bigg|\\
\le C\|\phi\|_{W^{1,\infty}}\bigg(\int_{\tilde\B_{\e,R}^{r_0,r}\setminus \tilde\B_{\e,R}^{r_0}}a\chi_R^z\Big(|\nabla u_\e-iu_\e N_\e\!\vi_\e\!|^2+\frac{a}{2\e^2}(1-|u_\e|^2)^2\Big)- \log(\tfrac r{r_0})\int_{\R^2}a\chi_R^z\nu_{\e,R}^{r_0}\\
+CrN_\e^2+CrR^{-1}\int_{B_{2R}(z)}|\nabla u_\e-iu_\e N_\e\!\vi_\e\!|^2+Cr\int_{B_{2R}(z)}\frac{a}{\e^2}(1-|u_\e|^2)^2\bigg).
\end{multline}

Let us estimate the last right-hand side term of~\eqref{eq:increasball-(iv)-rewr}. Applying the lower bound~\eqref{eq:lowerboundref} with $\e$ replaced by $2\e$ (with $\e<1/2$), together with the optimal energy bound, we obtain, for $r\ge r_0$ with $e^{-o(N_\e)}\le r\ll1$,
\begin{multline*}
\frac{\Log}2\int_{\R^2} a\chi_R^z |\nu_{\e,R}^{r_0}|-\frac{\log2}2\int_{\R^2} a\chi_R^z |\nu_{\e,R}^{r_0}|-o(N_\e^2)=\frac{|\!\log(2\e)|}2\int_{\R^2} a\chi_R^z |\nu_{\e,R}^{r_0}|-o(N_\e^2)\\
\le \frac12\int_{\tilde\B_{\e,R}^{r_0,r}}a\chi_R^z\Big(|\nabla u_\e-iu_\e N_\e \!\vi_\e\!|^2+\frac{a}{2(2\e)^2}(1-|u_\e|^2)^2\Big)\\
\le \D_{\e,R}^*+\frac{\Log}2\int_{\R^2} a\chi_R^z\mu_\e-\frac3{16\e^2}\int_{\tilde\B_{\e,R}^{r_0,r}}a^2\chi_R^z(1-|u_\e|^2)^2.
\end{multline*}
Using~\eqref{eq:munuexch}, the bound of item~(ii) on the number of vortices, and the choice of $r_0$, we then find
\begin{eqnarray*}
\lefteqn{\frac3{16\e^2}\int_{\tilde\B_{\e,R}^{r_0,r}}a^2\chi_R^z(1-|u_\e|^2)^2}\\
&\le& \D_{\e,R}^*+\frac{\Log}2\int_{\R^2} a\chi_R^z(\mu_\e-\nu_{\e,R}^{r_0})+\frac{\log2}2\int_{\R^2} a\chi_R^z |\nu_{\e,R}^{r_0}|+o(N_\e^2)\\
&\le& \D_{\e,R}^*+o(N_\e^2)~\lesssim\, N_\e^2.
\end{eqnarray*}
Combining this with the result~\eqref{eq:energyoutballs2} of item~(v), we deduce the (suboptimal) estimate
\begin{align}\label{eq:boundprepoho}
\sup_z\int_{\R^2}\frac{\chi_R^z}{\e^2}(1-|u_\e|^2)^2\lesssim N_\e^2.
\end{align}
Injecting this result into~\eqref{eq:increasball-(iv)-rewr}, together with the optimal energy bound and the choice $R\gtrsim\Log$, we find
\begin{multline}\label{eq:reincreasball-(iv)}
\bigg|\frac12\int_{\tilde\B_{\e,R}^{r_0,r}\setminus\tilde\B_{\e,R}^{r_0}}\phi\Big(|\nabla u_\e-iu_\e N_\e\!\vi_\e\!|^2+\frac{a}{2\e^2}(1-|u_\e|^2)^2\Big)-\frac{\log (\tfrac r{r_0})}2\int_{\R^2}\phi\nu_{\e,R}^{r_0}\bigg|\\
\le C\|\phi\|_{W^{1,\infty}}\bigg(\int_{\tilde\B_{\e,R}^{r_0,r}\setminus \tilde\B_{\e,R}^{r_0}}a\chi_R^z\Big(|\nabla u_\e-iu_\e N_\e\!\vi_\e\!|^2+\frac{a}{2\e^2}(1-|u_\e|^2)^2\Big)\\
-\log(\tfrac r{r_0})\int_{\R^2}a\chi_R^z\nu_{\e,R}^{r_0}
+CrN_\e^2\bigg).
\end{multline}
Combining this with~\eqref{eq:(iv)-elt1}, and recalling that by definition $\B_{\e,R}^{r_0}\subset\tilde\B_{\e,R}^{r_0}$, we deduce
\begin{multline*}
\bigg|\frac12\int_{\tilde\B_{\e,R}^{r_0,r}}\phi\Big(|\nabla u_\e-iu_\e N_\e\!\vi_\e\!|^2+\frac{a}{2\e^2}(1-|u_\e|^2)^2\Big)-\frac{\log (\tfrac r\e)}2\int_{\R^2}\phi\nu_{\e,R}^{r_0}\bigg|\\
\le C\|\phi\|_{W^{1,\infty}}\bigg(\int_{\tilde\B_{\e,R}^{r_0,r}}a\chi_R^z\Big(|\nabla u_\e-iu_\e N_\e\!\vi_\e\!|^2+\frac{a}{2\e^2}(1-|u_\e|^2)^2\Big)- \log(\tfrac r\e)\int_{\R^2}a\chi_R^z\nu_{\e,R}^{r_0}+o(N_\e^2)\bigg).
\end{multline*}
Using~\eqref{eq:munuexch} to replace $\nu_{\e,R}^{r_0}$ by $\mu_\e$ up to an error of order $(r_0N_\e\Log+1)\|\phi\|_{W^{1,\infty}}\ll N_\e^2\|\phi\|_{W^{1,\infty}}$, the result~\eqref{eq:excessestim} follows.

\medskip
\noindent\substep{8.7} Proof of~(vi).

We adapt an argument by Struwe~\cite{Struwe-94} (see also~\cite[Proof of Lemma~4.7]{SS-12}).
Recalling that $|B_{2R}(z)\cap\B_{\e,R}^r|\lesssim r^2$, a direct application of the Hölder inequality yields
\[\int_{\B_{\e,R}^r}\chi_R^z|\nabla u_\e-iu_\e N_\e\!\vi_\e\!|^p\lesssim r^{2-p}\Big(\int_{\B_{\e,R}^r}\chi_R^z|\nabla u_\e-iu_\e N_\e\!\vi_\e\!|^2\Big)^{p/2}\lesssim r^{2-p}(N_\e\Log)^{p/2},\]
which only implies the result if we are allowed to choose the total radius $r$ small enough.
Otherwise, it is useful to rather work on dyadic ``annuli''.
For each integer $0\le k\le K_\e:=\lfloor\log_2(\frac{r}{\e^{1/2}})\rfloor$, define the ``annulus'' $E_k:=\B_{\e,R}^{r2^{-k}}\setminus \B_{\e,R}^{r2^{-k-1}}$. We set for simplicity $s_{k}:=r2^{-k}$. Applying the Hölder inequality separately on each annulus yields
\begin{multline*}
\int_{\B_{\e,R}^r}\chi_R^z|\nabla u_\e-iu_\e N_\e\!\vi_\e\!|^p
\le\Big(\int_{\B_{\e,R}^{\e^{1/2}}}\chi_R^z|\nabla u_\e-iu_\e N_\e\!\vi_\e\!|^2\Big)^{p/2}|B_{2R}(z)\cap\B_{\e,R}^{\e^{1/2}}|^{1-p/2}\\
+\sum_{k=0}^{K_\e}\Big(\int_{E_k}\chi_R^z|\nabla u_\e-iu_\e N_\e\!\vi_\e\!|^2\Big)^{p/2}|B_{2R}(z)\cap E_k|^{1-p/2}.
\end{multline*}
Using that $|B_{2R}(z)\cap\B_{\e,R}^{\e^{1/2}}|\lesssim \e$, that $|B_{2R}(z)\cap E_k|\lesssim s_{k}^2$, and that the integral over $\B_{\e,R}^{\e^{1/2}}$ in the right-hand side is bounded by $\Ec_{\e,R}^z\lesssim N_\e\Log $, we deduce
\begin{multline}\label{eq:struwedecomp}
\int_{\B_{\e,R}^r}\chi_R^z|\nabla u_\e-iu_\e N_\e\!\vi_\e\!|^p\\
\lesssim\e^{1-p/2}(N_\e\Log)^{p/2}
+\sum_{k=0}^{K_\e}s_{k}^{2-p}\Big(\int_{\R^2\setminus\B_{\e,R}^{s_{k+1}}}\chi_R^z|\nabla u_\e-iu_\e N_\e\!\vi_\e\!|^2\Big)^{p/2}.
\end{multline}
It remains to estimate the last integrals.
Using Lemma~\ref{lem:ballconstr}(i)--(ii) in the forms~\eqref{eq:lowerboundclaim1} and~\eqref{eq:boundvortnumb}, together with the optimal energy bound, we obtain
\begin{multline*}
\frac12\int_{\B_{\e,R}^{s_{k+1}}}a\chi_R^z\Big(|\nabla u_\e-iu_\e N_\e\!\vi_\e\!|^2+\frac{a}{2\e^2}(1-|u_\e|^2)^2\Big)\\
\ge \frac\Log2\int_{\R^2} a\chi_R^z\nu_{\e,R}^{s_{k+1}}-O\big(N_\e|\!\log s_{k+1}|+s_{k+1}N_\e\Log\big)-o(N_\e^2),
\end{multline*}
and hence, using~\eqref{eq:munuexch} to replace $\nu_{\e,R}^{s_{k+1}}$ by $\mu_\e$,
\begin{align*}
\frac12\int_{\R^2\setminus\B_{\e,R}^{s_{k+1}}}a\chi_R^z|\nabla u_\e-iu_\e N_\e\!\vi_\e\!|^2&\le \D_{\e,R}^z+O(N_\e|\!\log s_{k+1}|+s_{k+1}N_\e\Log)+o(N_\e^2).
\end{align*}
If $r\ll \frac{N_\e}\Log$, then $s_k\le r\ll \frac{N_\e}\Log$ for all $k$, so that we find
\begin{align}\label{eq:struweboundballcompl}
\int_{\R^2\setminus\B_{\e,R}^{s_{k+1}}}\chi_R^z|\nabla u_\e-iu_\e N_\e\!\vi_\e\!|^2&\lesssim N_\e^2+N_\e(|\!\log r|+k).
\end{align}
Inserting this into~\eqref{eq:struwedecomp} yields for all $p<2$, with $r\ll \frac{N_\e}\Log$,
\begin{eqnarray*}
\lefteqn{\int_{\B_{\e,R}^r}\chi_R^z|\nabla u_\e-iu_\e N_\e\!\vi_\e\!|^p}\\
&\lesssim&\e^{1-p/2}(N_\e\Log)^{p/2}+\sum_{k=0}^{K_\e}(r2^{-k})^{2-p}\Big(N_\e^p+N_\e^{p/2}|\!\log r|^{p/2}+N_\e^{p/2}k^{p/2}\Big)\\
&\lesssim_p&\e^{1-p/2}(N_\e\Log)^{p/2}+r^{2-p}N_\e^p+r^{2-p}N_\e^{p/2}|\!\log r|^{p/2}.
\end{eqnarray*}
In the regime $N_\e\gg\log\Log$, we may choose $e^{-o(N_\e)}\le r\ll \frac{N_\e}\Log$, and the above yields for that choice
\begin{align}\label{eq:preendproofstruwe}
\int_{\B_{\e,R}^r}\chi_R^z|\nabla u_\e-iu_\e N_\e\!\vi_\e\!|^p\ll_p N_\e^p,
\end{align}
that is, \eqref{eq:Lpstruwe}.

We now consider the regime $1\ll N_\e\lesssim\log\Log$. In that case, we need to prove~\eqref{eq:preendproofstruwe} for larger values of the radius $r\ge e^{-o(N_\e)}$, and the above argument no longer  holds. Given $\e^{1/2}< r_0\ll \frac{N_\e}\Log$, we replace the initial total radius $\e^{1/2}$ by $r_0$, and for $r_0\le r\ll1$ we consider the modified dyadic ``annuli'' $\tilde E_k:=\tilde\B_{\e,R}^{r_0,r2^{-k}}\setminus\tilde\B_{\e,R}^{r_0,r2^{-k-1}\vee r_0}$, with $0\le k\le K:=\lfloor\log_2(\frac r{r_0})\rfloor$. We set for simplicity $\tilde s_k:=(r2^{-k})\vee r_0$. The decomposition~\eqref{eq:struwedecomp} is then replaced by
\begin{multline}\label{eq:struwedecomp2}
\int_{\tilde\B_{\e,R}^{r_0,r}}\chi_R^z|\nabla u_\e-iu_\e N_\e\!\vi_\e\!|^p\\
\lesssim r_0^{2-p}(N_\e\Log)^{p/2}
+\sum_{k=0}^{K}s_{k}^{2-p}\Big(\int_{\R^2\setminus\tilde\B_{\e,R}^{r_0,\tilde s_{k+1}}}\chi_R^z|\nabla u_\e-iu_\e N_\e\!\vi_\e\!|^2\Big)^{p/2},
\end{multline}
where it remains to adapt the estimate~\eqref{eq:struweboundballcompl} for the last integrals. The lower bound~\eqref{eq:lowerboundrefpre} of Step~6 together with the optimal energy bound and with the bound of item~(ii) on the number of vortices yields
\begin{multline*}
\frac12\int_{\tilde\B_{\e,R}^{r_0,\tilde s_{k+1}}}a\chi_R^z\Big(|\nabla u_\e-iu_\e N_\e \!\vi_\e\!|^2+\frac{a}{2\e^2}(1-|u_\e|^2)^2\Big)\ge \frac{\log(\tfrac{\tilde s_{k+1}}\e)}2\int_{\R^2} a\chi_R^z \nu_{\e,R}^{r_0}-o(N_\e^2)\\
\ge \frac{\Log}2\int_{\R^2} a\chi_R^z\nu_{\e,R}^{r_0}-O(N_\e|\!\log s_{k+1}|)-o(N_\e^2),
\end{multline*}
and hence, using~\eqref{eq:munuexch} to replace $\nu_{\e,R}^{r_0}$ by $\mu_\e$,
\begin{align*}
\frac12\int_{\R^2\setminus\tilde\B_{\e,R}^{r_0,\tilde s_{k+1}}}a\chi_R^z|\nabla u_\e-iu_\e N_\e \!\vi_\e\!|^2&\le \D_{\e,R}^z+O(N_\e|\!\log s_{k+1}|+r_0N_\e\Log)+o(N_\e^2).
\end{align*}
The choice $r_0\ll \frac{N_\e}\Log$ then yields
\begin{align*}
\frac12\int_{\R^2\setminus\tilde\B_{\e,R}^{r_0,\tilde s_{k+1}}}a\chi_R^z|\nabla u_\e-iu_\e N_\e \!\vi_\e\!|^2&\lesssim N_\e^2+N_\e(|\!\log r|+k).
\end{align*}
Inserting this into~\eqref{eq:struwedecomp2}, the result~\eqref{eq:Lpstruwe2} follows.
\end{proof}

Given the above ball construction, we state the following approximation result, which is obtained as in~\cite[Proposition~9.6]{SS-book}.
\begin{lem}\label{lem:approx}
Let $\e^{1/2}<r_0\le r<\bar r$, and let $\B_{\e,R}^r$ and $\tilde\B_{\e,R}^{r_0,r}$ denote the collections of the balls constructed in Proposition~\ref{prop:ballconstr}.
Then, given $\Gamma_\e\in W^{2,\infty}(\R^2)^2$, there exist approximate vector fields $\bar\Gamma_\e,\tilde\Gamma_\e\in W^{2,\infty}(\R^2)^2$ such that $\bar\Gamma_\e$ is constant in each ball of the collection $\B_{\e,R}^r$ and $\tilde\Gamma_\e$ is constant in each ball of the collection $\tilde\B_{\e,R}^{r_0,r}$, such that $\|\bar\Gamma_\e\|_{\Ld^\infty}\le\|\Gamma_\e\|_{\Ld^\infty}$ and $\|\tilde\Gamma_\e\|_{\Ld^\infty}\le\|\Gamma_\e\|_{\Ld^\infty}$, such that for all $0\le\gamma\le1$,
\[\|\bar\Gamma_\e-\Gamma_\e\|_{C^\gamma}+\|\tilde\Gamma_\e-\Gamma_\e\|_{C^\gamma}\lesssim r^{1-\gamma}\|\nabla\Gamma_\e\|_{\Ld^\infty},\]
and such that for all $R\ge1$,
\[\sup_z\|\nabla(\bar\Gamma_\e-\Gamma_\e)\|_{\Ld^1(B_R(z))}+\sup_z\|\nabla(\tilde\Gamma_\e-\Gamma_\e)\|_{\Ld^1(B_R(z))}\lesssim rR^2\|\nabla\Gamma_\e\|_{W^{1,\infty}}.\qedhere\]
\end{lem}
\subsection{Additional results}
In order to control the velocity of the vortices, the following quantitative version of the ``product estimate'' of~\cite{SS-prod} is needed; the proof is omitted, as it is a direct adaptation of~\cite[Appendix~A]{Serfaty-15} (further deforming the metric in a non-constant way in the time direction; see also~\cite[Section~III]{SS-prod} and~\cite[Theorem~1.3]{Roman-17}).

\begin{lem}[Product estimate]\label{lem:prodest}
Denote by $M_\e$ any quantity such that for all $q>0$,
\[\lim_{\e\downarrow0}\e^qM_\e=\lim_{\e\downarrow0}\Log M_\e^{-q}=\lim_{\e\downarrow0}\Log ^{-1}\log M_\e=0.\]
Let $u_\e:[0,T]\times\R^2\to\C$, $\vi_\e:[0,T]\times\R^2\to\R^2$, and $\pre_{\e}:[0,T]\times\R^2\to\R$.
Assume that $\Ec_{\e,R}^{*,t}\lesssim \Log ^2$ for all $t$ and that $\bar\Ec_{\e,R}^{*,T}\le M_\e$, where we have set
\[\bar\Ec_{\e,R}^{*,T}:=\sup_z\int_0^T\bigg(\Ec_{\e,R}^{z,t}+\int_{\R^2}\chi_R^z|\partial_tu_\e^t-iu_\e^t N_\e\!\pre_{\e}^t\!|^2\bigg)dt.\]
Then, for all $X\in W^{1,\infty}([0,T]\times\R^2)^2$ and $Y\in W^{1,\infty}([0,T]\times\R^2)$, we have for all $z\in \R^2$,
\begin{align*}
&\bigg|\int_0^T\int_{\R^2}\chi_R^z \tilde V_{\e}\cdot XY\bigg|\\
&\hspace{0.5cm}\le\frac{1+C\frac{\log M_\e}\Log}{\Log }\bigg(\int_0^T\int_{\R^2}\chi_R^z|(\partial_tu_\e-iu_\e N_\e\!\pre_{\e})Y|^2\\
&\hspace{7cm}+\int_0^T\int_{\R^2}\chi_R^z|(\nabla u_\e-iu_\e N_\e\!\vi_\e)\cdot X|^2\bigg)\\
&\hspace{1cm}+C\big(1+\|(X,Y)\|_{W^{1,\infty}([0,T]\times \R^2)}^5\big)\big(M_\e^{-1/8}+\e N_\e\big)\Big(\bar\Ec_{\e,R}^{*,T}+\sup_{0\le \tau\le T}\Ec_{\e,R}^{*,\tau}+N_\e^2\Big).\qedhere
\end{align*}
\end{lem}

We now turn to some useful a priori estimates on the solution $u_\e$ of equation~\eqref{eq:GL-1}. We start with the following (suboptimal) a priori bound on the velocity of the vortices, adapted from~\cite[Lemma~4.1]{Serfaty-15}.

\begin{lem}[A priori bound on velocity]\label{lem:velapbound}
Let $\alpha\ge0$, $\beta\in\R$, and let $h:\R^2\to\R$, $a:=e^h$, $F:\R^2\to\R^2$, $f:\R^2\to\R$ satisfy~\eqref{eq:scalingshFf}. Let $u_\e:\R^+\times\R^2\to\C$ and $\vi_\e:[0,T)\times\R^2\to\R^2$ be the solutions of~\eqref{eq:GL-1} and~\eqref{eq:GLv1} as in Proposition~\ref{prop:globGL}(i) and in Assumption~\ref{as:apveps}(a),
respectively, for some $T>0$.
Let $0<\e\ll1$, $1\le N_\e\lesssim \e^{-1}$, and $R\ge1$ with $\e R^\theta\ll1$ for some $\theta>0$, and assume that $\Ec_{\e,R}^{*,t}\lesssim_t N_\e\Log$ for all $t$. Then, in each of the considered regimes~\GLu, \GLd, \GLt, \GLup, and~\GLdp, we have for all $\theta>0$ and $t\in[0,T)$,
\begin{align*}
\alpha^2 \sup_z\int_0^t\int_{\R^2} a\chi_R^z|\partial_tu_\e|^2&\lesssim_{t,\theta} N_\e\Log^3+R^\theta N_\e^2\Log^2\\
&\lesssim R^\theta N_\e(N_\e+\Log)\Log^2.\qedhere
\end{align*}
\end{lem}

\begin{proof}
Integrating identity~\eqref{eq:timederenergy} in time, reorganizing the terms, and setting $D_{\e,R}^{z,t}:=\int_0^t\int_{\R^2} a\chi_R^z|\partial_tu_\e|^2$, we obtain
\begingroup\allowdisplaybreaks
\begin{multline}\label{eq:1st-est-velapbound}
\lambda_\e\alpha D_{\e,R}^{z,t}=\hat\Ec_{\e,R}^{z,\circ}-\hat\Ec_{\e,R}^{z,t}\\
- \int_0^t\int_{\R^2} a\nabla \chi_R^z\cdot\langle \partial_t u_\e,\nabla u_\e-iu_\e N_\e \!\vi_\e\rangle+\int_0^t\int_{\R^2} N_\e\chi_R^z\langle \partial_tu_\e,iu_\e\rangle \Div(a\!\vi_\e)\\
+\int_{\R^2} \frac{aN_\e^2}2(1-|u_\e^t|^2)(\psi_{\e,R}^{z,t}-\chi_R^z|\!\vi_\e^t\!|^2)-\int_{\R^2} \frac{aN_\e^2}2(1-|u_\e^\circ|^2)(\psi_{\e,R}^{z,\circ}-\chi_R^z|\!\vi_\e^\circ\!|^2)\\
+\int_0^t\int_{\R^2} a\chi_R^z\Big(N_\e(N_\e \!\vi_\e-j_\e)\cdot\partial_t\!\vi_\e-N_\e \!\vi_\e\!\cdot\, V_\e-\frac{\Log }2F^\bot\cdot V_\e\Big).
\end{multline}
\endgroup
Noting that $|\nabla\chi_R^z|\lesssim R^{-1}(\chi_R^z)^{1/2}$, using the pointwise estimates of Lemma~\ref{lem:pointest} for $V_\e$ and $j_\e-N_\e\!\vi_\e$,
and using assumptions~\eqref{eq:scalingshFf}, the properties of $\vi_\e$ in Assumption~\ref{as:apveps}(a), the bound~\eqref{eq:boundpsi} on $\psi_{\e,R}^z$, and Lemma~\ref{lem:apestu} in the form $\hat\Ec_{\e,R}^{z,t}\lesssim \Ec_{\e,R}^{*,t}+o(N_\e^2)\lesssim_t N_\e\Log$, we find for $\theta>0$ small enough, in the considered regimes,
\begingroup\allowdisplaybreaks
\begin{eqnarray*}
\lefteqn{\lambda_\e\alpha D_{\e,R}^{z,t}}\\
&\lesssim_{t,\theta}& N_\e\Log+ R^{-1}(N_\e\Log)^{1/2}(D_{\e,R}^{z,t})^{1/2}+N_\e(1+\e(N_\e\Log)^{1/2})(D_{\e,R}^{z,t})^{1/2}\\
&&\qquad+\e N_\e^2(N_\e\Log)^{1/2}\Big(1+\frac\Log{N_\e}(\lambda_\e R^\theta+1\wedge\lambda_\e^{1/2}+R^{-1+\theta})\Big)\\
&&\qquad+N_\e(N_\e\Log)^{1/2}(1+\e N_\e)+\e \lambda_\e^{-1/2}N_\e^2\Log\\
&&\qquad+(N_\e+\lambda_\e\Log)\big((1+\e N_\e)(N_\e\Log)^{1/2}+N_\e R^\theta\big)(D_{\e,R}^{z,t})^{1/2}\\
&\lesssim_\theta& \Log(N_\e+\Log)+\big(N_\e\Log R^\theta+\Log(N_\e\Log)^{1/2}\big)(D_{\e,R}^{z,t})^{1/2}+o(1).
\end{eqnarray*}
\endgroup
Absorbing $(D_{\e,R}^{z,t})^{1/2}$ in the left-hand side, the result follows.
\end{proof}

The following optimal a priori estimate is also crucially needed in our analysis in the presence of pinning, due to the absence of a factor $\frac12$ in front of the quantity $\frac a{\e^2}(1-|u_\e|^2)^2$ as it appears in the term $I_{\e,\varrho,R}^H$ in Lemma~\ref{lem:decompcruc}.
A simple computation based on the energy lower bound in Proposition~\ref{prop:ballconstr} yields a similar bound with $N_\e$ replaced by $N_\e^2$ (cf.~indeed~\eqref{eq:boundprepoho}), but the optimal result below is much more subtle.
It is proved as a combination of the Pohozaev ball construction of~\cite[Section~5]{SS-book} together with some careful cut-off techniques inspired by~\cite[Proof of Proposition~13.4]{SS-book}.

\begin{lem}\label{lem:poho}
Let $\alpha\ge0$, $\beta\in\R$, and let $h:\R^2\to\R$, $a:=e^h$, $F:\R^2\to\R^2$, $f:\R^2\to\R$ satisfy~\eqref{eq:scalingshFf}. Let $u_\e:\R^+\times\R^2\to\C$ and $\vi_\e:[0,T)\times\R^2\to\R^2$ be the solutions of~\eqref{eq:GL-1} and~\eqref{eq:GLv1} as in Proposition~\ref{prop:globGL}(i) and in Assumption~\ref{as:apveps}(a), respectively, for some $T>0$.
Let $0<\e\ll1$, $1\le N_\e\lesssim\Log$, and $R\ge1$ with $\e R\Log^3\lesssim1$, and assume that $\Ec_{\e,R}^{*,t}\lesssim_t N_\e\Log$ for all $t$. Then, in the nondegenerate dissipative case, in each of the considered regimes~\GLu, \GLd, \GLup, and~\GLdp, we have for all $t\in[0,T)$,
\begin{equation}\label{eq:pohoconcl}
\alpha^2\sup_z\int_0^t\int_{\R^2}\frac{\chi_R^z}{\e^2}(1-|u_\e|^2)^2\lesssim_t N_\e.
\qedhere
\end{equation}
\end{lem}

\begin{proof}
To simplify notation, we focus on the case $z=0$, but the result of course holds uniformly with respect to the translation $z\in R\Z^2$. We split the proof into three steps.

\medskip
\noindent\step1 Pohozaev estimate on balls.
\nopagebreak

In this step, we prove the following Pohozaev-type estimate, adapted from~\cite[Theorem~5.1]{SS-book}: for any ball $B_r(x_0)$ with $r\le1$, we have
\begin{multline}
\alpha^2\int_0^t\int_{B_r(x_0)} \frac {a^2\chi_R}{2\e^2}(1-|u_\e|^2)^2\\
\lesssim_{t} r\lambda_\e N_\e\Log ^3
+r\int_0^t\int_{\partial B_r(x_0)}\frac{a\chi_R}2\Big(|\nabla u_\e-iu_\e N_\e\!\vi_\e\!|^2+\frac{a}{2\e^2}(1-|u_\e|^2)^2\\
+|1-|u_\e|^2|(N_\e^2|\!\vi_\e\!|^2+|f|)\Big).\label{eq:poho}
\end{multline}
For any smooth vector field $X$ and any bounded open set $U\subset\R^2$, we find by integration by parts
\begin{align*}
-\int_U\chi_R\nabla X:\tilde S_\e=\int_U\chi_R\Div\!\tilde S_\e\cdot X+\int_UX\cdot\tilde S_\e\cdot\nabla\chi_R
-\int_{\partial U}\chi_RX\cdot\tilde S_\e\cdot n,
\end{align*}
and hence, for $U= B_r(x_0)$, $r>0$, and $X= x-x_0$,
\begin{multline*}
-\int_{B_r(x_0)} \chi_R\Tr\tilde S_\e\\
=\int_{B_r(x_0)} \chi_R\Div\!\tilde S_\e\cdot(x-x_0)+\int_{B_r(x_0)}(x-x_0)\cdot\tilde S_\e\cdot\nabla\chi_R
-r\int_{\partial B_r(x_0)}\chi_R\,\tilde S_\e:n\otimes n.
\end{multline*}
By definition~\eqref{eq:defSmod} of the modulated stress-energy tensor $\tilde S_\e$, this means
\begin{multline*}
\int_{B_r(x_0)}a\chi_R\Big(\frac {a}{2\e^2}(1-|u_\e|^2)^2+(1-|u_\e|^2)f\Big)\\
=\int_{B_r(x_0)}\chi_R\Div\!\tilde S_\e\cdot(x-x_0)+\int_{B_r(x_0)}(x-x_0)\cdot\tilde S_\e\cdot\nabla\chi_R\\
+r\int_{\partial B_r(x_0)}\frac {a\chi_R}2\Big(|n^\bot\cdot(\nabla u_\e-iu_\e N_\e \!\vi_\e)|^2-|n\cdot(\nabla u_\e-iu_\e N_\e \!\vi_\e)|^2\\+\frac{a}{2\e^2}(1-|u_\e|^2)^2+(1-|u_\e|^2)\big(N_\e^2(|n^\bot\cdot \!\vi_\e\!|^2-|n\cdot \!\vi_\e\!|^2)+f\big)\Big),
\end{multline*}
so that we may simply estimate
\begin{multline}\label{eq:poho1}
\int_{B_r(x_0)} \frac {a^2\chi_R}{2\e^2}(1-|u_\e|^2)^2\\
\le r\int_{B_r(x_0)} |\!\Div\!\tilde S_\e|+r\int_{B_r(x_0)}|\nabla\chi_R||\tilde S_\e|+\int_{B_r(x_0)}a|1-|u_\e|^2||f|\\
+r\int_{\partial B_r(x_0)}\frac {a\chi_R}2\Big(|\nabla u_\e-iu_\e N_\e \!\vi_\e\!|^2+\frac{a}{2\e^2}(1-|u_\e|^2)^2\\
+|1-|u_\e|^2|(N_\e^2|\!\vi_\e\!|^2+|f|)\Big).
\end{multline}
It remains to estimate the first three right-hand side terms. Using the pointwise estimates of Lemma~\ref{lem:pointest}, and using assumption~\eqref{eq:scalingshFf} and the boundedness properties of $\vi_\e,\pre_{\e}$ in Assumption~\ref{as:apveps}(a), Lemma~\ref{lem:divS} directly yields in the considered regimes,
\begin{multline*}
|\!\Div\!\tilde S_\e|\lesssim \lambda_\e\Log |\partial_tu_\e||\nabla u_\e-iu_\e N_\e \!\vi_\e\!|\\
+N_\e(1+\lambda_\e^{1/2}\Log) (1+|1-|u_\e|^2|)|\nabla u_\e-iu_\e N_\e \!\vi_\e\!|\\
+{\lambda_\e}N_\e\Log |\partial_tu_\e|(1+|1-|u_\e|^2|)
+(N_\e+\lambda_\e\Log)|\nabla u_\e-iu_\e N_\e \!\vi_\e\!|^2\\
+\e^{-2}(1-|u_\e|^2)^2
+|1-|u_\e|^2|\big(N_\e^2(N_\e+\lambda_\e\Log)+\lambda_\e^2\Log^2\big)+N_\e^2(N_\e+\lambda_\e\Log),
\end{multline*}
which gives for $N_\e\lesssim\Log$,
\begin{multline*}
|\!\Div\!\tilde S_\e|\lesssim\lambda_\e|\partial_tu_\e|^2+\lambda_\e\Log^2|\nabla u_\e-iu_\e N_\e\!\vi_\e\!|^2\\
+\lambda_\e N_\e^2\Log^2(1+(1-|u_\e|^2)^2)+\e^{-2}(1-|u_\e|^2)^2.
\end{multline*}
By Lemma~\ref{lem:velapbound} with $R=1$, we deduce for all $r\le1$,
\begin{align*}
\alpha^2\int_0^t\int_{B_r(x_0)}|\!\Div \tilde S_\e|&\lesssim_{t} \lambda_\e N_\e\Log ^3+\lambda_\e N_\e^2\Log^2(1+\e^2N_\e\Log)\lesssim \lambda_\e N_\e\Log ^3.
\end{align*}
Inserting this into~\eqref{eq:poho1}, and noting that~\eqref{eq:scalingshFf} in the form $\|f\|_{\Ld^\infty}\lesssim\Log^2$ yields
\[\int_{B_r(x_0)} a|1-|u_\e|^2||f|\lesssim_t \e r(N_\e\Log )^{1/2}\|f\|_{\Ld^\infty}\lesssim \e r\Log^3,\]
and
\begin{multline*}
\int_{B_r(x_0)}|\nabla\chi_R||\tilde S_\e|\\
\lesssim R^{-1}\int_{B_r(x_0)}\Big(|\nabla u_\e-iu_\e N_\e \!\vi_\e\!|^2+\frac{1}{\e^2}(1-|u_\e|^2)^2+\e^2(N_\e^4|\!\vi_\e\!|^4+|f|^2)\Big)\\
\lesssim R^{-1}\big(N_\e\Log+\e^2(N_\e^4+\|f\|_{\Ld^\infty}^2)\big)\lesssim N_\e\Log,
\end{multline*}
the result~\eqref{eq:poho} follows.

\medskip
\noindent\step2 Estimate inside small balls.

In this step, we prove the desired estimate~\eqref{eq:pohoconcl} for the integral restricted to suitable small balls centered at the vortex locations.
More precisely, since we have by assumption $\Ec_{\e,R}^*\lesssim N_\e\Log \lesssim \Log ^2$, we may apply~\cite[Proposition~4.8]{SS-book} with $M=\e^{\kappa-1}$ and $\delta=\e^{\kappa/4}$ for any $\kappa\in(0,1)$.
This yields a finite union $\hat\B_{\e,0}$ of disjoint closed balls with total radius $r(\hat\B_{\e,0})=\e^{\kappa/2}$, covering the set $\{x\in B_{2R}:||u_\e(x)|-1|\ge\e^{\kappa/4}\}$. We then prove that
\begin{align}\label{eq:poho-concl}
\alpha^2\int_0^t\int_{\hat\B_{\e,0}}\frac{a^2\chi_R}{2\e^2}(1-|u_\e|^2)^2&\lesssim_t N_\e.
\end{align}

For that purpose, we let the initial collection of balls $\hat\B_{\e,0}$ grow, and we use the Pohozaev estimate of Step~1 as in~\cite[Proof of Theorem~5.1]{SS-book}.
By~\cite[Theorem~4.2]{SS-book}, there exists a monotone family $(\hat\B_\e^s)_{s\ge0}$ of unions of disjoint closed balls, such that $\hat\B_\e^0=\hat\B_{\e,0}$, $\hat\B_\e^s$ has total radius $r(\hat\B_\e^s)=e^sr(\hat\B_{\e,0})$ for all $s\ge0$,
and $\hat\B_\e^s=e^{s-r}\hat\B_{\e}^r$ for all $0\le r\le s$ with $[r,s]\subset\R^+\setminus \Tc_\e$, for some finite set $\Tc_\e\subset\R^+$ (corresponding to the merging times in the growth process).
For all $s\ge0$ with $r(\hat\B_\e^s)\le1$, the result~\eqref{eq:poho} of Step~1 gives the following estimate,
\begin{multline*}
\alpha^2\int_0^t\int_{\hat\B_\e^s} \frac {a^2\chi_R}{2\e^2}(1-|u_\e|^2)^2\\
\lesssim_{t} r(\hat\B_\e^s) N_\e\Log ^3
+\sum_{B_r(x)\in\hat\B_\e^s}r\int_0^t\int_{\partial B_r(x)}\frac {a\chi_R}2\Big(|\nabla u_\e-iu_\e N_\e \!\vi_\e\!|^2+\frac{a}{2\e^2}(1-|u_\e|^2)^2\\
+|1-|u_\e|^2|(N_\e^2|\!\vi_\e\!|^2+f)\Big).
\end{multline*}
Integrating this estimate over $s$ and applying~\cite[Proposition~4.1]{SS-book}, we find, for all $s\ge0$ with $r(\hat\B_{\e}(s))\le1$,
\begin{multline*}
s\alpha^2\int_0^t\int_{\hat\B_{\e,0}}\frac{a^2\chi_R}{2\e^2}(1-|u_\e|^2)^2\le\alpha^2\int_0^sdv\int_0^t\int_{\hat\B_\e^v}\frac{a^2\chi_R}{2\e^2}(1-|u_\e|^2)^2\\
\lesssim_{t} s\,r(\hat\B_\e^s) N_\e\Log ^3+\int_0^t\int_{\hat\B_\e^s\setminus\hat\B_{\e,0}}\frac{a\chi_R}2\Big(|\nabla u_\e-iu_\e N_\e \!\vi_\e\!|^2+\frac{a}{2\e^2}(1-|u_\e|^2)^2\\
+|1-|u_\e|^2|(N_\e^2|\!\vi_\e\!|^2+f)\Big),
\end{multline*}
and hence, using assumption~\eqref{eq:scalingshFf}, the boundedness of $\vi_\e$ in Assumption~\ref{as:apveps}(a), and the assumed energy bound, 
\begin{align*}
s\alpha^2\int_0^t\int_{\hat\B_{\e,0}}\frac{a^2\chi_R}{2\e^2}(1-|u_\e|^2)^2&\lesssim_{t} s\,r(\hat\B_\e^s)N_\e\Log ^3+N_\e\Log.
\end{align*}
Recalling that $r(\hat\B_\e^s)=e^s\e^{\kappa/2}$, this yields for all $s\ge1$ with $r(\hat\B_\e^s)\le1$,
\begin{align*}
\alpha^2\int_0^t\int_{\hat\B_{\e,0}}\frac{a^2\chi_R}{2\e^2}(1-|u_\e|^2)^2&\lesssim_{t} e^s\e^{\kappa/2} N_\e\Log ^3+\frac{N_\e\Log }{s},
\end{align*}
and the result~\eqref{eq:poho-concl} now follows for the choice $s=|\!\log \e^{\kappa/4}|$.

\medskip
\noindent\step3 Estimate outside small balls.

It remains to show that the desired estimate~\eqref{eq:pohoconcl} also holds for the integral restricted to the complement of the  small balls $\hat\B_{\e,0}$. More precisely, we prove that for all $\theta>0$,
\begin{align}\label{eq:poho-compl}
\alpha\int_0^t\int_{||u_\e|-1|\le\e^{\kappa/4}}\chi_R\Big( |\nabla|u_\e||^2+\frac{a(1-|u_\e|^2)^2}{2\e^2}\Big)\lesssim_{t,\theta}
\e^{\kappa/4}R^\theta\Log^2+\e R\Log^3.
\end{align}
The conclusion~\eqref{eq:pohoconcl} follows from this together with~\eqref{eq:poho-concl}, for $\theta>0$ small enough.

In order to prove~\eqref{eq:poho-compl}, we adapt the argument of~\cite[Proof of Proposition~13.4]{SS-book}.
For $0<\e\le2^{-4/\kappa}$, we define a cut-off function $\zeta_\e$ as follows,
\[\zeta_\e(y):=\begin{cases}y,&\text{if $0\le y\le \frac12$;}\\\frac12+\frac{y-\frac12}{1-2\e^{\kappa/4}},&\text{if $\frac12\le y\le 1-\e^{\kappa/4}$};\\1,&\text{if $1-\e^{\kappa/4}\le y\le 1+\e^{\kappa/4}$};\\1+\frac{y-1-\e^{\kappa/4}}{1-2\e^{\kappa/4}},&\text{if $1+\e^{\kappa/4}\le y\le \frac32$};\\y,&\text{if $y\ge\frac32$}.\end{cases}\]
Writing $u_\e:=\rho_\e e^{i\varphi_\e}$ locally, equation~\eqref{eq:GL-1} for $u_\e$ implies in particular
\begin{multline}\label{eq:GLrhoeps}
\alpha\lambda_\e\partial_t\rho_\e-\beta\lambda_\e\Log \rho_\e\partial_t\varphi_\e\\
=\triangle \rho_\e-\rho_\e|\nabla\varphi_\e|^2+\frac{a\rho_\e}{\e^2}(1-\rho_\e^2)+\nabla h\cdot\nabla \rho_\e-\rho_\e\Log F^\bot\cdot\nabla\varphi_\e+f\rho_\e.
\end{multline}
Testing this equation against $\chi_R(\zeta_\e(\rho_\e)-\rho_\e)$ and rearranging the terms,
we obtain
\begin{multline}\label{st1}
\int_{\R^2} \chi_R(1-\zeta_\e'(\rho_\e))|\nabla\rho_\e|^2+\int_{\R^2}\frac{a\chi_R}{\e^2}\rho_\e(\zeta_\e(\rho_\e)-\rho_\e)(1-\rho_\e^2)\\
=\alpha\lambda_\e\int_{\R^2}\chi_R(\zeta_\e(\rho_\e)-\rho_\e)\partial_t\rho_\e
-\beta\lambda_\e\Log \int_{\R^2} \chi_R\rho_\e(\zeta_\e(\rho_\e)-\rho_\e)\partial_t\varphi_\e\\
+\int_{\R^2}(\zeta_\e(\rho_\e)-\rho_\e)\nabla\chi_R\cdot\nabla\rho_\e+\int_{\R^2}\chi_R(\zeta_\e(\rho_\e)-\rho_\e)\rho_\e|\nabla\varphi_\e|^2\\
-\int_{\R^2}\chi_R(\zeta_\e(\rho_\e)-\rho_\e)\nabla h\cdot\nabla \rho_\e
+\Log \int_{\R^2}\chi_R\rho_\e (\zeta_\e(\rho_\e)-\rho_\e)F^\bot\cdot\nabla\varphi_\e\\
-\int_{\R^2} \chi_R(\zeta_\e(\rho_\e)-\rho_\e)f\rho_\e.
\end{multline}
Using that the cut-off function $\zeta_\e$ satisfies for all $y\ge0$,
\begin{gather*}
|\zeta_\e(y)-y|\lesssim \e^{\kappa/4}\mathds1_{|y-1|\le\frac12}, \qquad |\zeta_\e(y)-y|\le|1-y|\le|1-y^2|,\\
|\zeta_\e'(y)-1|\lesssim \mathds1_{|y-1|\le\e^{\kappa/4}}+\e^{\kappa/4}\mathds1_{|y-1|\le\frac12},\qquad(\zeta_\e(y)-y)(1-y)\ge0,
\end{gather*}
noting that
\begin{align*}
\int_{|\rho_\e-1|\le\e^{\kappa/4}}\frac{a\chi_R}{5\e^2}(1-\rho_\e^2)^2&\le\int_{|\rho_\e-1|\le\e^{\kappa/4}}\frac{a\chi_R}{\e^2}\rho_\e(1-\rho_\e)(1-\rho_\e^2)\\
&\le\int_{\R^2}\frac{a\chi_R}{\e^2}\rho_\e(\zeta_\e(\rho_\e)-\rho_\e)(1-\rho_\e^2),
\end{align*}
and using~\eqref{eq:scalingshFf}, we deduce from~\eqref{st1},
\begin{multline*}
\int_{|\rho_\e-1|\le\e^{\kappa/4}} \chi_R\Big(|\nabla\rho_\e|^2+\frac{a}{2\e^2}(1-\rho_\e^2)^2\Big)\lesssim\e^{\kappa/4}\int_{|\rho_\e-1|\le1/2} \chi_R(|\nabla\rho_\e|^2+\rho_\e^2|\nabla\varphi_\e|^2)\\
+\lambda_\e\Log\int_{|\rho_\e-1|\le1/2}\chi_R|1-\rho_\e^2|(|\partial_t\rho_\e|+\rho_\e|\partial_t\varphi_\e|)\\
+(1+\lambda_\e\Log)\int_{|\rho_\e-1|\le1/2}\chi_R|1-\rho_\e^2|(|\nabla \rho_\e|+\rho_\e|\nabla\varphi_\e|)\\
+\int_{|\rho_\e-1|\le1/2} \chi_R|f||1-\rho_\e^2|+\int_{|\rho_\e-1|\le1/2}|\nabla\chi_R||1-\rho_\e^2||\nabla\rho_\e|.
\end{multline*}
Noting that
$|\nabla u_\e|^2=|\nabla\rho_\e|^2+\rho_\e^2|\nabla\varphi_\e|^2$ and $|\partial_t u_\e|^2=|\partial_t\rho_\e|^2+\rho_\e^2|\partial_t\varphi_\e|^2$, and using~\eqref{eq:scalingshFf}, we obtain
\begin{multline*}
\int_{|\rho_\e-1|\le\e^{\kappa/4}} \chi_R\Big(|\nabla|u_\e||^2+\frac{a}{2\e^2}(1-|u_\e|^2)^2\Big)\\
\lesssim\e^{\kappa/4}\|\nabla u_\e\|_{\Ld^2(B_{2R})}^2+\lambda_\e\Log\|1-|u_\e|^2\|_{\Ld^2(B_{2R})}\|\partial_tu_\e\|_{\Ld^2(B_{2R})}\\
+(1+\lambda_\e\Log)\|1-|u_\e|^2\|_{\Ld^2(B_{2R})}\|\nabla u_\e\|_{\Ld^2(B_{2R})}\\
+R(1+\lambda_\e^2\Log^2)\|1-|u_\e|^2\|_{\Ld^2(B_{2R})}.
\end{multline*}
By the integrability properties of $\vi_\e$ in Assumption~\ref{as:apveps}(a), we have for all $\theta>0$,
\begin{gather*}
\|\nabla u_\e\|_{\Ld^2(B_{2R})}\lesssim_\theta\|\nabla u_\e-iu_\e N_\e\!\vi_\e\!\|_{\Ld^2(B_{2R})}+N_\e(R^\theta+\|1-|u_\e|^2\|_{\Ld^2(B_{2R})}),
\end{gather*}
hence, by Lemma~\ref{lem:velapbound} and the energy bound,
\begin{align*}
\alpha\int_0^t\int_{|\rho_\e-1|\le\e^{\kappa/4}} \chi_R\Big(|\nabla|u_\e||^2+\frac{a}{2\e^2}(1-|u_\e|^2)^2\Big)\lesssim_{t,\theta}\e^{\kappa/4}R^\theta\Log^2+\e R\Log^3,
\end{align*}
and the result~\eqref{eq:poho-compl} follows.
\end{proof}

\section{Mean-field limit in the dissipative case}\label{chap:MFL-GL}

In this section we prove Theorem~\ref{th:mainGL}, that is, the mean-field limit result in the dissipative mixed-flow case ($\alpha>0$) in the regimes \GLu, \GLd,
\GLup, and~\GLdp.
More precisely, we establish the following result, which states that the rescaled supercurrent density $\frac1{N_\e}j_\e$ remains close to the solution $\vi_\e$ of equation~\eqref{eq:GLv1}.
Combining this with the results of Section~\ref{chap:dissip-lim} (in particular, with Lemma~\ref{lem:lastlimGL}), the result of Theorem~\ref{th:mainGL} follows.
The proof consists in making use of the various estimates and technical tools for vortex analysis developed in Section~\ref{sec:vortex} in order to estimate the terms in the decomposition of $\partial_t\hat\D_{\e,R}$ in Lemma~\ref{lem:decompcruc}, and then deduce the smallness of the modulated energy excess $\hat\D_{\e,R}$ by a Grönwall argument.
(In this section, as we assume $\alpha>0$, all multiplicative constants are implicitly allowed to additionally depend on an upper bound on $\alpha^{-1}$.)

\begin{prop}\label{prop:mflGL}
Let $\alpha>0$, $\beta\in\R$, $\alpha^2+\beta^2=1$, and let $h:\R^2\to\R$, $a:=e^h$, $F:\R^2\to\R^2$, $f:\R^2\to\R$ satisfy~\eqref{eq:scalingshFf}. Let $u_\e:\R^+\times\R^2\to\C$ and $\vi_\e:[0,T)\times\R^2\to\R^2$ be solutions of~\eqref{eq:GL-1} and~\eqref{eq:GLv1} as in Propositions~\ref{prop:globGL}(i) and~\ref{prop:GLvprop}, respectively, for some $T>0$. Let $0<\e\ll1$, $1\ll N_\e\lesssim\Log$, $R\ge1$, $\frac\Log{N_\e}\ll R\lesssim \Log^n$, for some $n\ge1$, and assume that the initial modulated energy excess satisfies $\D_{\e,R}^{*,\circ}\ll N_\e^2$.
Then,
\begin{enumerate}[(i)]
\item If $\log\Log\ll N_\e\lesssim\Log$, in each of the regimes~\GLu, \GLd, \GLup, and~\GLdp, we have $\D_{\e,R}^{*,t}\ll_t N_\e^2$ for all $t\in[0,T)$.
\item If $1\ll N_\e\lesssim\log\Log$, in the parabolic case ($\alpha=1$, $\beta=0$), either in the regime~\GLu, or in the regime~\GLdp{} with $\lambda_\e\lesssim \frac{e^{o(N_\e)}}\Log$, the same conclusion $\D_{\e,R}^{*,t}\ll_t N_\e^2$ holds for all $t\in[0,T)$.
\end{enumerate}
In particular, in both cases, we deduce $\frac1{N_\e}j_\e-\vi_\e\to0$ in $\Ld^\infty_\loc([0,T);\Ld^1_\uloc(\R^2)^2)$ as $\e\downarrow0$.
If we further assume $\D_{\e,\infty}^{*,\circ}\ll N_\e^2$, then for any $\ell\ge1$ we obtain more precisely for all $t\in[0,T)$ and $L\ge1$,
\begin{equation}\label{eq:boundL1L2jNv}
\sup_z\|\tfrac1{N_\e}j_\e-\vi_\e\!\|_{(\Ld^1+\Ld^2)(B_L(z))}\ll_{t,\ell} \Big(1+\frac L{\Log^\ell}\Big)^2.
\qedhere
\end{equation}
\end{prop}

\begin{rem}\label{rems:GLMFL}
If we further assume $\|u_\e^t\|_{\Ld^\infty}\lesssim_t1$ for all $t$, then the proof shows that the convergence $\frac1{N_\e}j_\e-\vi_\e\to0$ actually holds in $\Ld^\infty_\loc([0,T);\Ld^p_\uloc(\R^2)^2)$ for all $p<2$.
In the parabolic case without applied current ($F\equiv0$, $f\equiv0$), a maximum principle type argument gives that $\|u_\e^\circ\|_{\Ld^\infty}\le1$ implies $\|u_\e^t\|_{\Ld^\infty}\le1$ for all $t\ge0$ (cf.\@ e.g.~\cite[Proposition~4.4]{Chen-Hoffmann-Liang-93}).
However, the same argument fails in the presence of an applied current.
Moreover, such a uniform $\Ld^\infty$-bound on $u_\e$ is expected to fail in the conservative case due to the time reversibility of the equation in that case, and similarly it is expected to fail as well in the parabolic mixed-flow case. We therefore systematically avoid the use of such $\Ld^\infty$-estimates.
\end{rem}

\begin{proof}[Proof of Proposition~\ref{prop:mflGL}]
We choose $R\gg\frac\Log{N_\e}$ with $R^{\theta_0}\lesssim\Log$ for some $\theta_0>0$.
Given the assumption $\D_{\e,R}^{*,\circ}\ll N_\e^2$ on the initial data, for all $\e>0$ we define $T_\e>0$ as the maximum time $\le T$ such that $\D_{\e,R}^{*,t}\le N_\e^2$ holds for all $t\le T_\e$.
By Lemma~\ref{lem:apestu} and Proposition~\ref{prop:ballconstr}, we deduce $\hat\D_{\e,R}^{*,\circ}\ll N_\e^2$ and for all $t\le T_\e$,
\begin{align}\label{eq:corTeps}
\Ec_{\e,R}^{*,t}\lesssim_t N_\e\Log,\quad \hat\Ec_{\e,R}^{*,t}\lesssim_tN_\e\Log,\quad \hat\D_{\e,R}^{*,t}\lesssim_t N_\e^2,\quad \D_{\e,R}^{*,t}\lesssim \hat\D_{\e,R}^{*,t}+o_t(N_\e^2).
\end{align}
The strategy of the proof consists in showing that for all $t\le T_\e$,
\begin{align}\label{eq:stratprGL}
\hat\D_{\e,R}^{*,t}\lesssim_t o(N_\e^2)+\int_0^t\hat\D_{\e,R}^{*}.
\end{align}
By the Grönwall inequality, this implies $\hat\D_{\e,R}^{*,t}\ll_t N_\e^2$, hence $\D_{\e,R}^{*,t}\ll_tN_\e^2$ for all $t\le T_\e$. This gives in particular $T_\e=T$ for all $\e>0$ small enough and the main conclusion follows.

To simplify notation, we focus on~\eqref{eq:stratprGL} with the left-hand side $\hat\D_{\e,R}^t$ centered at $z=0$, but the result of course holds uniformly with respect to the translation.
We start with the general mixed-flow case in the regime $\log\Log\ll N_\e\lesssim\Log$. The proof of~\eqref{eq:stratprGL} in that case is split into three steps, while the additional statements are deduced in Step~4. Finally, Step~5 describes the modifications needed in the parabolic case for $1\ll N_\e\lesssim\log\Log$.

Let us first introduce some notation.
In the regime $\log\Log\ll N_\e\lesssim\Log$, for all $t\le T_\e$, as we are in the framework of Proposition~\ref{prop:ballconstr} with $u_\e^t,\vi_\e^t$, we let $\B_\e^t:=\B_{\e,R}^t$ denote the constructed collection of disjoint closed balls $\B_{\e,R}^{r_\e}(u_\e^t,\vi_\e^t)$ with total radius $r_\e:=\Log^{-4}e^{-\sqrt{N_\e}}$, hence $e^{-o(N_\e)}\le r_\e\ll \frac{N_\e}\Log$. Let then $\bar\Gamma_\e^t$ denote the corresponding approximation of $\Gamma_\e^t$ given by Lemma~\ref{lem:approx}.
We decompose $\Gamma_\e:=\alpha\Gamma_{\e,0}-\beta\Gamma_{\e,0}^\bot$ with
\[\Gamma_{\e,0}:=\lambda_\e^{-1}\Big(\nabla^\bot h-F^\bot-\frac{2N_\e}\Log\!\vi_\e\Big).\]

\medskip
\noindent\step1 Time derivative of the modulated energy excess.

Lemma~\ref{lem:decompcruc} yields the following decomposition,
\begin{align}\label{eq:decompGL0}
\partial_t\hat\D_{\e,R}=~&I_{\e,R}^S+I_{\e,R}^V+I_{\e,R}^E+I_{\e,R}^D+I_{\e,R}^H+I_{\e,R}^d+I_{\e,R}^g+I_{\e,R}^n+I_{\e,R}',
\end{align}
where the eight first terms are as in the statement of Lemma~\ref{lem:decompcruc},
and where the error $I_{\e,R}'$ is estimated as follows (cf.~\eqref{eq:decompDeR-rest}) in the considered regimes,
\begin{align*}
\int_0^t|I_{\e,R}'|&\lesssim_{t} \e R(N_\e\Log)^{1/2}\Log^2= o(N_\e^2).
\end{align*}

\medskip
\noindent\step2 Bound on the error terms.
\nopagebreak

In this step, we consider the regime $\log\Log\ll N_\e\lesssim\Log$, we study the three error terms $I_{\e,R}^d$, $I_{\e,R}^g$, and $I_{\e,R}^n$, and we prove for all $t\le T_\e$,
\begin{align}\label{eq:errorterms}
\int_0^t(I_{\e,R}^d+I_{\e,R}^g+I_{\e,R}^n)&\lesssim_t o(N_\e^2)+ o\Big(\frac{N_\e}\Log\Big)\int_0^t\int_{\R^2}\chi_R|\partial_tu_\e-iu_\e N_\e\!\pre_{\e}\!|^2.
\end{align}

We start with the bound on $I_{\e,R}^n$.
Using~\eqref{eq:corTeps}, Lemma~\ref{lem:velapbound}, and the boundedness properties of $\pre_{\e}$ (cf.~Proposition~\ref{prop:GLvprop}), the quantity $\bar\Ec_{\e,R}^*$ defined in Lemma~\ref{lem:prodest} is estimated as follows in the considered regimes, for all $\theta>0$,
\begin{multline*}
\bar\Ec_{\e,R}^{*,t}\lesssim\sup_z\int_0^t\Ec_{\e,R}^{z}+\sup_z\int_0^t\int_{\R^2}\chi_R^z\big(|\partial_tu_\e|^2+N_\e^2|\!\pre_{\e}\!|^2+N_\e^2|1-|u_\e|^2||\!\pre_{\e}\!|^2\big)\\
\lesssim_{t,\theta} R^\theta N_\e\Log^3+\lambda_\e^{-1}N_\e^2\lesssim R^\theta\Log^4,
\end{multline*}
hence, for $\theta>0$ small enough, $\bar\Ec_{\e,R}^{*,t}\lesssim_t\Log^5$. Using $|\nabla\chi_R|\lesssim R^{-1}\chi_R^{1/2}$, Lemma~\ref{lem:prodest} then yields
\begin{multline*}
\bigg|\int_0^t\int_{\R^2} a\tilde V_{\e}\cdot\nabla^\bot\chi_R\bigg|\lesssim_t \Log^{-1}\\
+R^{-1}\Log^{-1}\Big(\int_0^t\int_{\R^2}\chi_{R}|\partial_tu_\e-iu_\e N_\e\pre_{\e}\!|^2+\int_0^t\int_{B_{2R}}|\nabla u_\e-iu_\e N_\e\!\vi_\e\!|^2\Big),
\end{multline*}
and hence,
\begin{multline*}
\Big|\int_0^tI_{\e,R}^n\Big|\lesssim_t 1+R^{-1}\int_0^t\int_{\R^2}\chi_{R}|\partial_tu_\e-iu_\e N_\e\pre_{\e}\!|^2\\
+R^{-1}\int_0^t\int_{B_{2R}}\Big(|\nabla u_\e-iu_\e N_\e\!\vi_\e\!|^2+\frac a{2\e^2}(1-|u_\e|^2)^2+|1-|u_\e|^2|(N_\e^2|\!\vi_\e\!|^2+|f|)\Big).
\end{multline*}
Using~\eqref{eq:corTeps}, \eqref{eq:scalingshFf}, and the integrability properties of $\vi_\e$ (cf.\@ Proposition~\ref{prop:GLvprop}), with the choice $R\gg\frac\Log{N_\e}$, we conclude
\begin{multline}\label{st3}
\Big|\int_0^tI_{\e,R}^n\Big|\lesssim_t 1+R^{-1}N_\e\Log+R^{-1}\int_0^t\int_{\R^2}\chi_{R}|\partial_tu_\e-iu_\e N_\e\pre_{\e}|^2\\
\lesssim o(N_\e^2)+o\Big(\frac{N_\e}\Log\Big)\int_0^t\int_{\R^2}\chi_{R}|\partial_tu_\e-iu_\e N_\e\pre_{\e}|^2.
\end{multline}
We turn to the bound on $I_{\e,R}^g$.
Using~\eqref{eq:scalingshFf} and the pointwise estimates of Lemma~\ref{lem:pointest},
\begingroup\allowdisplaybreaks
\begin{multline*}
|I_{\e,R}^g|\lesssim \|\Gamma_\e-\bar\Gamma_\e\|_{\Ld^\infty}\bigg(N_\e\int_{B_{2R}} (|\nabla u_\e-iu_\e N_\e \!\vi_\e\!|+N_\e|1-|u_\e|^2|)|\curl\!\vi_\e\!|\\
+N_\e\int_{B_{2R}} |1-|u_\e|^2||\nabla u_\e-iu_\e N_\e \!\vi_\e\!|\\
+\lambda_\e\int_{\R^2} \chi_R\Big(|\nabla u_\e-iu_\e N_\e\!\vi_\e\!|^2+\frac a{\e^2}(1-|u_\e|^2)^2\Big)\\
+\lambda_\e\Log \int_{\R^2} \chi_R  |\partial_tu_\e-iu_\e N_\e\!\pre_{\e,\varrho}\!||\nabla u_\e-iu_\e N_\e \!\vi_\e\!|\\
+(N_\e+\lambda_\e\Log)\int_{\R^2} \chi_R(|\nabla u_\e-iu_\e N_\e \!\vi_\e\!|^2+N_\e^2|1-|u_\e|^2||\!\vi_\e\!|^2)\\
+N_\e^2\int_{\R^2} \chi_R|\!\vi_\e\!|^2(N_\e|\!\vi_\e\!|+\Log|F|)\\
+{\lambda_\e}N_\e\Log|\beta|\int_{\R^2} \chi_R|\partial_tu_\e-iu_\e N_\e\!\pre_{\e}\!|(|\!\vi_\e\!|+|1-|u_\e|^2|)\bigg).
\end{multline*}
\endgroup
By~\eqref{eq:corTeps}, by Lemma~\ref{lem:approx} in the form $\|\Gamma_\e-\bar\Gamma_\e\|_{\Ld^\infty}\lesssim r_\e=\Log^{-4}e^{-\sqrt{N_\e}}$, and by the integrability properties of $\vi_\e$ (cf.\@ Proposition~\ref{prop:GLvprop}), we deduce in the considered regimes, for all $\theta>0$,
\begin{align}\label{eq:Igestpreconcl}
|I_{\e,R}^g|&\lesssim_{t,\theta} \frac{e^{-\sqrt{N_\e}}}{\Log^4}R^\theta N_\e\Log^2\Big(1+\int_{\R^2} \chi_R  |\partial_tu_\e-iu_\e N_\e\!\pre_{\e}\!|^2\Big)^{1/2},
\end{align}
and hence, for $\theta>0$ small enough,
\begin{align}\label{eq:Igestconcl0}
|I_{\e,R}^g|&\lesssim_t o(N_\e^2)+o\Big(\frac{N_\e}\Log\Big)\int_{\R^2} \chi_R  |\partial_tu_\e-iu_\e N_\e\!\pre_{\e}\!|^2.
\end{align}
Regarding the last term~$I_{\e,R}^d$, the definition of the pressure $\pre_\e$ in~\eqref{eq:GLv1} simply yields $I_{\e,R}^d=0$, and the conclusion~\eqref{eq:errorterms} follows.

\medskip
\noindent\step3 Bound on the dominant terms.
\nopagebreak

In this step, we consider the regime $\log\Log\ll N_\e\lesssim\Log$ and we turn to the estimation of the five first terms in~\eqref{eq:decompGL0}, showing more precisely that for all $t\le T_\e$,
\begin{align}\label{eq:conclstep3Dt}
\hat\D_{\e,R}^t\lesssim_to(N_\e^2)+\int_0^t\hat\D_{\e,R}.
\end{align}
As this holds uniformly with respect to translations of the cut-off functions, the conclusion~\eqref{eq:stratprGL} follows.

We start with the bound on the first term $I_{\e,R}^S$.
Since for all $t$ the field $\bar\Gamma_\e^t$ is constant in each ball of the collection $\B_\e^t$ and satisfies $\|\nabla\bar\Gamma_\e^t\|_{\Ld^\infty}\lesssim\|\nabla\Gamma_\e^t\|_{\Ld^\infty}$, we find
\begin{multline*}
|I_{\e,R}^{S}|\lesssim \int_{\R^2\setminus\B_\e}\chi_R|\tilde S_\e|\lesssim \int_{\R^2\setminus\B_\e}a\chi_R\Big(|\nabla u_\e-iu_\e N_\e\!\vi_\e\!|^2+\frac{a}{2\e^2}(1-|u_\e|^2)^2\Big)\\
+\int_{\R^2}\chi_R|1-|u_\e|^2|(N_\e^2|\!\vi_\e\!|^2+|f|).
\end{multline*}
Since $\B_\e$ has total radius $r_\e:=\Log^{-4}e^{-\sqrt{N_\e}}$, and since the choice $N_\e\gg\log\Log$ ensures $r_\e\ge e^{-o(N_\e)}$, we may apply Proposition~\ref{prop:ballconstr}(v), which shows that the first integral in the above right-hand side is bounded by $\D_{\e,R}^*+o(N_\e^2)$. Further using~\eqref{eq:corTeps}, \eqref{eq:scalingshFf}, and the integrability properties of $\vi_\e$ (cf.\@ Proposition~\ref{prop:GLvprop}), we obtain in the considered regimes,
\begin{align}\label{eq:estIS}
|I_{\e,R}^{S}|&\lesssim \D_{\e,R}+o(N_\e^2)+\e(N_\e\Log)^{1/2}(N_\e^2+R\lambda_\e^2\Log^2)\lesssim \hat\D_{\e,R}+o(N_\e^2).
\end{align}
We turn to $I_{\e,R}^H$. Since $\|(\Gamma_\e,\nabla h)\|_{\Ld^{\infty}}\lesssim_t1$, Lemma~\ref{lem:poho} yields
\begin{multline*}
\int_0^tI_{\e,R}^H=O_t(N_\e)\\
+\int_0^t\int_{\R^2} \frac {a\chi_R}2\Gamma_\e^\bot\cdot\nabla h\Big(|\nabla u_\e-iu_\e N_\e \!\vi_\e\!|^2+\frac{a}{2\e^2}(1-|u_\e|^2)^2-\Log \mu_\e\Big),
\end{multline*}
and hence, by Proposition~\ref{prop:ballconstr}(iv) and by~\eqref{eq:corTeps},
\begin{align}\label{eq:estIH}
\int_0^tI_{\e,R}^H&\lesssim_{t}o(N_\e^2)+\int_0^t\D_{\e,R}\lesssim_{t}o(N_\e^2)+\int_0^t\hat\D_{\e,R}.
\end{align}
The term $I_{\e,R}^D$ is simply estimated by
\begin{align}\label{eq:estID}
I_{\e,R}^D\le-\frac{\lambda_\e \alpha}2\int_{\R^2}{a\chi_R}|\partial_tu_\e-iu_\e N_\e\!\pre_{\e}\!|^2+\frac{\lambda_\e \alpha}2\int_{\R^2}{a\chi_R}|(\nabla u_\e-iu_\e N_\e \!\vi_\e)\cdot\Gamma_\e^\bot|^2.
\end{align}
We finally turn to $I_{\e,R}^V$. Using $\alpha^2+\beta^2=1$, we have by definition
\[\Gamma_{\e,0}-\beta \Gamma_\e^\bot=\Gamma_{\e,0}-\beta(\alpha\Gamma_{\e,0}^\bot+\beta\Gamma_{\e,0})=\alpha^2\Gamma_{\e,0}-\alpha\beta \Gamma_{\e,0}^\bot=\alpha\Gamma_\e,\]
so that $I_{\e,R}^V$ takes on the following guise,
\begin{align}\label{eq:rewrite-term-IV}
I_{\e,R}^V=\lambda_\e\Log\int_{\R^2} \frac {a\chi_R}2\tilde V_{\e}\cdot (\Gamma_{\e,0}-\beta \Gamma_\e^\bot)=\lambda_\e\alpha\Log\int_{\R^2} \frac {a\chi_R}2\tilde V_{\e}\cdot \Gamma_{\e}.
\end{align}
As shown in Step~2, the quantity $\bar\Ec_{\e,R}^*$ defined in Lemma~\ref{lem:prodest}  satisfies $\bar\Ec_{\e,R}^{*,t}\lesssim_t\Log^5$. In the regime $\log\Log\ll N_\e\lesssim\Log$, choosing e.g.\@ $M_\e:=\exp((N_\e\log\Log)^{1/2})$, Lemma~\ref{lem:prodest} yields for any $\Lambda\simeq1$,
\begin{multline*}
\Big|\int_0^tI_{\e,R}^V\Big|\le o_t(1)+{\lambda_\e\alpha}\Big(1+\frac{C_t(N_\e\log\Log)^{1/2}}{\Log}\Big)\\
\times\bigg(\frac1\Lambda\int_0^t\int_{\R^2} a\chi_R|\partial_tu_\e-iu_\e N_\e\!\pre_{\e}\!|^2+\frac \Lambda4\int_0^t\int_{\R^2} a\chi_R|(\nabla u_\e-iu_\e N_\e\!\vi_\e)\cdot \Gamma_\e|^2\bigg),
\end{multline*}
and thus, using the optimal energy bound~\eqref{eq:corTeps}, we obtain in the considered regimes,
\begin{multline}\label{eq:applprodest}
\Big|\int_0^tI_{\e,R}^V\Big|\le o_t(N_\e^2)+\Big(\lambda_\e+o\Big(\frac{N_\e}\Log\Big)\Big)\frac{\alpha}\Lambda\int_0^t\int_{\R^2} a\chi_R|\partial_tu_\e-iu_\e N_\e\!\pre_{\e}\!|^2\\+\frac{\lambda_\e\alpha\Lambda}4\int_0^t\int_{\R^2} a\chi_R|(\nabla u_\e-iu_\e N_\e\!\vi_\e)\cdot \Gamma_\e|^2.
\end{multline}
We distinguish between two cases,
\begin{align}
\text{Case 1:}\quad\int_0^t\int_{\R^2} a\chi_R|\partial_tu_\e-iu_\e N_\e\!\pre_{\e}\!|^2\le 5\int_0^t\int_{\R^2} a\chi_R|(\nabla u_\e-iu_\e N_\e\!\vi_\e)\cdot\Gamma_\e|^2,\label{eq:case1bound}\\
\text{Case 2:}\quad\int_0^t\int_{\R^2} a\chi_R|\partial_tu_\e-iu_\e N_\e\!\pre_{\e}\!|^2> 5\int_0^t\int_{\R^2} a\chi_R|(\nabla u_\e-iu_\e N_\e\!\vi_\e)\cdot\Gamma_\e|^2.\label{eq:case2bound}
\end{align}
In Case~1, choosing $\Lambda=2$ in~\eqref{eq:applprodest} yields
\begin{multline*}
\Big|\int_0^tI_{\e,R}^V\Big|\le o_t(N_\e^2)+\Big(\lambda_\e+o\Big(\frac{N_\e}\Log\Big)\Big)\frac{\alpha}2\int_0^t\int_{\R^2} {a\chi_R}|\partial_tu_\e-iu_\e N_\e\!\pre_{\e}\!|^2\\+\frac{\lambda_\e\alpha}2\int_0^t\int_{\R^2} {a\chi_R}|(\nabla u_\e-iu_\e N_\e\!\vi_\e)\cdot \Gamma_{\e}|^2.
\end{multline*}
In Case~2, the condition~\eqref{eq:case2bound} can be rewritten as
\begin{multline*}
\frac14\int_0^t\int_{\R^2} a\chi_R|\partial_tu_\e-iu_\e N_\e\!\pre_{\e}\!|^2+\int_0^t\int_{\R^2} a\chi_R|(\nabla u_\e-iu_\e N_\e\!\vi_\e)\cdot\Gamma_{\e}|^2\\
\le\Big(\frac14+\frac1{10}\Big)\int_0^t\int_{\R^2} a\chi_R|\partial_tu_\e-iu_\e N_\e\!\pre_{\e}\!|^2+\frac12\int_0^t\int_{\R^2} {a\chi_R}|(\nabla u_\e-iu_\e N_\e\!\vi_\e)\cdot\Gamma_{\e}|^2,
\end{multline*}
and choosing $\Lambda=4$ in~\eqref{eq:applprodest} then yields, with $\frac{N_\e}\Log\lesssim\lambda_\e$ in the considered regimes,
\begin{multline*}
\Big|\int_0^tI_{\e,R}^V\Big|\le o_t(N_\e^2)
+{\lambda_\e\alpha}\bigg(\Big(\frac14+\frac1{10}+o(1)\Big)\int_0^t\int_{\R^2} a\chi_R|\partial_tu_\e-iu_\e N_\e\!\pre_{\e}\!|^2\\
+\frac12\int_0^t\int_{\R^2} {a\chi_R}|(\nabla u_\e-iu_\e N_\e\!\vi_\e)\cdot\Gamma_{\e}|^2\bigg).
\end{multline*}
Further noting that in Case~1 the condition~\eqref{eq:case1bound} together with the energy bound~\eqref{eq:corTeps} yields
\begin{align*}
o\Big(\frac{N_\e}\Log\Big)\int_{\R^2} a\chi_R|\partial_tu_\e-iu_\e N_\e\!\pre_{\e}\!|^2&\le o\Big(\frac{N_\e}\Log\Big)\int_0^t\int_{\R^2} a\chi_R|\nabla u_\e-iu_\e N_\e\!\vi_\e\!|^2\ll_t N_\e^2,
\end{align*}
and combining this with~\eqref{eq:errorterms} and~\eqref{eq:estID}, we observe an exact recombination of the terms, and obtain in Case~1,
\begin{multline}\label{eq:preconclcase1}
\int_0^t(I_{\e,R}^V+I_{\e,R}^D+I_{\e,R}^d+I_{\e,R}^g+I_{\e,R}^n+I'_{\e,R})\\
\le\frac{\lambda_\e \alpha}2\int_0^t\int_{\R^2}{a\chi_R}|\nabla u_\e-iu_\e N_\e \!\vi_\e|^2|\Gamma_\e|^2+o_t(N_\e^2),
\end{multline}
and in Case~2,
\begin{multline*}
\int_0^t(I_{\e,R}^V+I_{\e,R}^D+I_{\e,R}^d+I_{\e,R}^g+I_{\e,R}^n+I'_{\e,R})\\
\le-\frac{\lambda_\e \alpha}2\Big(\frac12-\frac1{5}-o(1)\Big)\int_0^t\int_{\R^2}{a\chi_R}|\partial_tu_\e-iu_\e N_\e\!\pre_{\e,\varrho}\!|^2\\
+\frac{\lambda_\e \alpha}2\int_0^t\int_{\R^2}{a\chi_R}|\nabla u_\e-iu_\e N_\e \!\vi_\e|^2|\Gamma_\e|^2+o_t(N_\e^2),
\end{multline*}
so that~\eqref{eq:preconclcase1} holds in both cases for $\e>0$ small enough.
Using $\alpha^2+\beta^2=1$, we find $\Gamma_\e\cdot\Gamma_{\e,0}=\alpha|\Gamma_{\e,0}|^2=\alpha|\Gamma_\e|^2$, so that the term $I_{\e,R}^E$ takes on the following guise,
\[I_{\e,R}^E=-\frac{\lambda_\e}2\Log\int_{\R^2} {a\chi_R}\Gamma_\e\cdot\Gamma_{\e,0}\,\mu_\e=-\frac{\lambda_\e\alpha}2\Log\int_{\R^2} {a\chi_R}|\Gamma_\e|^2\mu_\e.\]
Together with~\eqref{eq:preconclcase1}, this leads to
\begin{multline*}
\int_0^t(I_{\e,R}^V+I_{\e,R}^E+I_{\e,R}^D+I_{\e,R}^d+I_{\e,R}^g+I_{\e,R}^n+I'_{\e,R})\\
\le\frac{\lambda_\e \alpha}2\int_0^t\int_{\R^2}{a\chi_R}\big(|\nabla u_\e-iu_\e N_\e \!\vi_\e\!|^2-\Log\mu_\e\big)|\Gamma_\e|^2+o_t(N_\e^2).
\end{multline*}
Combining this with~\eqref{eq:decompGL0}, \eqref{eq:estIS}, \eqref{eq:estIH}, and with $\hat\D_{\e,R}^{*,\circ}\ll N_\e^2$, we conclude
\begin{align*}
\hat\D_{\e,R}^t\le o_t(N_\e^2)+C_t\int_0^t\hat\D_{\e,R}+\frac{\lambda_\e\alpha}2\int_0^t\int_{\R^2}{a\chi_R}\big(|\nabla u_\e-iu_\e N_\e \!\vi_\e\!|^2-\Log\mu_\e\big)|\Gamma_\e|^2,
\end{align*}
and the result~\eqref{eq:conclstep3Dt} now follows from Proposition~\ref{prop:ballconstr}(iv).

\medskip
\noindent\step4 Consequences.

In the previous steps, the results $T_\e=T$ and $\D_{\e,R}^{*,t}\ll_tN_\e^2$ for all $t\in[0,T)$ are established in the setting of item~(i) of the statement (that is, in the regime $\log\Log\ll N_\e\lesssim\Log$). We now show that it implies the stated convergence $\frac1{N_\e}j_\e-\vi_\e\to0$.

For all $t\in[0,T)$, since there holds $\D_{\e,R}^{*,t}\ll_tN_\e^2$, Proposition~\ref{prop:ballconstr}(v)--(vi) implies
\[\sup_z\int_{\R^2\setminus\B_{\e}}\chi_R^z|\nabla u_\e-iu_\e N_\e\!\vi_\e\!|^2\ll_t N_\e^2,\]
and for all $1\le p<2$,
\[\sup_z\int_{\B_{\e}}\chi_R^z|\nabla u_\e-iu_\e N_\e\!\vi_\e\!|^p\ll_t N_\e^p.\]
Using the pointwise estimates of Lemma~\ref{lem:pointest}, we deduce
\begin{multline*}
\sup_z\int_{B(z)}|j_\e-N_\e\!\vi_\e\!|\lesssim_t \sup_z\int_{B(z)}|\nabla u_\e-iu_\e N_\e \!\vi_\e\!|+\e N_\e\Log\\
\lesssim_t \sup_z\int_{\B_{\e}}\chi_R^z|\nabla u_\e-iu_\e N_\e\!\vi_\e\!|+\sup_z\Big(\int_{B(z)\setminus\B_{\e}}|\nabla u_\e-iu_\e N_\e\!\vi_\e\!|^2\Big)^{1/2}+o(N_\e)\ll_t N_\e,
\end{multline*}
hence $\frac1{N_\e}j_\e-\vi_\e\to0$ in $\Ld^\infty_\loc([0,T);\Ld^1_\uloc(\R^2)^2)$.
More precisely, for all $L\ge1$, we may decompose
\begin{multline*}
\sup_z\|j_\e-N_\e\!\vi_\e\!\|_{(\Ld^1+\Ld^2)(B_L(z))}\\
\lesssim_t \sup_z\|\nabla u_\e-iu_\e N_\e\!\vi_\e\!\|_{\Ld^1(\B_{\e}\cap B_L(z))}+\sup_z\|\nabla u_\e-iu_\e N_\e\!\vi_\e\!\|_{\Ld^2(B_L(z)\setminus\B_{\e})}\\
+N_\e\sup_z\|1-|u_\e|^2\|_{\Ld^2(B_L(z))}+\sup_z\|1-|u_\e|^2\|_{\Ld^2(B_L(z))}\|\nabla u_\e-iu_\e N_\e\!\vi_\e\!\|_{\Ld^2(B_L(z))},
\end{multline*}
hence
\begin{multline*}
\sup_z\|j_\e-N_\e\!\vi_\e\!\|_{(\Ld^1+\Ld^2)(B_L(z))}\\
\lesssim_t o(N_\e)(1+\tfrac LR)^2+\e N_\e(N_\e\Log)^{1/2}(1+\tfrac LR)+\e N_\e\Log(1+\tfrac LR)^2,
\end{multline*}
and the result~\eqref{eq:boundL1L2jNv} follows.
As mentioned in Remark~\ref{rems:GLMFL}, under the additional assumption that $\|u_\e^t\|_{\Ld^\infty}\lesssim_t1$, the convergence $\frac1{N_\e}j_\e-\vi_\e\to0$ also holds in $\Ld^\infty_\loc([0,T);\Ld^p_\loc(\R^2)^2)$ for all $1\le p<2$; this follows from a similar argument as above, replacing the pointwise estimate of Lemma~\ref{lem:pointest} for $j_\e-N_\e\!\vi_\e$ by
\begin{align*}
|j_\e-N_\e \!\vi_\e\!|&\le\|u_\e\|_{\Ld^\infty}|\nabla u_\e-iu_\e N_\e \!\vi_\e\!|+N_\e|1-|u_\e|^2||\!\vi_\e\!|.
\end{align*}

\medskip
\noindent\step5 Refinement in the parabolic case.

In this step, we consider the parabolic case ($\alpha=1$, $\beta=0$) both in the regime~\GLu{} and in the regime~\GLdp{} with $\lambda_\e\le \frac{e^{o(N_\e)}}\Log$, and we show that the additional assumption $N_\e\gg \log\Log$ can then be dropped.
In Steps~1--4 above, the main limitation comes from the fact that we need to use balls $\B_\e$ with a particularly small total radius $r_\e$ in order to obtain smallness of the error term $I_{\e,\varrho,R}^g$ in~\eqref{eq:Igestpreconcl}, while on the other hand the term $I_{\e,\varrho,R}^S$ corresponds to the energy outside the small balls $\B_\e$ so that we need to choose $r_\e\ge e^{-o(N_\e)}$ in order to apply Proposition~\ref{prop:ballconstr}(v). As we now show, the worst terms in $I_{\e,\varrho,R}^g$ vanish in the parabolic case, and the total radius $r_\e$ may then be chosen much larger.

We focus on the strongly dilute regime $1\ll N_\e\lesssim\log\Log$. Choose $\e^{1/2}<\tilde r_\e^0\ll \frac{N_\e}\Log$ and let $\tilde r_\e:=(\lambda_\e\Log)^{-2}\ge e^{-o(N_\e)}$.
For all $t\le T_\e$, as we are in the framework of Proposition~\ref{prop:ballconstr} with $u_\e^t,\vi_\e^t$, we let $\tilde\B_\e^t:=\tilde\B_{\e,R}^t$ denote the corresponding collection of disjoint closed balls $\tilde\B_{\e,R}^{\tilde r_\e^0,\tilde r_\e}(u_\e^t,\vi_\e^t)$. Let then $\tilde\Gamma_\e^t$ denote the associated approximation of $\Gamma_\e^t$ given by Lemma~\ref{lem:approx}.
As in Step~1, Lemma~\ref{lem:decompcruc} yields the following decomposition, with the approximate vector field $\bar\Gamma_\e$ replaced by $\tilde\Gamma_\e$,
\begin{align*}
\partial_t\hat\D_{\e,R}=~&I_{\e,R}^S+I_{\e,R}^V+I_{\e,R}^E+I_{\e,R}^D+I_{\e,R}^H+I_{\e,R}^d+I_{\e,R}^g+I_{\e,R}^n+I_{\e,R}',
\end{align*}
where all the terms are estimated just as above, except $I_{\e,R}^S$, $I_{\e,R}^V$, and $I_{\e,R}^g$. We start with the discussion of $I_{\e,R}^g$. For $\alpha=1$, $\beta=0$, this term takes on the following simpler form,
\begin{multline}\label{eq:Igdefpure}
I_{\e,R}^g=\int_{\R^2} a\chi_RN_\e(N_\e \!\vi_\e-j_\e)\cdot(\Gamma_\e-\tilde\Gamma_\e)\curl\!\vi_\e\\
+\int_{\R^2} \lambda_\e a\chi_R (\Gamma_\e-\tilde\Gamma_\e)^\bot\cdot\left\langle \partial_tu_\e-iu_\e N_\e\!\pre_{\e},\nabla u_\e-iu_\e N_\e \!\vi_\e\right\rangle\\
+\int_{\R^2} \frac {a\chi_R}2(\tilde\Gamma_\e-\Gamma_\e)^\bot\cdot\nabla h\Big(|\nabla u_\e-iu_\e N_\e \!\vi_\e\!|^2+\frac{a}{\e^2}(1-|u_\e|^2)^2\Big)\\+\int_{\R^2} a\chi_R(\tilde\Gamma_\e-\Gamma_\e)\cdot (N_\e \!\vi_\e+\tfrac12\Log F^\bot)\mu_\e.
\end{multline}
We estimate each of the four right-hand side terms separately.
We start with the first term. Using the pointwise estimates of Lemma~\ref{lem:pointest} and the integrability properties of $\vi_\e$ (cf.\@ Proposition~\ref{prop:GLvprop}), we find
\begin{multline*}
\int_{\R^2} a\chi_RN_\e(N_\e \!\vi_\e-j_\e)\cdot(\Gamma_\e-\tilde\Gamma_\e)\curl\!\vi_\e\\
\lesssim N_\e\|\Gamma_\e-\tilde\Gamma_\e\|_{\Ld^\infty}\bigg(\int_{\tilde\B_\e^t}\chi_R|\nabla u_\e-iu_\e N_\e\!\vi_\e\!|+\Big(\int_{\R^2\setminus\tilde\B_\e^t}\chi_R|\nabla u_\e-iu_\e N_\e\!\vi_\e\!|^2\Big)^{1/2}\bigg)\\
+N_\e\|\Gamma_\e-\tilde\Gamma_\e\|_{\Ld^\infty}\Big(\int_{\R^2}\chi_R|1-|u_\e|^2||\nabla u_\e-iu_\e N_\e\!\vi_\e\!|+N_\e\int_{\R^2}\chi_R|1-|u_\e|^2||\curl\!\vi_\e\!|\Big),
\end{multline*}
and hence, using~\eqref{eq:corTeps} and Proposition~\ref{prop:ballconstr}(v)--(vi) with $p=1$ to estimate the first two integrals in the right-hand side, and using Lemma~\ref{lem:approx} in the form $\|\Gamma_\e^t-\tilde\Gamma_\e^t\|_{\Ld^\infty}\lesssim_t \tilde r_\e\ll1$,
\begin{align*}
\int_{\R^2} a\chi_RN_\e(N_\e \!\vi_\e-j_\e)\cdot(\Gamma_\e-\tilde\Gamma_\e)\curl\!\vi_\e\lesssim N_\e^2\|\Gamma_\e-\tilde\Gamma_\e\|_{\Ld^\infty}\ll_t N_\e^2.
\end{align*}
For the second right-hand side term in~\eqref{eq:Igdefpure}, using~\eqref{eq:corTeps} and again Lemma~\ref{lem:approx}, with $\tilde r_\e\lambda_\e\ll \frac{N_\e}\Log$, we obtain
\begin{multline*}
\int_{\R^2} \lambda_\e a\chi_R (\Gamma_\e-\tilde\Gamma_\e)^\bot\cdot\left\langle \partial_tu_\e-iu_\e N_\e\!\pre_{\e},\nabla u_\e-iu_\e N_\e \!\vi_\e\right\rangle\\
\lesssim \lambda_\e({N_\e}\Log)^{1/2}\|\Gamma_\e-\tilde\Gamma_\e\|_{\Ld^\infty}\Big(\int_{\R^2} \chi_R |\partial_tu_\e-iu_\e N_\e\!\pre_{\e}\!|^2\Big)^{1/2}\\
\lesssim o(N_\e^2)+ o\Big(\frac{N_\e}\Log\Big)\int_{\R^2} \chi_R |\partial_tu_\e-iu_\e N_\e\!\pre_{\e}\!|^2.
\end{multline*}
For the third right-hand side term in~\eqref{eq:Igdefpure}, using~\eqref{eq:corTeps}, \eqref{eq:scalingshFf}, and Lemma~\ref{lem:approx} in the form $\|(\tilde\Gamma_\e-\Gamma_\e)^\bot\cdot\nabla h\|_{\Ld^\infty}\lesssim_t \tilde r_\e\lambda_\e\ll \frac{N_\e}\Log$, we find
\begin{align*}
\int_{\R^2} \frac {a\chi_R}2(\tilde\Gamma_\e-\Gamma_\e)^\bot\cdot\nabla h\Big(|\nabla u_\e-iu_\e N_\e \!\vi_\e\!|^2+\frac{a}{\e^2}(1-|u_\e|^2)^2\Big)\ll_tN_\e^2.
\end{align*}
It remains to estimate the fourth term in~\eqref{eq:Igdefpure}. Using~\eqref{eq:corTeps}, Proposition~\ref{prop:ballconstr}(iii) in the form~\eqref{eq:vortmass} with $\gamma=\frac12$, the regularity properties of $\vi_\e$ (cf.\@ Proposition~\ref{prop:GLvprop}), \eqref{eq:scalingshFf} in the form $\|F\|_{C^{1/2}}\lesssim\lambda_\e$, and Lemma~\ref{lem:approx} in the form $\|\tilde\Gamma_\e-\Gamma_\e\|_{C^{1/2}}\lesssim_t \tilde r_\e^{1/2}=(\lambda_\e\Log)^{-1}$, we obtain
\begin{multline*}
\int_{\R^2} a\chi_R(\tilde\Gamma_\e-\Gamma_\e)\cdot (N_\e \!\vi_\e+\tfrac12\Log F^\bot)\mu_\e\\
\lesssim N_\e\| a\chi_R(\tilde\Gamma_\e-\Gamma_\e)\cdot (N_\e \!\vi_\e+\tfrac12\Log F^\bot)\|_{C^{1/2}}\\
\lesssim N_\e(N_\e+\lambda_\e\Log)\|\tilde\Gamma_\e-\Gamma_\e\|_{C^{1/2}}\ll_tN_\e^2.
\end{multline*}
Inserting these various estimates into~\eqref{eq:Igdefpure} leads to
\[I_{\e,R}^g\lesssim_t o(N_\e^2)+o\Big(\frac{N_\e}\Log\Big)\int_{\R^2}\chi_R|\partial_tu_\e-iu_\e N_\e\!\pre_{\e}\!|^2,\]
proving that~\eqref{eq:Igestconcl0} again holds in the present setting.
We turn to the discussion of $I_{\e,R}^S$. Since the total radius satisfies $\tilde r_\e\ge e^{-o(N_\e)}$, we may apply Proposition~\ref{prop:ballconstr}(v), so that the same argument as in Step~3 leads to the estimate~\eqref{eq:estIS} for $I_{\e,R}^S$.
It remains to discuss the bound on the term $I_{\e,R}^V$.
In the regime $1\ll N_\e\lesssim\log\Log$, the assumption on $\lambda_\e$ leads to $\lambda_\e\lesssim \frac{e^{o(N_\e)}}\Log\ll \frac{N_\e}{\log\Log}$, that is, $\frac{N_\e}{\lambda_\e\log\Log}\gg1$. Writing $I_{\e,R}^V$ as in~\eqref{eq:rewrite-term-IV}, we may thus apply Lemma~\ref{lem:prodest} with the choice
\[M_\e:=\exp\bigg(\Big(\frac{N_\e}{\lambda_\e\log\Log}\Big)^{1/2}\log\Log\bigg),\]
and hence, for any $\Lambda\simeq1$, noting that $\lambda_\e\frac{\log M_\e}{\Log}=\frac{1}\Log(N_\e\lambda_\e\log\Log)^{1/2}=o(\frac{N_\e}{\Log})$,
\begin{multline*}
\Big|\int_0^tI_{\e,R}^V\Big|=\lambda_\e\Log\Big|\int_0^t\int_{\R^2}\frac{a\chi_R}2\tilde V_\e\cdot\Gamma_\e\Big|\\
\le o_t(1)+\Big(\lambda_\e+o\Big(\frac{N_\e}\Log\Big)\Big)\bigg(\frac1\Lambda\int_0^t\int_{\R^2} a\chi_R|\partial_tu_\e-iu_\e N_\e\!\pre_{\e}\!|^2\\
+\frac \Lambda4\int_0^t\int_{\R^2} a\chi_R|(\nabla u_\e-iu_\e N_\e\!\vi_\e)\cdot \Gamma_\e|^2\bigg).
\end{multline*}
Further using the energy bound~\eqref{eq:corTeps}, the estimate~\eqref{eq:applprodest} follows. With these ingredients at hand, we may now repeat the argument in Steps~2--3 and conclude with~\eqref{eq:stratprGL}. Finally, the convergence $\frac1{N_\e}j_\e-\vi_\e\to0$ follows as in Step~4, with $\B_\e$ replaced by $\tilde\B_\e$.
\end{proof}

\section{Mean-field limit in the nondilute parabolic case}\label{chap:mean-field-hd}
In this section we prove Theorem~\ref{th:main-hd}, that is, the mean-field limit result in the dissipative case ($\alpha>0$) in the nondilute regime~\GLt. More precisely, we
make use of the modulated energy strategy and
show that the rescaled supercurrent density $\frac1{N_\e}j_\e$ remains close to the solution $\vi_\e$ of equation~\eqref{eq:GLv1-hd}. Combining this with the convergence results of Section~\ref{chap:deg-eqn-exist-bound}, the result of Theorem~\ref{th:main-hd} follows.
Note that in this nondilute regime the proof of Proposition~\ref{prop:mflGL} indicates that we expect to find
\begin{align}\label{eq:bound-D-gen}
\hat\D_{\e,R}^t\le o_t(N_\e^2)+C_t(1+\alpha\lambda_\e)\int_0^t\hat\D_{\e,R}.
\end{align}
As $\lambda_\e\gg1$, the Grönwall inequality does of course not allow us to conclude $\hat\D_{\e,R}^t \ll_t N_\e^2$ for any $t>0$.
(In contrast, in the conservative case $\alpha=0$, the prefactor $\lambda_\e$ would disappear in~\eqref{eq:bound-D-gen}, cf.\@ Section~\ref{chap:MFL-GP}.)
In the sequel, the strategy consists in refining the magnitude of the error $o(N_\e^2)$ in~\eqref{eq:bound-D-gen} as much as possible, showing that it can be reduced to $O(N_\e^{2-\delta})$ for some $\delta>0$. For $\lambda_\e=\frac{N_\e}\Log\gg1$, the Grönwall inequality then still leads to $\hat\D_{\e,R}^t \ll_t N_\e^2$ for all $t\ge0$ in the regime $\Log\ll N_\e\ll\Log\log\Log$.
Since in~\cite{D-16} the well-posedness of the degenerate mean-field equation~\eqref{eq:GLv1-hd} could only be established in the parabolic case, we have to restrict to that case.

\subsection{Preliminary: vortex analysis}\label{chap:vort-anal-hd}

We adapt the crucial vortex analysis of Section~\ref{sec:vortex} to the present situation with a large number of vortices $N_\e\gg\Log$.
We start with establishing the following version of Proposition~\ref{prop:ballconstr}.

\begin{prop}[Refined lower bound]\label{prop:ballconstr-hd}
Let $h:\R^2\to\R$, $a:=e^h$, with $1\lesssim a\le1$ and $\|\nabla h\|_{\Ld^\infty}\lesssim1$, let $u_\e:\R^2\to\C$, $\vi_\e:\R^2\to\R^2$, with $\|\curl\!\vi_\e\!\|_{\Ld^1\cap\Ld^\infty},\|\!\vi_\e\!\|_{\Ld^\infty}\lesssim1$. Let $0<\e\ll1$, $N_\e\gtrsim\Log$, and $R\ge1$ with $\log N_\e\ll\Log$ and $\Log\lesssim R\lesssim\Log^n$ for some $n\ge1$, and assume that $\D_{\e,R}^*\lesssim N_\e^2$.
Then $\Ec_{\e,R}^*\lesssim N_\e^2$ holds for all $\e>0$ small enough.
Moreover, for some $\bar r\simeq1$,  for all $\e>0$ small enough and $r\in(\e^{1/2},\bar r)$, letting $\B_{\e,R}^r$ and $\nu_{\e,R}^r$ denote the locally finite union of disjoint closed balls and the point-vortex measure constructed in Lemma~\ref{lem:ballconstr}, the following properties hold,
\begin{enumerate}[(i)]
\item \emph{Lower bound:} In the regime $N_\e\gg\log\Log$, we have for all $\e^{1/2}<r<\bar r$ and $z\in \R^2$,
\begin{multline}\label{eq:lowerbound-hd}
\qquad\frac12\int_{\B_{\e,R}^r}a\chi_R^z\Big(|\nabla u_\e-iu_\e N_\e \!\vi_\e\!|^2+\frac{a}{2\e^2}(1-|u_\e|^2)^2\Big)\\
\ge \frac{\Log}2\int_{\R^2} a\chi_R^z |\nu_{\e,R}^r|-O\Big(rN_\e^2+\frac{N_\e^2}\Log(|\!\log r|+\log N_\e)\Big).
\end{multline}
\item \emph{Number of vortices:} For $\e^{1/2}<r\ll1$,
\begin{align}\label{eq:numbvort-hd}
\sup_z\int_{B_R(z)}|\nu_{\e,R}^r|~\lesssim~ \frac{N_\e^2}\Log.
\end{align}
\item \emph{Jacobian estimate:} For $\e^{1/2}<r\ll1$, for all $\gamma\in[0,1]$,
\begin{gather}
\sup_z\|\nu_{\e,R}^r-\tilde\mu_\e\|_{(C_c^\gamma(B_{R}(z)))^*}\lesssim r^\gamma \frac{N_\e^2}\Log+\e^{\gamma/2}N_\e^2,\label{eq:jacobest-hd}\\
\sup_z\|\mu_\e-\tilde\mu_\e\|_{(C_c^\gamma(B_R(z)))^*}\lesssim \e^\gamma N_\e^{2}\Log^n.\label{eq:jacobestbis-hd}
\end{gather}
\item \emph{Excess energy estimate:} For all $\phi\in W^{1,\infty}(\R^2)$ supported in a ball of radius $R$,
\begin{multline}\label{eq:excessestim-hd}
\qquad\int_{\R^2}\phi\Big(|\nabla u_\e-iu_\e N_\e\!\vi_\e\!|^2+\frac{a}{2\e^2}(1-|u_\e|^2)^2-\Log\mu_\e\Big)\\
\lesssim \bigg(\D_{\e,R}^*+\frac{N_\e^2}\Log\log N_\e\bigg)\|\phi\|_{W^{1,\infty}}.
\end{multline}
\item \emph{Energy outside small balls:} For all $\gamma\ge1$, $N_\e^{-\gamma}\le r<\bar r$, and $z\in\R^2$,
\begin{equation}\label{eq:energyoutballs-hd}
\quad~~\int_{\R^2\setminus\B_{\e,R}^r}a\chi_R^z\Big(|\nabla u_\e-iu_\e N_\e \!\vi_\e\!|^2+\frac{a}{2\e^2}(1-|u_\e|^2)^2\Big)
\le \D_{\e,R}^z+O_\gamma\Big(\frac{N_\e^2}\Log\log N_\e\Big).
\qedhere
\end{equation}
\end{enumerate}
\end{prop}

\begin{proof}
We split the proof into six steps. The main work consists in checking that the assumptions imply the optimal bound on the energy $\Ec_{\e,R}^*\lesssim N_\e^2$. This main conclusion is obtained in Step~5, while the various other claims are deduced in Step~6.

\medskip
\noindent\step1 Rough a priori estimate on the energy.

A direct adaptation of Step~1 of the proof of Proposition~\ref{prop:ballconstr} yields $\Ec_{\e,R}^*\lesssim N_\e^2+R^2\Log^2$, and hence by the choice of $R$ we deduce $\Ec_{\e,R}^*\lesssim N_\e^2+\Log^m$ for some $m\ge4$.

\medskip
\noindent\step2 Application of Lemma~\ref{lem:ballconstr}.

\nopagebreak
By assumption $\log N_\e\ll\Log$, the result of Step~1 yields in particular $\log\Ec_{\e,R}^*\ll\Log$, which allows to apply Lemma~\ref{lem:ballconstr}. For fixed $r\in(\e^{1/2},\bar r)$, let $\B_{\e,R}^r=\biguplus_jB^j$ denote the union of disjoint closed balls given by Lemma~\ref{lem:ballconstr}, and let $\nu_{\e,R}^r$ denote the associated point-vortex measure.
Using Lemma~\ref{lem:ballconstr}(ii) in the form
\begin{align}\label{eq:estnumbervort-hd}
\int_{B_R(z)}|\nu_{\e,R}^r|=\sum_{j:y_j\in B_R(z)}|d_j|\lesssim\frac{N_\e^2+\Ec_{\e,R}^*}\Log,
\end{align}
Lemma~\ref{lem:ballconstr}(i) gives,
for all $\phi\in W^{1,\infty}(\R^2)$ supported in a ball of radius $R$, with $\phi\ge0$,
\begin{multline}\label{eq:lowerboundclaim1b-hd}
\frac12\int_{\B_{\e,R}^r}\phi\Big(|\nabla u_\e-iu_\e N_\e \!\vi_\e\!|^2+\frac{a}{2\e^2}(1-|u_\e|^2)^2\Big)\\
\ge \frac{\Log}2\int_{\R^2} \phi |\nu_{\e,R}^r|-O(r\Ec_{\e,R}^*)\|\nabla\phi\|_{\Ld^{\infty}}\\-O\bigg(r^2N_\e^2+|\!\log r|\frac{N_\e^2+\Ec_{\e,R}^*}\Log+\frac{N_\e^2+\Ec_{\e,R}^*}\Log\log\Big(2+\frac{\Ec_{\e,R}^*}\Log\Big)\bigg)\|\phi\|_{\Ld^{\infty}}.
\end{multline}
Arguing as in Step~2 of the proof of Proposition~\ref{prop:ballconstr}, we then find for all $z\in\R^2$,
\begin{multline}\label{eq:energyR2minB-hd}
\int_{\R^2\setminus\B_{\e,R}^r}a\chi_R^z\Big(|\nabla u_\e-iu_\e N_\e \!\vi_\e\!|^2+\frac{a}{2\e^2}(1-|u_\e|^2)^2\Big)\\
\le \D_{\e,R}^z+O\bigg(1+(|\!\log r|+r\Log)\frac{N_\e^2+\Ec_{\e,R}^*}\Log+\frac{N_\e^2+\Ec_{\e,R}^*}\Log\log\Big(2+\frac{\Ec_{\e,R}^*}\Log\Big)\bigg),
\end{multline}
and in addition,
\begin{gather}
\Big|\int_{\R^2} \phi(\mu_\e-\nu_{\e,R}^r)\Big|\lesssim \bigg(r\frac{N_\e^2+\Ec_{\e,R}^*}\Log+\e^{1/3}\bigg)\|\phi\|_{W^{1,\infty}},\label{eq:munuexch-hd}\\
\Big|\int_{\R^2} \phi(\tilde\mu_\e-\mu_\e)\Big|\lesssim\e R N_\e(\Ec_{\e,R}^*)^{1/2}\|\phi\|_{W^{1,\infty}}\lesssim\e^{1/3}\|\phi\|_{W^{1,\infty}}.\label{eq:boundmumutildeequ-hd}
\end{gather}

\medskip
\noindent\step3 Energy and number of vortices.

In this step, we show that~\eqref{eq:estnumbervort-hd} is essentially an equality, in the following sense: for all $\e^{1/2}<r\ll1$,
\begin{align}\label{eq:energyvort-hd}
\sup_z\int_{\R^2}\chi_R^z|\nu_{\e,R}^r|\lesssim \frac{N_\e^2+\Ec_{\e,R}^*}\Log \lesssim \frac{N_\e^2}\Log+\sup_z\int_{\R^2} \chi_R^z|\nu_{\e,R}^r|.
\end{align}
The lower bound follows from~\eqref{eq:estnumbervort-hd}. We turn to the upper bound.
Since the energy excess satisfies $\D_{\e,R}^z\lesssim N_\e^2$, we deduce from~\eqref{eq:munuexch-hd},
\begin{align}\label{eq:energyvort-pre-hd}
\Ec_{\e,R}^z\le \D_{\e,R}^z+\frac\Log2\int_{\R^2} a\chi_R^z\mu_\e\le \frac\Log2\int_{\R^2} a\chi_R^z\nu_{\e,R}^r+O\big(N_\e^2+r\Ec_{\e,R}^*\big).
\end{align}
Taking the supremum in $z$, and absorbing $\Ec_{\e,R}^*$ in the left-hand side with $r\ll1$, the upper bound in~\eqref{eq:energyvort-hd} follows.

\medskip
\noindent\step4 Bound on the total variation of the vorticity.

In this step, we prove that for all $e^{-o(\Log)}<r\ll1$,
\begin{align}\label{eq:signrough-hd}
\sup_z\int_{\R^2}\chi_R^z|\nu_{\e,R}^r|\le (1+o(1))\sup_z\int_{\R^2}\chi_R^z\nu_{\e,R}^r+O\Big(\frac{N_\e^2}\Log\Big).
\end{align}
The lower bound~\eqref{eq:lowerboundclaim1b-hd} of Step~2 with $\phi=a\chi_R^y$ yields for all $y\in\R^2$, using the upper bound in~\eqref{eq:energyvort-hd} to replace the energy $\Ec_{\e,R}^*$ in the error terms,
\begin{multline*}
\Ec_{\e,R}^y\ge \frac12\int_{\B_{\e,R}^r}a\chi_R^y\Big(|\nabla u_\e-iu_\e N_\e \!\vi_\e\!|^2+\frac{a}{2\e^2}(1-|u_\e|^2)^2\Big)\ge \frac{\Log}2\int_{\R^2} a\chi_R^y |\nu_{\e,R}^r|\\-O\Big(\frac{N_\e^2}\Log+\sup_z\int_{\R^2}\chi_R^z|\nu_{\e,R}^r\Big)\bigg(|\!\log r|+r\Log+\log\Big(2+\frac{N_\e^2+\Ec_{\e,R}^*}{\Log}\Big)\bigg).
\end{multline*}
For $e^{-o(\Log)}<r\ll1$, using the result of Step~1 in the form $\log(N_\e^2+\Ec_{\e,R}^*)\ll\Log$, we obtain for all $y\in\R^2$,
\begin{align}\label{eq:preboundinfenergy-hd}
\Ec_{\e,R}^y\ge \frac{\Log}2\int_{\R^2} a\chi_R^y |\nu_{\e,R}^r|-o(\Log)\sup_z\int_{\R^2} \chi_R^z|\nu_{\e,R}^r|-o(N_\e^2).
\end{align}
On the other hand, the upper bound~\eqref{eq:energyvort-pre-hd} yields
\begin{align}\label{eq:preboundsupenergy-hd}
\Ec_{\e,R}^y\le\frac\Log2\int_{\R^2} a\chi_R^y\nu_{\e,R}^r+O(N_\e^2)+o(1)\Ec_{\e,R}^*,
\end{align}
and thus, taking the supremum over $y$ and absorbing $\Ec_{\e,R}^*$ in the left-hand side,
\begin{align*}
\Ec_{\e,R}^*\le\frac\Log2\sup_z\int_{\R^2} a\chi_R^z|\nu_{\e,R}^r|+O(N_\e^2),
\end{align*}
so that~\eqref{eq:preboundsupenergy-hd} takes the form, for all $y\in\R^2$,
\begin{align*}
\Ec_{\e,R}^y\le\frac\Log2\int_{\R^2} a\chi_R^y\nu_{\e,R}^r+O(N_\e^2)+o(\Log)\sup_z\int_{\R^2} \chi_R^z|\nu_{\e,R}^r|.
\end{align*}
Combining this with~\eqref{eq:preboundinfenergy-hd}, dividing both sides by $\frac12\Log$, and taking the supremum over~$y$, we find
\begin{align*}
\sup_z\int_{\R^2} \chi_R^z (\nu_{\e,R}^r)^-\lesssim\sup_z\int_{\R^2} a\chi_R^z (|\nu_{\e,R}^r|-\nu_{\e,R}^r)\le O\Big(\frac{N_\e^2}\Log\Big)+o(1)\sup_z\int_{\R^2} \chi_R^z|\nu_{\e,R}^r|,
\end{align*}
hence
\begin{multline*}
\sup_z\int_{\R^2} \chi_R^z |\nu_{\e,R}^r|=\sup_z\int_{\R^2} \chi_R^z (\nu_{\e,R}^r+2(\nu_{\e,R}^r)^-)\\
\le \sup_z\int_{\R^2} \chi_R^z \nu_{\e,R}^r+O\Big(\frac{N_\e^2}\Log\Big)+o(1)\sup_z\int_{\R^2} \chi_R^z|\nu_{\e,R}^r|,
\end{multline*}
and the result~\eqref{eq:preboundinfenergy} follows after absorbing the last right-hand side term.

\medskip
\noindent\step5 Refined bound on the energy.

In this step, we prove the optimal energy bound $\Ec_{\e,R}^*\lesssim N_\e^2$. By~\eqref{eq:estnumbervort-hd} this yields in particular $\sup_z\int_{\R^2}\chi_R^z|\nu_{\e,R}^r|\lesssim \frac{N_\e^2}\Log$.

Let $e^{-o(\Log)}<r\ll1$ be suitably chosen later. Using~\eqref{eq:munuexch-hd}, the bound on the energy excess $\D_{\e,R}^*\lesssim N_\e^2$ yields for all $z\in R\Z^2$,
\begin{align*}
\Ec_{\e,R}^z&\le \D_{\e,R}^z+\frac\Log2\int_{\R^2} a\chi_R^z\mu_\e\lesssim N_\e^2+r\Ec_{\e,R}^*+\Log\int_{\R^2} \chi_R^z|\nu_{\e,R}^r|,
\end{align*}
and hence, using the result~\eqref{eq:signrough-hd} of Step~4 and absorbing $\Ec_{\e,R}^*$ with $r\ll1$,
\begin{align}\label{eq:boundenergypreconcl-hd}
\Ec_{\e,R}^*&\lesssim N_\e^2+\Log\sup_z\int_{\R^2} \chi_R^z\nu_{\e,R}^r\lesssim N_\e^2+\Log\sup_z\int_{\R^2}\chi_R^z\mu_\e.
\end{align}
It remains to estimate $\int_{\R^2}\chi_R^z\mu_\e$. Arguing as in Step~5 of the proof of Proposition~\ref{prop:ballconstr}, we find
\begin{multline}
\int_{\R^2} \chi_R^z\mu_\e\lesssim N_\e+\Big(\int_{\R^2\setminus\B_{\e,R}^r}\chi_{2R}^z|\nabla u_\e-iu_\e N_\e\!\vi_\e\!|^2\Big)^{1/2}\\
+rR^{-1}\Big(\int_{B_{2R}(z)}|\nabla u_\e-iu_\e N_\e\!\vi_\e\!|^2\Big)^{1/2},\label{eq:boundvortpreconcl-hd}
\end{multline}
and then using~\eqref{eq:energyR2minB-hd} to estimate the second right-hand side term,
\begin{align*}
\int_{\R^2} \chi_R^z\mu_\e&\lesssim N_\e+(\D_{\e,R}^*)^{1/2}+rR^{-1}(\Ec_{\e,R}^*)^{1/2}+r^{1/2}(N_\e^2+\Ec_{\e,R}^*)^{1/2}\\
&\hspace{4cm}+\Big(\frac{N_\e^2+\Ec_{\e,R}^*}\Log\Big)^{1/2}\bigg(|\!\log r|+\log\Big(2+\frac{\Ec_{\e,R}^*}\Log\Big)\bigg)^{1/2}\\
&\lesssim N_\e+ r^{1/2}(\Ec_{\e,R}^*)^{1/2}+o(1)\frac{N_\e^2+\Ec_{\e,R}^*}\Log+|\!\log r|^{1/2}\Big(\frac{N_\e^2+\Ec_{\e,R}^*}\Log\Big)^{1/2}.
\end{align*}
Combining this with~\eqref{eq:boundenergypreconcl-hd} leads to
\begin{align*}
\frac{\Ec_{\e,R}^*}\Log&\lesssim\frac{N_\e^2}\Log+ r^{1/2}(\Ec_{\e,R}^*)^{1/2}+o(1)\frac{\Ec_{\e,R}^*}\Log+|\!\log r|^{1/2}\Big(\frac{N_\e^2+\Ec_{\e,R}^*}\Log\Big)^{1/2},
\end{align*}
hence, 
\begin{align*}
\frac{\Ec_{\e,R}^*}\Log&\lesssim \frac{N_\e^2}\Log+|\!\log r|.
\end{align*}
and the result follows from the choice $r=\Log^{-1}$.

\medskip
\noindent\step6 Conclusion.

The optimal energy bound $\Ec_{\e,R}^*\lesssim N_\e^2$ is now proved. In the present step, we check that the remaining statements follow from this bound.
The result~\eqref{eq:lowerbound-hd} follows from~\eqref{eq:lowerboundclaim1b-hd} in Step~2 with $\phi=a\chi_R^z$, combined with the optimal energy bound.
The bound~\eqref{eq:numbvort-hd} on the number of vortices follows from the result~\eqref{eq:energyvort-hd} of Step~3 together with the optimal energy bound.
For $r=N_\e^{-\gamma}$, $\gamma\ge1$, the result~\eqref{eq:energyoutballs-hd} follows from~\eqref{eq:energyR2minB-hd} together with the optimal energy bound. Monotonicity of $\B_{\e,R}^r$ with respect to $r$ then implies~\eqref{eq:energyoutballs-hd} for all $r\ge N_\e^{-\gamma}$.
It remains to establish items~(iii) and~(iv). We split the proof into two further substeps.

\medskip
\noindent\substep{6.1} Proof of~(iii).

The Jacobian estimate~\eqref{eq:jacobest-hd} follows from Lemma~\ref{lem:ballconstr}(iii) together with the optimal energy bound, and the estimate~\eqref{eq:jacobestbis-hd} with $\gamma=1$ similarly follows from~\eqref{eq:boundmumutildeequ-hd} and from the bound $R\lesssim\Log^n$. As in Step~8.4 of the proof of Proposition~\ref{prop:ballconstr}, we further find for all $\phi\in\Ld^\infty(\R^2)$ supported in a ball $B_R(z)$, $z\in\R^2$,
\begin{multline}\label{eq:mumutildeC0-hd}
\Big|\int_{\R^2} \phi(\tilde\mu_\e-\mu_\e)\Big|
\lesssim N_\e\|\phi\|_{\Ld^\infty}\int_{B_R(z)} \Big(|1-|u_\e|^2||\curl\!\vi_\e\!|\\
+2|\!\vi_\e\!||1-|u_\e|^2||\nabla u_\e-iu_\e N_\e\!\vi_\e|
+2|\!\vi_\e\!||\nabla u_\e-iu_\e N_\e\!\vi_\e|\Big),
\end{multline}
hence $|\!\int_{\R^2} \phi(\tilde\mu_\e-\mu_\e)|\lesssim RN_\e^2\|\phi\|_{\Ld^\infty}$, and the result~\eqref{eq:jacobestbis-hd} follows by interpolation.

\medskip
\noindent\substep{6.2} Proof of~(iv).
\nopagebreak

Let $\e^{1/2}<r\ll1$ to be later optimized as a function of $\e$.
Arguing as in Step~8.5 of the proof of Proposition~\ref{prop:ballconstr}, we find for all $\phi\in W^{1,\infty}(\R^2)$ supported in the ball $B_R(z)$,
\begin{multline*}
\int_{\R^2}\phi\Big(|\nabla u_\e-iu_\e N_\e\!\vi_\e\!|^2+\frac{a}{2\e^2}(1-|u_\e|^2)^2-\Log\nu_{\e,R}^r\Big)\\
\le\|a^{-1}\phi\|_{\Ld^\infty}\int_{\R^2}a\chi_{R}^z\Big(|\nabla u_\e-iu_\e N_\e\!\vi_\e\!|^2+\frac{a}{2\e^2}(1-|u_\e|^2)^2-\Log\nu_{\e,R}^r\Big)\\
+O\Big(\frac{N_\e^2}\Log(|\!\log r|+\log N_\e)\Big)\|a^{-1}\phi\|_{\Ld^\infty}+O(rN_\e^2)\|a^{-1}\phi\|_{W^{1,\infty}}.
\end{multline*}
Using~\eqref{eq:munuexch-hd} to replace $\nu_{\e,R}^r$ by $\mu_\e$ in both sides up to an error of order $(1+rN_\e^2)\|\phi\|_{W^{1,\infty}}$, and choosing $r=N_\e^{-1}$, the conclusion~\eqref{eq:excessestim-hd} follows.
\end{proof}

We now establish the following version of the (suboptimal) a priori estimate of Lemma~\ref{lem:velapbound} on the velocity of the vortices in the nondilute regime $N_\e\gg\Log$.

\begin{lem}[A priori bound on velocity]\label{lem:velapbound-hd}
Let $\alpha\ge0$, $\beta\in\R$, and let $h:\R^2\to\R$, $a:=e^h$, $F:\R^2\to\R^2$, $f:\R^2\to\R$ satisfy~\eqref{eq:scalingshFf}.
Let $u_\e:\R^+\times\R^2\to\C$ and $\vi_\e:\R^+\times\R^2\to\R^2$ be the solutions of~\eqref{eq:GL-1} and~\eqref{eq:GLv1-hd} as in Propositions~\ref{prop:globGL}(i) and~\ref{prop:apriori-est-veps-hd}, respectively.
Let $0<\e\ll1$, $\Log\ll N_\e\lesssim\e^{-1}$, and $R\ge1$ with $\e R\lesssim1$, and assume that $\Ec_{\e,R}^{*,t}\lesssim_t N_\e^2$ for all $t\ge0$. Then, in the regime~\GLt, we have for all $\theta>0$ and $t\ge0$,
\[\alpha^2\sup_z\int_0^t\int_{\R^2} a\chi_R^z|\partial_t u_\e|^2\lesssim_{t,\theta} (1+\e RN_\e)N_\e\Log+R^\theta N_\e^2\Log^2\lesssim R^\theta N_\e^2\Log^2.\qedhere\]
\end{lem}

\begin{proof}
Set $D_{\e,R}^{z,t}:=\int_0^t\int_{\R^2} a\chi_R^z|\partial_tu_\e|^2$. From identity~\eqref{eq:1st-est-velapbound}, using $|\nabla\chi_R^z|\lesssim R^{-1}(\chi_R^z)^{1/2}$, the pointwise estimates of Lemma~\ref{lem:pointest} for $V_\e$ and $j_\e-N_\e\!\vi_\e$,
assumption~\eqref{eq:scalingshFf}, the bound~\eqref{eq:boundpsi} on $\psi_{\e,R}^z$, and the definition of $\hat\Ec_{\e,R}^{z,t}$, we find in the considered regime,
\begin{multline*}
\lambda_\e\alpha D_{\e,R}^{z,t}\lesssim_{t,\theta} N_\e^2\big(1+\|\!\vi_\e\!\|_{\Ld^\infty_t\Ld^4}^2\big)\big(1+\|\partial_t\!\vi_\e\!\|_{\Ld^\infty_t(\Ld^2\cap\Ld^\infty(B_R))}\big)\\
+\e R N_\e^3\big(1+\|\!\vi_\e\!\|_{\Ld^\infty_t\Ld^\infty}\big)\big(1+\|\Gamma_\e\|_{\Ld^\infty_t\Ld^\infty}\big)+\e N_\e^2\Log\|\!\Div(a\!\vi_\e)\|_{\Ld^\infty_t\Ld^2}\\
+N_\e^2\big(1+\|\!\vi_\e\!\|_{\Ld^\infty_t(\Ld^2\cap\Ld^\infty(B_{2R}))}^2+\|\Div(a\!\vi_\e)\|_{\Ld^\infty_t(\Ld^2\cap\Ld^\infty)}\big)(D_{\e,R}^{z,t})^{1/2}+R^{-1}N_\e(D_{\e,R}^{z,t})^{1/2},
\end{multline*}
and hence, using the properties of $\vi_\e$ in Proposition~\ref{prop:apriori-est-veps-hd}, for all $\theta>0$,
\begin{align*}
\lambda_\e\alpha D_{\e,R}^{z,t}\lesssim_{t,\theta} N_\e^2+\e RN_\e^3+N_\e^2R^\theta(D_{\e,R}^{z,t})^{1/2}.
\end{align*}
Absorbing $(D_{\e,R}^{z,t})^{1/2}$ in the left-hand side, the result follows.
\end{proof}

We finally turn to the adaptation of the crucial a priori estimate of Lemma~\ref{lem:poho} to the nondilute regime $N_\e\gg\Log$.

\begin{lem}\label{lem:poho-hd}
Let $\alpha\ge0$, $\beta\in\R$, and let $h:\R^2\to\R$, $a:=e^h$, $F:\R^2\to\R^2$, $f:\R^2\to\R$ satisfy~\eqref{eq:scalingshFf}. Let $u_\e:\R^+\times\R^2\to\C$ and $\vi_\e:\R^+\times\R^2\to\R^2$ be the solutions of~\eqref{eq:GL-1} and~\eqref{eq:GLv1-hd} as in Propositions~\ref{prop:globGL}(i) and~\ref{prop:apriori-est-veps-hd}, respectively.
Let $0<\e\ll1$, $\Log\ll N_\e\lesssim\e^{-1}$, and $R\ge1$ with $\e RN_\e^3\lesssim1$, and assume that $\Ec_{\e,R}^{*,t}\lesssim_t N_\e^2$ for all $t\ge0$. Then, in the regime~\GLt, we have for all $t\ge0$,
\[\alpha^2\sup_z\int_0^t\int_{\R^2}\frac{\chi_R^z}{\e^2}(1-|u_\e|^2)^2\lesssim_t \frac{N_\e^2}\Log.\qedhere\]
\end{lem}

\begin{proof}
Using the pointwise estimates of Lemma~\ref{lem:pointest}, assumption~\eqref{eq:scalingshFf}, and the properties of $\vi_\e$ in~\eqref{eq:apriori-est-veps-hd}, Lemma~\ref{lem:divS} directly yields
\begin{multline*}
|\Div\! \tilde S_\e|\lesssim \big((\lambda_\e+\beta N_\e)|\nabla u_\e-iu_\e N_\e \!\vi_\e\!|+\beta N_\e^2+\beta N_\e^2|1-|u_\e|^2|\big)\big(1+\|\!\vi_\e\!\|_{\Ld^\infty}\big)|\partial_tu_\e|\\
+\big((\lambda_\e+\beta N_\e) N_\e\|\!\pre_{\e}\!\|_{\Ld^\infty}+N_\e\|\curl\!\vi_\e\!\|_{\Ld^\infty}+N_\e^2\|\!\vi_\e\!\|_{\Ld^\infty}\big)(1+|1-|u_\e|^2|)|\nabla u_\e-iu_\e N_\e \!\vi_\e\!|\\
+N_\e(1+\|\!\vi_\e\!\|_{\Ld^\infty})^3(|\nabla u_\e-iu_\e N_\e \!\vi_\e\!|^2+(1-|u_\e|^2)^2+N_\e^2)\\
+N_\e^2|1-|u_\e|^2|\big(1+\|\!\vi_\e\!\|_{\Ld^\infty}\big)\big(N_\e(1+\|\!\vi_\e\!\|_{\Ld^\infty})^3+\lambda_\e\|\!\pre_{\e}\!\|_{\Ld^\infty}+\|\curl\!\vi_\e\!\|_{\Ld^\infty}\big).
\end{multline*}
Using the assumption $\Ec_{\e,R}^{*,t}\lesssim_t N_\e^2$, Lemma~\ref{lem:velapbound-hd} with $R=1$, and the properties of $\vi_\e$ in~\eqref{eq:apriori-est-veps-hd}, we find for $r\le1$,
\begin{align*}
\int_0^t\int_{B_r(x_0)}|\!\Div\!\tilde S_\e| \lesssim_{t} N_\e^4\Log(1+\beta\Log)\lesssim N_\e^4\Log^2.
\end{align*}
Further noting that assumption~\eqref{eq:scalingshFf} yields
\[\int_{B_r(x_0)} a|1-|u_\e|^2||f|\lesssim_t \e r N_\e\|f\|_{\Ld^\infty}\lesssim \e r N_\e^3,\]
and also
\begin{multline*}
\int_{B_r(x_0)}|\nabla\chi_R||\tilde S_\e|\\
\lesssim R^{-1}\int_{B_r(x_0)}\Big(|\nabla u_\e-iu_\e N_\e \!\vi_\e\!|^2+\frac{1}{\e^2}(1-|u_\e|^2)^2+\e^2(N_\e^4|\!\vi_\e\!|^4+|f|^2)\Big)\\
\lesssim R^{-1}\big(N_\e^2+\e^2(N_\e^4\|\!\vi_\e\!\|_{\Ld^\infty}^4+\|f\|_{\Ld^\infty}^2)\big)\lesssim_t R^{-1}N_\e^2,
\end{multline*}
and arguing as in Step~1 of the proof of Lemma~\ref{lem:poho}, we deduce the following Pohozaev type estimate, adapted from~\cite[Theorem~5.1]{SS-book}: for any ball $B_r(x_0)$ with $r\le1$,
\begin{multline*}
\int_0^t\int_{B_r(x_0)} \frac {a^2\chi_R}{2\e^2}(1-|u_\e|^2)^2\lesssim_t r N_\e^4\Log^2\\
+r\int_0^t\int_{\partial B_r(x_0)}\frac {a\chi_R}2\Big(|\nabla u_\e-iu_\e N_\e \!\vi_\e\!|^2+\frac{a}{2\e^2}(1-|u_\e|^2)^2+|1-|u_\e|^2|(N_\e^2|\!\vi_\e\!|^2+|f|)\Big).
\end{multline*}
With this estimate at hand, the conclusion follows from a direct adaptation of Steps~2--3 of the proof of Lemma~\ref{lem:poho}.
\end{proof}

\subsection{Modulated energy argument}

With the above vortex analysis at hand, in the nondilute regime~\GLt{} with $\Log\ll N_\e\ll\Log\log\Log$, we adapt the modulated energy argument of Section~\ref{chap:MFL-GL} and show that the rescaled supercurrent density $\frac1{N_\e}j_\e$ remains close to the solution $\vi_\e$ of equation~\eqref{eq:GLv1-hd}.
Although the well-posedness result of Section~\ref{chap:deg-eqn-exist-bound} for equation~\eqref{eq:GLv1-hd} (hence the final statement of Theorem~\ref{th:main-hd}) is reduced to the parabolic case,
we show that the modulated energy argument formally works in the mixed-flow case as well.
(As we assume $\alpha>0$, all multiplicative constants are implicitly allowed to additionally depend on an upper bound on $\alpha^{-1}$.)

\begin{prop}
Let $\alpha>0$, $\beta\in\R$, $\alpha^2+\beta^2=1$, and let $h:\R^2\to\R$, $a:=e^h$, $F:\R^2\to\R^2$, $f:\R^2\to\R$ satisfy~\eqref{eq:scalingshFf}. Let $u_\e:\R^+\times\R^2\to\C$ be the solution of~\eqref{eq:GL-1} as in Proposition~\ref{prop:globGL}(i). Assume that for some $T>0$ for all $\e>0$ there exists a solution $\vi_\e:[0,T)\times\R^2\to\R^2$ of the following mixed-flow version of~\eqref{eq:GLv1-hd},
\begin{gather}\label{eq:GLv1-hd-mix}
\partial_t\!\vi_\e=\nabla\!\pre_\e+\Gamma_\e\curl\!\vi_\e,\qquad \vi_\e\!|_{t=0}=\vi^\circ,\\
\Gamma_\e:=\lambda_\e^{-1}(\alpha-\Jb\beta)\Big(\nabla^\bot h- F^\bot-\frac{2N_\e}\Log\vi_\e\Big),
\qquad \pre_\e:=(\lambda_\e\alpha a)^{-1}\Div(a\!\vi_\e),\nonumber
\end{gather}
and assume that $\vi_\e$ satisfies the bounds~\eqref{eq:apriori-est-veps-hd} on $[0,T)$.
Let $0<\e\ll1$, $\Log\lesssim N_\e\ll\Log\log\Log$, and $\Log\lesssim R\lesssim\Log^n$ for some $n\ge1$. Assume that the initial modulated energy excess satisfies $\D_{\e,R}^{*,\circ}\lesssim N_\e^{2-\delta}$ for some $\delta>0$.
Then we have $\D_{\e,R}^{*,t}\ll_t N_\e^2$ for all $t\in[0,T)$, hence $\frac1{N_\e}j_\e-\vi_\e\to0$ in $\Ld^\infty_\loc([0,T);\Ld^1_\uloc(\R^2)^2)$ as $\e\downarrow0$.
\end{prop}

\begin{proof}
Let $\Log\lesssim N_\e\lesssim\Log^n$ and $\Log\lesssim R\lesssim\Log^n$ for some $n\ge1$.
Given the assumption $\D_{\e,R}^{*,\circ}\ll N_\e^{2}$ on the initial data, for all $\e>0$ we define $T_\e>0$ as the maximum time $\le T$ such that $\D_{\e,R}^{*,t}\le N_\e^2$ holds for all $t\le T_\e$.
By the proof of Lemma~\ref{lem:apestu} and by Proposition~\ref{prop:ballconstr-hd}, we deduce
for all $t\le T_\e$,
\begin{align}\label{eq:corTeps-hd}
\Ec_{\e,R}^{*,t}\lesssim_t N_\e^2,\quad \hat\Ec_{\e,R}^{*,t}\lesssim_tN_\e^2,\quad \hat\D_{\e,R}^{*,t}\lesssim_t N_\e^2,\quad \D_{\e,R}^{*,t}\lesssim \hat\D_{\e,R}^{*,t}+o_t(\e^{1/2}).
\end{align}
The strategy of the proof consists in showing that for all $t\le T_\e$,
\begin{align}\label{eq:stratprGL-hd}
\hat\D_{\e,R}^{*,t}\lesssim_t\hat\D_{\e,R}^{*,\circ}+N_\e\lambda_\e^3\log\Log+\lambda_\e N_\e\log N_\e+\lambda_\e\int_0^t\hat\D_{\e,R}^{*}.
\end{align}
Combined with~\eqref{eq:corTeps-hd} and with the Grönwall inequality, this implies
\[\D_{\e,R}^{*,t}\lesssim_t e^{C_t\lambda_\e}\big(\D^{*,\circ}_{\e,R}+N_\e\lambda_\e^3\log\Log+\lambda_\e N_\e\log N_\e\big).\]
Then choosing $\Log\lesssim N_\e\ll\Log\log\Log$ and $\D_{\e,R}^\circ\lesssim N_\e^{2-\delta}$ for some $\delta>0$, we deduce $\D_{\e,R}^{*,t}\ll_t N_\e^2$ for all $t\le T_\e$. This gives in particular $T_\e=T$ for $\e>0$ small enough, and the conclusion follows.
To simplify notation, we focus on~\eqref{eq:stratprGL-hd} with the left-hand side $\hat\D_{\e,R}^t$ centered at $z=0$, but the result of course holds uniformly with respect to the translation.

Let us first introduce some notation.
For all $t\le T_\e$, as we are in the framework of Proposition~\ref{prop:ballconstr-hd} with $u_\e^t,\vi_\e^t$, we let $\B_\e^t:=\B_{\e,R}^t$ denote the constructed collection of disjoint closed balls $\B_{\e,R}^{r_\e}(u_\e^t,\vi_\e^t)$ with total radius $r_\e:=N_\e^{-4}$. Let then $\bar\Gamma_\e^t$ denote the corresponding approximation of $\Gamma_\e^t$ given by Lemma~\ref{lem:approx}.
We decompose $\Gamma_\e:=\alpha\Gamma_{\e,0}-\beta\Gamma_{\e,0}^\bot$ with
\[\Gamma_{\e,0}:=\lambda_\e^{-1}\Big(\nabla^\bot h-F^\bot-\frac{2N_\e}\Log\!\vi_\e\Big).\]

\medskip
\noindent\step1 Time derivative of the modulated energy excess.
\nopagebreak

Lemma~\ref{lem:decompcruc} yields the following decomposition,
\begin{align}\label{eq:decompGL0-hd}
\partial_t\hat\D_{\e,R}=~&I_{\e,R}^S+I_{\e,R}^V+I_{\e,R}^E+I_{\e,R}^D+I_{\e,R}^H+I_{\e,R}^d+I_{\e,R}^g+I_{\e,R}^n+I_{\e,R}',
\end{align}
where the eight first terms are as in the statement of Lemma~\ref{lem:decompcruc} while the error $I_{\e,R}'$ is estimated as follows (cf.~\eqref{eq:decompDeR-rest}),
\begin{align*}
\int_0^t|I_{\e,\varrho,R}'|\lesssim_{t}\e R(N_\e^2+\Log^2)(\Ec_{\e,R}^*)^{1/2}\lesssim_t\e^{1/2}.
\end{align*}

\medskip
\noindent\step2 Bound on the error terms.
\nopagebreak

In this step, we prove for all $t\le T_\e$,
\begin{multline}\label{eq:errorterms-hd}
\int_0^t(I_{\e,R}^d+I_{\e,R}^g+I_{\e,R}^n)\\
\lesssim_t 1+R^{-1}N_\e^2+(R^{-1}+N_\e^{-2})\int_0^t\int_{\R^2}\chi_{R}|\partial_tu_\e-iu_\e N_\e\pre_{\e}\!|^2.
\end{multline}
We start with the bound on $I_{\e,R}^n$.
Using Lemma~\ref{lem:velapbound-hd} and the properties of $\vi_{\e}$ in~\eqref{eq:apriori-est-veps-hd}, the quantity $\bar\Ec_{\e,R}^*$ defined in Lemma~\ref{lem:prodest} is estimated as follows, for $\theta>0$ small enough,
\begin{multline*}
\bar\Ec_{\e,R}^{*,t}\lesssim\sup_z\int_0^t\Ec_{\e,R}^{z}+\sup_z\int_0^t\int_{\R^2}\chi_R^z\big(|\partial_tu_\e|^2+N_\e^2|\!\pre_{\e}\!|^2+N_\e^2|1-|u_\e|^2||\!\pre_{\e}\!|^2\big)\\
\lesssim_{t,\theta} N_\e^2+(1+\e RN_\e)N_\e\Log+R^\theta N_\e^2\Log^2+N_\e\Log\|\!\Div(a\!\vi_\e)\|_{\Ld^\infty_t(\Ld^2\cap\Ld^\infty)}^2\\
\lesssim_{t,\theta} \e RN_\e^2\Log+R^\theta N_\e^2\Log^2\lesssim N_\e^2\Log^3\lesssim\Log^{n+3}.
\end{multline*}
Noting that $|\nabla\chi_R|\lesssim R^{-1}\chi_R^{1/2}$ and using Lemma~\ref{lem:approx} in the form $\|\bar\Gamma_\e\|_{\Ld^\infty}\lesssim\|\Gamma_\e\|_{\Ld^\infty}\lesssim1$, Lemma~\ref{lem:prodest} then yields
\begin{multline*}
\bigg|\int_0^t\int_{\R^2} a\tilde V_{\e}\cdot\nabla^\bot\chi_R\bigg|\lesssim_t \Log^{-1}\\
+R^{-1}\Log^{-1}\bigg(\int_0^t\int_{\R^2}\chi_{R}|\partial_tu_\e-iu_\e N_\e\pre_{\e}\!|^2+\int_0^t\int_{B_{2R}}|\nabla u_\e-iu_\e N_\e\!\vi_\e\!|^2\bigg),
\end{multline*}
and hence,
\begin{multline*}
\Big|\int_0^tI_{\e,R}^n\Big|\lesssim_t 1+R^{-1}\int_0^t\int_{\R^2}\chi_{R}|\partial_tu_\e-iu_\e N_\e\pre_{\e}\!|^2\\
+R^{-1}\int_0^t\int_{B_{2R}}\Big(|\nabla u_\e-iu_\e N_\e\!\vi_\e\!|^2+\frac a{2\e^2}(1-|u_\e|^2)^2+|1-|u_\e|^2|(N_\e^2|\!\vi_\e\!|^2+|f|)\Big).
\end{multline*}
Using~\eqref{eq:corTeps-hd}, assumption~\eqref{eq:scalingshFf}, and the properties of $\vi_\e$ in~\eqref{eq:apriori-est-veps-hd}, we conclude
\begin{multline*}
\Big|\int_0^tI_{\e,R}^n\Big|\lesssim_t 1+R^{-1}N_\e^2+\e N_\e^3\big(1+\|\!\vi_\e\!\|_{\Ld^\infty_t\Ld^4}^2\big)+R^{-1}\int_0^t\int_{\R^2}\chi_{R}|\partial_tu_\e-iu_\e N_\e\pre_{\e}\!|^2\\
\lesssim_t 1+R^{-1}N_\e^2+R^{-1}\int_0^t\int_{\R^2}\chi_{R}|\partial_tu_\e-iu_\e N_\e\pre_{\e}\!|^2.
\end{multline*}
We turn to the bound on $I_{\e,R}^g$.
Using~\eqref{eq:scalingshFf} and the pointwise estimates of Lemma~\ref{lem:pointest},
we find
\begingroup\allowdisplaybreaks
\begin{multline*}
|I_{\e,R}^g|\lesssim \|\Gamma_\e-\bar\Gamma_\e\|_{\Ld^\infty}(1+\|\!\vi_\e\!\|_{\Ld^\infty})\bigg(N_\e\int_{\R^2} \chi_R\big(|\nabla u_\e-iu_\e N_\e\!\vi_\e\!|+N_\e|1-|u_\e|^2|\big)|\curl\!\vi_\e\!|\\
+N_\e\int_{\R^2} \chi_R\Big(|\nabla u_\e-iu_\e N_\e \!\vi_\e\!|^2+\frac{a}{\e^2}(1-|u_\e|^2)^2\Big)\\
+N_\e^3\int_{\R^2}\chi_R(1+|1-|u_\e|^2|)|\!\vi_\e\!|^2+\lambda_\e\int_{\R^2} \chi_R |\partial_tu_\e-iu_\e N_\e\!\pre_{\e}\!||\nabla u_\e-iu_\e N_\e \!\vi_\e\!|\\
+\beta N_\e\int_{\R^2} \chi_R |\partial_tu_\e-iu_\e N_\e\!\pre_{\e}\!|\big(|\nabla u_\e-iu_\e N_\e \!\vi_\e\!|+N_\e|1-|u_\e|^2|+N_\e|\!\vi_\e\!|\big)\bigg).
\end{multline*}
\endgroup
By Lemma~\ref{lem:approx} in the form $\|\Gamma_\e-\bar\Gamma_\e\|_{\Ld^\infty}\lesssim r_\e=N_\e^{-4}$ and by the properties of $\vi_\e$ in~\eqref{eq:apriori-est-veps-hd}, we deduce for $\theta>0$ small enough,
\begin{multline}\label{eq:Igestconcl}
|I_{\e,R}^g|\lesssim r_\e N_\e^3R^\theta+r_\e(\lambda_\e N_\e+ R^\theta N_\e^2)\Big(\int_{\R^2} \chi_R |\partial_tu_\e-iu_\e N_\e\!\pre_{\e}\!|^2\Big)^{1/2}\\
\lesssim 1+ N_\e^{-1}\Big(\int_{\R^2} \chi_R |\partial_tu_\e-iu_\e N_\e\!\pre_{\e}\!|^2\Big)^{1/2}\\
\lesssim 1+ N_\e^{-2}\int_{\R^2} \chi_R |\partial_tu_\e-iu_\e N_\e\!\pre_{\e}\!|^2.
\end{multline}
Regarding the last term $I_{\e,R}^d$, the definition of the pressure in~\eqref{eq:GLv1-hd-mix} simply yields $I_{\e,R}^d=0$, and the conclusion~\eqref{eq:errorterms-hd} follows.

\medskip
\noindent\step3 Bound on the dominant terms.
\nopagebreak

In this step, we turn to the estimation of the five first terms in~\eqref{eq:decompGL0-hd}, showing more precisely that for all $t\le T_\e$,
\begin{align}\label{eq:conclstep3Dt-hd}
\hat\D_{\e,R}^t\lesssim_t \hat\D_{\e,R}^{\circ}+N_\e\lambda_\e^3\log\Log+\lambda_\e N_\e\log N_\e+\lambda_\e\int_0^t\hat D_{\e,R}.
\end{align}
As this holds uniformly with respect to translations of the cut-off functions, the conclusion~\eqref{eq:stratprGL-hd} follows.

We start with the bound on the first term $I_{\e,R}^S$.
Since  for all $t$ the field $\bar\Gamma_\e^t$ is constant in each ball of the collection $\B_\e^t$ and satisfies $\|\nabla\bar\Gamma_\e^t\|_{\Ld^\infty}\lesssim\|\nabla\Gamma_\e^t\|_{\Ld^\infty}\lesssim1$, we find
\begin{multline*}
|I_{\e,R}^{S}|\lesssim \int_{\R^2\setminus\B_\e}\chi_R|\tilde S_\e|
\lesssim \int_{\R^2\setminus\B_\e}a\chi_R\Big(|\nabla u_\e-iu_\e N_\e\!\vi_\e\!|^2+\frac{a}{2\e^2}(1-|u_\e|^2)^2\Big)\\
+\int_{\R^2}\chi_R|1-|u_\e|^2|(N_\e^2|\!\vi_\e\!|^2+|f|).
\end{multline*}
Since $\B_\e$ has total radius $r_\e=N_\e^{-4}$, Proposition~\ref{prop:ballconstr-hd}(v) yields
\begin{align*}
|I_{\e,R}^{S}|&\lesssim \D_{\e,R}+\lambda_\e N_\e\log N_\e+\int_{\R^2}\chi_R|1-|u_\e|^2|(N_\e^2|\!\vi_\e\!|^2+|f|).
\end{align*}
Further using~\eqref{eq:corTeps-hd}, assumption~\eqref{eq:scalingshFf}, and the properties of $\vi_\e$ in~\eqref{eq:apriori-est-veps-hd}, we conclude
\begin{align}\label{eq:estIS-hd}
|I_{\e,R}^{S}|&\lesssim \hat\D_{\e,R}+\lambda_\e N_\e\log N_\e.
\end{align}
We turn to $I_{\e,R}^H$. Using the assumption~\eqref{eq:scalingshFf} and the properties of~$\vi_\e$ in~\eqref{eq:apriori-est-veps-hd}, Lemma~\ref{lem:poho-hd} yields
\begin{multline*}
\int_0^tI_{\e,R}^H=O_t(\lambda_\e N_\e)\\
+\int_0^t\int_{\R^2} \frac {a\chi_R}2\Gamma_\e^\bot\cdot\nabla h\Big(|\nabla u_\e-iu_\e N_\e \!\vi_\e\!|^2+\frac{a}{2\e^2}(1-|u_\e|^2)^2-\Log \mu_\e\Big),
\end{multline*}
and hence by Proposition~\ref{prop:ballconstr-hd}(iv) and by~\eqref{eq:corTeps-hd},
\begin{align}\label{eq:estIH-hd}
\int_0^tI_{\e,R}^H&\lesssim_{t}\lambda_\e N_\e\log N_\e+\int_0^t\D_{\e,R}\lesssim_{t}\lambda_\e N_\e\log N_\e+\int_0^t\hat\D_{\e,R}.
\end{align}
The term $I_{\e,R}^D$ is simply estimated by
\begin{align}\label{eq:estID-hd}
I_{\e,R}^D\le-\frac{\lambda_\e \alpha}2\int_{\R^2}{a\chi_R}|\partial_tu_\e-iu_\e N_\e\!\pre_{\e}\!|^2+\frac{\lambda_\e \alpha}2\int_{\R^2}{a\chi_R}|(\nabla u_\e-iu_\e N_\e \!\vi_\e)\cdot\Gamma_\e^\bot|^2.
\end{align}
We finally turn to $I_{\e,R}^V$. Using $\alpha^2+\beta^2=1$, we find $\Gamma_{\e,0}-\beta \Gamma_\e^\bot=\alpha\Gamma_\e$, so that $I_{\e,R}^V$ takes on the following guise,
\begin{align*}
I_{\e,R}^V=N_\e\int_{\R^2} \frac {a\chi_R}2\tilde V_{\e}\cdot (\Gamma_{\e,0}-\beta \Gamma_\e^\bot)=\alpha N_\e\int_{\R^2} \frac {a\chi_R}2\tilde V_{\e}\cdot \Gamma_{\e}.
\end{align*}
As shown in Step~2, the quantity $\bar\Ec_{\e,R}^*$ defined in Lemma~\ref{lem:prodest}  satisfies $\bar\Ec_{\e,R}^{*,t}\lesssim_t\Log^{n+3}$. Choosing $M_\e:=\exp((\lambda_\e\log\Log)\wedge\Log^{1/2})$, Lemma~\ref{lem:prodest} then yields for any $\Lambda\simeq1$,
\begin{multline*}
\Big|\int_0^tI_{\e,R}^V\Big|\le o_t(1)+\lambda_\e\alpha\bigg(1+O_t\Big(\Log^{-1/2}\wedge\frac{\lambda_\e\log\Log}{\Log}\Big)\bigg)\\
\times\bigg(\frac1\Lambda\int_0^t\int_{\R^2} a\chi_R|\partial_tu_\e-iu_\e N_\e\!\pre_{\e}\!|^2+\frac \Lambda4\int_0^t\int_{\R^2} a\chi_R|(\nabla u_\e-iu_\e N_\e\!\vi_\e)\cdot \Gamma_\e|^2\bigg),
\end{multline*}
and thus, using the optimal energy bound~\eqref{eq:corTeps-hd},
\begin{multline}\label{eq:applprodest-hd}
\Big|\int_0^tI_{\e,R}^V\Big|\le O_t\big(N_\e\lambda_\e^3\log\Log\big)\\
+\bigg(1+O_t\Big(\Log^{-1/2}\wedge\frac{\lambda_\e\log\Log}{\Log}\Big)\bigg)\frac{\lambda_\e\alpha}\Lambda\int_0^t\int_{\R^2} a\chi_R|\partial_tu_\e-iu_\e N_\e\!\pre_{\e}\!|^2\\+\frac{\lambda_\e\alpha\Lambda}4\int_0^t\int_{\R^2} a\chi_R|(\nabla u_\e-iu_\e N_\e\!\vi_\e)\cdot \Gamma_\e|^2.
\end{multline}
We distinguish between two cases,
\begin{align}
\text{Case 1:}\quad\int_0^t\int_{\R^2} a\chi_R|\partial_tu_\e-iu_\e N_\e\!\pre_{\e}\!|^2\le 5\int_0^t\int_{\R^2} a\chi_R|(\nabla u_\e-iu_\e N_\e\!\vi_\e)\cdot\Gamma_\e|^2,\label{eq:case1bound-hd}\\
\text{Case 2:}\quad\int_0^t\int_{\R^2} a\chi_R|\partial_tu_\e-iu_\e N_\e\!\pre_{\e}\!|^2> 5\int_0^t\int_{\R^2} a\chi_R|(\nabla u_\e-iu_\e N_\e\!\vi_\e)\cdot\Gamma_\e|^2.\label{eq:case2bound-hd}
\end{align}
In Case~1, choosing $\Lambda=2$ in~\eqref{eq:applprodest-hd} yields
\begin{multline*}
\Big|\int_0^tI_{\e,R}^V\Big|\le O_t\big(N_\e\lambda_\e^3\log\Log\big)+\frac{\lambda_\e\alpha}2\bigg(1+O_t\Big(\frac{\lambda_\e\log\Log}{\Log}\Big)\bigg)\int_0^t\int_{\R^2} a\chi_R|\partial_tu_\e-iu_\e N_\e\!\pre_{\e}\!|^2\\+\frac{\lambda_\e\alpha}2\int_0^t\int_{\R^2} a\chi_R|(\nabla u_\e-iu_\e N_\e\!\vi_\e)\cdot \Gamma_\e|^2.
\end{multline*}
In Case~2, the condition~\eqref{eq:case2bound-hd} can be rewritten as
\begin{multline*}
\frac14\int_0^t\int_{\R^2} a\chi_R|\partial_tu_\e-iu_\e N_\e\!\pre_{\e}\!|^2+\int_0^t\int_{\R^2} a\chi_R|(\nabla u_\e-iu_\e N_\e\!\vi_\e)\cdot\Gamma_{\e}|^2\\
\le\Big(\frac14+\frac1{10}\Big)\int_0^t\int_{\R^2} a\chi_R|\partial_tu_\e-iu_\e N_\e\!\pre_{\e}\!|^2+\frac12\int_0^t\int_{\R^2} {a\chi_R}|(\nabla u_\e-iu_\e N_\e\!\vi_\e)\cdot\Gamma_{\e}|^2,
\end{multline*}
and choosing $\Lambda=4$ in~\eqref{eq:applprodest-hd} then yields
\begin{multline*}
\Big|\int_0^tI_{\e,R}^V\Big|\le O_t\big(N_\e\lambda_\e^3\log\Log\big)+\lambda_\e\alpha\Big(\frac14+\frac1{10}+o_t(1)\Big)\int_0^t\int_{\R^2} a\chi_R|\partial_tu_\e-iu_\e N_\e\!\pre_{\e}\!|^2\\+\frac{\lambda_\e\alpha}2\int_0^t\int_{\R^2} a\chi_R|(\nabla u_\e-iu_\e N_\e\!\vi_\e)\cdot \Gamma_\e|^2.
\end{multline*}
Further noting that in Case~1 the condition~\eqref{eq:case1bound-hd} together with the energy bound~\eqref{eq:corTeps-hd} yields
\begin{multline*}
\Big(R^{-1}+N_\e^{-2}+\frac{\lambda_\e^2\log\Log}\Log\Big)\int_{\R^2} a\chi_R|\partial_tu_\e-iu_\e N_\e\!\pre_{\e}\!|^2\\
\lesssim\Big(R^{-1}+N_\e^{-2}+\frac{\lambda_\e^2\log\Log}\Log\Big)\int_0^t\int_{\R^2} a\chi_R|\nabla u_\e-iu_\e N_\e\!\vi_\e\!|^2\lesssim_t N_\e\lambda_\e^3 \log\Log,
\end{multline*}
and combining this with~\eqref{eq:errorterms-hd} and~\eqref{eq:estID-hd}, we observe an exact recombination of the terms, and obtain in Case~1,
\begin{multline}\label{eq:preconclcase1-hd}
\int_0^t(I_{\e,R}^V+I_{\e,R}^D+I_{\e,R}^d+I_{\e,R}^g+I_{\e,R}^n+I'_{\e,R})\\
\le\frac{\lambda_\e \alpha}2\int_0^t\int_{\R^2}{a\chi_R}|\nabla u_\e-iu_\e N_\e \!\vi_\e|^2|\Gamma_\e|^2+O_t(N_\e\lambda_\e^3\log\Log),
\end{multline}
and in Case~2,
\begin{multline*}
\int_0^t(I_{\e,R}^V+I_{\e,R}^D+I_{\e,R}^g+I_{\e,R}^n+I'_{\e,R})\\
\le-\frac{\lambda_\e \alpha}2\Big(\frac12-\frac1{5}-o_t(1)\Big)\int_0^t\int_{\R^2}{a\chi_R}|\partial_tu_\e-iu_\e N_\e\!\pre_{\e}\!|^2\\
+\frac{\lambda_\e \alpha}2\int_0^t\int_{\R^2}{a\chi_R}|\nabla u_\e-iu_\e N_\e \!\vi_\e|^2|\Gamma_\e|^2+O_t(N_\e\lambda_\e^3\log\Log),
\end{multline*}
so that~\eqref{eq:preconclcase1-hd} holds in both cases for $\e>0$ small enough.
Using $\alpha^2+\beta^2=1$, we find $\Gamma_\e\cdot\Gamma_{\e,0}=\alpha|\Gamma_{\e,0}|^2=\alpha|\Gamma_\e|^2$, so that the term $I_{\e,R}^E$ takes on the following guise,
\[I_{\e,R}^E=-\frac{\lambda_\e}2\Log\int_{\R^2} {a\chi_R}\Gamma_\e\cdot\Gamma_{\e,0}\,\mu_\e=-\frac{\lambda_\e\alpha}2\Log\int_{\R^2} {a\chi_R}|\Gamma_\e|^2\mu_\e.\]
Together with~\eqref{eq:preconclcase1-hd}, this yields
\begin{multline*}
\int_0^t(I_{\e,R}^V+I_{\e,R}^E+I_{\e,R}^D+I_{\e,R}^g+I_{\e,R}^n+I'_{\e,R})\\
\le\frac{\lambda_\e \alpha}2\int_0^t\int_{\R^2}{a\chi_R}\big(|\nabla u_\e-iu_\e N_\e \!\vi_\e\!|^2-\Log\mu_\e\big)|\Gamma_\e|^2+O_t(N_\e\lambda_\e^3\log\Log).
\end{multline*}
Combining this with~\eqref{eq:decompGL0-hd}, \eqref{eq:estIS-hd}, and~\eqref{eq:estIH-hd}, we conclude
\begin{multline*}
\hat\D_{\e,R}^t- \hat\D_{\e,R}^\circ\lesssim_t\int_0^t\hat\D_{\e,R}+\frac{\lambda_\e \alpha}2\int_0^t\int_{\R^2}{a\chi_R}\big(|\nabla u_\e-iu_\e N_\e \!\vi_\e\!|^2-\Log\mu_\e\big)|\Gamma_\e|^2\\
+N_\e\lambda_\e^3\log\Log+\lambda_\e N_\e\log N_\e,
\end{multline*}
and the result~\eqref{eq:conclstep3Dt-hd} now follows from Proposition~\ref{prop:ballconstr-hd}(iv).

\medskip
\noindent\step4 Conclusion.

As explained at the beginning of the proof, in the regime $\Log\lesssim N_\e\ll\Log\log\Log$ with $\D_{\e,R}^\circ\lesssim N_\e^{2-\delta}$ for some $\delta>0$, the estimate~\eqref{eq:stratprGL-hd} implies $T_\e=T$ and $\D_{\e,R}^{*,t}\ll_tN_\e^2$ for all $t\in[0,T)$. We now show that it implies the convergence $\frac1{N_\e}j_\e-\vi_\e\to0$.
For all $t\in[0,T)$, Proposition~\ref{prop:ballconstr-hd}(v) gives
\[\sup_z\int_{\R^2\setminus\B_{\e}}\chi_R^z|\nabla u_\e-iu_\e N_\e\!\vi_\e\!|^2\ll_t N_\e^2,\]
and for all $1\le p<2$,
\[\sup_z\int_{\B_{\e}}\chi_R^z|\nabla u_\e-iu_\e N_\e\!\vi_\e\!|^p\lesssim|\B_{\e}|^{1-p/2}(\Ec_{\e,R}^{*})^{p/2}\lesssim_t r_\e^{2-p}N_\e^p\ll_p N_\e^p.\]
Using the pointwise estimates of Lemma~\ref{lem:pointest}, we deduce
\begin{multline*}
\sup_z\int_{B(z)}|j_\e-N_\e\!\vi_\e\!|\lesssim_t \sup_z\int_{B(z)}|\nabla u_\e-iu_\e N_\e \!\vi_\e\!|+\e N_\e^2\\
\lesssim_t \sup_z\int_{\B_{\e}}\chi_R^z|\nabla u_\e-iu_\e N_\e\!\vi_\e\!|+\sup_z\Big(\int_{B(z)\setminus\B_{\e}}|\nabla u_\e-iu_\e N_\e\!\vi_\e\!|^2\Big)^{1/2}+\e N_\e^2\ll_t N_\e,
\end{multline*}
hence $\frac1{N_\e}j_\e-\vi_\e\to0$ in $\Ld^\infty_\loc([0,T);\Ld^1_\uloc(\R^2)^2)$.
\end{proof}

\section{Mean-field limit in the conservative case}\label{chap:MFL-GP}
In this section, we prove Theorem~\ref{th:mainGP}, that is, the mean-field limit result in the conservative case ($\alpha=0$, $\beta=1$) in the regime \GP. More precisely, the rescaled supercurrent density $\frac1{N_\e}j_\e$ is shown to remain close to the solution $\vi_\e$ of equation~\eqref{eq:limGP}.
Combining this with the results of Section~\ref{chap:GP-lim} (in particular, with Lemma~\ref{lem:lastlimGP}), the result of Theorem~\ref{th:mainGP} follows.

\subsection{Preliminary: vortex analysis}
In the present situation, it is not needed to adapt the ball-construction lower bound of Section~\ref{sec:vortex} to the nondilute regime $N_\e\gg\Log$: we only need the following elementary estimate on the number of vortices based on a bound on the modulated energy excess. Since the vector field $\nabla h$ is assumed here to decay at infinity, the proof is considerably reduced with respect to the corresponding statement in Section~\ref{chap:vort-anal-hd}.
Note that in the considered regime $N_\e\gg\Log$ we show that $\Ec_{\e,R}$ and $\D_{\e,R}$ are interchangeable up to an error of order $o(N_\e^2)$.

\begin{lem}\label{lem:ballconstrGP}
Let $h:\R^2\to\R$, $a:=e^h$, with $a\le1$ and $\|\nabla h\|_{\Ld^2\cap\Ld^\infty}\lesssim1$, let $u_\e:\R^2\to\C$, $\vi_\e:\R^2\to\R^2$, with $\|\curl\!\vi_\e\!\|_{\Ld^1\cap\Ld^\infty},\|\!\vi_\e\!\|_{\Ld^\infty}\lesssim1$. Let $0<\e\ll1$, $\Log\ll N_\e\lesssim\e^{-1}$, and $R\ge1$, and assume that $\D_{\e,R}^*\lesssim N_\e^2$. Then,
\[\sup_z\|\mu_\e\|_{(\dot H^1\cap W^{1,\infty}(B_{R}(z)))^*}\lesssim N_\e,\]
hence in particular
\[\sup_z|\Ec_{\e,R}^z-\D_{\e,R}^z|\lesssim N_\e\Log\ll N_\e^2.\qedhere\]
\end{lem}

\begin{proof}
Let $\phi\in \dot H^1\cap W^{1,\infty}(\R^2)$ be supported in a ball of radius $R$. We decompose
\begin{align*}
\int_{\R^2} \phi\mu_\e&=\int_{\R^2} \phi\big(N_\e\curl\!\vi_\e+\,\curl(j_\e-N_\e\!\vi_\e)\big)\\
&= N_\e\int_{\R^2} \phi\,\curl\! \vi_\e-\int_{\R^2} \nabla^\bot\phi\cdot(j_\e-N_\e\!\vi_\e),
\end{align*}
hence, using the pointwise estimates of Lemma~\ref{lem:pointest},
\begin{align}\label{eq:boundmupre}
\Big|\int_{\R^2} \phi\mu_\e\Big|&\lesssim N_\e\|\phi\|_{\Ld^\infty}+(1+\e N_\e)(\Ec_{\e,R}^*)^{1/2}\|\nabla\phi\|_{\Ld^2}+\e\Ec_{\e,R}^*\|\nabla\phi\|_{\Ld^\infty}.
\end{align}
In particular, using the assumptions $\D_{\e,R}^*\lesssim N_\e^2$ and $\|\nabla h\|_{\Ld^2\cap\Ld^\infty}\lesssim1$, we obtain
\[\Ec_{\e,R}^z=\D_{\e,R}^z+\Log\int_{\R^2} a\chi_R^z\mu_\e\lesssim N_\e^2+(1+\e N_\e)\Log(\Ec_{\e,R}^*)^{1/2}+\e\Log \Ec_{\e,R}^*,\]
which implies, taking the supremum in $z$ and absorbing $\Ec_{\e,R}^*$ in the left-hand side, for $\e>0$ small enough,
\[\Ec_{\e,R}^*\lesssim N_\e^2+(1+\e N_\e)^2\Log^2\lesssim N_\e^2.\]
Inserting this into~\eqref{eq:boundmupre} yields $|\int_{\R^2}\phi\mu_\e|\lesssim N_\e\|\phi\|_{\dot H^1\cap W^{1,\infty}}$, and the result follows.
\end{proof}

\subsection{Modulated energy argument}
By a modulated energy argument, we show that the rescaled supercurrent density $\frac1{N_\e}j_\e$ remains close to the solution $\vi_\e$ of equation~\eqref{eq:limGP}. The proof consists in estimating the different terms in the decomposition of $\partial_t\hat\D_{\e,\varrho,R}$ in Lemma~\ref{lem:decompcruc} and then deducing the smallness of the modulated energy $\hat\Ec_{\e,\varrho,R}$ by a Grönwall argument.
Note that in the nondilute regime $N_\e\gg\Log$ the situation is greatly simplified with respect to Section~\ref{chap:MFL-GL}, since the modulated energy $\Ec_{\e,R}$ and the excess $\D_{\e,R}$ are now interchangeable up to an error $o(N_\e^2)$ (cf.~Lemma~\ref{lem:ballconstrGP}). The different terms appearing in Lemma~\ref{lem:decompcruc} thus only need to be estimated by means of the modulated energy $\Ec_{\e,R}$ without having to take care to substract the correct vortex self-interaction energy. In particular, the vector field $\Gamma_\e$ does no longer need to be truncated on small balls around the vortex locations, and we simply set $\bar\Gamma_\e=\Gamma_\e$. For this choice, all the terms involving the vortex velocity $\tilde V_{\e,\varrho}$ in Lemma~\ref{lem:decompcruc} vanish. This simplification is crucial since in the present conservative case no good a priori control on the vortex velocity is available (apart from rough a priori estimates of the form $\|\partial_tu_\e-iu_\e N_\e\!\pre_{\e,\varrho}\!\|_{\Ld^2}\lesssim\e^{-2}$), which indeed prevents us from extending this modulated energy argument to the case $N_\e\lesssim\Log$.

\begin{prop}\label{prop:mflGP}
Let $\alpha=0$, $\beta=1$, and let $h:\R^2\to\R$, $a:=e^h$, $F:\R^2\to\R^2$, $f:\R^2\to\R$ satisfy~\eqref{eq:scalingshFfdec}. Let $u_\e:[0,T)\times\R^2\to\C$ and $\vi_\e:\R^+\times\R^2\to\R^2$ be solutions of~\eqref{eq:GL-1} and~\eqref{eq:limGP} as in Propositions~\ref{prop:globGL}(ii) and~\ref{prop:GPvprop}, respectively, for some $T>0$. Let $0<\e\ll1$, $\Log\ll N_\e\ll\e^{-1}$, $R\gtrsim\|\partial_tu_\e\|_{\Ld^\infty_T\Ld^2}+\Log^2$, and assume that the initial modulated energy satisfies $\Ec_{\e,R}^{*,\circ}\ll N_\e^2$. Then, in the regime \GP, we have $\Ec_{\e,R}^{*,t}\ll_t N_\e^2$ for all $t\in[0,T)$, hence $\frac1{N_\e}j_\e-\vi_\e\to0$ in $\Ld^\infty_\loc([0,T);\Ld^1_{\uloc}(\R^2)^2)$ as $\e\downarrow0$.
Under the stronger assumption $\Ec_{\e}^{*,\circ}\ll N_\e^2$, the same convergence holds in $\Ld^\infty_\loc([0,T);(\Ld^1+\Ld^2)(\R^2)^2)$.
\end{prop}

\begin{proof}
In the sequel, we choose $1\ll \varrho\le R$ with $\varrho^{\theta_0}\ll(\e N_\e)^{-1}$ for some $\theta_0>0$. Regarding the global truncation at the scale $R$, it is not really needed in the present context (as a consequence of the decay assumption for $\nabla h,F,f$) and can be sent to infinity arbitrarily fast; here it suffices to choose $R\ge\|\partial_tu_\e\|_{\Ld^\infty_T\Ld^2}+\Log^2$ (where the right-hand side is indeed finite by Proposition~\ref{prop:globGL}(ii)).
Given the assumption $\Ec_{\e,R}^{*,\circ}\ll N_\e^2$ on the initial data, for all $\e>0$ we define $T_\e>0$ as the maximum time $\le T$ such that $\Ec_{\e,R}^{*,t}\le N_\e^2$ holds for all $t\le T_\e$. By Lemmas~\ref{lem:apestu} and~\ref{lem:ballconstrGP}, we deduce $\hat\D_{\e,\varrho,R}^{*,\circ}\ll N_\e^2$ and for all $t\le T_\e$,
\begin{gather}\label{eq:corTepsGP}
\D_{\e,R}^{*,t}\lesssim_t N_\e^2,\qquad \hat\Ec_{\e,\varrho,R}^{*,t}\lesssim_tN_\e^2,\qquad \hat\D_{\e,\varrho,R}^{*,t}\lesssim_t N_\e^2,\nonumber\\
\Ec_{\e,R}^{*,t}\lesssim \hat\Ec_{\e,\varrho,R}^{*,t}+o_t(N_\e^2),\qquad \hat\Ec_{\e,\varrho,R}^{*,t}\lesssim \hat\D_{\e,\varrho,R}^{*,t}+o_t(N_\e^2).
\end{gather}
The strategy of the proof consists in showing that for all $t\le T_\e$,
\begin{align}\label{eq:stratprGP}
\hat\Ec_{\e,\varrho,R}^{*,t}\lesssim_t o(N_\e^2)+\int_0^t\hat\Ec_{\e,\varrho,R}^{*}.
\end{align}
This estimate is proved in Step~1 below. To simplify notation, we focus on~\eqref{eq:stratprGP} with the left-hand side $\hat\Ec_{\e,\varrho,R}^t$ centered at $z=0$, but the result of course holds uniformly with respect to the translation. By the Grönwall inequality, it implies $\hat\Ec_{\e,\varrho,R}^{*,t}\ll_t N_\e^2$, hence $\Ec_{\e,R}^{*,t}\ll_tN_\e^2$ for all $t\le T_\e$. This yields in particular $T_\e=T$ for all $\e>0$ small enough, and the main conclusion follows, while the additional statements are deduced in Step~2.

\medskip
\noindent\step1 Proof of~\eqref{eq:stratprGP}.

\nopagebreak
Using the constraint $0=a^{-1}\Div(a\!\vi_\e)=\Div\!\vi_\e+\vi_\e\!\cdot\,\nabla h$, and choosing $\bar\Gamma_\e:=\Gamma_\e$, the result of Lemma~\ref{lem:decompcruc} takes the following simpler form,
\begin{align}\label{eq:GPapplydecompcruc}
\partial_t\hat\D_{\e,\varrho,R}=~&I_{\e,\varrho,R}^S+I_{\e,\varrho,R}^V+I_{\e,\varrho,R}^E+I_{\e,\varrho,R}^H+I_{\e,\varrho,R}^n+I_{\e,\varrho,R}',
\end{align}
where we have set
\begin{eqnarray*}
I_{\e,\varrho,R}^S&:=&-\int_{\R^2}\chi_R\nabla\Gamma_\e^\bot: \tilde S_\e,\\
I_{\e,\varrho,R}^V&:=&\int_{\R^2} \frac {a\chi_R\Log}2\tilde V_{\e,\varrho}\cdot \Big(-\lambda_\e \Gamma_\e^\bot+\nabla^\bot h-F^\bot-\frac{2N_\e}\Log \!\vi_\e\Big),\\
I_{\e,\varrho,R}^E&:=&-\int_{\R^2} \frac {a\chi_R\Log}2\Gamma_\e\cdot \Big(\nabla^\bot h-F^\bot-\frac{2N_\e}\Log \!\vi_\e\Big)\mu_\e,\\
I_{\e,\varrho,R}^H&:=&\int_{\R^2} \frac {a\chi_R}2\Gamma_\e^\bot\cdot\nabla h\Big(|\nabla u_\e-iu_\e N_\e \!\vi_\e\!|^2+\frac{a}{\e^2}(1-|u_\e|^2)^2-\Log \mu_\e\Big),\\
I_{\e,\varrho,R}^n&:=&-\int_{\R^2}\nabla\chi_R\cdot\tilde S_\e\cdot\Gamma_\e^\bot\\
&&\qquad- \int_{\R^2} a\nabla\chi_R\cdot\Big(\langle\partial_tu_\e-iu_\e N_\e\!\pre_\e,\nabla u_\e-iu_\e N_\e \!\vi_\e\rangle+\frac{\Log }2\tilde V_{\e,\varrho}^\bot\Big),
\end{eqnarray*}
and where the error $I_{\e,\varrho,R}'$ is estimated as follows (cf.~\eqref{eq:decompDeR-rest-dec}),
\begin{align*}
|I_{\e,\varrho,R}'|\lesssim_{t,\theta}\e N_\e\Ec_{\e,R}^*+N_\e(\Ec_{\e,R}^*)^{1/2}\|\nabla(\pre_\e-\pre_{\e,\varrho})\|_{\Ld^2}+\e N_\e^2\varrho^\theta(\Ec_{\e,R}^*)^{1/2}.
\end{align*}
Choosing $\theta>0$ small enough, and using Proposition~\ref{prop:GPvprop} in the form $\|\nabla(\pre_\e^t-\pre_{\e,\varrho}^t)\|_{\Ld^2}\ll_t1$ (cf.~\eqref{eq:diffpressuretrunc}), we obtain
\begin{align}\label{eq:decompDeR-rest-dec-GP}
|I_{\e,\varrho,R}'|\lesssim_{t,\theta}\Ec_{\e,R}^*+o(N_\e)(\Ec_{\e,R}^*)^{1/2}.
\end{align}
The choice~\eqref{eq:limGP} of $\Gamma_\e$ yields $I_{\e,\varrho,R}^V=I_{\e,\varrho,R}^E=0$, hence
\begin{align}\label{eq:decompderivDeR}
\partial_t\hat\D_{\e,\varrho,R}=~&I_{\e,\varrho,R}^S+I_{\e,\varrho,R}^H+I_{\e,\varrho,R}^n+I_{\e,\varrho,R}'.
\end{align}
It remains to estimate the first three right-hand side terms.
By assumption~\eqref{eq:scalingshFfdec} in the form $\|f\|_{\Ld^2}\lesssim N_\e^2$ and by the integrability properties of $\vi_\e$ (cf.\@ Proposition~\ref{prop:GPvprop}), the first right-hand side term $I_{\e,\varrho,R}^S$ is estimated as follows, for all $t\le T_\e$,
\begin{multline}\label{eq:ISestGP}
I_{\e,\varrho,R}^S\lesssim\|\nabla\Gamma_\e\|_{\Ld^\infty}\int_{\R^2}a\chi_R\Big(|\nabla u_\e-iu_\e N_\e \!\vi_\e\!|^2+\frac a{\e^2}(1-|u_\e|^2)^2+|1-|u_\e|^2|(N_\e^2|\!\vi_\e\!|^2+|f|)\Big)\\
\lesssim_t\Ec_{\e,R}+\e N_\e^2(\Ec_{\e,R})^{1/2}\lesssim \Ec_{\e,R}+o(N_\e^2).
\end{multline}
We turn to the second right-hand side term in~\eqref{eq:decompderivDeR}.
Lemma~\ref{lem:ballconstrGP} yields
\begin{align*}
I_{\e,\varrho,R}^H&\le \|\Gamma_\e^\bot\cdot\nabla h\|_{\Ld^\infty}\int_{\R^2} \chi_R\Big(|\nabla u_\e-iu_\e N_\e\!\vi_\e\!|^2+\frac a{\e^2}(1-|u_\e|^2)^2\Big)\\
&\hspace{4cm}+\Log\Big|\int_{\R^2} a\chi_R\Gamma_\e^\bot\cdot\nabla h\,\mu_\e\Big|\\
&\lesssim\Ec_{\e,R}\|\Gamma_\e^\bot\cdot\nabla h\|_{\Ld^\infty}+N_\e\Log\|a\chi_R\Gamma_\e^\bot\cdot\nabla h\|_{\dot H^1\cap W^{1,\infty}},
\end{align*}
and hence, using assumption~\eqref{eq:scalingshFfdec} and the properties of $\vi_\e$ (cf.\@ Proposition~\ref{prop:GPvprop}),
\begin{align}\label{eq:IHestGP}
I_{\e,\varrho,R}^H&\lesssim_t\Ec_{\e,R}+N_\e\Log\le \Ec_{\e,R}+o(N_\e^2).
\end{align}
It remains to estimate the third right-hand side term in~\eqref{eq:decompderivDeR}.
By definition of $\tilde S_\e$ and $\tilde V_{\e,\varrho}$, we find
\begin{multline*}
I_{\e,\varrho,R}^n\lesssim R^{-1}\Log\int_{B_{2R}}|\partial_t u_\e-iu_\e N_\e\!\pre_{\e,\varrho}\!||\nabla u_\e-iu_\e N_\e\!\vi_\e\!|\\
+R^{-1}\|\Gamma_\e\|_{\Ld^\infty}\int_{B_{2R}}\Big(|\nabla u_\e-iu_\e N_\e \!\vi_\e\!|^2+\frac a{\e^2}(1-|u_\e|^2)^2+|1-|u_\e|^2|(N_\e^2|\!\vi_\e\!|^2+|f|)\Big),
\end{multline*}
and hence, using assumption~\eqref{eq:scalingshFfdec}, the properties of $\vi_\e$ (cf.\@ Proposition~\ref{prop:GPvprop}), and the bound $\Ec_{\e,2R}^*\lesssim\Ec_{\e,R}^*$ (cf.~\eqref{eq:boundbyEmod}),
\begin{align*}
I_{\e,\varrho,R}^n
\lesssim_t \Ec_{\e,R}^*+R^{-1}\Log (\Ec_{\e,R}^*)^{1/2}\|\partial_tu_\e-iu_\e N_\e\!\pre_{\e,\varrho}\!\|_{\Ld^2(B_{2R})}+o(N_\e^2).
\end{align*}
The properties of $\pre_\e$ (cf.\@ Proposition~\ref{prop:GPvprop}) yield for all $\theta>0$,
\begin{eqnarray*}
\lefteqn{\|\partial_tu_\e-iu_\e N_\e\!\pre_{\e,\varrho}\!\|_{\Ld^2(B_{2R})}}\\
&\lesssim&\|\partial_tu_\e\|_{\Ld^2(B_{2R})}+N_\e\|\!\pre_{\e,\varrho}\!\|_{\Ld^2(B_{2R})}+N_\e\|\!\pre_{\e,\varrho}\!\|_{\Ld^\infty(B_{2R})}\|1-|u_\e|^2|\|_{\Ld^2(B_{2R})}\\
&\lesssim_{t,\theta}& \|\partial_tu_\e\|_{\Ld^2(B_{2R})}+N_\e\varrho^{\theta}+\e N_\e(\Ec_{\e,R}^*)^{1/2},
\end{eqnarray*}
so that the above takes the form
\begin{align*}
I_{\e,\varrho,R}^n&\lesssim_{t,\theta} \Ec_{\e,R}^*+R^{-2}\Log^2\|\partial_tu_\e\|_{\Ld^2(B_{2R})}^2+R^{-2(1-\theta)}N_\e^2\Log^2+o(N_\e^2).
\end{align*}
Using the choice $R\gtrsim\|\partial_tu_\e\|_{\Ld^2}+\Log^2$, and choosing $\theta>0$ small enough, we deduce $I_{\e,\varrho,R}^n\lesssim_t \Ec_{\e,R}^*+o(N_\e^2)$. Combining this with~\eqref{eq:decompDeR-rest-dec-GP}, \eqref{eq:decompderivDeR}, \eqref{eq:ISestGP}, and~\eqref{eq:IHestGP}, we conclude
\[\partial_t\hat\D_{\e,\varrho,R}\lesssim_t \Ec_{\e,R}^*+o(N_\e^2).\]
Integrating this in time with $\hat\D_{\e,\varrho,R}^{*,\circ}\ll N_\e^2$, using~\eqref{eq:corTepsGP}, and noting that the result holds uniformly with respect to translations of the cut-off functions, the conclusion~\eqref{eq:stratprGP} follows.

\medskip
\noindent\step2 Conclusion.

\nopagebreak
As explained, the result of Step~1 implies $T_\e=T$ and $\Ec_{\e,R}^{*,t}\ll_t N_\e^2$ for all $t\in[0,T)$. We now show that it implies $\frac1{N_\e}j_\e-\vi_\e\to0$. Using the pointwise estimates of Lemma~\ref{lem:pointest}, we obtain
\begin{multline*}
\|j_\e-N_\e\!\vi_\e\!\|_{(\Ld^1+\Ld^2)(B_R(z))}\\
\lesssim \|\nabla u_\e-iu_\e N_\e\!\vi_\e\!\|_{\Ld^2(B_R(z))}\big(1+\|1-|u_\e|^2\|_{\Ld^2(B_R(z))}\big)+N_\e\|1-|u_\e|^2\|_{\Ld^2(B_R(z))}\\
\ll_t N_\e(1+\e N_\e)~\lesssim N_\e,
\end{multline*}
and the conclusion follows, letting $R\uparrow\infty$.
\end{proof}

\section{Homogenization regimes}\label{chap:homog}

In this section, we briefly examine homogenization regimes and we prove the few rigorous results mentioned in Section~\ref{sec:homog}.
We focus on the dissipative case and for simplicity we restrict to the periodic setting, that is,
\begin{align}\label{eq:fastoscpin}
\hat a(x):=\hat a^0\big(x,\tfrac1{\eta_\e}x\big)^{\eta_\e},
\end{align}
with $\hat a^0:\R^d\times Q\to[\frac1C,1]$ periodic in its second variable.
We set $\hat h:=\log\hat a$ and $\hat h^0:=\log\hat a^0$.

\subsection{Homogenization diagonal result}\label{chap:modarghomog}

In this section, we adapt the modulated energy approach to the case with wiggly pinning weight~\eqref{eq:fastoscpin}. As the first term in the decomposition of $\partial_t\hat\D_{\e,\varrho,R}$ in Lemma~\ref{lem:decompcruc} involves the gradient of the mean-field driving vector field $\Gamma_\e$ (cf.~\eqref{eq:GLv1}), the wiggly pinning force leads to a divergent prefactor $O(\eta_\e^{-1})$ that destroys the Grönwall relation on $\hat\D_{\e,\varrho,R}$. For that reason, such an argument can only work in a suitable diagonal regime, as stated in Corollary~\ref{cor:homogdiag}. Note that the choice of the diagonal regime $\eta_{\e,0}\ll\eta_\e\ll1$ could be made more explicit, but this is left to the reader.

\begin{proof}[Proof of Corollary~\ref{cor:homogdiag}]
Given a fast oscillating pinning potential~\eqref{eq:fastoscpin}, we consider the regimes~\GLu, \GLd, \GLup, and \GLdp, and in the regime~\GLd{} we restrict to the parabolic case $\beta=0$.
We now denote by $\vi_\e$ the unique local (smooth) solution of~\eqref{eq:GLv1} with wiggly pinning force
\begin{align}\label{eq:fastoscpin+}
\nabla\hat h(x):=\eta_\e\nabla_1\hat h^0(x,\tfrac1{\eta_\e}x)+\nabla_2 \hat h^0(x,\tfrac1{\eta_\e}x).
\end{align}
We further denote by
$\tilde\vi_\e$ the unique global (smooth) solution of the corresponding mean-field equation~\eqref{eq:GL1lim}--\eqref{eq:GL2'lim} with $\nabla\hat h(x)$ replaced by $\nabla_2 \hat h^0(x,\frac1{\eta_\e}x)$.
We split the proof into three steps.

\medskip
\step1 Grönwall relation.

In this step, we show that $\vi_\e$ is defined on the time interval $[0,T_\e)$, with $T_\e^0:=\eta_\e T$ and with $T$ as in Proposition~\ref{prop:GLvprop}. In addition, we adapt the proof of Proposition~\ref{prop:mflGL}: with the same restrictions on the regimes, we show that there exist $\sigma>0$ and an increasing bijection $\theta:\R^+\to\R^+$ such that for all $t\ge0$ the conditions $\D_{\e,R}^{*,\circ}=o(N_\e^2)$ and $\sup_{0\le s\le t}\hat\D_{\e,R}^{*,s}\le N_\e^2$ imply
\begin{align}\label{eq:mfletagron}
\hat \D_{\e,R}^{*,t}\le \theta(\tfrac t{\eta_\e})\Big(\eta_\e^{-\sigma}o(N_\e^2)+\eta_\e^{-1}\int_0^t\hat \D_{\e,R}^*\Big).
\end{align}

We first check how $\vi_\e$ depends on the small parameter $\eta_\e$, thus adapting Proposition~\ref{prop:GLvprop}. A scaling argument shows that the solution $\vi_\e$ exists up to time $\eta_\e T$, where $T$ is as in Proposition~\ref{prop:GLvprop}.
Moreover, an inspection of the proofs in~\cite{D-16} together with a scaling argument shows that all the estimates in Proposition~\ref{prop:GLvprop} still hold up to multiplicative constants of the form $\eta_\e^{-\sigma}\theta(\frac t{\eta_\e})$ for all $t\in[0,\eta_\e T)$, for some $\sigma\ge0$ and some increasing bijection $\theta:\R^+\to\R^+$.
A scaling argument yields more precisely, for all $t\in[0,T_\e)$,
\[\|\Gamma_\e^t\|_{\Ld^\infty}\le \theta(\tfrac{t}{\eta_\e}),\qquad\|\nabla\Gamma_\e^t\|_{\Ld^\infty}\le \eta_\e^{-1}\theta(\tfrac{t}{\eta_\e}).\]
With such estimates at hand, repeating the proof of Proposition~\ref{prop:mflGL} leads to the claim~\eqref{eq:mfletagron}.

\medskip
\step2 Grönwall argument.

In this step, we show that there exists $\eta_{\e,0}\ll1$ (possibly depending on all the data of the problem) such that for $\eta_{\e,0}\ll\eta_\e\ll1$ the conclusions of Proposition~\ref{prop:mflGL} hold in each of the corresponding regimes.

Since in the regime~\GLd{} we restrict to the parabolic case, we deduce that there exists $\eta_{\e,0}\ll1$ such that for $\eta_{\e,0}\ll\eta_\e\ll1$ the time $T_\e^0$ in Step~1 diverges as $\e\downarrow0$.
Given the assumption $\D_{\e,R}^{*,\circ}\ll N_\e^2$ on the initial data, for all $\e>0$ we define $T_\e>0$ as the maximum time $\le T_\e^0$ such that $\D_{\e,R}^{*,t}\le N_\e^2$ holds for all $t\le T_\e$. The result of Step~1 then yields for all $0\le t\le T_\e$,
\begin{align*}
\hat \D_{\e,R}^t\le \theta(\tfrac t{\eta_\e})\Big(\eta_\e^{-\sigma}o(N_\e^2)+\eta_\e^{-1}\int_0^t\hat \D_{\e,R}\Big),
\end{align*}
and hence, by the Grönwall inequality,
\begin{align*}
\hat\D_{\e,R}^t\le \eta_\e^{-\sigma}\psi(\tfrac t{\eta_\e})\,o(N_\e^2),\qquad \psi(t):=\theta(t)e^{\int_0^t\theta}.
\end{align*}
Choosing e.g.\@ $\eta_{\e,0}:=\big[\psi^{-1}\big(\sqrt{\frac{N_\e^2}{o(N_\e^2)}}\big)\big]^{-1/(\sigma\vee1)}$, we deduce for $\eta_{\e,0}\ll\eta_\e\ll1$ that $\hat\D_{\e,R}^t\ll N_\e^2$ holds for all $0\le t\le T_\e^0$, and the claim follows as in Step~4 of the proof of Proposition~\ref{prop:mflGL}.

\medskip
\step3 Conclusion.

It remains to show that there exists $\eta_{\e,0}\ll1$ such that for $\eta_{\e,0}\ll\eta_\e\ll1$ there holds $\vi_\e-\tilde\vi_\e\to0$ in $\Ld^\infty_\loc(\R^+;\Ld^1_\uloc(\R^2)^2)$.
This convergence result directly follows from the computations in the proof of Lemma~\ref{lem:lastlimGL}, now taking into account the $\eta_\e$-dependence of $\vi_\e$ and $\tilde\vi_\e$ as in Step~1 and applying a Grönwall argument in a suitable diagonal regime.
\end{proof}

\subsection{Mesoscopic initial-boundary layer}

In non-diagonal regimes, the Grönwall relation~\eqref{eq:mfletagron} only yields conclusions in the short timescale $t=O(\eta_\e)$. This allows to rigorously explore the mesoscopic initial-boundary layer that occurs in that timescale: in each mesoscopic periodicity cell, the vorticity gets projected onto the support of the invariant measure for the cell dynamics associated with the initial mean-field driving vector field $\Gamma_\e^\circ$ (cf.~\eqref{eq:GLv1}). This is captured in terms of $2$-scale convergence.
The proof is particularly easy as the nonlinearity plays no role yet in that timescale.

\begin{prop}\label{prop:locrelax}
Let the same assumptions hold as in Theorem~\ref{th:mainGL}, with wiggly pinning weight~\eqref{eq:fastoscpin}. In the regime~\GLd, we restrict to the parabolic case.
For all $\e>0$ let $u_\e$ be the unique global solution of~\eqref{eq:GL-1} as in Proposition~\ref{prop:globGL}(i),
and for all $x\in\R^2$ let $\md_0(x,\cdot)$ denote the unique global solution of the following continuity equation in the torus $Q$,
\begin{gather}\label{eq:limloctimem0}
\partial_t\!\md_0(x,\cdot)=-\Div\!\!_y\big(\Gamma^\circ(x,\cdot)^\bot\md_0(x,\cdot)\big),\qquad
\md_0(x,\cdot)|_{t=0}=\curl\!\vi^\circ(x),\\
\Gamma^\circ(x,y):=(\alpha-\Jb\beta)\big(\nabla_2^\bot \hat h^0(x,y)-\hat F(x)^\bot-2\kappa\!\vi^\circ(x)\big),\nonumber
\end{gather}
where $\kappa:=1$ in the regime~\GLu, $\kappa:=\lambda$ in the regime~\GLd, and $\kappa:=0$ in the regimes~\GLup{} and~\GLdp.
Then there exists a sequence $\eta_{\e,0}\downarrow0$ (depending on all the data of the problem) such that for $\eta_{\e,0}\ll\eta_\e\ll1$ the slowed-down rescaled vorticity $\frac1{N_\e}\mu_\e^{\eta_\e t}$ $2$-scale converges to $\md_0^t$, that is, for all $\phi\in C_c^\infty([0,T)\times\R^2;C^\infty_\per(Q))$,
\[\lim_{\e\downarrow0}\,\iint_{\R^+\times\R^2}\phi(t,x,\tfrac{x}{\eta_\e})\, \tfrac1{N_\e}\mu_\e^{\eta_\e t}(x)\,dxdt=\iiint_{\R^+\times\R^2\times Q}\phi(t,x,y)\,\md_0^t(x,y)\,dydxdt.\qedhere\]
\end{prop}

\begin{proof}
As in Step~1 of the proof of Corollary~\ref{cor:homogdiag} above,
the solution $\vi_\e$ is defined on the time interval $[0,\eta_\e T)$ with $T$ as in Proposition~\ref{prop:GLvprop}, hence $T$ diverges as $\e\downarrow0$.
Applying~\eqref{eq:mfletagron} and choosing $\eta_{\e,0}:=\big(\frac{o(N_\e^2)}{N_\e^2}\big)^{1/\sigma}$, we deduce for $\eta_{\e,0}\ll\eta_\e\ll1$, for all $t\in[0,T)$,
\[\hat\D_{\e,R}^{*,\eta_\e t}\lesssim_to(N_\e^2)+\int_0^{t}\hat\D_{\e,R}^{*,\eta_\e s}ds.\]
The Grönwall inequality then implies $\hat\D_{\e,R}^{*,\eta_\e t}=o(N_\e^2)$ for all $t\in[0,T)$. As in Step~4 of the proof of Proposition~\ref{prop:mflGL}, we deduce $\tfrac1{N_\e}j_\e^{\eta_\e t}-\vi_\e^{\eta_\e t}\to0$ in $\Ld^\infty_\loc(\R^+;\Ld^1_\uloc(\R^2)^2)$ as $\e\downarrow0$. We may then find a sequence $\eta_{\e,0}\ll\eta_{\e,0}'\ll1$ such that for $\eta_{\e,0}'\ll\eta_\e\ll1$ we have for all $T_0,R_0>0$,
\begin{align}\label{eq:locrelax-est1}
\lim_{\e\downarrow0}\eta_\e^{-1}\int_0^{T_0}\int_{B_{R_0}}|\tfrac1{N_\e}j_\e^{\eta_\e t}-\vi_\e^{\eta_\e t}\!|=0.
\end{align}
It remains to determine the asymptotic behavior of $\vi_\e^{\eta_\e t}$. We split the proof into two steps.

\medskip
\noindent\step1 $2$-scale convergence of $\curl\!\vi_\e^{\eta_\e t}$.

Let $\bar\vi_\e^t:=\vi_\e^{\eta_\e t}$ and $\bar\md_\e:=\curl\bar\vi_\e$. Taking the $\curl$ of both sides of~\eqref{eq:GLv1}, we deduce the following equation for $\bar\md_\e$,
\begin{gather}\label{eq:barmdeqn2sc}
\partial_t\bar\md_\e=-\eta_\e\Div(\hat\Gamma_\e^\bot\bar\md_\e),\qquad\bar\md_\e|_{t=0}=\curl\!\vi_\e^\circ\\
\bar\Gamma_\e:=\lambda_\e^{-1}(\alpha-\Jb\beta)\Big(\nabla^\bot h-F^\bot-\frac{2N_\e}\Log\bar\vi_\e\Big).
\end{gather}
By~\cite[Lemma~4.1(iii)]{D-16} in the dissipative case with $\|h\|_{W^{1,\infty}}$, $\|\lambda_\e^{-1}(\nabla^\bot h-F^\bot)\|_{\Ld^\infty}$, $\|\!\vi_\e^\circ\!\|_{\Ld^\infty}$, $\|\!\Div(a\!\vi_\e^\circ)\|_{\Ld^2}\lesssim1$, we have $\|\!\vi_\e^t-\vi_\e^\circ\!\|_{\Ld^2}^2\lesssim t$ for all $t\in[0,\eta_\e T)$. By~\cite[Lemmas~4.2--4.3]{D-16} and a scaling argument, we have
$\|\curl\!\vi_\e^t\!\|_{\Ld^\infty}\lesssim_{t/\eta_\e}1$. After time rescaling, these estimates yield for all $t\in[0,T)$,
\begin{align}\label{eq:nonlinconv0}
\|\bar\vi_\e^t-\vi_\e^\circ\!\|_{\Ld^2}^2\lesssim_t\eta_\e,\qquad\|\bar\md_\e^t\|_{\Ld^\infty}\lesssim_t1.
\end{align}
Nguetseng's $2$-scale compactness theorem~\cite{Nguetseng-89} (e.g.\@ in the form of~\cite[Theorem~3.2]{E-92}) then implies the existence of some $\bar\md_0\in\Ld^\infty_\loc(\R^+;\Ld^\infty(\R^2\times Q))$ such that up to a subsequence $\bar\md_\e$ $2$-scale converges to $\bar\md_0$, that is, for all $\phi\in C_c^\infty(\R^+\times\R^2;C^\infty_\per(Q))$,
\begin{align*}
\lim_{\e\downarrow0}\iint_{\R^+\times\R^2} \phi(t,x,\tfrac{x}{\eta_\e})\,\bar\md_\e^t(x)\,dxdt=\iiint_{\R^+\times\R^2\times Q} \phi(t,x,y)\,\bar\md_0^t(x,y)\,dydxdt.
\end{align*}
Testing equation~\eqref{eq:barmdeqn2sc} with $\phi(t,x,\tfrac x{\eta_\e})$, we find
\begin{multline*}
-\int_{\R^2}\phi(0,x,\tfrac x{\eta_\e})\,\curl\!\vi_\e^\circ(x)\,dx-\iint_{\R^+\times\R^2}\partial_t\phi(t,x,\tfrac x{\eta_\e})\bar\md_\e^t(x)dxdt\\=\iint_{\R^+\times\R^2}\bar\md_\e^t(x)\,(\eta_\e\nabla_1\phi(t,x,\tfrac x{\eta_\e})+\nabla_2\phi(t,x,\tfrac x{\eta_\e}))\cdot\bar\Gamma_\e^t(x)^\bot dxdt,
\end{multline*}
and hence, passing to the limit $\e\downarrow0$ along the subsequence and noting that~\eqref{eq:nonlinconv0} implies $\bar\vi_\e\to\vi^\circ$ in $\Ld^\infty_\loc(\R^+;\Ld^2_\uloc(\R^2))$,
\begin{multline*}
-\iint_{\R^2\times Q}\phi(0,x,y)\,\curl\!\vi^\circ(x)\,dydx-\iiint_{\R^+\times\R^2\times Q}\partial_t\phi(t,x,y)\,\bar\md_0^t(x,y)\,dydxdt\\
=\iiint_{\R^+\times\R^2\times Q}\bar\md_0^t(x,y)\,\nabla_2\phi(t,x,y)\cdot\Gamma^\circ(x,y)^\bot dydxdt.
\end{multline*}
This proves that $\bar\md_0$ satisfies the weak formulation of the linear continuity equation~\eqref{eq:limloctimem0} and is therefore its unique solution $\bar\md_0=\md_0$. 

\medskip
\noindent\step2 Conclusion.

Let $\phi\in C_c^\infty(\R^+\times\R^2;C^\infty_\per(Q))$, with $\phi(t,x,y)=0$ for $|x|>R_0$ or $|t|>T_0$. Integration by parts yields
\begin{multline}\label{eq:conclcurljm0}
\bigg|\iint_{\R^+\times\R^2}\phi(t,x,\tfrac{x}{\eta_\e})\,\curl(\tfrac1{N_\e}j_\e^{\eta_\e t})(x)\,dxdt-\iiint_{\R^+\times\R^2\times Q}\phi(t,x,y)\,\md_0^t(x,y)\,dydxdt\bigg|\\
\le\eta_\e^{-1}\|\nabla\phi\|_{\Ld^\infty}\int_0^{T_0}\int_{B_{R_0}}|\tfrac1{N_\e}j_\e^{\eta_\e t}-\bar\vi_\e^{t}|\\
+\bigg|\iint_{\R^+\times\R^2}\phi(t,x,\tfrac{x}{\eta_\e})\,\curl\bar\vi_\e^{t}(x)\,dxdt-\iiint_{\R^+\times\R^2\times Q}\phi(t,x,y)\,\md_0^t(x,y)\,dydxdt\bigg|.
\end{multline}
Combining this with~\eqref{eq:locrelax-est1} and with the result of Step~1, the conclusion follows.
\end{proof}

\subsection{Small applied force implies pinning}

In this section, we establish the following intuitive result: in the presence of a small applied force $\|F\|_{\Ld^\infty}\ll\|\nabla h\|_{\Ld^\infty}$, with a wiggly pinning potential, vortices are pinned. The proof is based on energy methods and is limited to the non-critical scalings~\GLup{} and~\GLdp.

\begin{prop}\label{prop:smallforcing}
Let $\alpha>0$, $\beta\in\R$, $\alpha^2+\beta^2=1$, let Assumption~\ref{as:main}(a) hold with the initial data $(u_\e^\circ,\vi_\e^\circ,\vi^\circ)$ satisfying the well-preparedness condition~\eqref{eq:convincond}, and assume that
\begin{gather*}
1\ll N_\e\ll\Log,\qquad \frac{N_\e}\Log\ll\lambda_\e\lesssim1,\qquad\frac{\e}{\lambda_\e(N_\e\Log)^{1/2}}\ll\eta_\e\ll1,\\
h(x):=\lambda_\e\eta_\e \hat h^0(x,\tfrac x{\eta_\e}),\qquad \|F\|_{W^{1,\infty}}\ll\lambda_\e,
\end{gather*}
with $\hat h^0$ independent of $\e$.
Let $u_\e:\R^+\times\R^2\to\C$ be the solution of~\eqref{eq:GL-1} as in Proposition~\ref{prop:globGL}(i).
We consider the regime~\GLup{} with $\vi_\e^\circ=\vi^\circ$ and the regime~\GLdp{} with $\Div(a\!\vi_\e^\circ)=0$.
Then $\frac1{N_\e}\mu_\e\cvf*\curl\!\vi^\circ$ in $\Ld^\infty_\loc(\R^+;(C^{\gamma}_c(\R^2))^*)$ for all $\gamma>0$.
\end{prop}

\begin{proof}
We choose $\vi_\e:=\vi_\e^\circ$ in the definition of the modulated energy~\eqref{der2}, thus redefining for all $z\in\R^2$,
\begin{gather*}
\Ec_{\e,R}^z:=\int_{\R^2}\frac{a\chi_R^z}2\Big(|\nabla u_\e-iu_\e N_\e\!\vi_\e^\circ\!|^2+\frac{a}{2\e^2}(1-|u_\e|^2)^2\Big),\\
\D_{\e,R}^z:=\Ec_{\e,R}^z-\frac\Log2\int_{\R^2} a\chi_R^z\mu_\e,
\end{gather*}
as well as $\Ec_{\e,R}^*:=\sup_z\Ec_{\e,R}^z$ and $\D_{\e,R}^*:=\sup_z\D_{\e,R}^z$ (where the suprema implicitly run over $z\in R\Z^2$). We further consider the following modification of this modulated energy, including suitable lower-order terms,
\begin{multline*}
\hat\Ec_{\e,R}^z:=\int_{\R^2}\frac{a\chi_R^z}2\Big(|\nabla u_\e-iu_\e N_\e\!\vi_\e^\circ\!|^2+\frac{a}{2\e^2}(1-|u_\e|^2)^2\\
+(1-|u_\e|^2)(f-N_\e^2|\!\vi_\e^\circ\!|^2-N_\e\Log\vi_\e^\circ\cdot F^\bot)\Big),
\end{multline*}
and $\hat\Ec_{\e,R}^*:=\sup_z\hat\Ec_{\e,R}^z$.
The lower bound assumption on the pin separation $\eta_\e$ allows to choose the cut-off length $R\ge1$ in such a way that
\[\lambda_\e^{-1}\ll R\ll\e^{-1}\frac{(N_\e\Log)^{1/2}}{\lambda_\e\Log^2},\qquad R\ll\eta_\e \e^{-1}(N_\e\Log)^{1/2}.\]
By Proposition~\ref{prop:ballconstr}, the well-preparedness condition~\eqref{eq:convincond} implies $\Ec_{\e,R}^{*,\circ}\le C_0N_\e\Log$ for some $C_0\simeq1$.
Let $T>0$ be fixed and define $T_\e>0$ as the maximum time $\le T$ such that the bound $\Ec_{\e,R}^{*,t}\le 2C_0N_\e\Log$ holds for all $t\le T_\e$.
Using~\eqref{eq:formf} in the form $\|f\|_{\Ld^\infty}\lesssim\lambda_\e\eta_\e^{-1}+\lambda_\e^2\Log^2$, the assumptions on $\vi_\e^\circ$, and the choice of $\eta_\e,R$, we deduce for all $t\le T_\e$,
\begin{multline}\label{eq:frozenhatpashat}
|\hat\Ec_{\e,R}^{z,t}-\Ec_{\e,R}^{z,t}|\le\int_{\R^2} \chi_R^z|1-|u_\e^t|^2|(|f|+N_\e^2|\!\vi_\e^\circ\!|^2+N_\e\Log|\!\vi_\e^\circ\!||F|)\\
\lesssim \e R(\lambda_\e\eta_\e^{-1}+\lambda_\e^2\Log^2)(\Ec_{\e,R}^{z,t})^{1/2}+\e R^\theta o(\lambda_\e N_\e\Log)(\Ec_{\e,R}^{z,t})^{1/2}\ll \lambda_\e N_\e\Log,
\end{multline}
hence in particular $\hat\Ec_{\e,R}^{*,t}\lesssim N_\e\Log$ for all $t\le T_\e$. We split the proof into three steps.

\medskip
\noindent\step1 Evolution of the modulated energy.
\nopagebreak

In this step, for all $\e>0$ small enough, we show that $T_\e=T$ and that for all $t\le T$,
\begin{align}\label{eq:frozenSt1}
\frac{\lambda_\e\alpha}4\int_0^t\int_{\R^2} a\chi_R^z|\partial_tu_\e|^2&\le \hat \Ec_{\e,R}^{z,\circ}-\hat \Ec_{\e,R}^{z,t}+o_t(\lambda_\e N_\e\Log)\lesssim_tN_\e\Log.
\end{align}
The time derivative of the modulated energy $\hat\Ec_{\e,R}^z$ is computed as follows, by integration by parts,
\begingroup\allowdisplaybreaks
\begin{align*}
\partial_t\hat\Ec_{\e,R}^z&=\int_{\R^2} a\chi_R^z\Big(\langle\nabla u_\e-iu_\e N_\e\!\vi_\e^\circ,\nabla \partial_tu_\e\rangle-N_\e\!\vi_\e^\circ\!\cdot\,\langle\nabla u_\e-iu_\e N_\e\!\vi_\e^\circ,i\partial_tu_\e\rangle\\
&\qquad\qquad-\frac{a}{\e^2}(1-|u_\e|^2)\langle u_\e,\partial_tu_\e\rangle-(f-N_\e^2|\!\vi_\e^\circ\!|^2-N_\e\Log\vi_\e^\circ\cdot F^\bot)\langle u_\e,\partial_t u_\e\rangle\Big)\\
&=-\int_{\R^2} a\chi_R^z\Big\langle\triangle u_\e+\frac{au_\e}{\e^2}(1-|u_\e|^2)+\nabla h\cdot\nabla u_\e+i\Log F^\bot\cdot\nabla u_\e+f u_\e,\partial_tu_\e\Big\rangle\\
&\quad+N_\e\int_{\R^2} a\chi_R^z(\vi_\e^\circ\cdot\nabla h+\Div\!\vi_\e^\circ)\langle\partial_tu_\e,iu_\e\rangle-\int_{\R^2} a\nabla\chi_R^z\cdot\langle\nabla u_\e-iu_\e N_\e\!\vi_\e^\circ,\partial_t u_\e\rangle\\
&\qquad-\int_{\R^2} a\chi_R^z(\Log F^\bot+2N_\e\!\vi_\e^\circ)\cdot\langle\nabla u_\e-iu_\e N_\e\!\vi_\e^\circ,i\partial_tu_\e\rangle,
\end{align*}
\endgroup
hence, inserting equation~\eqref{eq:GL-1} in the first right-hand side term,
\begin{multline*}
\partial_t\hat\Ec_{\e,R}^z=-\lambda_\e\alpha\int_{\R^2} a\chi_R^z|\partial_tu_\e|^2-\int_{\R^2} a\chi_R^z(\Log F^\bot+2N_\e\!\vi_\e^\circ)\cdot\,\langle\nabla u_\e-iu_\e N_\e\!\vi_\e^\circ,i\partial_tu_\e\rangle\\
+N_\e\int_{\R^2} \chi_R^z\Div(a\!\vi_\e^\circ)\langle\partial_tu_\e,iu_\e\rangle-\int_{\R^2} a\nabla\chi_R^z\cdot\langle\nabla u_\e-iu_\e N_\e\!\vi_\e^\circ,\partial_t u_\e\rangle.
\end{multline*}
In particular, using the energy bound $\Ec_{\e,2R}^{*,t}\lesssim\Ec_{\e,R}^{*,t}\lesssim N_\e\Log$, we find for all $t\le T_\e$,
\begin{align*}
\partial_t\hat\Ec_{\e,R}^z&\le-\frac{\lambda_\e\alpha}2\int_{\R^2} a\chi_R^z|\partial_tu_\e|^2-\int_{\R^2} a\chi_R^z(\Log F^\bot+2N_\e\!\vi_\e^\circ)\cdot\,\langle\nabla u_\e-iu_\e N_\e\!\vi_\e^\circ,i\partial_tu_\e\rangle\\
&\hspace{1cm}+C_t\lambda_\e^{-1}N_\e^2\int_{\R^2} \chi_R^z|\Div(a\!\vi_\e^\circ)|^2(1+|1-|u_\e|^2|)\\
&\hspace{4cm}+C_t\lambda_\e^{-1}R^{-2}\int_{B_{2R}(z)} |\nabla u_\e-iu_\e N_\e\!\vi_\e^\circ\!|^2\\
&\le-\frac{\lambda_\e\alpha}2\int_{\R^2} a\chi_R^z|\partial_tu_\e|^2-\int_{\R^2} a\chi_R^z(\Log F^\bot+2N_\e\!\vi_\e^\circ)\cdot\,\langle\nabla u_\e-iu_\e N_\e\!\vi_\e^\circ,i\partial_tu_\e\rangle\\
&\hspace{1cm}+C_t\lambda_\e^{-1}N_\e^2\|\Div(a\!\vi_\e^\circ)\|_{\Ld^2\cap\Ld^\infty(B_{2R})}^2+C_t\lambda_\e^{-1}R^{-2}N_\e\Log,
\end{align*}
so that the assumptions on $\Div(a\!\vi_\e^\circ)$ and the choice of the cut-off length $R$ yield
\begin{multline}\label{eq:derEhatF0}
\partial_t\hat\Ec_{\e,R}^z\le-\frac{\lambda_\e\alpha}2\int_{\R^2} a\chi_R^z|\partial_tu_\e|^2\\
-\int_{\R^2} a\chi_R^z(\Log F^\bot+2N_\e\!\vi_\e^\circ)\cdot\langle\nabla u_\e-iu_\e N_\e\!\vi_\e^\circ,i\partial_tu_\e\rangle
+o_t(\lambda_\e N_\e\Log).
\end{multline}
Using the Cauchy-Schwarz inequality to estimate the second right-hand side term, with $\|F\|_{\Ld^\infty}\lesssim\lambda_\e$ and $\|\!\vi_\e^\circ\!\|_{\Ld^\infty}\lesssim1$, we find the following rough estimate,
\begin{align*}
\partial_t\hat\Ec_{\e,R}^z&\le-\frac{\lambda_\e\alpha}4\int_{\R^2} a\chi_R^z|\partial_tu_\e|^2+C\lambda_\e\Log^2\int_{\R^2} a\chi_R^z|\nabla u_\e-iu_\e N_\e\!\vi_\e^\circ\!|^2+o_t(\lambda_\e N_\e\Log)\\
&\le-\frac{\lambda_\e\alpha}4\int_{\R^2} a\chi_R^z|\partial_tu_\e|^2+O_t(\lambda_\e N_\e\Log^3),
\end{align*}
and thus, integrating in time with $\lambda_\e\lesssim1$, we find for all $t\le T_\e$,
\begin{align*}
\frac{\lambda_\e\alpha}4\int_0^t\int_{\R^2} a\chi_R^z|\partial_tu_\e|^2&\le \hat\Ec_{\e,R}^{z,\circ}-\hat\Ec_{\e,R}^{z,t}+O_t(\Log^4)\lesssim_t \Log^4.
\end{align*}
This rough estimate now allows to apply Lemma~\ref{lem:prodest} (with $\vi_\e=\vi_\e^\circ$ and $\pre_{\e}=0$), using that $\Log\|F\|_{\Ld^\infty}+N_\e\ll\lambda_\e\Log$, to the effect of
\begin{eqnarray*}
\lefteqn{\Big|\int_0^t\int_{\R^2} a\chi_R^z(\Log F^\bot+2N_\e\!\vi_\e^\circ)\cdot\,\langle\nabla u_\e-iu_\e N_\e\!\vi_\e^\circ,i\partial_tu_\e\rangle\Big|}\\
&\lesssim&\frac{\Log\|F\|_{\Ld^\infty}+N_\e}\Log\Big(\int_0^t\int_{\R^2} a\chi_R^z|\partial_tu_\e|^2+\int_0^t\int_{\R^2} a\chi_R^z|\nabla u_\e-iu_\e N_\e\!\vi_\e^\circ\!|^2\Big)+o_t(1)\\
&\lesssim&o(\lambda_\e)\int_0^t\int_{\R^2} a\chi_R^z|\partial_tu_\e|^2+o_t(\lambda_\e N_\e\Log).
\end{eqnarray*}
Inserting this into~\eqref{eq:derEhatF0} and integrating in time, we find for all $t\le T_\e$,
\begin{align*}
\hat \Ec_{\e,R}^{z,t}-\hat \Ec_{\e,R}^{z,\circ}&\le-\Big(\frac{\lambda_\e\alpha}2-o(\lambda_\e)\Big)\int_0^t\int_{\R^2} a\chi_R^z|\partial_tu_\e|^2+o_t(\lambda_\e N_\e\Log),
\end{align*}
and the result~\eqref{eq:frozenSt1} follows for all $t\le T_\e$. In particular, combined with~\eqref{eq:frozenhatpashat}, this yields for all $t\le T_\e$,
\begin{multline*}
\Ec_{\e,R}^{z,t}\le\hat\Ec_{\e,R}^{z,t}+o(\lambda_\e N_\e\Log)\le\hat\Ec_{\e,R}^{z,\circ}+o_t(\lambda_\e N_\e\Log)\le\Ec_{\e,R}^{z,\circ}+o_t(\lambda_\e N_\e\Log)\\
\le (C_0+o_t(1))N_\e\Log,
\end{multline*}
and thus, taking the supremum in $z$, we conclude $T_\e=T$ for $\e>0$ small enough.

\medskip
\noindent\step2 Lower bound on the modulated energy.

In this step, we prove for all $t\le T$,
\begin{align*}
\Ec_{\e,R}^{z,t}&\ge\frac{\Log}2\int_{\R^2} a\chi_R^z\mu_\e^t-o_t(\lambda_\e N_\e\Log),
\end{align*}
and hence, combined with the well-preparedness assumption $\D_{\e,R}^{z,\circ}\ll N_\e^2$
and with~\eqref{eq:frozenhatpashat},
\begin{align*}
\hat\Ec_{\e,R}^{z,\circ}-\hat\Ec_{\e,R}^{z,t}\le \Ec_{\e,R}^{z,\circ}-\Ec_{\e,R}^{z,t}+o(\lambda_\e N_\e\Log)\le\frac\Log2\int_{\R^2} a\chi_R^z(\mu_\e^\circ-\mu_\e^t)+o(\lambda_\e N_\e\Log).
\end{align*}

As we show, this is a simple consequence of Lemma~\ref{lem:ballconstr}. (However note that we may not directly apply Proposition~\ref{prop:ballconstr}(i)--(iii) as the assumption $R\gtrsim\Log$ does not hold.) Noting that $\|\nabla(a\chi_R^z)\|_{\Ld^\infty}\lesssim\lambda_\e+R^{-1}\lesssim\lambda_\e$, we deduce from Lemma~\ref{lem:ballconstr}(i) with $\phi=a\chi_R^z$, with $\Ec_{\e,R}^*\lesssim_t N_\e\Log$, and with $\e^{1/2}< r\ll1$,\\
\begin{align*}
\Ec_{\e,R}^z&\ge \frac{\log(\tfrac r\e)}2\int_{\R^2} a\chi_R^z|\nu_{\e,R}^r|-O_t(\lambda_\e rN_\e\Log)-O_t(r^2N_\e^2)-O_t(N_\e\log N_\e)\\
&\ge\frac{\Log}2\int_{\R^2} a\chi_R^z|\nu_{\e,R}^r|-O(|\!\log r|)\int_{\R^2} \chi_R^z|\nu_{\e,R}^r|-o_t(\lambda_\e N_\e\Log),
\end{align*}
hence by Lemma~\ref{lem:ballconstr}(ii), for $e^{-N_\e}\lesssim r\ll1$,
\begin{multline*}
\Ec_{\e,R}^z\ge\frac{\Log}2\int_{\R^2} a\chi_R^z|\nu_{\e,R}^r|-O_t(N_\e|\!\log r|)-o_t(\lambda_\e N_\e\Log)\\
\ge\frac{\Log}2\int_{\R^2} a\chi_R^z\nu_{\e,R}^r-o_t(\lambda_\e N_\e\Log).
\end{multline*}
By Lemma~\ref{lem:ballconstr}(iii) in the form~\eqref{eq:lem3.2iii-lip} with $\gamma=1$, and by~\eqref{eq:boundmumutildeequ}, using $\|\nabla(a\chi_R^z)\|_{\Ld^\infty}\lesssim\lambda_\e$, we may replace $\nu_{\e,R}^r$ by $\mu_\e$ in the right-hand side,
\begin{multline*}
\Ec_{\e,R}^z\ge\frac{\Log}2\int_{\R^2} a\chi_R^z\mu_{\e}\\
-\lambda_\e\Log O_t\big(\e RN_\e(N_\e\Log)^{1/2}+rN_\e\big)
-\Log O_t(\e^{1/2} N_\e\Log\big)-o_t(\lambda_\e N_\e\Log),
\end{multline*}
and the result follows from the choice $R\ll\e^{-1}(N_\e\Log)^{-1/2}$.

\medskip
\noindent\step3 Estimate on the total vorticity.
\nopagebreak

In this step, we show for all $t\le T$,
\[\Big|\int_{\R^2} a\chi_R^z(\mu_\e^t-\mu_\e^\circ)\Big|\ll_t \lambda_\e N_\e.\]
We first prove (a weaker version of) the result with the weight $a$ replaced by $1$.
Using identity~\eqref{eq:identity-1}, we may decompose
\begin{multline*}
\int_{\R^2} \chi_R^z(\mu_\e^t-\mu_\e^\circ)=\int_0^t\int_{\R^2} \chi_R^z\partial_t\mu_\e^t=\int_0^t\int_{\R^2} \chi_R^z\,\curl V_\e^t=-\int_0^t\int_{\R^2} \nabla^\bot\chi_R^z\cdot V_\e^t\\
=-2\int_0^t\int_{\R^2} \nabla^\bot \chi_R^z\cdot \langle\nabla u_\e-iu_\e N_\e\!\vi_\e^\circ,i\partial_t u_\e\rangle+N_\e\int_0^t\int_{\R^2} \nabla^\bot\chi_R^z\cdot\vi_\e^\circ \partial_t(1-|u_\e|^2).
\end{multline*}
Applying Lemma~\ref{lem:prodest} as in Step~1, with $|\nabla\chi_R|\lesssim R^{-1}\chi_R^{1/2}$, we deduce for all $t\le T$ and $\Log^{-2}\lesssim K\lesssim\Log^2$,
\begin{eqnarray*}
\lefteqn{\Big|\int_{\R^2} \chi_R^z(\mu_\e^t-\mu_\e^\circ)\Big|}\\
&\lesssim& \frac1\Log\Big(K^{-2}\int_0^t\int_{\R^2} \chi_R^z|\partial_tu_\e|^2+K^2R^{-2}\int_0^t\int_{B_{2R}}|\nabla u_\e-iu_\e N_\e\!\vi_\e^\circ\!|^2\Big)\\
&&\hspace{1cm}+o_t(\Log^{-1})+N_\e\int_{\R^2} \big(|1-|u_\e^t|^2|+|1-|u_\e^\circ|^2|\big)|\nabla^\bot\chi_R^z|\\
&\lesssim_t& \frac{K^{-2}}\Log\int_0^t\int_{\R^2} \chi_R^z|\partial_tu_\e|^2+K^2R^{-2}N_\e+\e N_\e\Log+o(\Log^{-1}).
\end{eqnarray*}
Using~\eqref{eq:frozenSt1} to estimate the first right-hand side term, and choosing $\lambda_\e^{-1}\ll K^2\ll \lambda_\e R^2$, we obtain
\begin{align}\label{eq:semiresfrozenst3}
\Big|\int_{\R^2} \chi_R^z(\mu_\e^t-\mu_\e^\circ)\Big|&\lesssim_t \frac{K^{-2}}{\lambda_\e\Log}(\hat \Ec_{\e,R}^{z,\circ}-\hat \Ec_{\e,R}^{z,t})_++o(K^{-2}N_\e)+K^2R^{-2}N_\e+o(\Log^{-1})\nonumber\\
&\lesssim_t o(\Log^{-1})(\hat \Ec_{\e,R}^{z,\circ}-\hat \Ec_{\e,R}^{z,t})_++o(\lambda_\e N_\e).
\end{align}
It remains to smuggle the weight $a$ into the left-hand side. For all $t\le T$, applying
Lemma~\ref{lem:ballconstr}(iii) in the form~\eqref{eq:lem3.2iii-lip} with $\gamma=1$, as well as~\eqref{eq:boundmumutildeequ}, and using the choice of $R\ll \e^{-1}(N_\e\Log)^{-1/2}$, we find for $\e^{1/2}<r\ll1$,
\[\Big|\int_{\R^2}(1-a)\chi_R^z(\mu_\e^t-\nu_{\e,R}^{r,t})\Big|\lesssim_t \lambda_\e rN_\e+\e^{1/2}N_\e\Log+\lambda_\e\e RN_\e(N_\e\Log)^{1/2}\ll \lambda_\e N_\e,\]
and hence, by Lemma~\ref{lem:ballconstr}(ii) with $\|1-a\|_{\Ld^\infty}\lesssim \lambda_\e\eta_\e\ll\lambda_\e$,
\[\Big|\int_{\R^2}(1-a)\chi_R^z\mu_\e^t\Big|\lesssim \|1-a\|_{\Ld^\infty}\int_{\R^2}\chi_R^z|\nu_{\e,R}^{r,t}|+o(\lambda_\e N_\e)\ll\lambda_\e N_\e.\]
Combining this with~\eqref{eq:semiresfrozenst3} and with the result of Step~2, we deduce
\begin{multline*}
\Big|\int_{\R^2} a\chi_R^z(\mu_\e^t-\mu_\e^\circ)\Big|\lesssim_t o(\Log^{-1})(\hat \Ec_{\e,R}^{z,\circ}-\hat \Ec_{\e,R}^{z,t})_++o(\lambda_\e N_\e)\\
\lesssim_t o(1)\Big|\int_{\R^2} a\chi_R^z(\mu_\e^t-\mu_\e^\circ)\Big|+o(\lambda_\e N_\e),
\end{multline*}
and the result follows.

\medskip
\noindent\step4 Conclusion.
\nopagebreak

Combining the results of Steps~1--3, we find
\[\int_0^T\int_{\R^2} a\chi_R^z|\partial_tu_\e|^2\ll_T N_\e\Log.\]
Applying Lemma~\ref{lem:prodest} (see also~\cite[Proposition~4.8]{Serfaty-15}) then yields for $X\in W^{1,\infty}([0,T]\times\R^2)^2$ and $\Log^{-1}\lesssim K\lesssim\Log$,
\begin{eqnarray*}
\lefteqn{\Big|\int_0^T\int_{\R^2} \chi_R^zX\cdot V_\e\Big|}\\
&\lesssim& \frac1\Log\Big(\frac1K\int_0^T\int_{\R^2} \chi_R^z|\partial_tu_\e|^2+K\int_0^T\int_{\R^2}\chi_R^z|X\cdot(\nabla u_\e-iu_\e N_\e\!\vi_\e^\circ)|^2\Big)\\
&&\hspace{3cm}+o(1)\big(1+\|X\|_{W^{1,\infty}([0,T]\times\R^2)}^5\big)\\
&\lesssim_T&\big(o(K^{-1}N_\e)+KN_\e+o(1)\big)\big(1+\|X\|_{W^{1,\infty}([0,T]\times\R^2)}^5\big),
\end{eqnarray*}
hence, for a suitable choice of $K$,
\begin{align*}
\sup_{z}\Big|\int_0^T\int_{\R^2} \chi_R^zX\cdot V_\e\Big|&\ll_T N_\e\big(1+\|X\|_{W^{1,\infty}([0,T]\times\R^2)}^5\big).
\end{align*}
This implies $\frac1{N_\e}V_\e\cvf* 0$ in $(C^{1}_c([0,T]\times \R^2))^*$, so that identity~\eqref{eq:identity-1} yields $\partial_t(\frac1{N_\e}\mu_\e)=\frac1{N_\e}\curl V_\e\cvf*0$ in $(C^1([0,T];C^{2}_c(\R^2)))^*$.
Arguing as in Step~4 of the proof of Proposition~\ref{prop:mflGL}, the well-preparedness assumption on the initial data implies $\frac1{N_\e}j_\e^\circ\to\vi^\circ$ in $\Ld^1_\uloc(\R^2)^2$, hence $\frac1{N_\e}\mu_\e^\circ\cvf*\curl\!\vi^\circ$ in $(C^{1}_c(\R^2))^*$. We easily deduce $\frac1{N_\e}\mu_\e\cvf*\curl\!\vi^\circ$ in $(C([0,T];C^2_c(\R^2)))^*$. Noting that Lemma~\ref{lem:ballconstr}(iii) together with~\eqref{eq:jacobestbis} ensures that the sequence $(\frac1{N_\e}\mu_\e)_\e$ is bounded in $\Ld^\infty([0,T];(C^{\gamma}_c(\R^2))^*)$ for all $\gamma>0$, the conclusion follows.
\end{proof}

\appendix
\section{Well-posedness of the mesoscopic model}\label{app:wellposed}

In this appendix, we address the global well-posedness of the mesoscopic model~\eqref{eq:GL-1}, establishing Proposition~\ref{prop:globGL} as well as additional regularity. We start with the decaying setting, that is, when $\nabla h,F,f$ decay at infinity.
Note that in this setting no advection is expected to occur at infinity. As is classical since the work of Bethuel and Smets~\cite{Bethuel-Smets-07} (see also~\cite{Miot-09}), we consider solutions $u_\e$ in the affine space $\Ld^\infty_\loc(\R^+;U_\e+H^1(\R^2;\C))$ for some ``reference map'' $U_\e$, which is typically chosen smooth and equal (in polar coordinates) to $e^{iN_\e\theta}$ outside a ball at the origin, for some given $N_\e\in\Z$, thus imposing for $u_\e$ a fixed total degree $N_\e$ at infinity. More generally, we consider the following spaces of ``admissible'' reference maps, for $k\ge0$,
\begin{multline*}
E_k(\R^2):=\big\{U\in \Ld^{\infty}(\R^2;\C): \nabla^2 U\in H^{k}(\R^2;\C),\,\nabla|U|\in\Ld^2(\R^2),\,1-|U|^2\in\Ld^2(\R^2),\\
\nabla U\in\Ld^p(\R^2;\C)~\forall p>2\big\}.
\end{multline*}
(Note that this definition slightly differs from the usual one in~\cite{Bethuel-Smets-07}, but this form is more adapted in the presence of pinning and applied current.)
The map $U_{N_\e}:=U_\e$ above clearly belongs to the space $E_\infty(\R^2)$. Global well-posedness and regularity in this framework are provided by the following proposition. Note that the proof requires a stronger decay of $\nabla h,F,f$ in the conservative case, but we do not know whether this is necessary.

\begin{samepage}
\begin{prop}[Well-posedness of~\eqref{eq:GL-1}, decaying setting]\label{prop:globGLapp}
Set $a:=e^h$ with $h:\R^2\to\R$.
\begin{enumerate}[(i)]
\item \emph{Dissipative case ($\alpha>0$, $\beta\in\R$):}\\
Given $h\in W^{1,\infty}(\R^2)$, $F\in\Ld^\infty(\R^2)^2$, $f\in \Ld^2\cap\Ld^\infty(\R^2)$, with $\nabla h,F\in\Ld^p(\R^2)^2$ for some $p<\infty$, and $u^\circ_\e\in U+H^1(\R^2;\C)$ for some $ U\in E_0(\R^2)$, there exists a unique global solution $u_\e\in\Ld^\infty_\loc(\R^+; U+H^1(\R^2;\C))$ of~\eqref{eq:GL-1} in $\R^+\times\R^2$ with initial data $u^\circ_\e$.
Moreover, if for some $k\ge0$ we have $h\in W^{k+1,\infty}(\R^2)$, $F\in W^{k,\infty}(\R^2)^2$, $f\in H^{k}\cap W^{k,\infty}(\R^2)$, with $\nabla h,F\in W^{k,p}(\R^2)^2$ for some $p<\infty$, and $U\in E_{k}(\R^2)$, then $u_\e\in\Ld^\infty_\loc([\delta,\infty); U+H^{k+1}(\R^2;\C))$ for all $\delta>0$. If in addition $u_\e^\circ\in U+H^{k+1}(\R^2;\C)$, then $u_\e\in\Ld^\infty_\loc(\R^+; U+H^{k+1}(\R^2;\C))$.
\item \emph{Conservative case ($\alpha=0$, $\beta=1$):}\\
Given $h\in W^{2,\infty}(\R^2)$, $\nabla h\in H^1(\R^2)^2$, $F\in H^2\cap W^{2,\infty}(\R^2)^2$, $f\in\Ld^2\cap\Ld^\infty(\R^2)$, with $\Div F=0$, and $u^\circ_\e\in U+H^1(\R^2;\C)$ for some $ U\in E_0(\R^2)$, there exists a unique global solution $u_\e\in\Ld^\infty_\loc(\R^+;U+H^1(\R^2;\C))$ of~\eqref{eq:GL-1} in $\R^+\times\R^2$ with initial data~$u^\circ_\e$.
Moreover, if for some $k\ge0$ we have $h\in W^{k+2,\infty}(\R^2)$, $\nabla h\in H^{k+1}(\R^2)^2$, $F\in H^{k+2}\cap W^{k+2,\infty}(\R^2)^2$, $f\in H^{k+1}\cap W^{k+1,\infty}(\R^2)$, with $\Div F=0$, and $u^\circ_\e\in U+H^{k+1}(\R^2;\C)$ with $ U\in E_{k+1}(\R^2)$, then $u_\e\in \Ld^\infty_\loc(\R^+;U+H^{k+1}(\R^2;\C))$.
\qedhere
\end{enumerate}
\end{prop}
\end{samepage}

The proof below is based on arguments by~\cite{Bethuel-Smets-07,Miot-09}, which need to be adapted in the present setting with both pinning and applied current. The conservative case is however more delicate, and we then use the structure of the equation to make a crucial change of variables that transforms the first-order terms into zeroth-order ones. As shown in the proof, in the dissipative case, the decay assumption $\nabla h,F\in\Ld^p(\R^2)^2$ (for some $p<\infty$) can be replaced by $(|\nabla h|+|F|)\nabla U\in\Ld^2(\R^2;\C)^2$.

\begin{proof}[Proof of Proposition~\ref{prop:globGLapp}]
We split the proof into seven steps. We start with the (easiest) case $\alpha>0$, and then turn to the conservative case $\alpha=0$ in Steps~4--7.

\medskip
\noindent\step1 Local existence in $U+H^{k+1}(\R^2;\C)$ for $\alpha>0$.

In this step, given $k\ge0$, we assume $h\in W^{k+1,\infty}(\R^2)$, $F\in W^{k,\infty}(\R^2)^2$, $f\in H^{k}\cap W^{k,\infty}(\R^2)$, $\nabla h,F\in W^{k,p}(\R^2)$ for some $p<\infty$, and $u_\e^\circ\in U+H^{k+1}(\R^2;\C)$ for some $U\in E_{k}(\R^2)$, and we prove that there exists some $T>0$ and a unique solution $u_\e\in \Ld^\infty([0,T);U+H^{k+1}(\R^2;\C))$ of~\eqref{eq:GL-1} in $[0,T)\times\R^2$. To simplify notation, we replace equation~\eqref{eq:GL-1} by its rescaled version
\begin{align}\label{eq:GLrescale}
(\alpha+i\beta)\partial_tu=\triangle u+{au}(1-|u|^2)+\nabla h\cdot \nabla u+iF^\bot\cdot\nabla u+fu,\qquad u|_{t=0}=u^\circ.
\end{align}
We start with the case $k=0$, and briefly comment afterwards on the adaptations needed for $k\ge1$.
We argue by a fixed-point argument in the set $E_{U,u^\circ}(C_0,T):=\{u:\|u-U\|_{\Ld^\infty_TH^1}\le C_0,\,u|_{t=0}=u^\circ\}$, for some $C_0,T>0$ to be suitably chosen.
We denote by $C\ge1$ any constant that only depends on an upper bound on $\alpha$, $\alpha^{-1}$, $|\beta|$, $\|h\|_{W^{1,\infty}}$, $\|(F,f,U)\|_{\Ld^\infty}$, $\|1-|U|^2\|_{\Ld^2}$, $\|\triangle U\|_{\Ld^2}$, $\|f\|_{\Ld^2}$, and $\|(|F|+|\nabla h|)\nabla U\|_{\Ld^2}$, and we add a subscript to indicate dependence on further parameters.

For $\alpha>0$, the kernel of the semigroup operator $e^{(\alpha+i\beta)^{-1}t\triangle}$ is given explicitly by
\[S^t(x):=(\alpha+i\beta)(4\pi t)^{-1}e^{-(\alpha+i\beta)|x|^2/(4t)},\]
which decays just like the standard heat kernel,
\begin{align}\label{eq:estkernel0}
|S^t(x)|\le Ct^{-1}e^{-\alpha|x|^2/(4t)},
\end{align}
and we have the following obvious estimates, for all $1\le r\le\infty$, $k\ge1$,
\begin{align}\label{eq:estkernel}
\|S^t\|_{\Ld^r}\le Ct^{\frac1r-1},\qquad\|\nabla^kS^t\|_{\Ld^r}\le C_kt^{\frac1r-1-\frac k2}.
\end{align}
Setting $\hat u:=u-U$, we may rewrite equation~\eqref{eq:GLrescale} as follows,
\begin{multline}
(\alpha+i\beta)\partial_t\hat u=\triangle \hat u+\triangle U+{a(\hat u+U)}(1-| U|^2)-2a(\hat u+U)\langle U,\hat u\rangle-a(\hat u+U)|\hat u|^2\\+\nabla h\cdot \nabla \hat u+\nabla h\cdot \nabla U+iF^\bot\cdot\nabla \hat u+iF^\bot\cdot\nabla U+f\hat u+fU,
\end{multline}
with initial data $\hat u|_{t=0}=\hat u^\circ:=u^\circ-U$.
Any solution $\hat u\in \Ld^\infty([0,T);H^1(\R^2;\C))$ satisfies the Duhamel formula $\hat u=\Xi_{U,\hat u^\circ}(\hat u)$, where we have set
\begin{align*}
\Xi_{U,\hat u^\circ}(\hat u)^t:=&~S^t\ast \hat u^\circ+(\alpha+i\beta)^{-1}\int_0^tS^{t-s}\ast Z_{U,\hat u^\circ}(\hat u^s)ds,\\
Z_{U,\hat u^\circ}(\hat u^s):=&~\triangle  U+{a(\hat u^s+ U)}(1-| U|^2)-2a(\hat u^s+ U)\langle U,\hat u^s\rangle-a(\hat u^s+ U)|\hat u^s|^2\\
&\hspace{2cm}+\nabla h\cdot \nabla \hat u^s+\nabla h\cdot \nabla  U+iF^\bot\cdot\nabla \hat u^s+iF^\bot\cdot\nabla  U+f\hat u^s+f U.\nonumber
\end{align*}
Let us examine the map $\Xi_{U,\hat u^\circ}$ more closely. Using~\eqref{eq:estkernel} in the forms $\|S^t\|_{\Ld^1}\le C$ and $\|\nabla S^{t}\|_{\Ld^1}\le Ct^{-1/2}$, we obtain by the triangle inequality
\begin{multline*}
\|\Xi_{U,\hat u^\circ}(\hat u)^t\|_{H^1}\le\|S^t\|_{\Ld^1}\|\hat u^\circ\|_{H^1}\\
+C\int_0^t(1+(t-s)^{-1/2})\Big(1+\|\hat u^s\|_{\Ld^2}+\|\hat u^s\|^3_{\Ld^6}+\|\nabla \hat u^s\|_{\Ld^2}\Big)ds,
\end{multline*}
hence, by Sobolev embedding in the form $\|\hat u^s\|_{\Ld^6}\le C\|\hat u^s\|_{H^1}$, for all $\hat u\in -U+E_{U,u^\circ}(C_0,T)$,
\begin{align*}
\|\Xi_{U,\hat u^\circ}(\hat u)\|_{\Ld^\infty_TH^1}\le C\|\hat u^\circ\|_{H^1}+C(T+T^{1/2})(1+C_0^3).
\end{align*}
Similarly, again using the Sobolev embedding, we easily find for all $\hat u,\hat v\in -U+E_{U,u^\circ}(C_0,T)$,
\begin{eqnarray*}
\lefteqn{\|\Xi_{U,\hat u^\circ}(\hat u)-\Xi_{U,\hat u^\circ}(\hat v)\|_{\Ld^\infty_TH^1}}\\
&\le& C\int_0^t(1+(t-s)^{-1/2})(1+\|\hat u^s\|_{H^1}^2+\|\hat v^s\|_{H^1}^2)\|\hat u^s-\hat v^s\|_{H^1}ds\\
&\le& C(T+T^{1/2})(1+C_0^2)\|\hat u-\hat v\|_{\Ld^\infty_TH^1}.
\end{eqnarray*}
Choosing $C_0:=1+C\|\hat u^\circ\|_{H^{1}}$ and $T:=1\wedge (4C(1+C_0^3))^{-2}$, we deduce that $\Xi_{U,\hat u^\circ}$ maps the set $-U+E_{U,u^\circ}(C_0,T)$ into itself and is contracting on that set. The conclusion follows from a fixed-point argument.

We now briefly comment on the case $k\ge1$ and explain how to adapt the above argument.
We again proceed by a fixed point argument, but this time we estimate $\Xi_{U,\hat u^\circ}(w)$ in $H^{k+1}(\R^2;\C)$ as follows,
\[\|\Xi_{U,\hat u^\circ}(\hat u)^t\|_{H^{k+1}}\le\|S^t\|_{\Ld^1}\|\hat u^\circ\|_{H^{k+1}}+C\int_0^t(\|S^{t-s}\|_{\Ld^1}+\|\nabla S^{t-s}\|_{\Ld^1})\|Z_{U,\hat u^\circ}(\hat u^s)\|_{H^{k}},\]
where we easily check with the Sobolev embedding that
\begin{align}\label{eq:boundHkZ}
\|Z_{U,\hat u^\circ}(\hat u^s)\|_{H^{k}}\le C_k(1+\|\hat u^s\|_{H^{k+1}}^3),
\end{align}
for some constant $C_k\ge1$ that only depends on an upper bound on $\alpha$, $\alpha^{-1}$, $|\beta|$, $k$, $\|h\|_{W^{k+1,\infty}}$, $\|F\|_{W^{k,\infty}}$, $\|f\|_{H^{k}\cap W^{k,\infty}}$, $\| U\|_{\Ld^{\infty}}$, $\|\nabla| U|\|_{\Ld^2}$, $\|\nabla^2  U\|_{H^{k}}$, $\|1-| U|^2\|_{\Ld^2}$, and on $\sum_{j\le k}\|(|\nabla^jF|+|\nabla^j\nabla h|)\nabla U\|_{\Ld^2}$. Similarly estimating the $H^{k+1}$-norm of the difference $\Xi_{U,\hat u^\circ}(\hat u)-\Xi_{U,\hat u^\circ}(\hat v)$, the result follows.

\medskip
\noindent\step2 Regularizing effect for $\alpha>0$.

In this step, given $k\ge0$, we assume $h\in W^{k+1,\infty}(\R^2)$, $F\in W^{k,\infty}(\R^2)^2$, $f\in H^{k}\cap W^{k,\infty}(\R^2)$, $\nabla h,F\in W^{k,p}(\R^2)^2$ for some $p<\infty$, and $ U\in E_{k}(\R^2)$, and we prove that any solution $u\in \Ld^\infty([0,T); U+H^1(\R^2;\C))$ of~\eqref{eq:GLrescale} satisfies $u\in \Ld^\infty([\delta,T); U+H^{k+1}(\R^2;\C))$ for all $\delta>0$. We denote by $C_k\ge1$ any constant that only depends on an upper bound on $\alpha$, $\alpha^{-1}$, $|\beta|$, $k$, $\|h\|_{W^{k+1,\infty}}$, $\|F\|_{W^{k,\infty}}$, $\|f\|_{H^{k}\cap W^{k,\infty}}$, $\| U\|_{\Ld^{\infty}}$, $\|1-| U|^2\|_{\Ld^2}$, $\|\nabla| U|\|_{\Ld^2}$, $\|\nabla^2  U\|_{H^{k}}$, $\sum_{j\le k}\|(|\nabla^jF|+|\nabla^j\nabla h|)\nabla U\|_{\Ld^2}$, and $\|u^\circ- U\|_{H^1}$. We write $C$ for such a constant in the case $k=1$. We denote by $C_{k,t}\ge1$ any such constant that additionally depends on an upper bound on $t$, $t^{-1}$, and $\|u- U\|_{\Ld^\infty_tH^1}$. We add a subscript to indicate dependence on further parameters.

Let $u\in \Ld^\infty([0,T); U+H^1(\R^2;\C))$ be a solution of~\eqref{eq:GLrescale}, and let $\hat u:=u-U$. We prove by induction that $\|\hat u^t\|_{H^{k+1}}\le C_{k,t}$ for all $t\in(0,T)$ and $k\ge0$. As it is obvious for $k=0$, we assume that it holds for some $k\ge0$ and we then deduce that it also holds for $k$ replaced by $k+1$. Using the Duhamel formula $\hat u=\Xi_{U,\hat u^\circ}(\hat u)$ as in Step~1, we find
\begin{multline}\label{eq:boundHkhatustep2}
\|\nabla^{k+1}\hat u^t\|_{\Ld^2}\le \|\nabla^{k}S^t\|_{\Ld^1}\|\nabla \hat u^\circ\|_{\Ld^2}\\
+C\int_{t/2}^t\|\nabla S^{t-s}\ast\nabla^kZ_{U,\hat u^\circ}(\hat u^s)\|_{\Ld^2}ds
+C\int_0^{t/2}\|\nabla^{k+1} S^{t-s}\ast Z_{U,\hat u^\circ}(\hat u^s)\|_{\Ld^2}ds.
\end{multline}
A finer estimate than~\eqref{eq:boundHkZ} is now needed. Arguing as in~\cite[Lemma~2]{Bethuel-Smets-07} by means of various Sobolev embeddings, we find for all $1<r<2$,
\begin{align}\label{eq:estL2LrZ}
\|\nabla Z_{U,\hat u^\circ}(\hat u^t)\|_{\Ld^2+\Ld^r}\le C_r(1+\|\hat u^t\|_{H^1}^3+\|\hat u^t\|_{H^2}).
\end{align}
(Note that we cannot choose $r=2$ here due to terms of the form $\||\hat u^s|^2\nabla\hat u^s\|_{\Ld^r}$, and the term $\|\hat u^t\|_{H^2}$ in the right-hand side comes from the forcing terms $(\nabla h+i F^\bot)\cdot\nabla \hat u^t$ appearing in the expression for $Z_{U,\hat u^\circ}(\hat u^t)$.)
By a similar argument (cf.\@ e.g.~\cite[Step~1 of the proof of Proposition~A.8]{Miot-09}), we find for all $k\ge0$ and $1<r<2$,
\begin{align}\label{eq:dersupZ}
\|\nabla^kZ_{U,\hat u^\circ}(\hat u^t)\|_{\Ld^2+\Ld^r}\le C_{k,r} (1+\|\hat u^t\|_{H^{k}}^3+\|\hat u^t\|_{H^{k+1}}).
\end{align}
We may then deduce from~\eqref{eq:boundHkhatustep2}, together with Young's convolution inequality and~\eqref{eq:estkernel}, for all $1<r<2$,
\begin{align*}
\|\nabla^{k+1}\hat u^t\|_{\Ld^2}&\le \|\nabla^{k}S^t\|_{\Ld^1}\|\nabla \hat u^\circ\|_{\Ld^2}+C\int_{\frac12t}^t\|\nabla S^{t-s}\|_{\Ld^1\cap\Ld^{\frac{2r}{3r-2}}}\|\nabla^kZ_{U,\hat u^\circ}(\hat u^s)\|_{\Ld^2+\Ld^r}ds\\
&\hspace{3cm}+C\int_0^{\frac12t}\|\nabla^{k+1} S^{t-s}\|_{\Ld^1}\|Z_{U,\hat u^\circ}(\hat u^s)\|_{\Ld^2}ds\\
&\le Ct^{-k/2}+C_{k,r}\int_{\frac12t}^t((t-s)^{-1/2}+(t-s)^{-1/r})(1+\|\hat u^s\|_{H^{k}}^3+\|\hat u^s\|_{H^{k+1}})ds\\
&\hspace{3cm}+C\int_0^{\frac12t}(t-s)^{-(k+1)/2}(1+\|\hat u^s\|_{H^1}^3)ds\\
&\le C_{k,t}+C_{k,t}\sup_{\frac12t\le s\le t}\|\hat u^s\|_{H^k}^3+C_{k,t}\bigg(\int_0^t\|\nabla^{k+1}\hat u^s\|_{\Ld^2}^3ds\bigg)^{1/3}.
\end{align*}
By induction hypothesis, this yields $\|\nabla^{k+1}\hat u^t\|_{\Ld^2}^3\le C_{k,t}+C_{k,t}\int_0^t\|\nabla^{k+1}\hat u^s\|_{\Ld^2}^3ds$,
and the result follows from the Grönwall inequality.

\medskip
\noindent\step3 Global existence for $\alpha>0$.

In this step, we assume $h\in \Ld^\infty(\R^2)$, $f\in \Ld^2\cap\Ld^\infty(\R^2)$, $\nabla h,F\in \Ld^p\cap\Ld^\infty(\R^2)$ for some $p<\infty$, $u^\circ\in U+H^1(\R^2;\C)$, and $ U\in E_0(\R^2)$, and we prove that~\eqref{eq:GLrescale} admits a unique global solution $u\in \Ld^\infty_\loc(\R^+; U+H^1(\R^2;\C))$.
We denote by $C>0$ any constant that only depends on an upper bound on $\alpha$, $\alpha^{-1}$, $|\beta|$, $\|h\|_{W^{1,\infty}}$, $\|(F,U)\|_{\Ld^\infty}$, $\|1-| U|^2\|_{\Ld^2}$, $\|\triangle  U\|_{\Ld^2}$, $\|f\|_{\Ld^2\cap\Ld^\infty}$, and $\|(|F|+|\nabla h|)\nabla U\|_{\Ld^2}$.

Given $T>0$ and a solution $u\in\Ld^\infty([0,T); U+H^1(\R^2;\C))$ of~\eqref{eq:GLrescale}, we claim that the following a priori estimate holds for all $t\in[0,T)$,
\begin{multline}\label{eq:globalexGL}
\frac\alpha2\int_0^t\int_{\R^2}|\partial_t u|^2+\frac12\int_{\R^2}\Big(|\nabla (u^t- U)|^2+\frac a2(1-|u^t|^2)^2+|u^t- U|^2\Big)\\
\le Ce^{Ct} (1+\|u^\circ- U\|_{H^1}^2).
\end{multline}
Combining this with the local existence result of Step~1 in the space $ U+H^1(\R^2;\C)$, we deduce that local solutions can be extended globally in that space, and the result follows. It remains to prove the claim~\eqref{eq:globalexGL}.
For simplicity, we assume in the computations below that $u\in\Ld^\infty([0,T); U+H^2(\R^2;\C))$, which in particular implies $\partial_tu\in\Ld^\infty([0,T);\Ld^2(\R^2;\C))$ by~\eqref{eq:GLrescale}. The general result then follows from an approximation argument based on the local existence result of Step~1 in the space $U+H^{2}(\R^2;\C)$.

We set for simplicity $(\alpha+i\beta)^{-1}=\alpha'+i\beta'$, $\alpha'>0$. Using equation~\eqref{eq:GLrescale}, we compute the following time derivative, suitably organizing the terms and integrating by parts,
\begin{eqnarray*}
\lefteqn{\hspace{-0.6cm}\frac12\partial_t\int_{\R^2} |u- U|^2=\int_{\R^2} \langle u- U,(\alpha'+i\beta')(\triangle u+au(1-|u|^2)+\nabla h\cdot\nabla u+iF^\bot\cdot\nabla u+fu)\rangle}\\
&=&-\alpha'\int_{\R^2} |\nabla(u- U)|^2+\alpha'\int_{\R^2} a|u- U|^2(1-|u|^2)\\
&&+\int_{\R^2} \langle u- U,(\alpha'+i\beta')(\nabla h\cdot\nabla (u- U)+iF^\bot\cdot\nabla (u- U)+f(u- U))\rangle\\
&&+\int_{\R^2} \langle u- U,(\alpha'+i\beta')(\triangle  U+a U(1-|u|^2)+\nabla h\cdot\nabla  U+iF^\bot\cdot\nabla  U+f U)\rangle,
\end{eqnarray*}
which is estimated as follows,
\begin{align*}
\frac12\partial_t\int_{\R^2} |u- U|^2&\le-\alpha'\int_{\R^2} |\nabla(u- U)|^2+C\int_{\R^2} |u- U|^2+C\int_{\R^2} |u- U||\nabla (u- U)|\\
&\hspace{2cm}+\int_{\R^2} |u- U|(|\triangle  U|+|1-|u|^2|+(|\nabla h|+|F|)|\nabla U|+|f|)\\
&\le-\frac{\alpha'}2\int_{\R^2} |\nabla(u- U)|^2+C+C\int_{\R^2} |u- U|^2+C\int_{\R^2} (1-|u|^2)^2.
\end{align*}
On the other hand, again using the equation and integrating by parts, we compute
\begingroup\allowdisplaybreaks
\begin{eqnarray*}
\lefteqn{\frac12\partial_t\int_{\R^2}|\nabla (u- U)|^2=\int_{\R^2}\langle\nabla (u- U),\nabla\partial_t u\rangle=-\int_{\R^2}\langle\triangle (u- U),\partial_t u\rangle}\\
&=&-\int_{\R^2}\langle(\alpha+i\beta)\partial_tu-\triangle U-au(1-|u|^2)-\nabla h\cdot\nabla u-iF^\bot\cdot\nabla u-fu,\partial_t u\rangle\\
&=&-\alpha\int_{\R^2}|\partial_t u|^2-\frac14\partial_t\int_{\R^2} {a}(1-|u|^2)^2\\
&&\qquad+\int_{\R^2}\langle\nabla h\cdot\nabla(u- U)+iF^\bot\cdot\nabla(u- U)+f(u- U),\partial_t u\rangle\\
&&\qquad\qquad\qquad+\int_{\R^2}\langle \triangle U+\nabla h\cdot\nabla U+iF^\bot\cdot\nabla U+f U,\partial_t u\rangle,
\end{eqnarray*}
\endgroup
and hence
\begingroup\allowdisplaybreaks
\begin{eqnarray*}
\lefteqn{\frac12\partial_t\int_{\R^2}|\nabla (u- U)|^2+\frac14\partial_t\int_{\R^2} {a}(1-|u|^2)^2}\\
&\le&-\alpha\int_{\R^2}|\partial_t u|^2+C\int_{\R^2}|\partial_tu|(|u- U|+|\nabla(u- U)|)\\
&&\hspace{4cm}+C\int_{\R^2}|\partial_tu|(|\triangle U|+(|\nabla h|+|F|)|\nabla U|+|f|)\\
&\le&-\frac\alpha2\int_{\R^2}|\partial_t u|^2+C+C\int_{\R^2}|u- U|^2+C\int_{\R^2}|\nabla(u- U)|^2.
\end{eqnarray*}
\endgroup
Combining the above yields
\begin{multline*}
\frac\alpha2\int_{\R^2}|\partial_tu|^2+\partial_t\int_{\R^2}\Big(\frac12|\nabla (u- U)|^2+\frac a4(1-|u|^2)^2+\frac12|u- U|^2\Big)\\
\le C+ C\int_{\R^2}\Big(\frac12|\nabla (u- U)|^2+\frac a4(1-|u|^2)^2+\frac12|u- U|^2\Big),
\end{multline*}
and the claim~\eqref{eq:globalexGL} follows from the Grönwall inequality.

\medskip
\noindent\step4 A useful change of variables.

We turn to the conservative case $\alpha=0$. The first-order forcing terms in the right-hand side of equation~\eqref{eq:GL-1} can no longer be treated as errors since the lost derivative is not retrieved by the Schrödinger operator, and the proof of local existence in Step~1 can thus not be adapted to this case. The global estimates in Step~3 similarly fail, as no dissipation is available to absorb the first-order terms. To remedy this, we start by performing a useful change of variables transforming first-order terms into zeroth-order ones, which are much easier to deal with.
Since by assumption $\Div F=0$ with $F\in\Ld^\infty(\R^2)^2$, we deduce from a Hodge decomposition that there exists $\psi\in H^1_\loc(\R^2)$ such that $F=-2\nabla^\bot\psi$. Using the relation $a=e^h$, and setting $w_\e:=\sqrt a u_\e e^{i\Log\psi}$, a straightforward computation shows that equation~\eqref{eq:GL-1} for $u_\e$ is equivalent to
\begin{align}\label{eq:GL-1b}\begin{cases}
\lambda_\e(\alpha+i\Log\beta)\partial_t w_\e=\triangle w_\e+\frac{w_\e}{\e^2}(a-|w_\e|^2)+(f_0+ig_0)w_\e,&\text{in $\R^+\times\R^2$,}\\
w_\e|_{t=0}=w_\e^\circ:=\sqrt ae^{i\Log\psi}u_\e^\circ.
\end{cases}\end{align}
where we have set
\[f_0:=f-\frac{\triangle\sqrt a}{\sqrt a}+\frac14\Log^2|F|^2,\qquad g_0:=\frac12\Log a^{-1}\curl(aF).\]
We look for solutions $w_\e$ in the class $W+H^1(\R^2;\C)$ for some ``weighted reference map'' $W$, that is, an element of
\begin{multline*}
E_k^a(\R^2):=\{W\in \Ld^{\infty}(\R^2;\C):\nabla^2W\in H^{k}(\R^2;\C),\nabla|W|\in\Ld^2(\R^2),\\
a-|W|^2\in\Ld^2(\R^2),\nabla W\in \Ld^p(\R^2;\C)~\forall p>2\}.
\end{multline*}
For $k\ge0$, and $\nabla h,\nabla\psi\in H^{k+1}(\R^2)^2$, we indeed observe that $w_\e$ is a solution of~\eqref{eq:GL-1b} in $\Ld^\infty([0,T);W+H^{k+1}(\R^2;\C))$ for some $W\in E_k^a$ if and only if $u_\e$ is a solution of~\eqref{eq:GL-1} in $\Ld^\infty([0,T);U+H^{k+1}(\R^2;\C))$ for some $U\in E_k$.

\medskip
\noindent\step5 Local existence for $\alpha=0$.
\nopagebreak

In this step, given $k\ge0$, we assume $h\in W^{k+1,\infty}(\R^2)$, $\nabla h\in H^{k}(\R^2)^2$, $f_0,g_0\in H^{k+1}\cap W^{k+1,\infty}(\R^2)$, and $w^\circ\in W+H^{k+1}(\R^2;\C)$ for some $W\in E_{k+1}^a(\R^2)$, and we prove that there exists some $T>0$ and a unique solution $w_\e\in \Ld^\infty([0,T);W+H^{k+1}(\R^2;\C))$ of~\eqref{eq:GL-1b} in $[0,T)\times\R^2$. To simplify notation, we replace equation~\eqref{eq:GL-1b} (with $\alpha=0$) by its rescaled version
\begin{align}\label{eq:GLrescaleb}
i\partial_t w=\triangle w+{w}(a-| w|^2)+(f_0+ig_0) w,\qquad  w|_{t=0}= w^\circ.
\end{align}
We start with the case $k=0$, and comment afterwards on the adaptations needed for $k\ge1$.
We argue by a fixed-point argument in the set $E_{W,w^\circ}(C_0,T):=\{w:\|w-W\|_{\Ld^\infty_TH^1}\le C_0,w|_{t=0}=w^\circ\}$, for some $C_0,T>0$ to be suitably chosen.
We denote by $C\ge1$ any constant that only depends on an upper bound on $\|\nabla h\|_{\Ld^2\cap\Ld^\infty}$, $\|(f_0,g_0)\|_{H^1\cap W^{1,\infty}}$, $\|(h,W)\|_{\Ld^\infty}$, $\|a-|W|^2\|_{\Ld^2}$, $\|\nabla|W|\|_{\Ld^2}$, and $\|\triangle W\|_{H^1}$, and we add a subscript to indicate dependence on further parameters.

Let $S^t$ denote the kernel of the semigroup operator $e^{-it\triangle}$. Setting $\hat w:=w-W$, we may rewrite equation~\eqref{eq:GLrescaleb} as follows,
\begin{multline*}
i\partial_t\hat w=\triangle \hat w+\triangle W+{(\hat w+W)}(a-| W|^2)-2(\hat w+W)\langle W,\hat w\rangle-(\hat w+W)|\hat w|^2\\
+(f_0+ig_0)\hat w+(f_0+ig_0)W,
\end{multline*}
with initial data $\hat w|_{t=0}=\hat w^\circ:=w^\circ-W$.
Any solution $\hat w\in \Ld^\infty([0,T);H^1(\R^2;\C))$ satisfies the Duhamel formula $\hat w=\Xi_{W,\hat w^\circ}(\hat w)$, where we have set
\begin{align*}
\Xi_{W,\hat w^\circ}(\hat w)^t:=&~S^t\ast \hat w^\circ-i\int_0^tS^{t-s}\ast Z_{W,\hat w^\circ}(w^s)ds,\\
Z_{W,\hat w^\circ}(\hat w^s):=&~\triangle  W+(\hat w^s+ W)(a-| W|^2)-2(\hat w^s+ W)\langle W,\hat w^s\rangle-(\hat w^s+ W)|\hat w^s|^2\\
&\hspace{6cm}+(f_0+ig_0)\hat w^s+(f_0+ig_0) W.\nonumber
\end{align*}
Similarly as in Step~1, we find $\|Z_{W,\hat w^\circ}(\hat w^s)\|_{\Ld^2}\le C(1+\|\hat w^s\|_{H^1}^3)$. On the other hand, arguing as in~\cite[Lemma~2]{Bethuel-Smets-07} by means of various Sobolev embeddings, we obtain the following version of~\eqref{eq:estL2LrZ}: we may decompose $\nabla Z_{W,\hat w^\circ}(\hat w^s)=Z^{1}_{W,\hat w^\circ}(\hat w^s)+Z^{2}_{W,\hat w^\circ}(w^s)$, such that for all $1<r<2$,
\begin{align}\label{eq:estL2LrZ2}
\|\nabla Z_{W,\hat w^\circ}(\hat w^s)\|_{\Ld^2+\Ld^r}&\le\|Z^{1}_{W,\hat w^\circ}(\hat w^s)\|_{\Ld^2}+\|Z^{2}_{W,\hat w^\circ}(\hat w^s)\|_{\Ld^r}\nonumber\\
&\le C_r(1+\|\hat w^s\|_{H^1}^3).
\end{align}
(Recall that we cannot choose $r=2$ here due to terms of the form $\||\hat w^s|^2\nabla\hat w^s\|_{\Ld^r}$.)
Let us now examine the map $\Xi_{W,\hat w^\circ}$ more closely. We have
\begin{multline*}
\|\Xi_{W,\hat w^\circ}(\hat w)^t\|_{H^1}\le \|S^t\ast (\hat w^\circ,\nabla \hat w^\circ)\|_{\Ld^2}\\
+\bigg\|\int_0^te^{-i(t-s)\triangle}(Z_{W,\hat w^\circ}(\hat w^s),Z^1_{W,\hat w^\circ}(\hat w^s),Z^2_{W,\hat w^\circ}(\hat w^s))ds\bigg\|_{\Ld^2},
\end{multline*}
and hence by the Strichartz estimates for the Schrödinger operator~\cite{Keel-Tao-98}, for all $1<r\le2$,
\begin{multline*}
\|\Xi_{W,\hat w^\circ}(\hat w)\|_{\Ld^\infty_TH^1}\le C\|\hat w^\circ\|_{H^1}\\
+C\|(Z_{W,\hat w^\circ}(\hat w),Z^1_{W,\hat w^\circ}(\hat w))\|_{\Ld^{1}_T\Ld^2}+C_r\|Z^2_{W,\hat w^\circ}(\hat w)\|_{\Ld^{\frac{2r}{3r-2}}_T\Ld^r}.
\end{multline*}
Injecting~\eqref{eq:estL2LrZ2} then yields for all $1<r<2$,
\begin{align*}
\|\Xi_{W,\hat w^\circ}(\hat w)\|_{\Ld^\infty_TH^1}&\le C\|\hat w^\circ\|_{H^1}+(CT+C_rT^{\frac32-\frac1r})(1+\|\hat w\|_{\Ld^\infty_TH^1}^3).
\end{align*}
Choosing $r=\frac43$, this yields in particular, for all $\hat w\in -W+E_{W,\hat w^\circ}(C_0,T)$,
\begin{align*}
\|\Xi_{W,\hat w^\circ}(\hat w)\|_{\Ld^\infty_TH^1}&\le C\|\hat w^\circ\|_{H^1}+C(T+T^{3/4})(1+C_0^3).
\end{align*}
Similarly, again using Sobolev embeddings and Strichartz estimates, we easily for all $\hat v,\hat w\in -W+E_{W,\hat w^\circ}(C_0,T)$,
\begin{align*}
\|\Xi_{W,\hat w^\circ}(\hat v)-\Xi_{W,\hat w^\circ}(\hat w)\|_{\Ld^\infty_TH^1}&\le C(T+T^{3/4})(1+C_0^2)\|\hat v-\hat w\|_{\Ld^\infty_TH^1}.
\end{align*}
Choosing $C_0:=1+C\|\hat w^\circ\|_{H^{1}}$ and $T:=1\wedge (4C(1+C_0^3))^{-4/3}$, we deduce that $\Xi_{W,\hat w^\circ}$ maps the set $-W+E_{W,\hat w^\circ}(C_0,T)$ into itself and is contracting on that set. The conclusion follows from a fixed-point argument.

We now briefly comment on the case $k\ge1$ and explain how to adapt the above argument. We again proceed by a fixed point argument, estimating this time $\Xi_{W,\hat w^\circ}(\hat w)$ and $Z_{W,\hat w^\circ}(\hat w)$ in $H^{k+1}(\R^2;\C)$. Arguing similarly as in~\cite[Step~1 of the proof of Proposition~A.8]{Miot-09} by means of various Sobolev embeddings, we obtain the following version of~\eqref{eq:dersupZ}, for all $k\ge1$ and $1<r<2$,
\begin{align}\label{eq:dersupZ2}
\|\nabla^{k+1}Z_{W,\hat w^\circ}(\hat w)\|_{\Ld^\infty_t(\Ld^2+\Ld^r)}\le C_{k,r} (1+\|\hat w\|_{\Ld^\infty_tH^{k+1}}^3),
\end{align}
for some constant $C_{k,r}\ge1$ that only depends on an upper bound on $k$, $\|\nabla h\|_{H^{k}\cap W^{k,\infty}}$, $\|(h,W)\|_{\Ld^\infty}$, $\|(f_0,g_0)\|_{H^{k+1}\cap W^{k+1,\infty}}$, $\|a-|W|^2\|_{\Ld^2}$, $\|\nabla|W|\|_{\Ld^2}$, $\|\nabla^2W\|_{H^{k+1}}$, $(r-1)^{-1}$, and $(2-r)^{-1}$. The result then easily follows as above.

\medskip
\noindent\step6 Global existence for $\alpha=0$.

In this step, we assume $h\in \Ld^{\infty}(\R^2)$, $f_0\in\Ld^2\cap\Ld^\infty(\R^2)$, $g_0\in H^1\cap W^{1,\infty}(\R^2)$, and $w^\circ\in W+H^1(\R^2;\C)$ for some $ W\in E_0^a(\R^2)$, and we prove that~\eqref{eq:GLrescaleb} admits a unique global solution $w\in \Ld^\infty_\loc(\R^+; W+H^1(\R^2;\C))$.
We denote by $C>0$ any constant that only depends on an upper bound on $\|h\|_{\Ld^\infty}$, $\|f_0\|_{\Ld^2\cap\Ld^\infty}$, $\|g_0\|_{H^1\cap W^{1,\infty}}$, $\|W\|_{\Ld^\infty}$, $\|1-| W|^2\|_{\Ld^2}$, and $\|\triangle  W\|_{\Ld^2}$.

Given a solution $w\in\Ld^\infty([0,T); W+H^1(\R^2;\C))$ of~\eqref{eq:GLrescaleb}, we claim that the following a priori estimate holds for all $t\in[0,T)$,
\begin{align}\label{eq:globalexGLb}
\int_{\R^2}\Big(|\nabla (w^t- W)|^2+\frac 12(a-|w^t|^2)^2+|w^t-W|^2\Big)\le Ce^{Ct} (1+\|w^\circ- W\|_{H^1}^2).
\end{align}
Combining this with the local existence result of Step~5 in the space $ W+H^1(\R^2;\C)$, we deduce that local solutions can be extended globally in that space, and the result follows. It remains to prove the claim~\eqref{eq:globalexGLb}.
For simplicity, we assume in the computations below that $w\in\Ld^\infty([0,T); W+H^2(\R^2;\C))$, which in particular implies $\partial_tw\in\Ld^\infty([0,T);\Ld^2(\R^2;\C))$ by~\eqref{eq:GLrescaleb}. The general result then follows from a simple approximation argument based on the local existence result of Step~5 in the space $W+H^{2}(\R^2;\C)$.

Using equation~\eqref{eq:GLrescaleb}, we compute the following time derivative, suitably organizing the terms and integrating by parts,
\begin{eqnarray}
\lefteqn{\frac12\partial_t\int_{\R^2}|w-W|^2=\int_{\R^2}\langle i(w- W), \triangle w+w(a-|w|^2)+f_0w+ig_0w\rangle}\nonumber\\
&\qquad=&\int_{\R^2}\langle i(w-W), \triangle  W+ W(a-|w|^2)+f_0W+ig_0W\rangle+\int_{\R^2} g_0|w-W|^2\nonumber\\
&\qquad\le& C+C\int_{\R^2} |w-W|^2+C\int_{\R^2}(a-|w|^2)^2.\label{eq:dertempuUL2}
\end{eqnarray}
Likewise, we compute
\begingroup\allowdisplaybreaks
\begin{eqnarray}
\lefteqn{\partial_t\int_{\R^2}|\nabla(w-W)|^2=2\int_{\R^2} \langle \nabla(w-W),\nabla\partial_t w\rangle}\nonumber\\
&=&-2\int_{\R^2} \langle \triangle(w-W),\partial_t w-g_0w\rangle\nonumber\\
&&\qquad+2\int_{\R^2} \langle \nabla(w-W), g_0 \nabla(w-W)+g_0\nabla W+(w- W)\nabla g_0+ W\nabla g_0\rangle\nonumber\\
&\le&-2\int_{\R^2} \langle \triangle(w-W),\partial_t w-g_0w\rangle\nonumber\\
&&\qquad+C+C\int_{\R^2}|\nabla(w-W)|^2+C\int_{\R^2}|w-W|^2,\label{eq:dertempuUH1-0}
\end{eqnarray}
\endgroup
where we have
\begin{eqnarray*}
\lefteqn{-2\int_{\R^2} \langle \triangle(w- W),\partial_t w-g_0w\rangle}\\
&=&-2\int_{\R^2} \langle i(\partial_tw-g_0w)-w(a-|w|^2)-f_0w-\triangle W,\partial_t w-g_0w\rangle\\
&=&~2\int_{\R^2} \langle w(a-|w|^2)+f_0w+\triangle W,\partial_t w-g_0w\rangle\\
&=&-\partial_t\int_{\R^2}\Big( \frac12(a-|w|^2)^2- f_0 |w|^2-2 \langle \triangle W, w\rangle\Big)\\
&&\qquad+2\int_{\R^2} g_0(a-|w|^2)^2-2\int_{\R^2} ag_0(a-|w|^2)-2\int_{\R^2} f_0g_0|w|^2-2\int_{\R^2} g_0\langle \triangle W,w\rangle\\
&\le&-\partial_t\int_{\R^2} \Big( \frac12(a-|w|^2)^2- f_0 |w-W|^2-2\langle w,\triangle W+f_0W\rangle\Big)\\
&&\qquad+C+C\int_{\R^2}(a-|w|^2)^2+C\int_{\R^2}|w-W|^2.
\end{eqnarray*}
Combining this with~\eqref{eq:dertempuUL2} and~\eqref{eq:dertempuUH1-0}, we obtain
\begin{multline*}
\partial_t\int_{\R^2}\Big((C-f_0)|w-W|^2+|\nabla(w-W)|^2+\frac12(a-|w|^2)^2-2\langle w,\triangle  W+f_0W\rangle\Big)\\
\le C+C\int_{\R^2}\big(|w-W|^2+|\nabla(w- W)|^2+(a-|w|^2)^2\big),
\end{multline*}
and the result easily follows from the Grönwall inequality, choosing a large enough constant $C$ in the left-hand side.

\medskip
\noindent\step7 Propagation of regularity for $\alpha=0$.

In this step, given $k\ge0$, we assume $h\in W^{k+1,\infty}(\R^2)$, $\nabla h\in H^{k}(\R^2)^2$, $f_0,g_0\in H^{k+1}\cap W^{k+1,\infty}(\R^2)$, and $w^\circ\in W+H^{k+1}(\R^2;\C)$ for some $W\in E_{k+1}^a(\R^2)$, and we prove that the global solution $w$ of Step~6 belongs to $\Ld^\infty_\loc(\R^+; W+H^{k+1}(\R^2;\C))$. We denote by $C_k\ge1$ any constant that only depends on an upper bound on $k$, $\|\nabla h\|_{H^{k}\cap W^{k,\infty}}$, $\|(f_0,g_0)\|_{H^{k+1}\cap W^{k+1,\infty}}$, $\|(h,W)\|_{\Ld^\infty}$, $\|a-|W|^2\|_{\Ld^2}$, $\|\nabla|W|\|_{\Ld^2}$, and $\|\nabla^2W\|_{H^{k+1}}$. We add a subscript to indicate dependence on further parameters.

Let $w\in \Ld^\infty([0,T); W+H^1(\R^2;\C))$ be a solution of~\eqref{eq:GLrescale} and let $\hat w:=w-W$. We argue by induction: as the result is obvious for $k=0$, we assume that it holds for some $k\ge0$ and we deduce that it then also holds for $k$ replaced by $k+1$.
By a similar argument as e.g.~in~\cite[Lemma~4]{Bethuel-Smets-07} or in~\cite[Step~1 of the proof of Proposition~A.8]{Miot-09}, we obtain the following version of~\eqref{eq:dersupZ} (which generalizes~\eqref{eq:estL2LrZ2} to higher derivatives): for all $k\ge0$ we may decompose $\nabla^{k+1}Z_{W,\hat w^\circ}(\hat w^t)=\nabla^{k}Z_{W,\hat w^\circ}^1(\hat w^t)+\nabla^{k}Z_{W,\hat w^\circ}^2(w^t)$ such that for all $1<r<2$,
\begin{align*}
\|\nabla^{k+1}Z_{W,\hat w^\circ}(\hat w^t)\|_{\Ld^2+\Ld^r}&\le\|\nabla^{k}Z_{W,\hat w^\circ}^1(\hat w^t)\|_{\Ld^2}+\|\nabla^{k}Z_{W,\hat w^\circ}^2(\hat w^t)\|_{\Ld^r}\\
&\le C_{k,r} (1+\|\hat w^t\|_{H^{k+1}}^3),
\end{align*}
or more precisely,
\begin{align}\label{eq:dersupZ2}
\|\nabla^{k+1}Z_{W,\hat w^\circ}(\hat w^t)\|_{\Ld^2+\Ld^r}\le C_{k,r} (1+\|\hat w^t\|_{H^{k}}^2)(1+\|\hat w^t\|_{H^{k+1}}).
\end{align}
Using Duhamel's formula $\hat w=\Xi_{W,\hat w^\circ}(\hat w)$ and applying the Strichartz estimates for the Schrödinger operator~\cite{Keel-Tao-98} as in Step~5, we find for all $k\ge0$ and $1<r\le2$,
\begin{multline*}
\|\nabla^{k+1}\hat w^t\|_{\Ld^2}\le \|S^t\ast \nabla^{k+1}\hat w^\circ\|_{\Ld^2}+\bigg\|\int_0^t S^{t-s}\ast\nabla^{k+1}Z_{W,\hat w^\circ}(\hat w^s)ds\bigg\|_{\Ld^2}\\
\le C\|\nabla^{k+1}\hat w^\circ\|_{\Ld^2}+C\|\nabla^{k+1}Z^1_{W,\hat w^\circ}(\hat w)\|_{\Ld^1_t\Ld^2}+C_r\|\nabla^{k+1}Z^2_{W,\hat w^\circ}(\hat w)\|_{\Ld^{2r/(3r-2)}_t\Ld^r},
\end{multline*}
and hence, by~\eqref{eq:dersupZ2}, for all $k\ge0$,
\begin{align*}
\|\hat w^t\|_{H^{k+1}}&\le C_k\|\hat w^\circ\|_{H^{k+1}}+C_{k,r}(1+t)(1+\|\hat w\|_{\Ld^\infty_tH^k}^2)(1+\|\hat w\|_{\Ld^{2r/(3r-2)}_tH^{k+1}}).
\end{align*}
The result then follows from the induction hypothesis and the Grönwall inequality.
\end{proof}

In the dissipative case, we now prove a well-posedness result for equation~\eqref{eq:GL-1} in the general non-decaying setting, that is, without decay assumption on the data $\nabla h,F,f$. In this case, subtle advection forces may occur at infinity, preventing the solution $u_\e$ from staying in the same affine space $\Ld^\infty_\loc(\R^+;U+H^1(\R^2;\C))$ for any reference map $U$. The well-posedness result below is rather obtained in the space $\Ld^\infty(\R^+;H^1_\uloc(\R^2;\C))$, which yields no information at all on the behavior of the constructed solution at infinity. It is in particular completely unclear whether the total degree of the solution remains well-defined for positive times. In the proof, the key observation is that the Grönwall argument in Step~3 of the proof of Proposition~\ref{prop:globGLapp} can be localized by means of an exponential cut-off. Note that the same argument does not seem adaptable to the conservative case.

\begin{prop}[Well-posedness of~\eqref{eq:GL-1}, non-decaying setting]\label{prop:globGLappnondec}
Set $a:=e^h$ with $h:\R^2\to\R$. In the dissipative case ($\alpha>0$, $\beta\in\R$), given $h\in W^{1,\infty}(\R^2)$, $F\in\Ld^\infty(\R^2)^2$, $f\in\Ld^\infty(\R^2)$, and $u_\e^\circ\in H^1_\uloc(\R^2;\C)$, there exists a unique global solution $u_\e\in\Ld^\infty_\loc(\R^+;H^1_\uloc(\R^2;\C))$ of~\eqref{eq:GL-1} in $\R^+\times\R^2$ with initial data $u_\e^\circ$, and this solution satisfies $\partial_tu_\e\in \Ld^\infty_\loc(\R^+;\Ld^2_\uloc(\R^2;\C))$.
Moreover, if for some $k\ge0$ we have $h\in W^{k+1,\infty}(\R^2)$, $F\in W^{k,\infty}(\R^2)^2$, $f\in W^{k,\infty}(\R^2)$, and $u_\e^\circ\in H^{k+1}_\uloc(\R^2;\C)$, then $u_\e\in\Ld^\infty_\loc(\R^+;H^{k+1}_\uloc(\R^2;\C))$ and $\partial_tu_\e\in \Ld^\infty_\loc(\R^+;H^k_\uloc(\R^2;\C))$.
\end{prop}

\begin{proof}
We split the proof into four steps. We denote by $\xi^z(x):=e^{-|x-z|}$ the exponential cut-off centered at $z\in\Z^2$, and $\xi(x):=\xi^0(x)=e^{-|x|}$. To simplify notation, we replace equation~\eqref{eq:GL-1} by its rescaled version
\begin{align}\label{eq:GLrescale-2}
(\alpha+i\beta)\partial_tu=\triangle u+{au}(1-|u|^2)+\nabla h\cdot \nabla u+iF^\bot\cdot\nabla u+fu,\qquad u|_{t=0}=u^\circ.
\end{align}

\medskip
\noindent\step1 Global existence in $H^1_\uloc(\R^2;\C)$.

In this step, we assume $h\in W^{1,\infty}(\R^2)$, $F\in \Ld^{\infty}(\R^2)^2$, $f\in \Ld^\infty(\R^2)$, and $u^\circ\in H^1_\uloc(\R^2;\C)$, and we prove that there exists a global solution $u\in\Ld^\infty_\loc(\R^+;H^{1}_\uloc(\R^2;\C))$ of~\eqref{eq:GLrescale-2} in $\R^+\times\R^2$ with initial data $u^\circ$.
We denote by $C\ge1$ any constant that only depends on an upper bound on $\alpha$, $\alpha^{-1}$, $|\beta|$, $\|(h,\nabla h,F,f)\|_{\Ld^\infty}$, and $\|u^\circ\|_{H^1_\uloc}$.

We argue by approximation: for $n\ge1$, we define $\chi_n:=\chi(\cdot/n)$ for some cut-off function~$\chi$ with $\chi|_{B_1}\equiv1$ and $\chi|_{\R^2\setminus B_2}\equiv0$, and we set $h_n:=\chi_n h$, $a_n:=e^{h_n}$, $F_n:=\chi_nF$, and $f_n:=\chi_nf$.
Note that by construction $\|(h_n,\nabla h_n,F_n,f_n)\|_{\Ld^{\infty}}\le C$.
We also need to approximate the initial data $u^\circ\in H^1_{\uloc}(\R^2;\C)$: for $n\ge1$, we define $\rho_n:=n^{2}\rho(n\cdot)$ for some $\rho\in C^\infty_c(\R^2)$ with $\int_{\R^2}\rho=1$, and we set $u_n^\circ:=\chi_n(u^\circ\ast\rho_n)+1-\chi_n$. By definition, we have $u_n^\circ\in E_0$, the sequence $(u_n^\circ)_n$ is bounded in $H^1_\uloc(\R^2;\C)$, and as $n\uparrow\infty$ we obtain $u_n^\circ\to u^\circ$ in $H^1_\loc(\R^2;\C)$ and $a_n\to a$, $\nabla h_n\to\nabla h$, and $F_n\to F$ in $\Ld^\infty_\loc(\R^2)^2$.
By Proposition~\ref{prop:globGLapp}, there exists a unique global solution $u_n\in\Ld^\infty_\loc(\R^+;u_n^\circ+H^1(\R^2;\C))$ of the following truncated equation in $\R^+\times\R^2$,
\begin{align}\label{eq:GLrescale-n}
(\alpha+i\beta)\partial_tu_n=\triangle u_n+{a_nu_n}(1-|u_n|^2)+\nabla h_n\cdot \nabla u_n+iF_n^\bot\cdot\nabla u_n+f_nu_n,
\end{align}
with initial data $u_n|_{t=0}=u_n^\circ$.
In order to pass to the limit $n\uparrow\infty$ in (the weak formulation of) this equation, we prove the boundedness of the sequence $(u_n)_n$ in $\Ld^\infty_\loc(\R^+;H^1_\uloc(\R^2;\C))$, that is, we claim that the following a priori estimate holds for all $t\ge0$,
\begin{align}\label{eq:approxH1ulocest}
\|u_n^t\|_{H^1_\uloc}\le\sup_z\|u_n^t\|_{H^1(B(z))}+\alpha^{1/2}\sup_z\|\partial_tu_n\|_{\Ld^2_t\Ld^2(B(z))}\le Ce^{Ct}.
\end{align}
Before proving this estimate, we show how to conclude. Up to a subsequence, $u_n$ converges weakly-* to some $u$ in $\Ld^\infty_\loc(\R^+;H^1_\uloc(\R^2;\C))$. As $\partial_tu_n$ is bounded in $\Ld^2_\loc(\R^+;\Ld^2(B(z);\C))$, uniformly in $z$, and as $H^1(B(z);\C)$ is compactly embedded into $\Ld^3(B(z);\C)$, we deduce from the Aubin-Simon lemma that $u_n\to u$ strongly in $\Ld^\infty_\loc(\R^+;\Ld^3_\uloc(\R^2;\C))$. This allows to pass to the limit in the weak formulation of equation~\eqref{eq:GLrescale-n}, and deduce that the limit $u$ is a global solution of~\eqref{eq:GLrescale-2} in $\R^+\times\R^2$ with initial data $u^\circ$.

It remains to prove~\eqref{eq:approxH1ulocest}. We set for simplicity $(\alpha+i\beta)^{-1}=\alpha'+i\beta'$, $\alpha'>0$. Using equation~\eqref{eq:GLrescale-n}, integrating by parts, and using $|\nabla\xi^z|\le \xi^z$, we compute the following time derivative, for all $z\in R\Z^2$,
\begingroup\allowdisplaybreaks
\begin{eqnarray*}
\lefteqn{\frac12\partial_t\int_{\R^2}\xi^z|u_n|^2}\\
&=&\int_{\R^2}\xi^z\langle u_n,(\alpha'+i\beta')(\triangle u_n+{a_nu_n}(1-|u_n|^2)+\nabla h_n\cdot \nabla u_n+iF_n^\bot\cdot\nabla u_n+f_nu_n)\rangle\\
&\le&\int_{\R^2}\xi^z\langle u_n,(\alpha'+i\beta')\triangle u_n\rangle+\alpha'\int_{\R^2} a_n\xi^z|u_n|^2(1-|u_n|^2)\\
&&\hspace{5cm}+C\int_{\R^2}\xi^z|u_n||\nabla u_n|+C\int_{\R^2}\xi^z|u_n|^2\\
&\le&-\alpha'\int_{\R^2}\xi^z|\nabla u_n|^2+C\int_{\R^2}\xi^z|u_n||\nabla u_n|+C\int_{\R^2}\xi^z|u_n|^2,
\end{eqnarray*}
\endgroup
and hence
\begin{align*}
\frac12\partial_t\int_{\R^2}\xi^z|u_n|^2&\le-\frac{\alpha'}2\int_{\R^2}\xi^z|\nabla u_n|^2+C\int_{\R^2}\xi^z|u_n|^2.
\end{align*}
On the other hand, integration by parts yields
\begin{align*}
\frac12\partial_t\int_{\R^2}\xi^z|\nabla u_n|^2&=\int_{\R^2}\xi^z\langle\nabla u_n,\nabla\partial_t u_n\rangle=-\int_{\R^2}\xi^z\langle\triangle u_n,\partial_tu_n\rangle-\int_{\R^2}\nabla\xi^z\cdot\langle \nabla u_n,\partial_tu_n\rangle,
\end{align*}
hence, inserting equation~\eqref{eq:GLrescale-n} in the first right-hand side term,
\begin{eqnarray*}
\lefteqn{\frac12\partial_t\int_{\R^2}\xi^z|\nabla u_n|^2}\\
&=&-\int_{\R^2}\xi^z\langle(\alpha+i\beta)\partial_tu_n-a_nu_n(1-|u_n|^2)-\nabla h_n\cdot\nabla u_n-iF_n^\bot\cdot\nabla u_n-f_nu_n,\partial_tu_n\rangle\\
&&\hspace{2cm}-\int_{\R^2}\nabla\xi^z\cdot\langle \nabla u_n,\partial_tu_n\rangle\\
&\le&-\alpha\int_{\R^2}\xi^z|\partial_tu_n|^2-\frac14\partial_t\int_{\R^2} a_n\xi^z(1-|u_n|^2)^2+C\int_{\R^2} \xi^z(|u_n|+|\nabla u_n|)|\partial_tu_n|,
\end{eqnarray*}
and thus
\begin{multline*}
\frac12\partial_t\int_{\R^2}\xi^z|\nabla u_n|^2+\frac14\partial_t\int_{\R^2} a_n\xi^z(1-|u_n|^2)^2\\
\le-\frac\alpha2\int_{\R^2}\xi^z|\partial_tu_n|^2+C\int_{\R^2} \xi^z(|u_n|^2+|\nabla u_n|^2).
\end{multline*}
We may then conclude
\begin{multline*}
\frac12\partial_t\int_{\R^2}\xi^z(|u_n|^2+|\nabla u_n|^2)+\frac14\partial_t\int_{\R^2} a_n\xi^z(1-|u_n|^2)^2+\frac\alpha2\int_{\R^2}\xi^z|\partial_tu_n|^2\\
\le C\int_{\R^2}\xi^z(|u_n|^2+|\nabla u_n|^2).
\end{multline*}
By the Grönwall inequality, this yields for all $t\ge0$ and $z\in R\Z^2$,
\begin{multline*}
\int_{\R^2}\xi^z(|u_n^t|^2+|\nabla u_n^t|^2)+\frac12\int_{\R^2} a_n\xi^z(1-|u_n^t|^2)^2+\alpha\int_0^t\int_{\R^2}\xi^z|\partial_tu_n|^2\\\le e^{Ct}\Big(\int_{\R^2}\xi^z(|u_n^\circ|^2+|\nabla u_n^\circ|^2)+\frac12\int_{\R^2} a_n\xi^z(1-|u_n^\circ|^2)^2\Big),
\end{multline*}
and hence, using the Sobolev embedding for $H^1_\uloc(\R^2)$ into $\Ld^4_\uloc(\R^2)$ (cf.\@ \eqref{eq:Sobolevexp} below),
\begin{multline*}
\int_{\R^2}\xi^z(|u_n^t|^2+|\nabla u_n^t|^2)+\frac12\int_{\R^2} a_n\xi^z(1-|u_n^t|^2)^2+\alpha\int_0^t\int_{\R^2}\xi^z|\partial_tu_n|^2\\
\le Ce^{Ct}\Big(1+\int_{\R^2}\xi^z(|u_n^\circ|^2+|\nabla u_n^\circ|^2)\Big)^2.
\end{multline*}
The claim~\eqref{eq:approxH1ulocest} then follows from the boundedness of $u_n^\circ$ in $H^1_\uloc(\R^2;\C)$, noting that
\begin{align}\label{eq:redefinuloc}
\|\zeta\|_{\Ld^2_\uloc}^2\simeq\sup_{z\in\R^2}\int_{\R^2} \xi^z|\zeta|^2.
\end{align}

\medskip
\noindent\step2 Global existence in $H^{k+1}_\uloc(\R^2;\C)$.

In this step, given $k\ge0$, we assume $h\in W^{k+1,\infty}(\R^2)$, $F\in W^{k,\infty}(\R^2)^2$, $f\in W^{k,\infty}(\R^2)$, and $u^\circ\in H^{k+1}_\uloc(\R^2;\C)$, and we prove that the global solution $u$ constructed in Step~1 then belongs to $\Ld^\infty_\loc(\R^+;H^{k+1}_\uloc(\R^2;\C))$.
We denote by $C_k\ge1$ any constant that only depends on an upper bound on $k$, $\alpha$, $\alpha^{-1}$, $|\beta|$, $\|(h,\nabla h,F,f)\|_{W^{k,\infty}}$, and $\|u^\circ\|_{H^{k+1}_\uloc}$, and we write $C_{k,t}$ if it additionally depends on an upper bound on~$t$.

We argue again by approximation. We consider the truncations $h_n,a_n,F_n,f_n,u_n^\circ$ defined in Step~1, as well as the solution $u_n$ to the corresponding equation~\eqref{eq:GLrescale-n}. We claim that for all $k\ge0$ and $t\ge0$,
\begin{align}\label{eq:inducbounddernodec}
\|u_n^t\|_{H^{k+1}_\uloc}+\|\partial_tu_n\|_{\Ld^2_tH^{k}_\uloc}\le C_{k,t}.
\end{align}
The conclusion then follows by passing to the limit $n\uparrow\infty$. This result is proved by induction on $k$. As for $k=0$ the result already follows from Step~1, we assume that $\|u_n^t\|_{H^{k}_\uloc}\le C_{k,t}$ holds for some $k\ge1$, and we deduce that~\eqref{eq:inducbounddernodec} also holds for this $k$.
Integrating by parts, we find
\begin{multline*}
\frac12\partial_t\int_{\R^2} \xi^z|\nabla^{k+1}u_n|^2=\int_{\R^2}\xi^z\langle\nabla^{k+1}u_n,\nabla^{k+1}\partial_tu_n\rangle\\
\le C\int_{\R^2}\xi^z|\nabla^{k+1} u_n||\nabla^{k}\partial_tu_n|-\int_{\R^2}\xi^z\langle\nabla^{k}\triangle u_n,\nabla^{k}\partial_tu_n\rangle,
\end{multline*}
hence, inserting equation~\eqref{eq:GLrescale-n} in the first right-hand side term and developing the terms,
\begin{eqnarray*} 
\lefteqn{\frac12\partial_t\int_{\R^2} \xi^z|\nabla^{k+1}u_n|^2}\\
&\le& -\alpha\int_{\R^2}\xi^z|\nabla^{k}\partial_tu_n|^2+C\int_{\R^2}\xi^z|\nabla^{k+1} u_n||\nabla^{k}\partial_tu_n|\\
&&\hspace{1cm}+\int_{\R^2}\xi^z\big\langle\nabla^{k}\big(a_nu_n(1-|u_n|^2)+\nabla h_n\cdot\nabla u_n+iF_n^\bot\cdot\nabla u_n+f_nu_n\big),\nabla^{k}\partial_tu_n\big\rangle\\
&\le& -\alpha\int_{\R^2}\xi^z|\nabla^{k}\partial_tu_n|^2+\,C\int_{\R^2}\xi^z|u_n|^2|\nabla^ku_n||\nabla^{k}\partial_tu_n|\\
&&\hspace{1cm}+C_k\sum_{j=0}^{k+1}\int_{\R^2}\xi^z|\nabla^{j}u_n||\nabla^{k}\partial_tu_n|+C_k\sum_{j=0}^{k-1}\int_{\R^2}\xi^z|\nabla^ju_n|^3|\nabla^{k}\partial_tu_n|\\
&\le& -\frac\alpha2\int_{\R^2}\xi^z|\nabla^{k}\partial_tu_n|^2+C\int_{\R^2}\xi^z|u_n|^4|\nabla^ku_n|^2\\
&&\hspace{1cm}+C_k\sum_{j=0}^{k+1}\int_{\R^2}\xi^z|\nabla^{j}u_n|^2+C_k\sum_{j=0}^{k-1}\int_{\R^2}\xi^z|\nabla^ju_n|^6.
\end{eqnarray*}
Note that the Sobolev embedding in the balls $B_2(x)$ yields
\begingroup\allowdisplaybreaks
\begin{eqnarray}\label{eq:Sobolevexp}
\int_{\R^2} \xi^z|\nabla^ju_n|^6&\lesssim&\sum_{x\in\Z^2}\xi^z(x)\int_{B_2(x)}|\nabla^ju_n|^6\nonumber\\
&\lesssim&\sum_{x\in\Z^2}\xi^z(x)\Big(\int_{B_2(x)}(|\nabla^ju_n|^2+|\nabla^{j+1}u_n|^2)\Big)^3\nonumber\\
&\lesssim&\Big(\sum_{x\in\Z^2}\xi^z(x)\int_{B_2(x)}(|\nabla^ju_n|^2+|\nabla^{j+1}u_n|^2)\Big)^3\nonumber\\
&\lesssim&\Big(\int_{\R^2}\xi^z(|\nabla^ju_n|^2+|\nabla^{j+1}u_n|^2)\Big)^3,
\end{eqnarray}
\endgroup
and similarly,
\begin{multline*}
\int_{\R^2}\xi^z|u_n|^4|\nabla^ku_n|^2\le\Big(\int_{\R^2}\xi^z|u_n|^8\Big)^{1/2}\Big(\int_{\R^2}\xi^z|\nabla^ku_n|^4\Big)^{1/2}\\
\lesssim\Big(\int_{\R^2}\xi^z|\nabla u_n|^2\Big)^{2}\Big(\int_{\R^2}\xi^z(|\nabla^ku_n|^2+|\nabla^{k+1}u_n|^2)\Big).
\end{multline*}
Inserting these estimates in the above, and using~\eqref{eq:redefinuloc}, we obtain
\begin{eqnarray*} 
\lefteqn{\partial_t\int_{\R^2} \xi^z|\nabla^{k+1}u_n|^2+\alpha\int_{\R^2}\xi^z|\nabla^{k}\partial_tu_n|^2}\\
&\le& C_k\sum_{j=0}^{k}\Big(1+\int_{\R^2}\xi^z|\nabla^ju_n|^2\Big)^3+C_k\Big(1+\int_{\R^2}\xi^z|\nabla u_n|^2\Big)^{2}\int_{\R^2}\xi^z|\nabla^{k+1}u_n|^2\\
&\le& C_k\big(1+\|u_n\|_{H^k_\uloc}^6)+C_k\big(1+\|u_n\|_{H^1_\uloc}^4)\int_{\R^2} \xi^z|\nabla^{k+1}u_n|^2.
\end{eqnarray*}
By the induction hypothesis, we deduce for all $t\ge0$,
\begin{align*} 
\partial_t\int_{\R^2} \xi^z|\nabla^{k+1}u_n^t|^2+\alpha\int_{\R^2}\xi^z|\nabla^{k}\partial_tu_n^t|^2&\le C_{k,t}+C_{k,t}\int_{\R^2}\xi^z|\nabla^{k+1}u_n^t|^2,
\end{align*}
and the result~\eqref{eq:inducbounddernodec} follows from the Grönwall inequality.

\medskip
\noindent\step3 Uniqueness.\nopagebreak

In this step, we assume $h\in W^{1,\infty}(\R^2)$, $F\in \Ld^{\infty}(\R^2)^2$, and $f\in \Ld^\infty(\R^2)$, and we prove that there exists at most one global solution $u\in\Ld^\infty_\loc(\R^+;H^{1}_\uloc(\R^2;\C))$ of~\eqref{eq:GLrescale-2} in $\R^+\times\R^2$ with given initial data $u^\circ$.
We denote by $C\ge1$ any constant that only depends on an upper bound on $\alpha$, $\alpha^{-1}$, $|\beta|$, and $\|(h,\nabla h,F,f)\|_{\Ld^\infty}$.

Let $u_1,u_2\in\Ld^\infty_\loc(\R^+;H^{1}_\uloc(\R^2;\C))$ denote two solutions as above. We set for simplicity $(\alpha+i\beta)^{-1}=\alpha'+i\beta'$, $\alpha'>0$. Using equation~\eqref{eq:GLrescale-2} and integrating by parts, we find
\begingroup\allowdisplaybreaks
\begin{eqnarray}\label{eq:uniquenodec}
\lefteqn{\frac12\partial_t\int_{\R^2}\xi^z|u_1-u_2|^2}\nonumber\\
&\le&-\alpha'\int_{\R^2}\xi^z|\nabla(u_1-u_2)|^2+C\int_{\R^2}\xi^z|u_1-u_2||\nabla(u_1-u_2)|+C\int_{\R^2}\xi^z|u_1-u_2|^2\nonumber\\
&&\hspace{2cm}+\int_{\R^2} a\xi^z\big\langle u_1-u_2,(\alpha'+i\beta')\big(u_1(1-|u_1|^2)-u_2(1-|u_2|^2)\big)\big\rangle\nonumber\\
&\le&-\frac{\alpha'}2\int_{\R^2}\xi^z|\nabla(u_1-u_2)|^2+C\int_{\R^2} \xi^z|u_1-u_2|^2(1+|u_1|+|u_2|)^2.
\end{eqnarray}
\endgroup
It remains to estimate the last integral. For that purpose, we decompose
\begin{align*}
\int_{\R^2} \xi^z|u_1-u_2|^2(|u_1|+|u_2|)^2&\lesssim \sum_{x\in \Z^2} \xi^z(x)\int_{B_2(x)}|u_1-u_2|^2(|u_1|+|u_2|)^2\\
&\lesssim \sum_{x\in \Z^2} \xi^z(x)\Big(\int_{B_2(x)}|u_1-u_2|^4\Big)^{1/2}\Big(\int_{B_2(x)}(|u_1|+|u_2|)^4\Big)^{1/2},
\end{align*}
hence, using the Sobolev embedding for $H^{3/4}(B_2(x))$ (and $H^{1}(B_2(x))$) into $\Ld^4(B_2(x))$,
\begin{align*}
\int_{\R^2} \xi^z|u_1-u_2|^2(|u_1|+|u_2|)^2&\,\lesssim\, \|(u_1,u_2)\|_{H^1_\uloc}^2\sum_{x\in \Z^2} \xi^z(x)\|u_1-u_2\|_{H^{3/4}(B_2(x))}^2.
\end{align*}
Using interpolation and Young's inequality then yields for all $K\ge1$,
\begin{eqnarray*}
\lefteqn{\int_{\R^2} \xi^z|u_1-u_2|^2(|u_1|+|u_2|)^2}\\
&\lesssim& \|(u_1,u_2)\|_{H^1_\uloc}^2\sum_{x\in \Z^2} \xi^z(x)\|u_1-u_2\|_{H^{1}(B_2(x))}^{3/2}\|u_1-u_2\|_{\Ld^2(B_2(x))}^{1/2}\\
&\lesssim& K^{-1}\int_{\R^2} \xi^z|\nabla(u_1-u_2)|^2+K^{3}(1+\|(u_1,u_2)\|_{H^1_\uloc}^8)\int_{\R^2}\xi^z|u_1-u_2|^2.
\end{eqnarray*}
Inserting this into~\eqref{eq:uniquenodec} with $K\simeq1$ large enough, we find
\begin{align*}
\frac12\partial_t\int_{\R^2}\xi^z|u_1-u_2|^2&\le C\big(1+\|(u_1,u_2)\|_{H^1_\uloc}^8\big)\int_{\R^2}\xi^z|u_1-u_2|^2,
\end{align*}
and the conclusion $u_1=u_2$ follows from the Grönwall inequality.
\end{proof}

\bigskip
\bibliographystyle{plain}
\bibliography{biblio}
\vskip .5cm

\begin{samepage}
\noindent
\textsc{Mitia Duerinckx}\\
Universit\'e Libre de Bruxelles, boulevard du Triomphe, 1050 Brussels, Belgium,\\
\& Sorbonne Universit\'e, CNRS, UMR 7598, Laboratoire Jacques-Louis Lions, 4 place Jussieu, 75005 Paris, France.\\
{\tt mduerinc@ulb.ac.be}
\vspace{.2cm}

\noindent
\textsc{Sylvia Serfaty}\\
Courant Institute, New York University, 251 Mercer street, New York, NY 10012, USA,\\
\& Sorbonne Universit\'e, CNRS, UMR 7598, Laboratoire Jacques-Louis Lions, 4 place Jussieu, 75005 Paris, France,\\
\& Institut Universitaire de France.\\
{\tt serfaty@cims.nyu.edu}
\end{samepage}
\end{document}